%% file: Random_Hermite.tex
\documentclass[a4paper,11pt,english]{smfart}
\usepackage[applemac]{inputenc}
\usepackage[english]{babel}
\usepackage{amsfonts}
\usepackage{amsmath} \usepackage{amsthm}
\usepackage{hyperref} 
\usepackage{url}
\usepackage{latexsym}
\usepackage{array}
\usepackage{amssymb}
\usepackage{enumerate}
\usepackage{mathrsfs}
\usepackage{smfthm}
\usepackage{graphicx}
\usepackage{color}
\usepackage{stmaryrd}

\def\pun{\mathbf{1}}
\def\ep{\varepsilon}
\newcommand{\ligne}{\vspace{1\baselineskip}}
\newcommand{\ph}{\phantomsection}
\newcommand{\cal}{\mathcal}
\newcommand{\di}{\displaystyle}

\newcommand{\R}{\mathbb  R}
\newcommand{\C}{\mathbb  C}
\newcommand{\N}{\mathbb  N}

\renewcommand{\P}{\mathbb{P}}

\newcommand{\eps}{\varepsilon}
\renewcommand{\epsilon}{\varepsilon}
\newcommand{\e}{  \text{e}   }

\newcommand{\Z}{  \mathbb{Z}   }

\newcommand{\E}{\mathbb{E} }

\renewcommand{\H}{  \mathcal{H}   }

\newcommand{\T}{  \mathbb{T} }

\newcommand{\dis}{\displaystyle}

\newcommand\Sum{ \displaystyle\sum}
\newcommand\Int{ \displaystyle\int}

\newcommand{\om}{  \omega   }

\renewcommand{\a}{  \alpha   }

\newcommand{\lessim}{  \lesssim  }
\renewcommand{\phi}{  \varphi  }
\renewcommand{\L}{  \mathcal{L}   }

\numberwithin{equation}{section}
\theoremstyle{plain}

\def\beq{\begin{equation}}   \def\eeq{\end{equation}}
\def\bea{\begin{eqnarray}}  \def\eea{\end{eqnarray}}

\renewcommand{\theequation}{\thesection.\arabic{equation}}
\newcounter{hran} \renewcommand{\thehran}{\thesection.\arabic{hran}}

\def\bmini{\setcounter{hran}{\value{equation}}
    \refstepcounter{hran}\setcounter{equation}{0}
    \renewcommand{\theequation}{\thehran\alph{equation}}\begin{eqnarray}}

\def\bminiG#1{\setcounter{hran}{\value{equation}}
\refstepcounter{hran}\setcounter{equation}{-1}
\renewcommand{\theequation}{\thehran\alph{equation}}
\refstepcounter{equation}\label{#1}\begin{eqnarray}}

\urladdr{}
\author{Rafik Imekraz}
\address{Institut de Math\'ematiques de Bordeaux, UMR 5251 du CNRS, Universit\'e de Bordeaux 1, 351, cours de la Lib\'eration F33405 Talence Cedex, France     }
\email{rafik.imekraz@math.u-bordeaux1.fr}
\author{ Didier Robert}
\address{Laboratoire de Math\'ematiques J. Leray, UMR  6629 du CNRS, Universit\'e de Nantes, 
2, rue de la Houssini\`ere,
44322 Nantes Cedex 03, France}
\email{didier.robert@univ-nantes.fr}
\author{ Laurent Thomann }
\address{Laboratoire de Math\'ematiques J. Leray, UMR  6629 du CNRS, Universit\'e de Nantes, 
2, rue de la Houssini\`ere,
44322 Nantes Cedex 03, France}
\email{laurent.thomann@univ-nantes.fr}

 \title[On random Hermite series]{On random Hermite series}

\begin{document}
\input{fig4tex}
\frontmatter

 \begin{abstract}
We  study integrability and  continuity properties of random series of Hermite functions. We get optimal results which are   analogues to classical  results  concerning Fourier series, like the Paley-Zygmund or the Salem-Zygmund theorems.   We also consider the case of series of radial Hermite functions, which are not so well-behaved. In this context, we prove some\;$L^p$ bounds of radial Hermite functions, which are optimal when $p$ is large.
  \end{abstract}
\thanks{D. R.  was partly supported  by the  grant  ``NOSEVOL''   ANR-2011-BS01019 01.\\L.T. was partly supported  by the  grant  ``HANDDY'' ANR-10-JCJC 0109 \\
 \indent \quad and by the  grant  ``ANA\'E'' ANR-13-BS01-0010-03.}
\maketitle
\mainmatter


 \section{Introduction}
 
 In this paper we prove  some optimal integrability and regularity results on the convergence of random Hermite expansions, i.e. on series of eigenfunctions of the harmonic oscillator with random coefficients. \ligne

 Before we enter in the details, let us recall an old result on the 1-D  torus  $\T=\R/2\pi\Z$. Let $\di{u(x) = \sum_{n\in\Z}c_n{\rm e}^{inx}}$  and define the Sobolev  space $H^{s}(\T)$ by the norm  
 $\di{\Vert u\Vert^2_{H^s(\T)} = \sum_{n\in\Z}(1+\vert n\vert)^{2s}\vert c_n\vert^2}$.
 By the usual Sobolev embeddings, if $u\in H^{1/2-1/p}(\T)$ with $p\geq 2$ then
 $u\in L^p(\T)$, but in general $u\notin \mathcal{C}(\T)$. Paley and Zygmund (1930) have improved this result allowing  random  coefficients.
 \begin{theo}[Paley-Zygmund]\ph \label{pazy}
 Let  $u^\omega(x) = \di{\sum_{n\in\Z}\ep_n(\omega)c_n{\rm e}^{inx}}$  where $(\ep_n)_{n\in \Z}$ is a sequence of independent Rademacher random variables.
 If $u\in L^2(\T)$ then for all $2\leq p<+\infty$, a.s  $u^\omega\in L^p(\T)$. 
 
 Moreover if for some $\alpha >1$, $\di{\sum_{n\in\Z}\ln^\alpha(1+\vert n\vert)\vert c_n\vert^2 <+\infty}$ then a.s $u^\omega\in{\cal C}(\T)$.
\end{theo}
Many  other results concerning random   trigonometric  series  were  obtained   by Paley and Zygmund, as it is   detailed
in the   book of J-P. Kahane  \cite{JPK}. The study has been extended to  random Fourier series on  Lie groups  (see  Marcus-Pisier \cite{MP}) and 
to Riemannian compact manifolds  for orthonormal basis of eigenfunctions of the Laplace-Beltrami operator (see Tzvetkov \cite{Tzvetkov4} and references therein). 

On the torus $\T^d=\R^d/(2\pi\Z)^d$,  there is a natural choice of the basis for the expansion, namely the $(\e^{in\cdot x})_{n\in \Z}$. In our context (or more generally, if one study expansions on eigenfunctions of the Laplacian on a compact manifold) it is not clear which basis to choose, and the convergence properties of the random series $u^\om(x)=\sum c_n X_n(\om)\phi_n(x)$ might depend on the choice of the basis $(\phi_n)_{n\geq 0}$. For instance, an analogous result to Theorem~\ref{pazy} has been obtained by   Tzvetkov \cite[Theorem 5]{Tzvetkov4} in compact manifolds with a condition depending on the $L^{\infty}$ bound of the $\phi_n$.\ligne

Here we show  that by adding a squeezing condition (see condition~\eqref{condi3} below), we can use the intrinsic estimates of the spectral function, and obtain a convergence  condition on the $(c_n)$ which does not depend on the choice of the basis of Hermite functions. The idea to take profit of the bounds of the spectral function and of the Weyl law comes from   \cite{ShZe, bule} and has been fruitful in different contexts (see \cite{PRT1,PRT2,PRT3}), where results have been obtained for a large class of probability laws.  Here we extend this approach by working  in a space\;$\mathcal{Z}^s_{\phi}$ (instead  of using  condition~\eqref{condi3})   and  which also enables to exploit the estimates of the spectral function and which is compatible with the L\'evy contraction principle of random series. We refer to the next paragraph for more details.\ligne

Let us now briefly describe our main contribution in this paper: \vspace{4pt}

We first study integrability properties of the random series $u^{\om}$. We then detail the case of series of radial Hermite functions, for which the situation is different than in the general case.

In a second time, we prove regularity results of the random series. We prove a Salem-Zygmund theorem which describes the behaviour of partial sums. We are then able to obtain an analogous result to  Theorem~\ref{pazy} in our context, and we show that the $\ln$ factor is optimal. Finally, we state in Theorem~\ref{hold} some more precise regularity results. Notice that due to dispersive effects of the harmonic oscillator on $\R^d$, the randomisation yields  better estimates than on the torus.\ligne

In Proposition~\ref{estrad} we state some $L^p$ bounds of radial Hermite functions and which are optimal at least for $p\geq2$ large enough. Even if the proof is elementary, using the well-known asymptotic estimates of Laguerre functions, we did not find the result in the literature. Therefore, we have written down the details, since the estimates we obtain are better than the bounds of general Hermite functions.\ligne

Finally, we point out that the previous results have analogues for random series of eigenfunctions of the Laplacian on a Riemannian compact manifold or for the Laplacian on $\R^{d}$ with a confining potential. These results can be obtained with the same strategy by using the corresponding bounds of the spectral function.

\subsection{Functional analysis}\label{sect11}
\subsubsection{Some elements on the harmonic oscillator}
We consider the multidimensional harmonic oscillator $H:=-\Delta+ |x|^2$ on $L^2(\R^d)$ with $d\geq 1$.
The spectrum of $H$ is $d+2\N$ and we consider the sequence of eigenvalues $(\lambda_n)_{n\geq 0}$ by counting multiplicities: 
$$      d=\lambda_0 <\lambda_1\leq \lambda_2 \leq \lambda_3\leq \dots                $$
Fix any  orthonormal basis  $(\varphi_n)_{n\geq 0}$ of  normalized eigenfunctions for the harmonic  oscillator $H$ such that $H\varphi_n=\lambda_n \varphi_n$.
For $j\geq 1$ denote by 
     \begin{equation*} 
     I(j)=\big\{n\in \N,\;2j\leq \lambda_{n}< 2(j+1)\big\}.
     \end{equation*}
     Observe that for all $j \geq d/2$, $I(j)\neq \emptyset$ and that $\#I(j)\sim C_{d}j^{d-1}$ when $j \longrightarrow +\infty$, and therefore $\lambda_{n}\sim c_d n^{1/d}$. 
Though $(\varphi_n)_{n\geq 0}$ is arbitrary, the vector space spanned by $\{\varphi_n, n\in I(j)\}$ is independent of the choice of the Hilbert basis.

Now, we recall what are the natural Sobolev spaces for $H$:
\begin{equation*}
\forall s\geq 0 \quad \forall p\in [1,+\infty)\cup \{+\infty\} \quad \mathcal{W}^{s,p}(\R^d):=\{u\in \mathcal{S}' (\R^d), \quad H^{s/2} u \in L^p(\R^d)\}.
\end{equation*} 
Therefore, we define 
\begin{equation}\label{wdp}
\|u\|_{\mathcal{W}^{s,p}(\R^d)}:= \|H^{\frac{s}{2}}u\|_{L^p(\R^d)}.
\end{equation}
It turns out (see \cite[Lemma~2.4]{YajimaZhang2}) that a functional characterisation of $\mathcal{W}^{s,p}(\R^d)$ for $1\leq p<+\infty$ and $s\geq 0$ is given by 
\begin{equation*}
u\in \mathcal{W}^{s,p}(\R^d) \quad \Leftrightarrow \quad \|(I-\Delta)^{\frac{s}{2}} u\|_{L^p(\R^d)} + \|\langle x\rangle^s u\|_{L^p(\R^d)}<+\infty.
\end{equation*}
In the Hilbertian framework, we have 
\begin{equation*}    \mathcal{H}^s(\R^d):=
\mathcal{W}^{s,2}(\R^d)=\{u\in H^s(\R^d), \quad  \langle x \rangle^s u \in L^2(\R^d) \}             \end{equation*}
where $H^s(\R^d)=\mbox{Dom}((I-\Delta)^{s/2})$ is the classical Sobolev space.
Thus, up to an equivalence of norm, one can define 
\begin{equation}\label{defi-norm}   \|u\|_{\mathcal{H}^s(\R^d)} = \|H^{s/2} u \|_{L^2(\R^d)} = \|u\|_{H^s(\R^d)}+\|\langle x \rangle^s u \|_{L^2(\R^d)} .     \end{equation}
Consequently, one can check that $\H^s(\R^d)$ is an algebra if $s>\frac{d}{2}$ and is included in $L^\infty(\R^d)$.\ligne 

We will need   the $L^{\infty}$ estimate of the spectral function given by Thangavelu/Karadzhov (see \cite[Lemma 3.5]{PRT1}) which reads 
\begin{equation}\label{kt}
\|\Pi_j\|_{L^2(\R^d)\rightarrow L^\infty(\R^d)}^2=  \sup_{x\in \R^d} \sum_{n\in I(j)} | \phi_{n}\big(x\big)\big|^{2}\leq C j^{\gamma{(d)}},
\end{equation}
with $\gamma(1)=-1/6$ and $\gamma(d)=d/2-1$ for $d\geq 2$ and 
where $\Pi_j$ is the spectral projector of $H$ on the eigenspace associated to the unique eigenvalue which belongs to $I(j)$.
It is classical that the function defined in \eqref{kt} does not depend on the choice of the $(\phi_n)_{n\geq 0}$. For $d=1$, \eqref{kt} comes from the simplicity of the spectrum of $H$ and the classical estimate of the normalized Hermite functions: $\|\phi_j\|_{L^\infty(\R)}\lesssim j^{-\frac{1}{12}}$.

In the sequel we will also need the notation $\beta(d)=d-1-\gamma(d)$  as follows 
\begin{equation}\label{kt2}
\begin{array}{|c|c|c|c|} \hline
 & \gamma(d) & \beta(d)  \\ \hline 
 d=1 & -\frac{1}{6} & \frac{1}{6}  \\[2mm] 
 d\geq 2 & \frac{d}{2}-1 & \frac{d}{2} \\ \hline 
\end{array}
\end{equation}
~

\subsubsection{The space $\mathcal{Z}_{\phi}^{s}(\R^{d})$}
   Given a Hilbertian basis of Hermite functions $(\phi_{n})_{n\geq 0}$ and  $s\in \R$,  any $u\in \H^{s}(\R^{d})$ can be written in a unique fashion 
     \begin{equation*} 
     u= \sum_{n\geq 0}c_{n}\phi_{n}, \quad  \sum_{n \geq 0} \lambda_n^{s} |c_n|^2 <+\infty. 
     \end{equation*}
     We define the space $\mathcal{Z}_{\phi}^{s}(\R^{d})$ by the norm 
     \begin{equation*}
     \|u\|^{2}_{\mathcal{Z}_\phi^{s}}=\sum_{j\geq 1} j^{s+d-1}   \max\limits_{n\in I(j)} |c_n|^2,
     \end{equation*}
     and we stress that this space depends on the choice of the basis $(\phi_{n})$.
     It is clear that we have the strict embeddings
     \begin{equation*}
   \H^{s+d-1}(\R^{d})\subset   \mathcal{Z}_{\phi}^{s}(\R^{d})\subset \H^{s}(\R^{d}).
     \end{equation*}
     
  In the works \cite{PRT1,PRT2,PRT3}, the following assumption on the coefficients of  $u\in \H^{s}(\R^{d})$ was made   
    \begin{equation}\label{condi3} 
  |c_{k}|^{2}\leq \frac{C}{\# I(j)}\sum_{n\in I(j)}|c_{n}|^{2},\quad \forall k\in I(j),\quad \forall  j\geq 1.
  \end{equation}
Let us explain why the condition $u\in \mathcal{Z}_{\phi}^{s}(\R^{d})$ is more natural.
Firstly, observe that if the coefficients of  $u\in \H^{s}(\R^{d})$  satisfy \eqref{condi3}   then $u\in \mathcal{Z}_{\phi}^{s}(\R^{d})$.
Secondly, consider two functions $u,v\in \H^{s}(\R^{d})$
$$  u = \sum_{n=0}^{+\infty} c_{n}\phi_{n}, \quad v = \sum_{n=0}^{+\infty} \gamma_n c_{n}\phi_{n},$$ 
where $(\gamma_n)$ is a real bounded sequence.
The contraction principle for the random series (see Theorem~\ref{theocontrac}) states roughly that if one can prove an almost sure convergence for the random series coming from $u$ (see below \eqref{defi-rand}), then the same is true for $v$.
But it is easy to see that condition \eqref{condi3} is not stable by multiplication by bounded sequences whereas $u\in \mathcal{Z}_\phi^s(\R^d)$ is the most general condition which is implied by \eqref{condi3} and stable by multiplication by bounded sequences.

Sometimes, we also need  the stronger condition  
    \begin{equation}\label{condi4} 
 \frac{C_1}{\# I(j)}\sum_{n\in I(j)}|c_{n}|^{2}\leq   |c_{k}|^{2}\leq \frac{C_2}{\# I(j)}\sum_{n\in I(j)}|c:_{n}|^{2},\quad \forall k\in I(j),\quad \forall  j\geq 1.
  \end{equation}  
  
  \subsection{Probabilistic setting} Consider a   probability space
$(\Omega, {\cal F}, {\P})$  and  let $(X_n)_{n\geq 0}$ be  independent and identically distributed random variables which are not constant almost surely.
All random variables are real valued.
In all the paper (except in the annex \ref{annex-proba}), we will make two different assumptions depending on whether we study integrability or regularity results: 
\begin{equation}\label{mom-bor}
\E[X_1]=0 \quad \mbox{and} \quad \forall k\geq 1 \quad \E[|X_1|^k]<+\infty.
\end{equation}
\begin{equation}\label{defi-subno}
\exists \sigma >0, \quad \forall r>0, \quad \E[e^{rX_1}] \leq e^{\frac{1}{2} \sigma^2 r^2 }.
\end{equation}
One checks that \eqref{defi-subno} implies \eqref{mom-bor}.
The usual laws we have in mind fulfill \eqref{defi-subno}: 
the real Gaussian law $\mathcal{N}_{\R}(0,1)$ or the Rademacher law (in that case, we will write $X_n=\ep_n$). More generally, any centered and bounded r.v. satisfies \eqref{defi-subno}. 


  
We explain now the way we introduce randomness in Sobolev spaces.
Let $(c_{n})_{n\geq 0}$ be such that $\sum_{n\geq 0} \lambda_{n}^{s}|c_{n}|^{2}<+\infty$. Then we can  define a random variable $u^{\om}$ by 
 \begin{equation}\label{defi-rand}
  u^\omega = \sum_{n=0}^{+\infty}  X_n(\om)c_{n}\phi_{n}.
 \end{equation}
It is clear that we have  
 \begin{equation*}
\E\left[ \| u^{\om}\|^{2}_{\H^{s}(\R^{d})} \right]= \E \left[ \sum_{n=0}^{+\infty} \lambda_{n}^{s}|c_{n}|^{2}|X_{n}|^{2} \right] 
 \leq \E[X_1^2]  \sum_{n= 0}^{+\infty}\lambda_n^{s} |c_{n}|^{2}<+\infty,
 \end{equation*}
In other words $\omega \mapsto u^\omega$ belongs to $L^2(\Omega, \mathcal{H}^s(\R^d))$ and almost surely $u^\omega$ belongs to $\mathcal{H}^s(\R^d)$.


  \section{Main  results of the paper}

\subsection{\bf Integrability results for random Hermite series }

We state here convergence results in the $L^p(\R^d)$ scale with $p\in [2,\infty)$.
The following result (used in a slightly weaker form in \cite{Grivaux}) will play a key role. It is a combination of results of Hoffman-Jorgensen, Maurey-Pisier \cite{maurey-pisier76} and the fact that\;$L^{p}(\R^d)$ has finite cotype.

\begin{prop}\ph\label{equi-cv}
Let $p\in [2,+\infty)$ and $(f_n)_{n\geq 0}$ be a sequence of $L^p(\R^d)$.
Assume that the sequence $(X_n)_{n\geq 0}$ fulfills \eqref{mom-bor}, the following statements are equivalent: 
\begin{enumerate}[(i) ]
\item the series $\sum \ep_n f_n$ converges almost surely in $L^p(\R^d)$,
\item the series $\sum X_n f_n$ converges almost surely in $L^p(\R^d)$,
\item the function $\displaystyle\sum_{n\geq 0} |f_n|^2$ belongs to $L^{\frac{p}{2}}(\R^d)$.
 \end{enumerate}
\end{prop}

This proposition is a synthesis of known results on the convergence of random series in Banach spaces. For the reader's convenience, we have gathered the elements of the proof in Section~\ref{annex-proba}.\ligne

Here is our first result involving random Hermite series. Recall the definition \eqref{kt2} of $\gamma$ and $\beta$.

\begin{theo}\ph \label{proplp}
Let $d\geq 1$ and $2\leq p<+\infty$.  We  assume that the r.v.~$(X_n)_{n\geq 0}$ fulfill \eqref{mom-bor} and that $u=\displaystyle\sum_{n\geq 0} c_n \phi_n$ belongs to $\mathcal{Z}_\phi^{-2\beta(d)\left( \frac{1}{2}-\frac{1}{p}\right)}(\R^d)$, i.e. the sequence $(c_n)_{n\geq 0}$ is such that 
\begin{equation} \label{condi5}
  \sum_{j= 1}^{+\infty}j^{\gamma(d)+\frac{2\beta(d)}{p}} \max_{n\in I(j)}|c_{n}|^{2}<+\infty.
\end{equation} 
Then $u^\omega=\displaystyle\sum_{n\geq 0} X_n c_n \phi_n$ converges almost surely in ${L}^{p}(\R^d)$.
\end{theo}

We will see in the proof that the exponent $\gamma{(d)}+\frac{2\beta(d)}{p}$ which appears in \eqref{condi5} is such that
\begin{equation*}
\Big\Vert  \sum_{n\in I(j)} | \phi_{n}(x)|^{2}\Big\Vert_{L^{p/2}(\R^d)}\leq C j^{\gamma{(d)}+\frac{2\beta(d)}{p}},
\end{equation*}

We refer to \cite[Proposition 2.1]{PRT2} where a result similar to Theorem \ref{proplp} was given using the condition\;\eqref{condi3}. \ligne
 
By considering radial functions as in Ayache-Tzvetkov \cite{AT} and in Grivaux \cite{Grivaux}, we introduce now a natural example for which the gain of integrability may not hold in all the spaces $L^p(\R^d)$, in this case condition \eqref{condi3} does not hold true and we may have $u\in \mathcal{Z}^{s}(\R^{d})$ and $u\in \H^{s+d-1}(\R^{d})\backslash \H^{s+d-1+\eps}(\R^{d})$.
   
Let  $d\geq 2$ and $L_{rad}^{2}(\R^{d})$   be the subspace of $L^2(\R^d)$ invariant  by  the action of the rotation group~$SO(d)$.
One can prove that there exists  a Hilbertian basis  $(\psi_{n})_{n\geq 0 }$  of  $L_{rad}^{2}(\R^{d})$ of eigenfunctions of~$H$. 
Indeed, we have $H\psi_{n}=(4n+d) \psi_{n}$, each eigenspace has dimension 1 and $\psi_n$ may be expressed with Laguerre polynomials (see Section \ref{Sect3} for more details). 

\begin{theo}\ph\label{theo24} Let $d\geq 2$, assume that $(X_n)_{n\geq 0}$ verifies \eqref{mom-bor} and that $u_{rad}:=\Sum_{n\geq 0} c_n \psi_n$ belongs to $\bigcap\limits_{\ep>0} \mathcal{H}^{-\ep}(\R^d)$. The random series 
\begin{equation*} 
u_{rad}^{\om}(x)=\sum_{n=0}^{+\infty}  X_{n}(\om)c_{n}\psi_{n}(x).
\end{equation*}
converges almost surely in $L^p(\R^d)$ for any $p\in ]2,\frac{d}{\alpha_\star(c)}[$ and diverges almost surely for any $p>\frac{d}{\alpha_\star(c)}$ where 
\begin{equation*}
 \a_{\star}(c):=\inf\big\{ \a>0\,: \sum_{n=0}^{N}n^{\frac{d}{2}-1}|c_{n}|^{2} =\mathcal{O}(N^\a)\big\}. \end{equation*}
 \end{theo}

Let us give some examples:
    
   \begin{itemize}
   \item[$\bullet$] If $d=1$, then by Theorem \ref{proplp}, the series $u_{rad}^{\om}$ (defined in the obvious way) converges a.s. in\;$L^{p}(\R^{d})$ for all $p<\infty$. 
   \item[$\bullet$] If $(c_{n})$ is such that $\sum_{n\geq 0} n^{\frac{d}{2}-1}|c_{n}|^{2}<+\infty$, then $\frac{d}{\alpha_\star(c)}=+\infty$. Therefore\;$u_{rad}^{\om}$ converges a.s. in\;$L^{p}(\R^{d})$ for all $p<\infty$.   
      \item[$\bullet$]  Assume that $c_{n}\sim n^{-\kappa}$ with $\kappa\geq 1/2$, then $ \a_{\star}(c)=\max \left(\frac{d}{2}-2\kappa,0 \right)$ and then 
       \begin{equation*} 
\frac{d}{\alpha_\star(c)}=\left\{\begin{array}{ll} 
\dis \min\Big(\frac{2d}{d-4\kappa},+\infty\Big) ,\quad &\text{if} \quad \kappa<\frac{d}{4}, \\[6pt]  
\dis +\infty,  &\text{if} \quad \kappa\geq \frac{d}{4}.
\end{array} \right.
\end{equation*}

   \end{itemize} 
     
     An analogous result to Theorem~\ref{theo24}, but with a different numerology, was first obtained in \cite{AT,Grivaux} for the family of the radial eigenfunctions of the Laplacian on the unit disc in $\R^{d}$
where the analogue value of $\frac{d}{\alpha_\star(c)}$ is called the critical convergence exponent of $c$. 
We will follow the main lines of~\cite{Grivaux}, the difference in the proof involves 
the study of $L^p(\R^d)$ bounds of the radial Hermite functions.
     
     \begin{prop}\ph\label{estrad}
    Let $d\geq 2$. Consider the family $(\psi_{n})_{n\geq 0}$ of the $L^{2}$-normalized radial Hermite functions which satisfies $H\psi_n=(4n+d)\psi_n.$ Then 
    \begin{enumerate}[(i)]
    \item Assume that   $\frac{2d}{d-1}<p\leq +\infty$. Then 
      \begin{equation*}
 c_p n^{ \frac{d}{2}(\frac{1}{2}-\frac{1}{p})-\frac{1}{2} }\leq    \|\psi_n\|_{L^p(\R^d)}\leq C_p n^{ \frac{d}{2}(\frac{1}{2}-\frac{1}{p})-\frac{1}{2} }.
    \end{equation*}
    \item Assume that   $p=\frac{2d}{d-1}$. Then 
      \begin{equation*}
     \|\psi_n\|_{L^p(\R^d)}\leq C_p n^{-\frac{1}{4}} \ln^{\frac{1}{p}}(n).
    \end{equation*}
    \item Assume that   $2\leq p<\frac{2d}{d-1}$. Then 
      \begin{equation*}
     \|\psi_n\|_{L^p(\R^d)}\leq C_p n^{- \frac{d}{2}(\frac{1}{2}-\frac{1}{p})}.
    \end{equation*}
    
    \end{enumerate}  
     \end{prop}
     
The proof uses asymptotic estimates of Laguerre functions proved by Erdelyi (such a method has been used in \cite[Lemma 3.1]{deng2012} for $d=2$ and is indicated in \cite[Chapter 1]{Thanga}).

We do not know if the estimates stated in $(ii)$ and $(iii)$ are optimal or not. To get the lower bound in $(i)$  we show that there exist $\eps,c>0$ such that for all $n\geq 1$ and all $|x|\leq \frac{\eps}{\sqrt{n}}$, $|\psi_{n}(x)|\geq c n^{\frac{d}{4}-\frac{1}{2}}$, and the result follows by integrating this estimate.  
 
    
  In the figures below, we represent the estimates of Proposition~\ref{estrad}. The dashed lines represent the bounds of  Koch-Tataru \cite[Corollary 3.2]{KOTA} obtained for general Hermite functions as defined in Section\;\ref{sect11}.   We see that   in the range  $2<p<\frac{d-2}{2d}$ the radial functions enjoy   better bounds than in the general case, but not  in the regime $\frac{d-2}{2d}<p\leq +\infty$.

 
 \ligne
 
  \figinit{cm}
\figpt 1:(0,0)\figpt 2:(0,1)\figpt 3:(10,0)\figpt 4:(0,4)
\figpt 5:(3,0) \figpt 6:(6,0) \figpt 7:(4.5,0) \figpt 8:(4.5,-2.66)  
 \figpt 10:(9,0)   \figpt 11:(9,0)   \figpt 12:(0,-3,5) \figpt 13:(0,-2,66) \figpt 14:(6,-.8) \figpt 15:(0,-.8) 
\figvectDBezier 9:1,0.25[1,2,3,4]
\figdrawbegin{}
  \figdrawarrow[12,2]\figdrawarrow[1,3]
\figset (width=1.5)
\figdrawline[1,8,10]
\figset (width=1)
\figset (dash=5, color=\Grayrgb)\figdrawline[7,8,13] \figdrawline[6,14,15] \figset (dash=default,  color=default)
\figset  (dash=2)\figdrawline[1,14,10]\figset (dash=default)
\figset (width=1)\figset (width=default)
\figdrawend
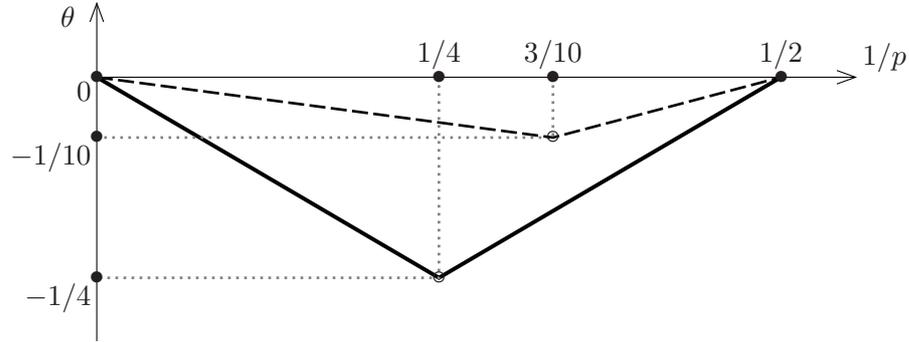
\figvisu{\figBoxA}{$L^p$ estimates of radial Hermite functions: the case $d=2$}{
\figwritene 3:$1/p$(0.1)
\figwritesw 2:$\theta$\;\;(0.1)
\figset write(mark=$\bullet$)
\figwritesw 1:$0$(0.1)
\figwriten 6:$3/10$(0.1)
\figwriten 7:$1/4$(0.1)
\figwriten 10:$1/2$(0.1)
\figwritesw 13:$-1/4$(0.1)
\figwritesw 15:$-1/10$(0.1)
\figset write(mark=$\circ$)
\figwritesw 14:$$(0.1)
\figwritesw  8:$$(0.1)
}
\centerline{\box\figBoxA}
\ligne
\ligne

\figinit{cm}
\figpt 1:(0,0)\figpt 2:(0,5)\figpt 3:(10,0)\figpt 4:(0,4)
\figpt 5:(3,0) \figpt 6:(4,0) \figpt 7:(5,0) \figpt 8:(5,-2.66)  
 \figpt 10:(9,0)   \figpt 11:(9,0)   \figpt 12:(0,-3,5) \figpt 13:(0,-2,66) \figpt 14:(4,-.8) \figpt 15:(0,-.8) 
\figvectDBezier 9:1,0.25[1,2,3,4]
\figdrawbegin{}
  \figdrawarrow[12,2]\figdrawarrow[1,3]
\figset (width=1.5)
\figdrawline[4,5,8,10]
\figset (width=1)
\figset (dash=5, color=\Grayrgb)\figdrawline[7,8,13] \figdrawline[6,14,15] \figset (dash=default, color=default)
\figset  (dash=2)\figdrawline[5,14,10]\figset (dash=default)
\figset (width=1)\figset (width=default)
\figdrawend
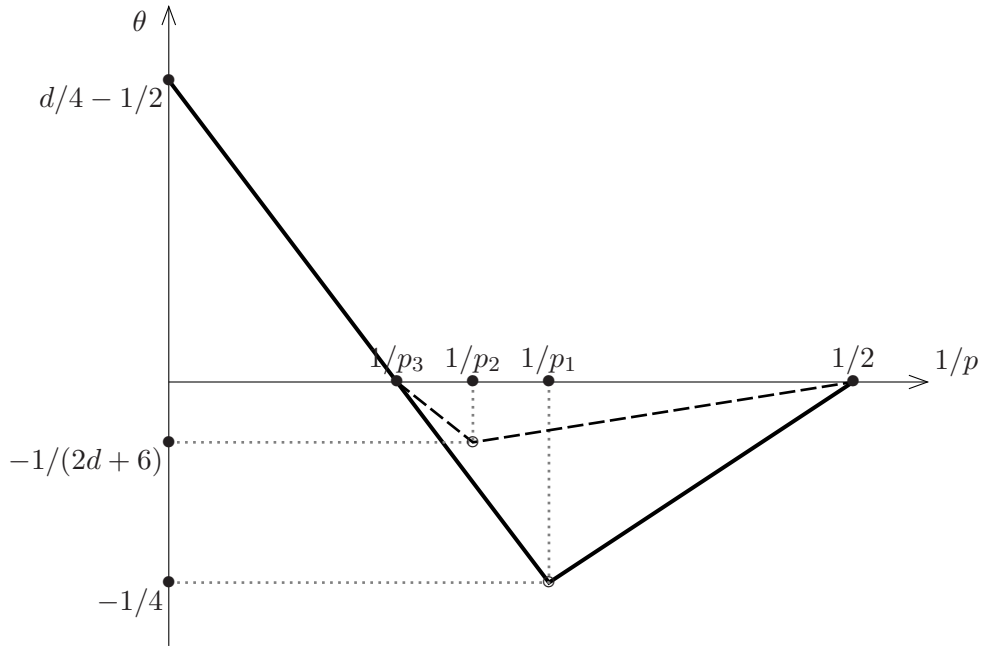
\figvisu{\figBoxA}{$L^p$ estimates of radial Hermite functions: the case $d\geq 3$}{
\figwritene 3:$1/p$(0.1)
\figwritesw 2:$\theta$\;\;(0.1)
\figset write(mark=$\bullet$)
\figwriten 5:$1/p_{3}$(0.1)
\figwriten 6:$1/p_{2}$(0.1)
\figwriten 7:$1/p_{1}$(0.1)
\figwriten 10:$1/2$(0.1)
\figwritesw 4:$d/4-1/2$(0.1)
\figwritesw 13:$-1/4$(0.1)
\figwritesw 15:$-1/{(}2d+6{)} $(0.1)
\figset write(mark=$\circ$)
\figwritesw  8:$$(0.1)
\figwritesw 14:$$(0.1)
}
\centerline{\box\figBoxA}
\ligne
 
 In the second figure we have set
\begin{equation*}
2<p_1:=\frac{2d}{d-1}\leq p_2:=\frac{2(d+3)}{d+1}\leq p_3:=\frac{2d}{d-2}.
\end{equation*} 

    \subsection{\bf Continuity results for random Hermite series}

We are concerned with the random behavior of the partial sums of \eqref{defi-rand} in the space $L^\infty(\R^d)$.    
 Let us define for any $\lambda \geq d$ 
\begin{equation}\label{rand-part}
 u_\lambda^\omega(x)=\sum_{\lambda_n\leq \lambda} c_n X_n(\omega) \varphi_n(x) .
\end{equation}

There is not an equivalent of Proposition~\ref{equi-cv} for the space $L^\infty(\R^d)$
 (the reason is that $L^\infty(\R^d)$ is not a Banach space with finite cotype, see Annex \ref{annex-proba}).
Hence, we will use other methods to get probabilistic results, like the following one which is in the spirit of the Salem-Zygmund inequality (see \cite[Theorem, page 55]{JPK}).

\begin{theo}\ph\label{salem-zy}
Assume  that    $(X_n)_{n\geq 0}$  is an i.i.d. family of   r.v. which  satisfies the subnormality condition\;\eqref{defi-subno} with a real number $\sigma>0$.
For any positive integer $N>0$, there is $C:=C(N,d,\sigma)>0$ such that for any $\lambda \gg 1$ one has for any sequence $(c_n)_{n\geq 0}$
\begin{equation}\label{pz} \P\left[ \Big\Vert \sum_{\lambda_n\leq \lambda } c_n X_n \phi_n \Big\Vert_{L^\infty(\R^d)}^2 
\leq C \ln (\lambda)   \sum_{j\leq [\lambda/2]} j^{\gamma(d)} \max_{n\in I(j)}|c_{n}|^{2} \right] \geq  1 -\frac{1}{\lambda^N},        \end{equation}
where $\gamma(d)$ is defined in \eqref{kt2}.
Furthermore, if $d\geq 2$ holds and if the   $(X_n)_{n\geq 0}$ are independant Gaussians $\mathcal{N}_{\R}(0,1)$, then one can find a sequence $(c_n)_{n\geq 0}$ 
such that we cannot replace the function $\lambda \mapsto \ln(\lambda)$ with a slower function of order\;$o(\ln(\lambda))$. 
 \end{theo}
 
In particular the previous result shows that   there  exists  $C>0$   such that  almost  surely   we have 
\begin{equation*}
\limsup_{\lambda\rightarrow +\infty}
\frac{\Big\Vert \sum_{\lambda_n\leq \lambda } c_n X_n \phi_n \Big\Vert_{L^\infty(\R^d)}}{\sqrt{\ln(\lambda)}}
\leq C\big(\sum_{j\geq 0} j^{\gamma(d)} \max_{n\in I(j)}|c_{n}|^{2}\big)^{1/2}. 
\end{equation*}
Furthermore  there   exists  a sequence  $\{c_n\}$  and  $c>0$ such that 
$\sum_{j\geq 0} j^{\gamma(d)} \max_{n\in I(j)}|c_{n}|^{2} =1$ and 
\begin{equation*}
\liminf_{\lambda\rightarrow +\infty}
\frac{\Big\Vert \sum_{\lambda_n\leq \lambda } c_n X_n \phi_n \Big\Vert_{L^\infty(\R^d)}}{\sqrt{\ln(\lambda)}}
\geq c.
\end{equation*}




It is straightforward that  if the coefficients $(c_n)_{n\geq 0}$ satisfy \eqref{condi3}, then \eqref{pz} implies
\begin{equation} \label{pz2}
 \P\left[ \Big\Vert \sum_{\lambda_n\leq \lambda } c_n X_n \phi_n \Big\Vert_{L^\infty(\R^d)}^2 
\leq C \ln (\lambda)   \sum_{\lambda_n\leq \lambda}   \lambda_n^{-\beta(d)} |c_n|^2 \right] \geq  1 -\frac{1}{\lambda^N},       
 \end{equation}
 with $\beta(1)=1/6$ and $\beta(d)=d/2$ for $d\geq 2.$
 The Salem-Zygmund inequality in the classical case of random trigonometric polynomials 
is similar to \eqref{pz2} but holds for $\beta(d)=0$.
Thus, \eqref{pz2} shows that randomness for Hermite series has a much more smoothing effect than for Fourier series.
Indeed, this is a consequence of a better behavior of the spectral function \eqref{kt} of $H$ in the space $L^\infty(\R^d)$.

Let us add that the proof of the classical Salem-Zygmund inequality \cite[Theorem 1, page 55]{JPK} uses in an essential way that the torus $\mathbb{T}$ is compact.
In our setting, the non-compactness of $\R^d$ is counterbalanced by the localization of \eqref{rand-part} in any subset or $\R^d$ which contains strictly the ball $B(0,\sqrt{\lambda})$ (here we will choose the closed ball $\overline{B}(0,\lambda)$ which is much bigger than 
$B(0,\sqrt{\lambda})$). \ligne

Our next result gives a sufficient condition to get almost surely continuity as in Theorem \ref{pazy}.

 \begin{theo}\ph\label{propPZ}
Let $\gamma(d)$ be defined by \eqref{kt2}, and  let  $(c_n)_{n\in \N}$ be such that 
\begin{equation}\label{cond}
\exists\, \alpha>1, \quad \sum_{j= 1}^{+\infty}j^{\gamma(d)}(\ln j)^{\alpha}\max_{n\in I(j)}|c_{n}|^{2}<+\infty.
\end{equation} 
Assume that $(X_n)_{n\geq 0}$ is an i.i.d. family of symmetric r.v. such that \eqref{defi-subno} holds.  
Denote by $$u^\omega_{\lambda} = \sum_{\lambda_n\leq \lambda}c_nX_{n}(\omega)\varphi_n.$$
Then  $\di u^\omega_{\lambda}\longrightarrow  u^\omega$  in $L^\infty(\R^d)$ almost surely when $\lambda\longrightarrow +\infty$.

In particular  for almost all $\om \in \Omega$, $u^{\om}$ is a bounded continuous  function on $\R^d$.

\end{theo}


In the particular case where  $(c_n)_{n\geq 0}$ are such that \eqref{condi3} holds, then the assumption \eqref{cond} becomes
\begin{equation*}
\exists\, \alpha>1, \quad \sum_{n= 0}^{+\infty}\lambda_n^{-\beta(d)}(\ln \lambda_n)^{\alpha}|c_{n}|^{2}<+\infty,
\end{equation*} 
with $\beta(1)=1/6$ and $ \beta(d)=d/2$ for $d\geq 2.$ This shows that for $d\geq 2$,  $u$ is in a slightly smaller space, denoted by $\H^{-d/2+}(\R^d)$ (with a log correction), than $\H^{-d/2}(\R^d)$. In other words, under condition \eqref{condi3} almost all series\;$u^{\om}$ in the very irregular distribution space $\H^{-d/2+}(\R^d)$  is actually a continuous function on $\R^{d}$.  

It is interesting to notice that if we forget the logarithmic term in 
in the assumption \eqref{cond}, we find exactly the assumption \eqref{condi5} of Theorem \ref{proplp} as $p$ tends to infinity although methods of proofs are different.

The symmetry assumption of the r.v. is only needed for the convergence of the partial sums, but the continuity result holds without this assumption.

We shall  give two   different proofs of Theorem~\ref{propPZ}: one is  an application of the Salem-Zygmund inequality (Theorem~\ref{salem-zy}), and the other   relies on an entropy criterion (see Section \ref{sect6}). \ligne

From the Salem-Zygmund inequality  we  can  also  get   a sufficient    condition so that 
$u^\omega(x)$  satisfies a   global H\"older  continuity  condition.
Recall  the  definition   of the modulus  of  continuity   of $ u: \R^d \rightarrow \C$, 
$$
m_{u}(h) = \sup_{\substack{\vert x-y\vert\leq h \\ x,y \in\R^d}}\vert u(x) - u(y)\vert,\;\; h>0.
$$
\begin{theo}\ph\label{hold}
Let  $(c_n)_{n\geq 0}$   such that there  exists $C>0$ such that 
\begin{equation}\label{condh}
  \sum_{k=2^j}^{2^{j+1}} \,\max_{n\in I(k)} |c_{n}|^{2}  \leq C2^{(-\gamma(d)-\mu) j}j^{2\nu},\;\; \forall j\geq 0, 
\end{equation} 
with $C>0$, ($\nu \in \R$ and  $0< \mu \leq 1$) or  ($\nu <-1$ and $\mu = 0$).
Assume that $(X_n)_{n\geq 0}$ is an i.i.d. family of r.v. such that~\eqref{defi-subno} holds. 
Then   we have, almost   surely  in $\omega$,
\begin{equation*}
m_{u^\omega}(h) = \mathcal {O}(h^\mu\vert\ln h\vert^\theta)
\end{equation*}
where
 \begin{itemize}
 \item[$\bullet$]
$\theta = \frac{1}{2} +\nu$  if $0<\mu< 1$
\item[$\bullet$] $\theta = 1+\nu$  if  $\mu =0$
\item[$\bullet$]
if $\mu = 1$ then
$$
\left\lbrace
\begin{array}{l}
\theta = 1+\nu,\; {\rm if}\; \nu \geq -\frac{1}{2}\\
\theta = \frac{1}{2}\; \; {\rm if}\;  -1\leq \nu \leq-\frac{1}{2}\\
u^\omega\; {\rm is\;  a.e \;  differentiable \; if}\; \nu < -1.
\end{array}
\right.
$$
\end{itemize}
\end{theo}

In particular, if   $(c_n)_{n\geq 0}$  is a sequence  which satisfies
\eqref{condi3} and  such that there  exists $C>0$ such that 
\begin{equation*} 
 \sum_{n \,:\, 2^j\leq \lambda_{n} < 2^{j+1}} |c_{n}|^{2}  \leq C2^{(\beta(d)-\mu) j}j^{2\nu},\;\; \forall j\geq 0, 
\end{equation*} 
then \eqref{condh} is satisfied.

\begin{rema}
With a slight modification of the proof of Theorem \ref{hold} we can get the following extension of Theorem \ref{propPZ}. If
\begin{equation*} 
 \sum_{j= 1}^{+\infty}j^{\gamma(d)+\mu}(\ln j)^{\a}\max_{n\in I(j)}|c_{n}|^{2}<+\infty,
\end{equation*} 
then  almost   surely  in $\omega$,
\begin{equation*}
m_{u^\omega}(h) = \mathcal {O}(h^\mu\vert\ln h\vert^\theta)
\end{equation*}
where
 \begin{itemize}
 \item[$\bullet$]
$\theta = \frac{1}{2} -\a+\eps$ for all $\eps>0$  if $0<\mu< 1$
\item[$\bullet$] $\theta = 1-\a+\eps$ for all $\eps>0$  if  $\mu =0$.
\end{itemize}

\end{rema}
 \subsection{Notations and plan of the paper} 
 
 \begin{enonce*}{Notations}
 In this paper $c,C>0$ denote constants the value of which may change
from line to line. These constants will always be universal, or uniformly bounded with respect to the other parameters. We write $a\lesssim b$ if $a\leq C b$ and $a\approx b$ if $ca\leq  b\leq C a$, for some $c,C>0$. 
\end{enonce*}

The rest of the paper is organised as follows. In Section \ref{Sect3} we prove the integrability results on the Hermite series. Section \ref{Sect4} is devoted to the proof of the regularity results (Theorems~\ref{salem-zy},~\ref{propPZ} and~\ref{hold}). 
In Section  \ref{annex-proba} we review some results we need about the convergence of random series in Banach spaces.    Finally, in Section \ref{sect6} we give an alternative proof of Theorem \ref{propPZ}.


\section{Proof of the integrability results}\label{Sect3}

\subsection{Proof of Theorem~\ref{proplp}}

We see Theorem~\ref{proplp} as a consequence of Proposition~\ref{equi-cv}, and it is equivalent to check 
%
\begin{equation}\label{plp-key}
\sum_{n\geq 0}  |c_n|^2 |\phi_n|^2        \in L^{\frac{p}{2}}(\R^d).
\end{equation}
By interpolating the $L^{\frac{p}{2}}$ norm and using that $I(j)\sim C_j j^{d-1}$, we get 
\begin{equation*}
\Big\Vert \sum_{n\in I(j)} |\phi_n|^2  \Big\Vert_{L^{\frac{p}{2} }(\R^d)}
\leq \Big\Vert \sum_{n\in I(j)} |\phi_n|^2  \Big\Vert_{L^{1 }(\R^d)}^{\frac{2}{p} }
 \Big\Vert \sum_{n\in I(j)} |\phi_n|^2 \Big\Vert_{L^{\infty} (\R^d)}^{1-\frac{2}{p}}
\lesssim j^{\frac{2}{p}(d-1)+\left(1-\frac{2}{p} \right) \gamma(d)}=j^{\gamma(d)+2\beta(d)/p},
\end{equation*}
as a consequence
\begin{equation*}
\Big\Vert\sum_{n\geq 0}  |c_n|^2 |\phi_n|^2 \Big\Vert_{L^{\frac{p}{2}}(\R^d)} \leq \sum_{j\geq 1}   \big(\max\limits_{n\in I(j)} |c_n|^2 \big)\Big\Vert \sum_{n\in I(j)} |\phi_n|^2 \Big\Vert_{L^{\frac{p}{2}}(\R^d)} \lesssim \sum_{j\geq 1} j^{\gamma(d)+2\beta(d)/p} \max\limits_{n\in I(j)} |c_n|^2 <+\infty .
\end{equation*}
We get \eqref{plp-key} and hence conclude.

\subsection{Proof of Proposition~\ref{estrad}}

 Let us first recall some results concerning the Laguerre polynomials, see \cite[Chapter 1]{Thanga} or \cite{Szego}. For $\a>-1$, the Laguerre polynomial $L^{(\a)}_n$ of type $\a$ and degree~$n\geq 0$ is defined by 
 \begin{equation}\label{defp}
 \e^{-r}r^{\a}L^{(\a)}_n(r)=\frac1{n !}\frac{d^n}{dr^n}\big( \e^{-r}r^{n+\a}  \big),\quad x\in \R.
 \end{equation}
We need the following identities (see \cite[lines (5.1.1),(5.1.3),(5.1.7) and (5.1.14)]{Szego}): 
 \begin{equation}\label{ortho}
 \int_0^{+\infty}L^{(\a)}_n(r)L^{(\a)}_m(r) \e^{-r}r^{\a}dr=\frac{\Gamma(n+\a+1) }{\Gamma(n+1)} \delta_{nm},
 \end{equation}
 \begin{equation}\label{val}
 L^{(\a)}_n(0)=\frac{\Gamma(n+\a+1) }{\Gamma(n+1)\Gamma(\a+1)}\approx n^{\alpha},
 \end{equation}
 \begin{equation}\label{deri}
 \forall n\geq 1, \qquad \frac{d}{dr}L^{(\a)}_n(r)=-L^{(\a+1)}_{n-1}(r),
 \end{equation}
 \beq\label{eqdiff}
 r \frac{d^2 L_n^{(\alpha)} }{dr^2} + (\alpha+1-r) \frac{d L_n^{(\alpha)}}{dr}  + nL_{n}^{(\alpha)} = 0.
  \eeq
We will need the following lemma 
\begin{lemm}\ph\label{szeg}
For any $\alpha>-1$ there are $c,\epsilon>0$ such that 
\begin{equation*}
\forall n\geq 1, \quad \forall r\in \Big(0,\frac{\epsilon^2}{n}\Big), \quad |L_n^{(\alpha)}(r)|\geq c n^{\alpha}.
\end{equation*}
\end{lemm}
\begin{proof}
As in \cite[p. 176]{Szego}, we introduce the function
\begin{equation*}
r\mapsto n  L_{n-1}^{(\alpha+1)}(r)^2+r \left(\frac{d}{dr} L_{n-1}^{(\alpha+1)}(r) \right)^2
\end{equation*}
whose derivative is $2\big( r-\frac{3}{2}-\alpha  \big)\big(\frac{d}{dr} L_{n-1}^{(\alpha+1) }\big)^2$ thanks to \eqref{eqdiff}. Thus, one has
\begin{equation*}
\forall r\in \Big[0,\alpha+\frac{3}{2} \Big],\quad  |L_{n-1}^{(\alpha+1)}(r)|\leq |L_{n-1}^{(\alpha+1)}(0)| \lesssim n^{\alpha+1}.
\end{equation*}
By using \eqref{deri}, we have 
\begin{equation*}
\forall r\in \Big[0,\alpha+\frac{3}{2} \Big],\quad | L_{n}^{(\alpha)}(r)-L_n^{(\alpha)}(0)| \lesssim r n^{\alpha +1}.
\end{equation*}
We can conclude by using a new time \eqref{val}.
\end{proof}
  
Because of the orthogonality condition \eqref{ortho}, it is usual to introduce the Laguerre functions normalized in $L^2(0,+\infty)$ 
\begin{equation}\label{defi-Ln}
\mathcal{L}_n^{\alpha}(r):=\frac{\sqrt{n!}}{\sqrt{\Gamma(n+\alpha+1)}} L_n^{(\alpha)}(r)\e^{-r/2} r^{\alpha/2} ,\quad    \frac{\sqrt{n!}}{\sqrt{\Gamma(n+\alpha+1)}}\approx n^{-\alpha/2}.    
\end{equation}

Those functions satisfy the following uniform estimates (see \cite{erdelyi,askey1965,mucken}).

\begin{prop}\ph\label{erdelyi}
For any $\alpha>-1$, there are $C=C(\alpha)$ and $\gamma=\gamma(\alpha)>0$ such that, by denoting $\nu=4n+2\alpha+2$, one has 
\begin{equation*}|\mathcal{L}_n^{(\alpha)}(r)| \leq \left\{ \begin{array}{lcl}
C(r\nu)^{\alpha/2} & \mbox{if} & 0\leq r \leq \frac{1}{\nu} \\[2mm]
C(r\nu)^{-1/4} & \mbox{if} & \frac{1}{\nu} \leq r \leq \frac{\nu}{2} \\[2mm]
C\nu^{-1/4}  (\nu^{1/3}+|\nu-r|)^{-1/4} & \mbox{if} & \frac{\nu}{2} \leq r \leq \frac{3\nu}{2} \\[2mm]
Ce^{-\gamma r} & \mbox{if} &  \frac{3\nu}{2}\leq r. 
\end{array} \right. \end{equation*}
\end{prop}

Now, denote by $\psi_n$ the $n$th $L^2(\R^d)$-normalized radial Hermite function for $d\geq 2$.
One can prove that $H\psi_{n}=(4n+d)\psi_{n}$ holds and that 
$\psi_n$ is proportional to $L^{(d/2-1)}_{n}(\vert x \vert^2)\e^{-\vert x\vert^2/2}$ (see for instance \cite[Corollary 3.4.1]{Thanga}).
By using the orthogonality of Laguerre functions $\mathcal{L}_n^{(\frac{d}{2}-1)}$, one easily gets
\begin{equation}\label{defi-psin}
\psi_n(x):=   c(d)\mathcal{L}_{n}^{\left(\frac{d}{2}-1\right)}(\vert x \vert^2) \vert x \vert^{-(\frac{d}{2}-1)}=c(d)\frac{\sqrt{n!}}{\sqrt{\Gamma(n+\frac{d}{2})}}   L_{n}^{( \frac{d}{2}-1)}(\vert x\vert^2)\e^{-|x|^2/2}   
\end{equation}
with $c(d):=\frac{\sqrt{2}}{\sqrt{\mbox{Vol}(S^{d-1})}}$ (see below \eqref{normLp} for $p=2$).

Let us estimate the $L^p(\R^d)$ norm of $\psi_n$ for $p\geq 2$ by using Proposition~\ref{erdelyi} with $\nu\sim n$ and $\alpha=\frac{d}{2}-1$.

The case $p=\infty$ is the easiest, and we get directly that $|\mathcal{L}_n^{(\alpha)}(r)|r^{-\frac{\alpha}{2}}\leq C \nu^{\frac{\alpha}{2}}$,
 in other terms $\Vert \psi_n \Vert_{L^\infty(\R^d)}\lesssim n^{\frac{d}{4}-\frac{1}{2}}$.
To get the lower bound, it is sufficient to combine Lemma~\ref{szeg} with \eqref{defi-psin} and the equivalent in \eqref{defi-Ln}.
 
We now consider  $p\in [2,+\infty)$. Then  we have 
\begin{equation}\label{normLp}
\begin{array}{rcl}
\Vert \psi_n \Vert_{L^p(\R^d)}^p & =& c(d)^p \mbox{Vol}(S^{d-1})\Int_{0}^{+\infty}  |\mathcal{L}_n^{(d/2-1)}(r^2)|^p r^{-p\left(\frac{d}{2}-1 \right)} r^{d-1} dr  \\[3mm]
& = &  c(d)^{p-2}\Int_{0}^{+\infty}  |\mathcal{L}_n^{(d/2-1)}(r)|^p r^{-\left(\frac{p}{2}-1\right)\left(\frac{d}{2}-1 \right)} dr .
\end{array}
\end{equation}
We begin by the following integrals: 
\begin{equation}\label{prem-int}\begin{array}{rcl}
\Int_{0}^{\frac{1}{\nu}}  |\mathcal{L}_n^{(d/2-1)}(r)|^p r^{-\left(\frac{p}{2}-1\right)\left(\frac{d}{2}-1 \right)} dr & \lesssim&  n^{\frac{p}{2}\left( \frac{d}{2}-1 \right)} \Int_{0}^{\frac{1}{\nu}} r^{\frac{d}{2}-1}dr 
 \lesssim n^{\frac{p}{2}\left( \frac{d}{2}-1 \right)-\frac{d}{2}}=n^{\frac{pd}{2}\left(\frac{1}{2}-\frac{1}{p}\right)-\frac{p}{2}}  \\[10pt]
\Int_{\frac{3\nu}{2}}^{+\infty}  |\mathcal{L}_n^{(d/2-1)}(r)|^p r^{-\left(\frac{p}{2}-1\right)\left(\frac{d}{2}-1 \right)} dr & \lesssim &           \Int_{\frac{3\nu}{2}}^{+\infty} e^{-\gamma r} r^{-\left(\frac{p}{2}-1\right)\left(\frac{d}{2}-1\right)}dr = \mathcal{O}(n^{-\infty}).
\end{array}  \end{equation}
To study the integrals over the others intervals given by Proposition~\ref{erdelyi}, we have to consider several subcases.

$\bullet$ If $p>\frac{2d}{d-1}$ holds, one has obviously  $\frac{1}{2}-\frac{1}{p}>\frac{1}{2d}$ and the comparison of different exponents of $n$ will rely on:
\begin{equation}\label{compa1}\frac{-d}{2}\left(\frac{1}{2}-\frac{1}{p} \right) <  \frac{d}{2}\left(\frac{1}{2}-\frac{1}{p} \right) -\frac{1}{2}  .\end{equation}
Notice that one also has
\begin{equation}\label{azqs}\frac{p}{4} +\left( \frac{p}{2}-1\right)\left( \frac{d}{2}-1\right)>
\frac{d}{2(d-1)}+\left(\frac{d}{d-1}-1 \right)\frac{d-2}{2}=1,\end{equation}
which implies that the following integral is of interest near $r=0$: 
\begin{multline} 
\Int_{\frac{1}{\nu}}^{\frac{\nu}{2}}  |\mathcal{L}_n^{(d/2-1)}(r)|^p r^{-\left(\frac{p}{2}-1\right)\left(\frac{d}{2}-1 \right)} dr  \lesssim  n^{-\frac{p}{4}} \Int_{\frac{1}{\nu}}^{\frac{\nu}{2}} r^{-\frac{p}{4} -\left( \frac{p}{2}-1\right)\left( \frac{d}{2}-1\right)} dr\lessim\\
 \lesssim  n^{-\frac{p}{4}+\frac{p}{4}+\left(\frac{p}{2}-1 \right)\left(\frac{d}{2}-1\right)-1}  \approx  n^{\frac{pd}{2}\left( \frac{1}{2}-\frac{1}{p} \right)-\frac{p}{2}}\label{int3} .
\end{multline}
The integral over $[\frac{\nu}{2},\frac{3\nu}{2}]$ is bounded by 
\begin{equation*}\begin{array}{rcl}
 \Int_{\frac{\nu}{2}}^{\frac{3\nu}{2}}  |\mathcal{L}_n^{(d/2-1)}(r)|^p r^{-\left(\frac{p}{2}-1\right)\left(\frac{d}{2}-1 \right)} dr  & \lesssim & n^{-\frac{p}{4}} \Int_{\frac{\nu}{2}}^{\frac{3\nu}{2}} \frac{dr}{(\sqrt[3]{\nu}+|\nu-r| )^{\frac{p}{4}} r^{\left(\frac{p}{2}-1 \right)\left(\frac{d}{2}-1 \right)} } \\[3mm]
& \lesssim &  n^{-\frac{p}{4}-\left(\frac{p}{2}-1 \right)\left(\frac{d}{2}-1  \right)   }  \Int_{\frac{\nu}{2}}^{\frac{3\nu}{2}} \frac{dr}{(\sqrt[3]{\nu}+|\nu-r| )^{\frac{p}{4}}} \\[3mm]
& \lesssim & n^{\frac{p}{4} +\frac{d}{2}-\frac{pd}{4}-1   }  \Int_{0}^{\frac{\nu}{2}} \frac{dr}{(\sqrt[3]{\nu}+r )^{\frac{p}{4}}} \\[3mm]
& \lesssim & n^{\frac{p}{4} +\frac{d}{2}-\frac{pd}{4}-1   }  \nu^{\frac{1}{3}-\frac{p}{12}}\Int_{0}^{\frac{1}{2} \nu^{\frac{2}{3}}} \frac{dr}{(1+r)^{\frac{p}{4}}} \\[3mm]
&\lesssim & 
n^{\frac{p}{6}+\frac{d}{2}-\frac{pd}{4}-\frac{2}{3}} \Int_{0}^{\frac{1}{2} \nu^{\frac{2}{3}}} \frac{dr}{(1+r)^{\frac{p}{4}}}.
 \end{array}\end{equation*}

We have to use now the following fact if $p>4$ holds (which is necessary for $d=2$): 
\begin{equation*} 
\frac{d}{2}\left(\frac{1}{2}-\frac{1}{p} \right) -\frac{1}{2}   - \left(
 \frac{1}{6}+\frac{d}{2p}-\frac{d}{4}-\frac{2}{3p}\right)= \frac{d}{2}-\frac{2}{3}-\left(d-\frac{2}{3} \right) \frac{1 }{p} > \frac{d}{4}-\frac{1}{2} \geq 0.  
\end{equation*}
That brings us to 
\begin{equation}\label{int2}
 \Int_{\frac{\nu}{2}}^{\frac{3\nu}{2}}  |\mathcal{L}_n^{(d/2-1)}(r)|^p r^{-\left(\frac{p}{2}-1\right)\left(\frac{d}{2}-1 \right)} dr \lesssim \left\{\begin{array}{cc} n^{\frac{p}{6}+\frac{d}{2}-\frac{pd}{4}-\frac{2}{3}}\leq n^{p\left[ \frac{d}{2}\left(\frac{1}{2}-\frac{1}{p} \right) -\frac{1}{2}  \right] } & \mbox{ if } p>4 \\
 n^{-\frac{d}{2}} \ln(n) \lesssim n^{\frac{d}{2}}=n^{-\frac{pd}{2}\left( \frac{1}{2}-\frac{1}{p}\right) } & \mbox{ if } p=4 \\
n^{\frac{p}{6}+\frac{d}{2}-\frac{pd}{4}-\frac{2}{3}+\frac{2}{3}\left(1-\frac{p}{4} \right)}= n^{-\frac{pd}{2}\left( \frac{1}{2}-\frac{1}{p}\right) } & \mbox{ if } p<4 . 
 \end{array}\right.
\end{equation}

Thanks to \eqref{compa1}, the comparison of exponents in \eqref{prem-int}, 
\eqref{int3} and \eqref{int2} gives $\|\psi_n\|_{L^p(\R^d)}\lesssim n^{\frac{d}{2}\left( \frac{1}{2}-\frac{1}{p}\right)-\frac{1}{2}}$.

$\bullet$ If $p<\frac{2d}{d-1}$ holds, one has $\frac{1}{2}-\frac{1}{p}<\frac{1}{2d}$ and the contrary of \eqref{compa1} and \eqref{azqs} hold: 
\begin{equation}\label{compa2}   \frac{d}{2}\left(\frac{1}{2}-\frac{1}{p} \right) -\frac{1}{2}  < \frac{-d}{2}\left(\frac{1}{2}-\frac{1}{p} \right)  \end{equation}
\begin{equation*}\frac{p}{4} +\left( \frac{p}{2}-1\right)\left( \frac{d}{2}-1\right)<1. \end{equation*}
Hence, the integral over $\left[\frac{1}{\nu},\frac{\nu}{2} \right]$ is of interest for  $r\gg 1$: 
\begin{multline}
 \Int_{\frac{1}{\nu}}^{\frac{\nu}{2}}  |\mathcal{L}_n^{(d/2-1)}(r)|^p r^{-\left(\frac{p}{2}-1\right)\left(\frac{d}{2}-1 \right)} dr   \lesssim   n^{-\frac{p}{4}}\Int_{\frac{1}{\nu}}^{\frac{\nu}{2}} r^{-\frac{p}{4}-\left( \frac{p}{2}-1\right)\left(\frac{d}{2}-1\right)} dr  \lesssim \\
  \lesssim  n^{-\frac{p}{4}+1-\frac{p}{4}-\left( \frac{p}{2}-1\right)\left(\frac{d}{2}-1\right)}  
  \lesssim  n^{-\frac{pd}{2}\left( \frac{1}{2}-\frac{1}{p} \right) }  . \label{int2b} \end{multline}

We deal with the integral over $\left[ \frac{\nu}{2},\frac{3\nu}{2} \right]$ by the same way with the help of \eqref{int2} and by noticing that $p<\frac{2d}{d-1}\leq 4$ holds. Hence we get 
\begin{equation}\label{int3b} \int_{\frac{\nu}{2}}^{\frac{3\nu}{2}}  |\mathcal{L}_n^{(d/2-1)}(r)|^p r^{\left(1-\frac{p}{2}\right)\left(\frac{d}{2}-1 \right)} dr  \lesssim 
n^{-\frac{pd}{2}\left( \frac{1}{p}-\frac{1}{2} \right) }. \end{equation}
Once again, we compare the exponents in \eqref{prem-int}, \eqref{int2b} and \eqref{int3b} with the help of \eqref{compa2} and we get 
$\|\psi_n\|_{L^p(\R^d)}\lesssim n^{-\frac{d}{2}\left( \frac{1}{2}-\frac{1}{p}\right)}$.

$\bullet$ If $p=\frac{2d}{d-1}$ holds, we follow the previous analysis and we see that $ \|\psi_n\|_{L^p}^p \lesssim n^{-\frac{p}{4}} \ln(n) $.

We have finished the proof of Proposition~\ref{estrad}.

\subsection{Proof of Theorem~\ref{theo24}}

We will use Proposition~\ref{equi-cv} and Proposition~\ref{estrad}.

We consider $p>\frac{d}{\alpha_\star(c)}$ and we write
\begin{equation*}
\begin{array}{rcl}
\displaystyle\int_{\R^d} \Big(\Sum_{n\geq 0} |c_n|^2 |\psi_n(x)|^2 \Big)^{\frac{p}{2}} dx & \geq & \sup\limits_{N\geq 1} \displaystyle\int_{|x|\leq \frac{\ep}{\sqrt{N}}}  \Big( \sum_{n=1}^N |c_n|^2 |\psi_n(x)|^2  \Big)^{\frac{p}{2}} dx
\\
& \geq &   C \sup\limits_{N\geq 0} \frac{\ep^d}{N^{\frac{d}{2}}}  \Big( \displaystyle \sum_{n=1}^N  |c_n|^2 n^{\frac{d}{2}-1} \Big)^{\frac{p}{2}}    \\
& \geq & +\infty.
\end{array}
\end{equation*}

We consider $p<\frac{d}{\alpha_\star(c)}$ and we write
\begin{eqnarray*}
\Big\Vert \displaystyle \sum_{n\geq 0} |c_n|^2 |\psi_n|^2 \Big\Vert_{L^{p/2}(\R^d)} \leq  \displaystyle   \displaystyle   \sum_{n\geq 0} |c_n|^2 \| \psi_n^2 \|_{L^{p/2}(\R^d)}   = \displaystyle   \sum_{n\geq 0} |c_n|^2 \| \psi_n \|_{L^{p}(\R^d)}^2.
\end{eqnarray*}
If $p$ belongs to $(2,\frac{2d}{d-1}]$, then $\|\psi_n\|_{L^p(\R^d)}$ is 
less than $n^{-\ep}$ for some $\ep>0$ (see Proposition~\ref{estrad}).
By using that $\Sum_{n\geq 0} c_n \psi_n$ belongs to $\bigcap\limits_{\ep>0} \mathcal{H}^{-\ep}(\R^d)$, it is clear that the series $\sum |c_n|^2 \|\psi_n\|_{L^p(\R^d)}^2$ converges.

If $p$ is greater than $\frac{2d}{d-1}$, 
we use first an Abel summation and then two times the inequality $\alpha_\star(c)<\frac{d}{p}$ to 
bound the sum of the series 
$\sum |c_n|^2 \| \psi_n \|_{L^{p}(\R^d)}^2$ by
\begin{multline*}
C|c_0|^2+\displaystyle\sum_{n\geq 1} |c_n|^2 n^{\frac{d}{2}-1}n^{-\frac{d}{p}} \leq \\
\begin{aligned}
 &\leq  C|c_0|^2+\Big(\lim\limits_{N\rightarrow +\infty} N^{-\frac{d}{p}} \displaystyle \sum_{n=1}^N |c_n|^2  n^{\frac{d}{2}-1} \Big) +\displaystyle\sum_{N\geq 1} \Big(\sum_{n=1}^N |c_n|^2 n^{\frac{d}{2}-1} \Big)
|N^{-\frac{d}{p}}-(N+1)^{-\frac{d}{p}}| \\[3mm]
 &  \lesssim C|c_0|^2+0  +\displaystyle\sum_{N\geq 1} \frac{1}{N^{1+\frac{d}{p}}}\Big(\sum_{n=1}^N |c_n|^2 n^{\frac{d}{2}-1} \Big) <+\infty .
\end{aligned}
\end{multline*}

\textbf{Remark.} If we define for any sequence $(c_n)_{n\geq 0}$
\begin{equation*}
\forall p> \frac{2d}{d-1} \qquad \|c\|_{d,p}:=|c_0|+\sup\limits_{N\geq 1} \frac{1}{N^{\frac{d}{2p}}} \left(\sum_{n=1}^N |c_n|^2 n^{\frac{d}{2}-1}\right)^{\frac{1}{2}},
\end{equation*}
then the previous proof shows that one has
\begin{equation*}
\forall \nu>0 \qquad    C \|c\|_{d,p}\leq \Bigg\Vert \sqrt{\sum_{n\geq 0} |c_n|^2 \psi_n^2}\Bigg\Vert_{L^p} \leq C(\nu) \|c\|_{d,p+\nu}.
\end{equation*}
It is not clear if one can find a more precise norm on the sequence $(c_n)_{n\geq 0}$ which is equivalent to $\left\Vert \sqrt{\sum_{n\geq 0} |c_n|^2 \psi_n^2}\right\Vert_{L^p} $.
Indeed, this is essentially equivalent to decide whether or not the almost sure convergence in $L^p(\R^d)$ holds if $p$ is the critical convergence exponent $\frac{d}{\alpha_\star(c)}$.

    
      \section{Proof of the regularity results}\label{Sect4} 

 \subsection{Proof of Theorem~\ref{salem-zy}} 

Let us begin by introducing the following notation: 
\begin{equation*}    \forall \lambda >0 \quad \mathcal{E}_H(\lambda):=\mbox{Span}\{\phi_j, \lambda_j\leq \lambda   \,    \}, \end{equation*}
and let us recall the following bound on the spectral function of $H$ (see \cite[Lemmas 3.1, 3.2 and 3.5]{PRT1}: there are constants $C,c>0$ such that for any $\lambda \geq 1$ and $x\in \R^d$ one has 
\begin{equation}\label{zop1}  \forall u \in \mathcal{E}_H(\lambda),\qquad  |u(x)|\leq C \lambda^{\frac{d}{4}} \exp \left(-c\frac{|x|^2}{2\lambda} \right) \|u\|_{L^2(\R^d)} ,        \end{equation}

The first tool we need to prove Theorem~\ref{salem-zy} is a Bernstein inequality for the harmonic oscillator.
In the Hilbertian framework, it is easy to check that one has 
\begin{equation*} \forall \lambda\geq 1 \quad \forall u\in \mathcal{E}_H(\lambda)\quad    \|\partial_x u \|_{L^2(\R^d)}\leq C \sqrt{\lambda} \|u\|_{L^2(\R^d)}        .     \end{equation*}
We need a version of the previous inequality by replacing the space $L^2(\R^d)$ with $L^\infty(\R^d)$.

\begin{lemm}\ph\label{bernstein}
For any dimension $d\geq 1$, there are $s(d)\geq 0$ and $C=C(d)>0$ such that the following inequalities hold
\begin{equation}\label{bern} \forall \lambda \geq 1 \quad \forall u\in \mathcal{E}_H(\lambda)\quad   \|\nabla u \|_{L^\infty(\R^d)}\leq C \lambda^{s(d)} \|u\|_{L^\infty(\R^d)} .       \end{equation}
\end{lemm}

\begin{proof}
For any real number $s>\frac{d}{2}$, the Sobolev embedding $\H^s(\R^d)\subset L^\infty(\R^d)$ allows us to write for each $\ell \in \{1,2,\dots,d\}$:
\begin{displaymath}
\begin{array}{rcl}
\|\partial_{x_\ell} u \|_{L^\infty(\R^d)} & \leq & C \| \partial_{x_\ell} u \|_{\H^s(\R^d)} \\[3mm]
& \leq & C \|u\|_{\H^{s+1}(\R^d)} \\[3mm]
& \leq & C \left( \Sum_{\lambda_j\leq \lambda} \lambda_j^{s+1} \left\vert \int_{\R^d} \phi_j(x) u(x)dx\right\vert^2 \right)^{\frac{1}{2}} \\[3mm]
& \leq & C \lambda^{s+1} \lambda^d  \|u\|_{L^\infty(\R^d)} \sup\limits_{\lambda_j\leq \lambda} \|\phi_j\|_{L^1(\R^d)}.
\end{array}
\end{displaymath}
In view to get a bound of $\|\phi_j\|_{L^1(\R^d)}$ we just use the Cauchy-Schwarz inequality:
\begin{displaymath} \begin{array}{rcl} \Int_{\R^d} |\phi_j(x)| dx &=& \Int_{\R^d} \langle x \rangle^{-\frac{d+1}{2}} \langle x \rangle^{\frac{d+1}{2}} |\phi_j(x)|dx  \\[3mm]
& \leq &  C \| \langle x \rangle^{\frac{d+1}{2}} \phi_j(x)       \|_{L^2(\R^d)}   \\[3mm]
& \leq & C \|\phi_j\|_{\H^{\frac{d+1}{2}}(\R^d)} \\[3mm]
& \leq & C \lambda_j^{\frac{d+1}{4}} .\end{array} \end{displaymath}
Thus \eqref{bern} is proved.
\end{proof}

It is not clear for us if the exponent $s(d)$ can be chosen to be independent of $d$ or if we can find the optimal value of $s(d)$.

\begin{coro}\ph\label{bern-coro}
If $\lambda$ is enough large, there is a constant $c>0$ which is independent of $\lambda$ such that for any $u\in \mathcal{E}_H(\lambda)$ there is $y\in \overline{B}(0,\lambda)$ for which we have 
\begin{enumerate}[(i) ] 
\item $\|u\|_{L^\infty(\overline{B}(0,\lambda))}=\|u\|_{L^\infty(\R^d)}$,
\item $ \forall x\in \overline{B}(y,c\lambda^{-s(d)} )\cap \overline{B}(0,\lambda), \quad |u(x)|\geq \frac{1}{2} \|u\|_{L^\infty(\overline{B}(0,\lambda))}          $,
\item by denoting $\mbox{Vol}$ the volume function, we have  \begin{equation*}\mbox{Vol}\left\{\overline{B}(y,c\lambda^{-s(d)} )\cap \overline{B}(0,\lambda)\right\}\geq \frac{1}{3} \mbox{Vol}\big\{ \overline{B}(y,c\lambda^{-s(d)} )\big\}.\end{equation*}
\end{enumerate}
\end{coro}
\begin{proof}
By the same argument we used in the proof of Lemma~\ref{bernstein}, we claim that there is a constant $\nu>0$, independent of $\lambda$, such that 
\begin{equation}\label{zop2}   \forall u\in \mathcal{E}_H(\lambda)\quad \|u\|_{L^2(\R^{d})}\leq C \lambda^\nu \|u\|_{L^\infty(\R^d)}   .   \end{equation}
By combining \eqref{zop1} and \eqref{zop2}, we understand that if $\lambda$ is enough large and if $|x|>\lambda$ holds then we have 
\begin{equation*} |u(x)| \leq C \lambda^{\frac{d}{4}+\nu} \exp \big(-\frac{c\lambda}{2} \big)  \|u\|_{L^\infty(\R^{d})} \leq \frac{1}{2} \|u\|_{L^\infty(\R^d)}  .\end{equation*}
That proves Point $(i)$.
Let us check Point $(ii)$.
By a compactness argument, there is $y\in \overline{B}(0,\lambda)$ which maximizes $u$ on the whole space $\R^d$.
For any $x\in \overline{B}(0,\lambda)$, Lemma~\ref{bernstein} gives us
\begin{equation*} |u(x)-u(y)| \leq C |x-y| \lambda^{s(d)} \|u\|_{L^\infty(\R^d)}       .\end{equation*}
If $|x-y|< \frac{C}{2} \lambda^{-s(d)}$ holds then $|u(x)|\geq \frac{1}{2} \|u\|_{L^\infty(\R^d)}$.

Point $(iii)$ is a consequence of a geometric fact. Indeed, it is quite clear that we have 
\begin{equation*}  \liminf\limits_{R\rightarrow 0} \inf\limits_{z\in \overline{B}(0,1)} \frac{\mbox{Vol}\big\{ \overline{B}(z,R)\cap \overline{B}(0,1)        \big\}  }{\mbox{Vol}\big\{ \overline{B}(z,R)\big\} }   =\frac{1}{2}       .\end{equation*}
Consequently, if $\lambda$ is enough large then Point $(iii)$ holds.
\end{proof}

We can prove Theorem~\ref{salem-zy} by following \cite[Theorem 1, page 55]{JPK}.
Our preliminaries allow us to deal with the non-compactness of $\R^d$.
We define the random maximum 
\begin{displaymath}
M_\lambda^\omega:=  \|u_\lambda^\omega(x)\|_{L_x^\infty(\R^d)}=\|u_\lambda^\omega(x)\|_{L_x^\infty(B(0,\lambda))}.
\end{displaymath}
We apply Lemma~\ref{bern-coro} to the random function $u_\lambda^\omega \in \mathcal{E}_H(\lambda)$.
If $x$ belongs to the random set $A_\lambda^\omega:=\overline{B}(y^\omega,c\lambda^{-s(d)})\cap \overline{B}(0,\lambda)$ then we have 
\begin{equation*}   u_\lambda^\omega(x) \geq \frac{1}{2} M_\lambda^\omega \quad \mbox{or} \quad -u_\lambda^\omega(x) \geq \frac{1}{2} M_\lambda^\omega                   .\end{equation*}
Thus, it comes for any $r>0$ 
\begin{displaymath}   \E\Big[    \exp\big( \frac{1}{2} r M_\lambda^\omega  \big)     \Big] 
\leq \E\Big[   \frac{1}{\mbox{Vol}(A_\lambda^\omega )}  \int_{A_\lambda^\omega} \exp\left(r u_\lambda^\omega(x) \right)+     \exp\left(-r u_\lambda^\omega(x) \right) dx        \Big]  .
    \end{displaymath}
From Point $(iii)$ of Lemma~\ref{bern-coro}, we get 
\begin{displaymath}   \begin{array}{rcl} \E\Big[    \exp\left( \frac{1}{2} r M_\lambda^\omega  \right)     \Big] 
& \leq & C \lambda^{d s(d)} \E\Big[ \Int_{A_\lambda^\omega} \exp\left(r u_\lambda^\omega(x) \right)+     \exp\left(-r u_\lambda^\omega(x) \right) dx        \Big]  \\[4mm]
& \leq & C \lambda^{d s(d)} \Int_{\overline{B}(0,\lambda)} \E \Big[ \exp\left(r u_\lambda^\omega(x) \right)+     \exp\left(-r u_\lambda^\omega(x) \right)\Big] dx .\end{array}     \end{displaymath}

By coming back to the definition \eqref{defi-rand} of $u_\lambda^\omega$, we can use the independence of the random variables\;$X_n$: 
\begin{displaymath} \E \Big[\exp\big(r u_\lambda^\omega (x) \big)\Big]  =
\prod_{\lambda_n\leq \lambda} \E\Big[ \exp \big(  r c_n X_n(\omega) \phi_n(x)     \big)      \Big]  .\end{displaymath}
Now we use \eqref{defi-subno} and \eqref{kt} to get 
\begin{eqnarray}
\E \Big[\exp\big(r u_\lambda^\omega (x) \big)\Big] & \leq &
\exp \Big(Cr^2\Sum_{\lambda_n\leq \lambda} |c_n|^2 |\phi_n(x)|^2        \Big) \nonumber  \\ 
& \leq & \exp\Big(      Cr^2 \Sum_{j\leq \lambda} \Big[  \max_{n\in I(j)} |c_n|^2   \Sum_{n\in I(j)} |\phi_n(x)|^2 \Big]          \Big)\nonumber \\
& \leq & \exp\left(C{r^2}    \rho_\lambda       \right), \label{inega}
\end{eqnarray}
where we have set 
\begin{equation*}    \rho_\lambda:=    \Sum_{j \leq \lambda}  j^{\gamma(d)} \max_{n\in I(j)} |c_n|^2.
\end{equation*}
Obviously, a similar argument gives the same bound for $\E \Big[\exp\big(-r u_\lambda^\omega(x) \big)\Big] $, and we have  obtained
\begin{displaymath}  \E\Big[    \exp\big( \frac{1}{2} r M_\lambda^\omega  \big)     \Big] \leq C \lambda^{ds(d)+d} \exp\big(C\frac{r^2}{2}    \rho_\lambda       \big),
       \end{displaymath}
which is totally equivalent to 
\begin{displaymath}
\forall L\geq 1 \quad \forall r>0 \quad \E \left[  \exp\left(  \frac{r}{2}  \Big(  M_\lambda^\omega -Cr\rho_\lambda^2 - \frac{2}{r}\ln (C \lambda^{ds(d)+d}L) \Big)   \right)         \right] \leq \frac{1}{L}.
\end{displaymath}
From Markov's inequality, it comes 
\begin{displaymath}
\P\left[ M_\lambda^\omega - Cr\rho_\lambda^2 - \frac{2}{r}\ln (C \lambda^{ds(d)+d}L) \geq 0\right] \leq \frac{1}{L}.
\end{displaymath}
Now we just have to optimize in $r$ by choosing $r^2=\frac{1}{\rho_\lambda^2} \ln (C \lambda^{ds(d)+d} L)$. For another constant $C>0$, we have 
\begin{displaymath}
\P\left[M_\lambda^\omega  \geq C \rho_\lambda \sqrt{\ln(C \lambda^{ds(d)+d}L) }\, \right] \leq \frac{1}{L}.
\end{displaymath}
The conclusion comes with the choice $L=\lambda^N$.

Finally, we have to see that the term $\ln(\lambda)$ is optimal in \eqref{pz} if 
$d\geq 2$ holds, and when the $(X_n)_{n\geq 0}$ are independent Gaussians $\mathcal{N}_{\R}(0,1)$.

Let us suppose the contrary and consider a function $\vartheta(\lambda)=o(\ln(\lambda))$ such that Theorem~\ref{salem-zy} holds true by  replacing $\ln(\lambda)$ with $\vartheta(\lambda)$.

To see that implies a contradiction, let us recall a result proved in \cite[Theorem 1.1]{PRT1}\footnote{\cite[Theorem 1.1]{PRT1}  is stated for complex Gaussians, but the result also holds for real r.v. (see  \cite[Assumption 1]{PRT1}).} (with the sequence $d_j=\lambda_j^{-\frac{d}{4}}c_j$ and assuming 
\eqref{condi4}), there are real numbers $C_0>0$ and $c>0$ such that for any $j \gg 1$ one has

\begin{equation*}
\P \Bigg[  C_0 \ln (j)   \Big\Vert  \sum_{n\in I(j)} \lambda_n^{-\frac{d}{4}} c_n X_n \phi_n \Big\Vert^2_{{L}^{2}(\R^2)} \leq 
  \Big\Vert  \sum_{n\in I(j)} \lambda_n^{-\frac{d}{4}} c_n X_n \phi_n \Big\Vert^2_{\mathcal{W}^{\frac{d}{2},\infty}(\R^d) } \Bigg] 
\geq 1-\frac{1}{(j+2)^{c}}.
\end{equation*}
From the definition \eqref{wdp} and Theorem~\ref{salem-zy} with the function $\vartheta$ and any chosen positive integer $N$, we have with probability greater than $1-\frac{1}{(j+2)^{N}}-\frac{1}{j^N}$
\begin{equation*}
\begin{array}{rcl} \Big\Vert  \Sum_{n\in I(j)}  \lambda_n^{-\frac{d}{4}} c_n X_n \phi_n \Big\Vert_{\mathcal{W}^{\frac{d}{2},\infty}(\R^d) }& =&  \Big\Vert 
  \Sum_{n\in I(j)} c_n X_n \phi_n \Big\Vert_{L^{\infty}(\R^d) } \\[3mm]
& \leq & \Big\Vert 
  \Sum_{\lambda_n <2j} c_n X_n \phi_n \Big\Vert_{L^{\infty}(\R^d) } + \Big\Vert 
  \Sum_{\lambda_n < 2j+2} c_n X_n \phi_n \Big\Vert_{L^{\infty}(\R^d) } \\[3mm]
& \leq & C \sqrt{\vartheta (2j)+\vartheta(2j+2)} \Big(\Sum_{\lambda_n < 2j+2} {\lambda_n^{-\frac{d}{2}}} {|c_n|^2} \Big)^{\frac{1}{2}}             .
  \end{array}
\end{equation*}
We have now to make use of the condition \eqref{condi4}: 
\begin{equation*}
\bigg(\sum_{n\in I(j)} \lambda_n^{-\frac{d}{2}} |c_n|^2 \bigg) \times \frac{1}{\#I(j)} \sum_{n\in I(j)} |X_n|^2 \leq C \sum_{n\in I(j)} |\lambda_n^{-\frac{d}{4}}c_n X_n|^2=C\Big\Vert  \sum_{n\in I(j)} \lambda_n^{-\frac{d}{4}} c_n X_n \phi_n \Big\Vert^2_{{L}^{2}(\R^2)}  .
\end{equation*}
By combining these arguments, we have with probability greater than $1-\frac{1}{(j+2)^c}-\frac{1}{(j+2)^N}-\frac{1}{j^N}$
\begin{equation*}
 \frac{1}{\#I(j)} \sum_{n\in I(j)} |X_n|^2  \leq C \frac{\vartheta(2j)+\vartheta(2j+2)}{\ln(j)} \bigg(\Sum_{\lambda_n < 2j+2}  \lambda_n^{-\frac{d}{2}} |c_n|^2 \bigg) \bigg( \Sum_{n \in I(j) }  \lambda_n^{-\frac{d}{2}} |c_n|^2\bigg)^{-1}.   
\end{equation*}
One can obviously choose the sequence $(c_n)_{n\geq 0}$ such that   \eqref{condi4} and the two following properties hold: 
\begin{equation*}
u\in \mathcal{H}^{-\frac{d}{2}}(\R^d), \quad \sum_{j\geq 1} \sum_{n\in I(j)} \lambda_n^{-\frac{d}{2}} |c_n|^2 <+\infty ,
\end{equation*}
\begin{equation*}
\#\bigg\{j\geq 1, \quad \Sum_{n \in I(j) }\lambda_n^{-\frac{d}{2}} |c_n|^2  \geq \left(\frac{\vartheta(2j)+\vartheta(2j+2)}{\ln(j)} \right)^{\frac{1}{2}} \bigg\}=+\infty.
\end{equation*}
Hence, we get for probability greater than $1-\frac{1}{(j+2)^c}-\frac{1}{(j+2)^N}-\frac{1}{j^N}$
\begin{equation}\label{lln}
 \frac{1}{\#I(j)} \sum_{n\in I(j)} |X_n|^2  \leq \varepsilon(j)
\end{equation}
where $\liminf\limits_{j\rightarrow +\infty} \varepsilon(j)=0$.
Since $\lim\limits_{j\rightarrow +\infty} \# I(j)=+\infty$ holds, the Law of Large Numbers ensures that the left side of \eqref{lln} converges almost surely to $\E\big[\vert X_1\vert^2\big]>0$.
Since the almost sure convergence implies the convergence in probability, we understand that \eqref{lln} cannot hold.
That proves that Theorem~\ref{salem-zy} is optimal for the function $\ln(\lambda)$.


\subsection{Proof of  Theorem~\ref{propPZ}}     

We give here an argument which uses the Salem-Zygmund theorem. In Section~\ref{sect6} we will present an alternative proof  relies on an entropy argument.

\subsubsection{Proof of Theorem~\ref{propPZ} using  the Salem-Zygmund Theorem~\ref{salem-zy}}     
For any positive integer $K$, we introduce $J(K):=\big\{n\in \N,\quad \lambda_n \in [2^{2^K},2^{2^{K+1}}-1]\big\}     $ and
\begin{equation*}
u_K^\omega:=\sum_{n\in J(K)} c_n X_n(\omega) \phi_n .
\end{equation*}
By using Theorem~\ref{salem-zy}, we have 
\begin{equation*}
\P \bigg[  \|u_K^\omega\|_{L^\infty(\R^d)} \geq C 2^{K/2}  \bigg( \sum_{j=2^{2^K-1}}^{2^{2^{K+1}-1}-1}   j^{\gamma(d)}\max_{k\in I(j)}|c_n|^2    \bigg)^{\frac{1}{2}}  \,   \bigg] \leq \frac{1}{2^{2^{K+1}-1}}.
\end{equation*}
The Borel-Cantelli lemma ensures that almost surely there is $C_\omega>0$ such that 
\begin{eqnarray*}
\|u_K^\omega\|_{L^\infty(\R^d)} & \leq&  C_\omega 2^{K/2}  \bigg( \sum_{j=2^{2^K-1}}^{2^{2^{K+1}-1}-1}   j^{\gamma(d)}\max_{k\in I(j)}|c_n|^2    \bigg)^{\frac{1}{2}}\\
& \leq& C_\omega \frac{1}{2^{K(\alpha-1)/2}} 
\bigg( \sum_{j=2^{2^K-1}}^{2^{2^{K+1}-1}-1} (\ln j)^{\a}  j^{\gamma(d)}\max_{k\in I(j)}|c_n|^2    \bigg)^{\frac{1}{2}}.
\end{eqnarray*}  
Now by \eqref{cond}
 and by the Cauchy-Schwarz inequality since $\alpha>1$  holds we get 
 \begin{equation*}
\sum_{K\geq 1}  \|u_K^\omega \|_{L^\infty(\R^d)} <+\infty \quad \mbox{a.s.}
 \end{equation*}
As a consequence, we have shown that a sub-sequence of the partial sum converges uniformly, a.s. This implies that $u^{\om}$ is a continuous and bounded function, a.s.

Now if we moreover assume that the $(X_n)$ are symmetric, we can apply \cite[Theorem II.5, p.120]{queff} which yields that 
\begin{equation*}
u_{\lambda}^\omega=\sum_{\lambda_n \leq \lambda} c_n X_n(\omega) \phi_n,
\end{equation*}
also converges in $L^{\infty}(\R^d)$, a.s.  for  $\lambda\rightarrow +\infty$.
\subsection{Proof  of Theorem~\ref{hold}}
The  proof  will  follow the  proof  of J.-P. Kahane \cite[Theorem 2, p.~66]{JPK}, with  the necessary  modifications in our  context.

Let $\kappa\geq 1$ and let  us introduce  the notations:
$$
\nu_j = \kappa 2^{j-1},\; N_j =2^{\nu_j},
$$
 \begin{eqnarray*}
 u^\omega_0(x) & = & \sum_{\lambda_n < N_1}c_nX_n(\omega)\varphi_n(x),\;\; {\rm for }\;\;j\geq 1, \\
 u^\omega_j(x) &= &\sum_{N_j\leq \lambda_n < N_{j+1}}c_nX_n(\omega)\varphi_n(x).
 \end{eqnarray*}
Using  the  triangle  inequality  and  the fundamental  calculus theorem  we have
\begin{equation*}
m_{u^\omega}(h) \leq h\Vert\nabla_xu^\omega_0\Vert_{L^\infty(\R^d)}
+  2\sum_{1\leq j< +\infty}\Vert u^\omega_j\Vert_{L^\infty(\R^d)}.
\end{equation*}
From  Theorem~\ref{salem-zy}  we have for $j\geq 1$, 
\beq\label{Tj}
\P\left[ \Big \Vert u^\omega_j\Vert_{L^\infty(\R^d)}\geq C(\ln N_{j+1})^{1/2}
\Big(\sum_{N_j/2\leq \ell<N_{j+1}/2}  \ell^{\gamma(d)}\max_{n\in I(\ell)}\vert c_n\vert^2\Big)^{1/2}\right] \leq \frac{1}{N_{j+1}^2}.
\eeq
The $j=0$  term  satisfies  the  following 
\begin{lemm}\ph\label{deriv}
There  exists  $C>0$ large  enough   such  that 
\beq\label{T0}
\P\left[ \big \Vert \nabla_xu^\omega_0\big\Vert_{L^\infty(\R^d)} \geq C(\ln N_{1})^{1/2}
\Big(\sum_{\ell<N_{1}/2}  \ell^{1+\gamma(d)}\max_{n\in I(\ell)}\vert c_n\vert^2\Big)^{1/2}\right] \leq \frac{1}{N_{1}^2}.
\eeq

\end{lemm}
The lemma  will  be proved  later.

\begin{rema}
More generally, we can get a similar bound for $a(x,D)u^{\om}_{0}$, when $a(x,\xi)$ is a polynomial in $(x,\xi) \in \R^{2d}$. We leave the details to the reader.
\end{rema}
  Using  this  lemma  we can  prove Theorem  \ref{hold}.\\
  Let  us denote $\Omega_j(\kappa)$ the event in (\ref{Tj}), $\Omega_0(\kappa)$
   the event in (\ref{T0})  and $\di{\Omega^{\infty}(\kappa) = \bigcup_{j\geq 0}\Omega_j(\kappa)}$.\\
   Using  the definition  of $N_j$  we have
   $$
   \P\Big[\Omega^{\infty}(\kappa)\Big]  \leq 2^{1-2\kappa}.
   $$
   Hence  using  the Borel-Cantelli lemma  we get  that 
   $$
  \P\Big[ \limsup_{\kappa\rightarrow +\infty}\Omega^{\infty}(\kappa) \Big]= 0.
   $$
   On the  other  side  denote by
\begin{eqnarray*}
  E_0 &=& (\ln N_{1})^{1/2}
\Big(\sum_{\ell<N_{1}/2}  \ell^{1+\gamma(d)}\max_{n\in I(\ell)}\vert c_n\vert^2\Big)^{1/2}\\
 E_j &=& (\ln N_{j+1})^{1/2}
\Big(\sum_{N_j/2\leq \ell<N_{j+1}/2}  \ell^{\gamma(d)}\max_{n\in I(\ell)}\vert c_n\vert^2\Big)^{1/2}.
\end{eqnarray*}
 Using  assumption (\ref{condh})   we  have
 \begin{eqnarray*}
E_0 &\leq &   (\ln N_{1})^{1/2}
\Big(\sum_{\ell=1}^{\kappa-1}\sum_{k=2^{\ell-1}}^{2^{\ell}}  k^{1+\gamma(d)}\max_{n\in I(k)}\vert c_n\vert^2\Big)^{1/2}\\
&\leq &  C \kappa^{1/2}
\Big(\sum_{\ell=1}^{\kappa-1} 2^{\ell(1+\gamma(d))} \sum_{k=2^{\ell-1}}^{2^{\ell}}   \max_{n\in I(k)}\vert c_n\vert^2\Big)^{1/2}\\
&\leq&  C\kappa^{1/2} \bigg( \,\sum_{  \ell=1}^{\kappa-1}2^{\ell(1-\mu)}\ell^{2\nu}\,\bigg)^{1/2}, 
\end{eqnarray*}
and for all $j\geq 1$
\begin{eqnarray*}
E_j &\leq &C (\kappa 2^j)^{1/2}  \bigg(\sum_{\nu_j\leq \ell < \nu_{j+1}}         \sum_{k=2^{\ell-1}}^{2^{\ell}}  k^{\gamma(d)}\max_{n\in I(k)}\vert c_n\vert^2 \bigg)^{1/2} \\
&\leq &C (\kappa 2^j)^{1/2}  \bigg(\sum_{\nu_j\leq \ell < \nu_{j+1}}    2^{\ell \gamma(d)}     \sum_{k=2^{\ell-1}}^{2^{\ell}}   \max_{n\in I(k)}\vert c_n\vert^2 \bigg)^{1/2} \\
&\leq &C (\kappa 2^j)^{1/2}  \bigg(\sum_{\nu_j\leq \ell < \nu_{j+1}}    2^{-\mu \ell}   \ell^{2\nu}  \bigg)^{1/2}.
\end{eqnarray*}
 
$\bullet$ Assume  that $0 < \mu < 1$. 
We easily   compute the following estimates
$$ E_0 \leq C \kappa^{\frac{1}{2} +\nu} 2^{(1-\mu)\kappa/2},\qquad  E_{j}\leq   C(\kappa2^j)^{\frac{1}{2}+\nu}2^{-\mu\kappa 2^{j-2}}
$$
and 
$$  
\sum_{j\geq 1}E_j \leq C\sum_{j\geq 1}(\kappa2^j)^{\frac{1}{2}+\nu}2^{-\mu\kappa 2^{j-2}}
\leq C\kappa^{\frac{1}{2} +\nu}2^{-\kappa\mu/2}.
$$
Now taking  $h=h_\kappa= 2^{-\kappa}$ we  have  proved  that   for  every 
$\di{\omega \notin \limsup_{\kappa\rightarrow +\infty}\Omega^{\infty}(\kappa)}$  and  \\ for
 every   $\kappa$
 large  enough 
$$
m_{u^\omega}(h_\kappa) \leq Ch_\kappa^\mu\vert\ln(h_\kappa)\vert^{\frac{1}{2}+\nu}.
$$
Using that  $m_{u^\omega}(h)$  is  non  increasing  in $h$  we  have proved Theorem~\ref{hold}   for  $0< \mu <1$.  

$\bullet$ Assume  that $\mu=0$ and $\nu<-1$. Then in this case we get 
 $$ E_0 \leq C \kappa^{\frac{1}{2} +\nu} 2^{\kappa/2},\qquad  E_{j}\leq   C\kappa^{\nu+1} 2^{(1+\nu)j}, \quad \sum_{j\geq 1}E_j \leq C\kappa^{\nu+1},
$$
and the end of the proof is similar.

$\bullet$  The  other  cases  are  proved   in the  same  way   (see \cite{JPK}) excepted  the last one ($\mu =1$,
$\nu < -1$)  where  the  result  is obtain by applying  Theorem~\ref{propPZ}
  to  the  partial  derivatives  $\partial_{x_j}u^\omega$, $1\leq j\leq d$.
  \qed
\ligne

Now  we   prove the  Lemma~\ref{deriv}.
\begin{proof}[Proof of Lemma~\ref{deriv}]
It  is more   convenient here  to index  the   Hermite  basis  by $\N^d$. So  we have
$$
u_0^\omega(x) = \sum_{2\vert\alpha\vert +d \leq N_1}c_\alpha X_\alpha(\om)\varphi_\alpha(x)
$$
where  we have denoted $\vert\alpha\vert  = \alpha_1 +\cdots +\alpha_d$.
We  have $H\varphi_\alpha = \lambda_\alpha\varphi_\alpha$,  with
$\lambda_\alpha = 2\vert\alpha\vert  + d$. 
It is  easier  to  consider  first  the  tensor  basis: 
$$
\varphi_\alpha(x) = \mathfrak h_\alpha (x)= \mathfrak h_{\alpha_1}(x_1)\cdots  \mathfrak h_{\alpha_d}(x_d).
$$
Recall that in 1D the Hermite functions satisfy   for all $t\in \R$
$$
 \frac{d}{dt}\mathfrak h_k(t) = 2^{-1/2}\left(\sqrt k\mathfrak h_{k-1}(t) - \sqrt{k+1}\mathfrak h_{k+1}(t)\right). 
 $$
So  we get
$$
\sqrt 2\partial_{x_1}u_0^\omega(x) =\sum_{2\vert\alpha\vert +d \leq N_1}\sqrt{\alpha_1}c_\alpha X_\alpha(\om)\mathfrak h_{\alpha-e_1}(x)
-\sum_{2\vert\alpha\vert +d \leq N_1}\sqrt{\alpha_1+1}c_\alpha X_\alpha(\om)\mathfrak h_{\alpha+e_1}(x)
$$
where $\{e_j\}_{1\leq j\leq d}$ is the  canonical basis  of $\R^d$. Applying   the Theorem~\ref{salem-zy}  to each term  of the sum  we   have  proved the Lemma~\ref{deriv}
for the  tensor basis~$\mathfrak h_\alpha$.

For a  general  orthonormal basis $(\varphi_\alpha)_{\a \in \N^{d}}$  of Hermite  functions,   we   write   
$$
\varphi_\alpha(x) = \sum_{\vert\alpha\vert = \vert\beta\vert}t_{\alpha, \beta}\mathfrak h_\beta(x)
$$
where $\{t_{\alpha, \beta}\}$ is  a unitary  matrix.    So  we have
\begin{multline*}
\sqrt 2\partial_{x_1}u_0^\omega(x) =\\
=\sum_{2\vert\alpha\vert +d \leq N_1}\sqrt{\alpha_1}c_\alpha X_\alpha(\om)\sum_{\vert\beta\vert = \vert\alpha\vert }t_{\alpha, \beta}\mathfrak h_{\beta-e_1}(x)
-\sum_{2\vert\alpha\vert +d \leq N_1}\sqrt{\alpha_1+1}c_\alpha X_\alpha(\om)\sum_{\vert\beta\vert = \vert\alpha\vert}t_{\alpha, \beta}\mathfrak h_{\beta+e_1}(x). 
\end{multline*}
Now we estimate   separately the  two sums by revisiting the  proof  of Theorem~\ref{salem-zy}.\\
It is enough to consider  the   first one  denoted $v_\lambda^\omega(x)$ where $\lambda =N_1$.
We have  to estimate:
\begin{equation*}
\E\Big[\exp(rv_\lambda^\omega)\Big] \leq \exp\Big(Cr^2\sum_{\lambda_\alpha \leq \lambda}\alpha_1\vert c_\alpha\vert^2
\sum_{\vert\beta\vert =\vert\alpha\vert}\left\vert t_{\alpha, \beta}\mathfrak h_{\beta-e_1}\right\vert^2\Big).
\end{equation*}
For $\lambda_\alpha\in I(j)$  we have  
\begin{equation*}
\sum_{\lambda_\alpha \in I(j)}\alpha_1\vert c_\alpha\vert^2
\big\vert \sum_{\vert\beta\vert =\vert\alpha\vert} t_{\alpha, \beta}\mathfrak h_{\beta-e_1}(x)\big\vert^2
 \leq   
 C(1+j)\max_{\lambda_\alpha\in I(j)}|c_\alpha|^{2}\Big\vert
 \sum_{\substack{\vert\beta\vert =\vert\alpha\vert\\ \lambda_\alpha\in I(j)}}t_{\alpha, \beta}\mathfrak h_{\beta-e_1}(x)\Big\vert^2.
\end{equation*}
Using that the matrix $\{t_{\alpha, \beta}\}$  is unitary we  have 
$$
\Big\vert
 \sum_{\substack{\vert\beta\vert =\vert\alpha\vert\\ \lambda_\alpha\in I(j)}}t_{\alpha, \beta}\mathfrak h_{\beta-e_1}(x)\Big\vert^2 \leq \sum_{\lambda_\alpha\in I(j-1)}|\mathfrak h_\alpha(x)|^{2} \leq C  j^{\gamma(d)}
$$
and 
$$
\E\Big[\exp(rv_\lambda^\omega)\Big]  \leq \exp\Big(Cr^2\sum_{\lambda_\alpha \leq \lambda}\lambda^{\gamma(d)+1}_{\alpha}\vert c_{\alpha}\vert^2
\Big).
$$
This inequality has the same form as \eqref{inega}, hence we can conclude here as in the proof of Theorem~\ref{salem-zy}.
\end{proof}


 \section{Annex: about random series in Banach spaces}\label{annex-proba}

We present here some elements on the theory of random series in Banach spaces. We refer the reader to the books \cite{HJ}, \cite{DJT} and \cite{queff} for more elements on this subject.

Let $B$ be a Banach space on the field of real or complex numbers.
Let $(\varepsilon_n)_{n\geq 0}$ be a sequence of Rademacher i.i.d. random variables and let us define
\begin{equation}\label{defi-ce}
\Sigma(B):=\left\{(b_n)_{n\geq 0}, \quad \sum \varepsilon_n b_n \quad \mbox{converges a.s.} \right\}
\end{equation}
It is clear that $\Sigma(B)$ is a vector subspace of $B^{\N}$.
The following theorem is well-known in the theory of Banach random series (see for instance \cite[Chapitre 3,   IV.2]{queff}): 
\begin{theo}\ph\label{radem-seri}
Let $B$ a Banach space and consider a sequence $(b_n)_{n\geq 0}$ in $B$. The following facts are equivalent 
\begin{enumerate}[(i) ]
\item the sequence $(b_n)_{n\geq 0}$ belongs to $\Sigma(B)$,
\item the random series $\sum \epsilon_n(\omega) b_n$ converges in probability,
\item the random series $\sum \epsilon_n(\omega) b_n$ converges in law,
\item there is some $p\geq 1$ such that the random series $\sum \epsilon_n (\omega) b_n$ converges in $L^p(\Omega,B)$,
\item for any $p\geq 1$, the random series $\sum \epsilon_n (\omega) b_n$ converges in $L^p(\Omega,B)$.
\end{enumerate}
\end{theo}
For instance, if $B$ is a Hilbert space, the previous theorem can be used to see that $\Sigma(B)$ is nothing else than $\ell^2(B)$ (see also \cite[Chapter 3]{JPK}).

A natural question is to study what happens for the almost sure convergence of $\sum X_n b_n$ if $(X_n)_{n\geq 0}$ is i.i.d. with another reference law.
A part of this question is solved by the following result proved by Hoffman-Jorgensen.

\begin{theo}[Hoffman-Jorgensen]\ph\label{quef}
Let $(X_n)_{n\geq 0}$ be a sequence of real, non-constant and i.i.d. random variables and $(b_n)_{n\geq 0}$ be a sequence which takes values in a general Banach space $B$,
 we assume that the series $\sum X_n(\omega) b_n$ converges almost surely in $B$.
Then the series $\sum \epsilon_n(\omega) b_n$ converges almost surely in $B$, in other words $(b_n)_{n\geq 0}$ belongs to $\Sigma(B)$. 
\end{theo}
We emphasize the fact that no integrability assumption is made on the law of $X_n$.
We do not know any published reference of Theorem~\ref{quef} and we give below a proof we learned from Herv\'e Queff\'elec.
The converse question is not easy and needs assumptions on the geometry of the Banach space $B$.
It is worthwhile now to recall Kahane-Khintchine's inequalities.
For any real numbers $q,p\geq 1$ and any finite sequence $(b_n)_{n\geq 0}$ in $B$ there is a constant $K(p,q)$ which depends only on $p$ and $q$ such that 
\begin{equation}\label{khin}
\E\bigg[ \Big\Vert \sum_{n\geq 0} \epsilon_n b_n \Big\Vert^q \bigg]^{1/q} \leq K(p,q)
\E\bigg[ \Big\Vert \sum_{n\geq 0} \epsilon_n b_n \Big\Vert^p \bigg]^{1/p}.
\end{equation}
For the specific case $B=\R$, those inequalities are called Khinthine's inequalities and we have 
\begin{equation*}
\E\bigg[ \Big\vert \sum_{n\geq 0} \epsilon_n b_n \Big\vert^2 \bigg]^{1/2}=\left(\sum_{n\geq 0} |b_n|^2\right)^{1/2}.
\end{equation*}

We can now define the notion of cotype of a Banach space.

\begin{defi}\ph
A Banach space $B$ has cotype $p\geq 2$ if there are real numbers $q\geq 1$ and $C_q>0$ such that for any finite sequence $(b_n)_{n\geq 0}$ in $B$ one has 
\begin{equation}\label{cotype}    \bigg(\sum_{n\geq 0} \|b_n\|^p \bigg)^{1/p}\leq C_q \E\bigg[ \Big\Vert \sum_{n\geq 0} \epsilon_n b_n \Big\Vert^q \bigg]^{1/q}     . \end{equation}
\end{defi}

Thanks to \eqref{khin}, notice that if \eqref{cotype} holds then it holds for any $q\geq 1$.
For instance, one can prove that for any $p\geq 1$ the Banach space $B:=L^p(\R^d)$ has cotype $\max(2,p)$.
To see this, we can make use of Kahane-Khintchine's inequalities for $q=p$:
\begin{equation*}
\begin{array}{rcl} \E \left[ \Big\Vert \Sum_{n=1}^N \epsilon_n f_n \Big\Vert_{L^p(\R^d)}^p      \right] &= &    \Int_{\R^d} \E \left[\Big\vert \Sum_{n=1}^N \epsilon_n(\omega) f_n(t)\Big\vert^p \right] dt   \\[3mm]
&\sim & C_p  \Int_{\R^d} \Big(\sum_{n=1}^N |f_n(t)|^2 \Big)^{\frac{p}{2}}dt . \end{array}
\end{equation*}
In the case $p\leq 2$, by denoting $\|\cdot\|_{2/p}$ the obvious norm of $\R^{N}$, we can write
\begin{equation*}
\begin{array}{rcl} \Int_{\R^d} \Big(\sum_{n=1}^N  |f_n(t)|^2 \Big)^{\frac{p}{2}}dt  &= & \Int_{\R^d} \Big\Vert \big(|f_1(t)|^p,\dots,|f_N(t)|^p\big) \Big\Vert_{2/p} dt   \\[3mm]
& \geq & \left\Vert \Int_{\R^d} (|f_1(t)|^p,\dots, |f_N(t)|^p) dt \right\Vert_{2/p} \\[3mm]
& \geq & \Big(\Sum_{n=1}^N \|f_n\|_{L^p(\R^d)}^2 \Big)^{\frac{p}{2}} .  \end{array}
\end{equation*}
In the case $p\geq 2$, we write 
\begin{equation*}
\Int_{\R^d} \Big(\sum_{n=1}^N |f_n(t)|^2 \Big)^{\frac{p}{2}}dt   \geq     \Int_{\R^d} \sum_{n=1}^N |f_n(t)|^p dt   = \Sum_{n = 1}^{N} \|f_n\|_{L^p(\R^d)}^p .
 \end{equation*}

As used in \cite{Grivaux} for Gaussian random variables, we have the following  astonishing result of Maurey and Pisier:

\begin{theo}[Maurey-Pisier] \ph\label{MP}
The following assertions are equivalent
\begin{enumerate}[(i) ]
\item the Banach space $B$ has finite cotype (that means that there is $p\geq 2$ such that $B$ has cotype\;$p$),
\item for any sequence  $(b_n)_{n\geq 0}$ of $B$, the almost sure convergence of $\sum \epsilon_n b_n$ implies the almost sure convergence of $\sum G_n b_n$,
where $(G_n)_{n\geq 0}$ is a sequence of i.i.d. $\mathcal{N}_\R(0,1)$ Gaussian random variables,
\item for any sequence $(b_n)_{n\geq 0}$ of $B$ the almost sure convergence of $\sum \ep_n b_n$ implies the almost sure convergence of $\sum X_n b_n$
where $(X_n)_{n\geq 0}$ is any sequence of real, centered and i.i.d random variables with finite moments of any order.
\end{enumerate}
\end{theo}
\begin{proof}
The equivalence $(i)$ $\Leftrightarrow$ $(ii)$ is done in \cite[Corollaire 1.3]{maurey-pisier76}.
Obviously, $(iii)$ $\Rightarrow$ $(ii)$ is true by choosing $X_n=G_n$.
Let us explain arguments which are not explicitly written in \cite[Corollaire 1.3]{maurey-pisier76}.
To see $(i)$ $\Rightarrow$ $(iii)$, we begin by assuming that the random variables $X_n$ are symmetric. The proof of \cite[Corollaire 1.3, a) $\Rightarrow$ b), page 69]{maurey-pisier76} shows that there is a positive constant $C$ which involves a moment $\E[|X_1|^q]$ (for some $q>0$) such that for any sequence $(b_n)_{n\geq 0}$ we have 
\begin{equation*}
\forall k,\ell\geq 1 \quad  \E \left[ \Big\Vert \sum_{n= k}^\ell X_n  b_n  \Big\Vert^2    \right] \leq C \E \left[ \Big\Vert \sum_{n= k}^\ell \ep_n  b_n  \Big\Vert^2    \right].
\end{equation*}
Since the series $\sum \ep_n b_n$ converges almost surely, it converges in $L^2(\Omega,B)$ (see Theorem~\ref{radem-seri}), so does $\sum X_n b_n$.
Now assume that $X_n$ are merely centered.
Clearly, $Z_n(\omega,\omega')=X_n(\omega)-X_n(\omega')$ is symmetric on the probability space $\Omega\times \Omega'$.
Therefore, the previous analysis shows that  $\sum Z_n(\omega,\omega') b_n$ converges in $L^2(\Omega\times \Omega',B)$ and also in $L^1(\Omega\times \Omega',B)$.
Now we use that random variables $X_n$ are centered: 
\begin{equation*}
\forall \ell\geq k \quad \E_\omega \left[ \Big\Vert    \Sum_{n= k}^\ell X_n(\omega) b_n  \Big\Vert   \right] \leq 
\E_{\omega,\omega'} \left[ \Big\Vert    \Sum_{n= k}^\ell X_n(\omega) b_n -X_n(\omega') b_n   \Big\Vert   \right] .
\end{equation*}
That means that $\sum X_n b_n$ converges in $L^1(\Omega,B)$, so converges in probability and almost surely in $B$ (see \cite[Th\'eor\`eme II.3]{queff}).
\end{proof}

\subsection{Proof of Proposition~\ref{equi-cv}}

Equivalence of $(i)$ and $(ii)$ comes from Theorem~\ref{quef}, Theorem~\ref{MP} and the fact that $L^p(\R^d)$ has finite cotype.
In view to check the link with $(iii)$, it is necessary and sufficient to study convergence in $L^p(\Omega,L^p(\R^d))$ (see Theorem~\ref{radem-seri}).
Cauchy criterion leads to handle terms of the following form: 
\begin{equation*}  \int_{\Omega}  \int_{\R^d} \Big\vert \sum_{n= k}^\ell \ep_n(\omega) f_n(x)    \Big\vert^p d\P(\omega) dx =\int_{\R^d} \E_\omega \left[
\Big\vert \sum_{n= k}^\ell \ep_n(\omega) f_n(x)  \Big\vert^p \right] dx. \end{equation*}
By Khintchine's inequalities \eqref{khin}, there exists $C_p\geq 1$ so that 
\begin{equation*}
\frac{1}{C_p} \int_{\R^d} \Big\vert \sum_{n= k}^\ell |f_n(x)|^2 \Big\vert^{p/2} dx\leq  \int_{\Omega}  \int_{\R^d} \Big\vert \sum_{n= k}^\ell \ep_n(\omega) f_n(x)    \Big\vert^p d\P(\omega) dx \leq C_p  \int_{\R^d} \Big\vert \sum_{n= k}^\ell |f_n(x)|^2 \Big\vert^{p/2} dx,
\end{equation*}
and we conclude easily.
 
\subsection{Proof of Theorem~\ref{quef}}

We need the contraction principle (see for instance \cite[Th\'eor\`eme III.1]{queff} or \cite[Chapter 2.6 in the Rademacher framework]{JPK}) and a few lemmas.

\begin{theo}[contraction principle]\ph\label{theocontrac}
Let $(X_n)_{n\geq 0}$ be a sequence of symmetric independent random variables which takes values in a Banach space $B$.
If $\sum X_n$ converges almost surely in $B$ then, for any bounded real sequence $(\lambda_n)_{n\geq 0}$, the series $\sum \lambda_n X_n$ converges almost surely in $B$.
\end{theo}

Let us recall a classical lemma in the probability theory.
\begin{lemm}\ph\label{non-cst}
Let $X$ be a real random variable, the following statements are equivalent: 
\begin{enumerate}[(i) ]
\item $X$ is not almost surely constant,
\item there is $\xi \in \R$ such that $|\E[\exp(i\xi X)]|<1$ holds,
\item the set $\big\{\xi\in \R, |\E[\exp(i\xi X)]|=1\big\}$ is countable.
\end{enumerate}
\end{lemm}
\begin{proof}
The implications $(iii)\Rightarrow (ii)$ and $(ii)\Rightarrow (i)$ are obvious.
Suppose now $(i)$ and let $\xi_0\neq \xi_1\in \R\backslash\{0\}$ be two numbers such that $|\E[\exp(i\xi_0 X)]|=|\E[\exp(i\xi_1 X)]|=1$.
Since $|\exp(i \xi_0 X)|\leq 1$ holds, the equality $|\E[\exp(i\xi_0 X)]|=1$ ensures there is $\alpha_0 \in \R$ such that one has $e^{i\xi_0 x}=e^{i\alpha_0}$ for $\mu$-almost all $x\in \R$ where $\mu$ is the law of $X$.
Hence, $x\in \frac{\alpha_0}{\xi_0}+\frac{2\pi}{\xi_0} \Z$ for $\mu$-almost all $x\in \R$.
The same is true by replacing $\xi_0$ with $\xi_1$ and $\alpha_0$ with $\alpha_1$.
Because $X$ is not constant almost surely, there are at least two numbers $x\neq y$ which both belong to $\big\{\frac{\alpha_0}{\xi_0}+\frac{2\pi}{\xi_0} \Z\big\} \cap \big\{ \frac{\alpha_1}{\xi_1}+\frac{2\pi}{\xi_1} \Z \big\} $.
We notice that $x-y\neq 0$ belongs to $\frac{2\pi}{\xi_0} \Z \cap \frac{2\pi}{\xi_1} \Z$.
Finally $\xi_0/\xi_1$ is rational and $(iii)$ is proved.
\end{proof}

\begin{lemm}\ph\label{lem-anx1}
For any sequence of real, non-constant and i.i.d. random variables $(Y_\ell)_{\ell\geq 1}$ we have 
\begin{equation*}
\lim\limits_{N\rightarrow +\infty} \P\big[ \,|Y_1+\dots+Y_N|\geq 1|\, \big]=1.
\end{equation*}
\end{lemm}
\begin{proof}
Let $\mu$ be the law of $Y_1$ and $\phi\in L^1(\R)$ be a function such that $\widehat{\phi}(x)\geq 1$ holds for any $x\in (-1,+1)$. 
It comes 
\begin{equation*}
\begin{array}{rcl} \P\big[\,|Y_1+\dots+Y_{N}|<1\,\big]&  =&  \Int_{\R} \pun_{(-1,1)}(x) d \overbrace{\mu\star\dots \star \mu}^{N \mbox{ times} }(x) \\[3mm]
& \leq & \Int_{\R} \widehat{\phi}(x) d \mu\star\dots \star \mu (x) =\Int_{\R} \phi(\xi) \widehat{\mu}(\xi)^N d\xi. \end{array}
\end{equation*}
Point $(iii)$ of Lemma~\ref{non-cst} ensures that $|\widehat{\mu}(\xi)|<1$ holds for almost all $\xi$ in the sense of Lebesgue.
We conclude by the dominated convergence theorem if $N$ tends to infinity.
\end{proof}

\begin{lemm}\ph\label{lem-anx2}
Let $G$ be a locally compact Abelian group, consider a subgroup $G_0\subset G$ which has a positive Haar measure and is everywhere dense.
Then $G_0$ is the whole group $G$.
\end{lemm}
\begin{proof}
It is sufficient to prove that $G_0$ is closed.
Steinhaus theorem states that $G_0-G_0\subset G_0$ contains an open neighbourhood of the origin.
By using translations of $G_0$, it turns out that $G_0$ is an open subgroup of $G$.
A classical argument from the theory of topological groups asserts that $G_0$ is also closed: we just write $G=\sqcup_{i\in I} (G_0+g_i)$
 where $(g_i)_{i\in I}$ is a family of elements of $G$ and $g_i=0$ for one $i\in I$, it appears that the complementary subset of $G_0$ is open.
\end{proof} 

We can now prove Theorem~\ref{quef}.

\begin{proof}[Proof of Theorem~\ref{quef}]

\textbf{Step 1.} 
It is well known that we can realize any sequence of independent real random variables on the probability space $[0,1]$ endowed with the Lebesgue measure \cite{williams} (p. 34 and p.43).
For any $n\geq 0$, we consider a sequence $(\widehat{Z}_{n,\ell})_{\ell\geq 1}$ of i.i.d. random variables on $[0,1]$ and such that $X_n=\widehat{Z}_{n,0}$ for any $n\geq 0$.
The following random variables 
\begin{equation*}
\begin{array}{rcl}
Z_{n,\ell}: [0,1]^\N & \rightarrow & \R \\
(\omega_0,\omega_1,\dots)& \mapsto & \widehat{Z}_{n,\ell}(\omega_n)
\end{array}
\end{equation*}
are i.i.d with the same law than the random variables $X_n$.
The assumption of Theorem~\ref{quef}
ensures that the series 
$\sum_{n\geq 0} Z_{n,\ell}(\omega) b_n=\sum_{n\geq 0} \widehat{Z}_{n,\ell}(\omega_n) b_n$ converges in $B$ almost surely in $\omega \in [0,1]^\N$.
By combining Lemma~\ref{non-cst} and the equations
\begin{equation*}
\forall \ell\geq 1 \quad \E\big[\exp(i\xi Z_{n,2\ell-1}-i\xi Z_{n,2\ell}  )\big]=\big|E[\exp(i\xi X_1)]\big|^2,\end{equation*}
we see that $Z_{n,2\ell-1}-Z_{n,2\ell}$ is not constant almost surely.
By using Lemma~\ref{lem-anx1} with the sequence $Y_\ell=Z_{n,2\ell-1}-Z_{n,2\ell}$, we see that there is an integer $N\geq 1$ which depends only on the law of $X_1$ such that 
\begin{equation*}
\frac{1}{2}\leq  \P\left[ |Z_{n,1}-Z_{n,2}+\dots +Z_{n,2N-1}-Z_{n,2N} |\geq 1 \right] \quad \mbox{ and is independent of } n.
\end{equation*}
By setting $S_n:=Z_{n,1}-Z_{n,2}+\dots +Z_{n,2N-1}-Z_{n,2N}$, we have the three properties: 
\begin{enumerate}[$(i)$ ]
\item the series $\sum_{n\geq 0} S_n b_n$ converges almost surely in $B$ in $\omega \in [0,1]^\N$,
\item $(S_n)_{n\geq 0}$ is a sequence of real, non-constant, \textbf{symmetric} and i.i.d. random variables,
\item for any $n\geq 0$ one has $\P[ |S_n| \geq  1 ] \geq \frac{1}{2}$.
\end{enumerate}
By construction, $S_n(\omega)=\widehat{S_n}(\omega_n)$ with $\widehat{S_n}:=
\widehat{Z}_{n,1}-\widehat{Z}_{n,2}+\dots +\widehat{Z}_{n,2N-1}-\widehat{Z}_{n,2N}$.

\textbf{Step 2.} 
On the probability space $[0,1]^{\N}\times [0,1]$, one checks that the sequence $( S_n(\omega)\epsilon_n(\omega'))_{n\geq 0}$ is i.i.d. and has the same common law than $S_1$.
From $(i)$ and $(ii)$, the series $\sum S_n(\omega) \epsilon_n(\omega') b_n$ converges
almost surely in $(\omega,\omega')\in [0,1]^{\N}\times [0,1]$.
Fubini's theorem ensures that almost surely in $\omega\in [0,1]^\N$ the sequence $(S_n(\omega) b_n)_{n\geq 0}$ belongs to $\Sigma(B)$ (see definition \eqref{defi-ce}).
Since $S_n(\omega)=\widehat{S}_n(\omega_n)$, we also have $\P(|\widehat{S}_n|\geq 1)=\P(|S_n|\geq 1)\geq \frac{1}{2}$.
Thus, we can consider a Borel subset $A_n\subset [0,1]$ such that 
\begin{equation*}\P(A_n)=\frac{1}{2} \quad \mbox{ and } \quad A_n\subset \big\{\omega_n\in [0,1] ,\quad |\widehat{S_n}(\omega_n)| \geq 1] \big\}.\end{equation*}
Let us define $\rho_n(\omega):=\pun_{A_n}(\omega_n)\leq |S_n(\omega)|$ for each $\omega \in [0,1]^\N$.
It is obvious that $(\rho_n)_{n\geq 0}$ is a sequence of i.i.d. random variables with the $\frac{1}{2}$-Bernoulli law.
From the contraction principle (Theorem~\ref{theocontrac}), we know that almost surely in~$\omega$ the sequence $(\rho_n(\omega) b_n)_{n\geq 0}$ belongs to $\Sigma(B)$.

\textbf{Step 3.} Let us identify $\Z /2\Z$ with $\{0,1\}$ and introduce the compact group $G:=\left(\Z/2\Z \right)^{\N}$ which becomes now our reference probability space.
It is clear that the maps $g\in G\mapsto g_n \in \{0,1\}$ seen as random variables 
are independent and identically distributed with a $\frac{1}{2}$-Bernoulli law.
Let us define $G_0\subset G$ the subset of elements $(g_n)_{n\geq 0}$ such that $(g_n b_n)_{n\geq 0}$ belongs to $\Sigma(B)$.
Since $\Sigma(B)$ is a vector space, $G_0$ is a subgroup of $G$.
We directly get from the previous analysis in Step 2 that $G_0$ has a full Haar measure in $G$.
Furthermore, $G_0$ contains obviously the everywhere dense subgroup of $G$ of elements $(g_n)_{n\geq 0}$ which satisfy $g_n=0$ for $n\gg 1$.
We use Lemma~\ref{lem-anx2} to conclude that $(1,1,\dots)$ belongs to $G_0$, in other words $(b_n)_{n\geq 0}$ belongs to $\Sigma(B)$.
\end{proof}

\section{Annex: An alternative proof of Theorem~\texorpdfstring{\ref{propPZ} }{2.6} inspired by \texorpdfstring{\cite{Tzvetkov4}}{[22]}}\label{sect6}

We give here a different proof of Theorem~\ref{propPZ} we learnt from \cite{Tzvetkov4}, which we decided to detail   for pedagogical reasons.

\begin{lemm}\ph \label{lem37} Let  $(\varphi_n)_{n\geq 0}$ be any  Hilbertian basis   of  eigenfunctions for the harmonic  oscillator~$H$.  Let $\gamma(1)=-1/6$ and $\gamma(d)=d/2-1$ for $d\geq 2$. Then  for all $j\geq 1$ and $x,y\in \R^d$ we have 
\begin{equation*}
\Big(\sum_{n\in I(j)}|\phi_{n}(y)-\phi_{n}(x)|^{2}\Big)^{1/2}\leq C |y-x| j^{\gamma(d)/2+1/2}.
\end{equation*}
\end{lemm}

\begin{proof}
By the Taylor formula and Cauchy-Schwarz we get, for $n\in I(j)$
\begin{eqnarray}
|\phi_{n}(y)-\phi_{n}(x)|^{2} &\leq & |y-x|^{2}\Big( \int_{0}^{1}\big|\nabla \phi_{n}\big(x+(y-x)t\big)\big|dt\Big)^{2}\nonumber\\
&\leq & |y-x|^{2} \int_{0}^{1}\big|\nabla \phi_{n}\big(x+(y-x)t\big)\big|^{2}dt\nonumber \\
&\leq &Cj |y-x|^{2} \int_{0}^{1}\big| \phi_{n}\big(x+(y-x)t\big)\big|^{2}dt,\label{ineqn}
\end{eqnarray}
where in the last line we used 
\begin{equation*}
\int_{0}^{1}\big|\nabla \phi_{n}\big(x+(y-x)t\big)\big|^{2}dt\leq \int_{0}^{1}\big| H^{1/2}\phi_{n}\big(x+(y-x)t\big)\big|^{2}dt=\lambda_{n}\int_{0}^{1}\big| \phi_{n}\big(x+(y-x)t\big)\big|^{2}dt.
\end{equation*}
Now we sum up the inequalities \eqref{ineqn} and get with \eqref{kt}
\begin{equation*}
 \sum_{n\in I(j)}|\phi_{n}(y)-\phi_{n}(x)|^{2} \leq Cj  |y-x|^{2}  \sup_{z\in \R} \sum_{n\in I(j)} | \phi_{n}\big(z\big)\big|^{2}\leq Cj^{\gamma(d)+1}  |y-x|^{2},
\end{equation*}
which was the claim
\end{proof}

We follow the main lines of the proof of N. Tzvetkov \cite[Theorem 5]{Tzvetkov4}. We define the pseudo-distance\;$\delta$ by 
\begin{equation*}
\delta(x,y)=\Big(\sum_{n\geq 0}|c_{n}|^{2}|\phi_{n}(y)-\phi_{n}(x)|^{2}\Big)^{1/2}.
\end{equation*}
 For $\a>1$, we define the function $\Phi_{\a}: (0,+\infty)\longrightarrow  (0,+\infty)$
\begin{equation*} 
\Phi_{\a}(t)=\left\{\begin{array}{ll} 
(-\ln t)^{\a/2} \quad &\text{if} \quad 0<t<1/a, \\[6pt]  
\Phi_{\a}(1/a)   &\text{if} \quad
t\geq 1/a,
\end{array} \right.
\end{equation*}
where $a>1$ is chosen in such a way that the function $t\mapsto t\Phi_{\a}(t)$ is increasing on $(0,+\infty)$. Observe also that  $t\mapsto \Phi_{\a}(t)$ is non-increasing on $(0,+\infty)$.  Then we have a result similar to \cite[Theorem 5]{Tzvetkov4}.

\begin{lemm} \ph \label{lemm38} Assume that the coefficients $(c_{n})$ satisfy \eqref{cond}, then 
\begin{equation*}
\delta(x,y)\leq \frac{C}{\Phi_{\a}(|y-x|)}.
\end{equation*}
\end{lemm}

\begin{proof}
We  clearly have 
\begin{equation*}
(\delta(x,y))^2\leq C \sum_{j=1}^{+\infty} \big(\max_{k\in I(j)}\vert  c_k\vert^2\big) \sum_{n\in I(j)}|\phi_{n}(y)-\phi_{n}(x)|^{2}.
\end{equation*}
We split the previous sum in two parts. Then, by Lemma~\ref{lem37}  
\begin{eqnarray}
I_1(x,y)&:=&\sum_{j:\, aj^{1/2}\leq \vert y-x\vert^{-1}}  \big(\max_{k\in I(j)}\vert  c_k\vert^2\big)\sum_{n\in I(j)}|\phi_{n}(y)-\phi_{n}(x)|^{2}\nonumber \\
&\leq& C \sum_{j:\, aj^{1/2}\leq \vert y-x\vert^{-1}} j^{\gamma(d)+1} \max_{k\in I(j)}\vert  c_k\vert^2|y-x|^{2}\nonumber \\
&=& \frac{C}{\Phi^2_{\a}(y-x)} \sum_{j:\, aj^{1/2}\leq \vert y-x\vert^{-1}}j^{\gamma(d)+1} \max_{k\in I(j)}\vert  c_k\vert^2\Big(|y-x| \Phi_{\a}(y-x)\Big)^{2}.\label{en3}
\end{eqnarray}
Now we use that the function $t\longmapsto t \Phi_{\a}(t)$ is increasing, thus for $aj^{1/2}\leq \vert y-x\vert^{-1}$ we have 
\begin{equation*}
\Big(|y-x| \Phi_{\a}(y-x)\Big)^{2} \leq j^{-1}(\ln j)^{\a},
\end{equation*}
therefore from \eqref{en3}  and the assumption \eqref{cond}  on the $c_n$, we get 
\begin{equation*}
I_1(x,y) \leq C \Phi^{-2}_{\a}(y-x)  \sum_{j=1}^{+\infty} j^{\gamma(d)} (\ln j)^{\a} \max_{k\in I(j)}\vert  c_k\vert^2\leq C \Phi^{-2}_{\a}(y-x). 
\end{equation*}
Next, by \eqref{kt}
\begin{eqnarray}
I_2(x,y)&:=&\sum_{j:\, aj^{1/2}> \vert y-x\vert^{-1}} \big(\max_{k\in I(j)}\vert  c_k\vert^2\big)\sum_{n\in I(j)}|\phi_{n}(y)-\phi_{n}(x)|^{2}\nonumber \\
&\leq & C \sum_{j:\, aj^{1/2}> \vert y-x\vert^{-1}} j^{\gamma(d)} \max_{k\in I(j)}\vert  c_k\vert^2.\label{en4}
\end{eqnarray}
Now we use that $\Phi_{\a}$ is non-increasing and for $aj^{1/2}> \vert y-x\vert^{-1}$ we get 
\begin{equation*}
\Phi_{\a}(y-x)\leq \Phi_{\a}(a^{-1}j^{-1/2})\leq C (\ln j)^{\a/2}.
\end{equation*}
As a consequence, from \eqref{en4} and the assumption \eqref{cond} on the $c_n$, we deduce that 
\begin{equation*}
I_2(x,y) \leq C \Phi^{-2}_{\a}(y-x)  \sum_{j=1}^{+\infty} j^{\gamma(d)} (\ln j)^{\a}  \max_{k\in I(j)}\vert  c_k\vert^2\leq C \Phi^{-2}_{\a}(y-x),
\end{equation*}
which completes the proof.
\end{proof}

\begin{proof}[Proof of Theorem~\ref{propPZ}]
It is enough to prove  that on every compact set $K\subset \R^d$,   a.e. in $\omega$, $u^\omega$ is continuous on~$K$.
Hence   we can follow  the proof given in  \cite[Theorem 5]{Tzvetkov4}  using an entropy argument (Dudley-Fernique criterion), together with the result of Lemma~\ref{lemm38}.
\end{proof}

\begin{rema} Let's compare the two different proofs. This proof relies on both a decomposition in space and in frequencies, while in the other proof one only needs a   decomposition  in frequencies. Observe also that in the first proof one moreover  gets that for almost all $\om \in \Omega$,  $u^{\om}$ is bounded.
\end{rema}

\textsc{Acknowledgments.} The  authors would like to thank warmly Herv\'e Queff\'elec for very interesting discussions about random Banach series, and in particular for the proof of Theorem~\ref{quef}.

\end{document}

%% file: fig4tex.tex
%
%
%
\ifx\figforTeXisloaded\relax \else\global\let\figforTeXisloaded=\relax\fi
\message{version 1.9}
\catcode`\@=11
\ifx\ctr@ln@m\undefined\else%
    \immediate\write16{*** Fig4TeX WARNING : \string\ctr@ln@m\space already defined.}\fi
\def\ctr@ln@m#1{\ifx#1\undefined\else%
    \immediate\write16{*** Fig4TeX WARNING : \string#1 already defined.}\fi}
\ctr@ln@m\ctr@ld@f
\def\ctr@ld@f#1#2{\ctr@ln@m#2#1#2}
\ctr@ld@f\def\ctr@ln@w#1#2{\ctr@ln@m#2\csname#1\endcsname#2}
{\catcode`\/=0 \catcode`/\=12 /ctr@ld@f/gdef/BS@{\}}
\ctr@ld@f\def\ctr@lcsn@m#1{\expandafter\ifx\csname#1\endcsname\relax\else%
    \immediate\write16{*** Fig4TeX WARNING : \BS@\expandafter\string#1\space already defined.}\fi}
\ctr@ld@f\edef\colonc@tcode{\the\catcode`\:}
\ctr@ld@f\edef\semicolonc@tcode{\the\catcode`\;}
\ctr@ld@f\def\t@stc@tcodech@nge{{\let\c@tcodech@nged=\z@%
    \ifnum\colonc@tcode=\the\catcode`\:\else\let\c@tcodech@nged=\@ne\fi%
    \ifnum\semicolonc@tcode=\the\catcode`\;\else\let\c@tcodech@nged=\@ne\fi%
    \ifx\c@tcodech@nged\@ne%
    \immediate\write16{}
    \immediate\write16{!!!=============================================================!!!}
    \immediate\write16{ Fig4TeX WARNING:}
    \immediate\write16{ The category code of some characters has been changed, which will}
    \immediate\write16{ result in an error (message "Runaway argument?").}
    \immediate\write16{ This probably comes from another package that changed the category}
    \immediate\write16{ code after Fig4TeX was loaded. If that proves to be exact, the}
    \immediate\write16{ solution is to exchange the loading commands on top of your file}
    \immediate\write16{ so that Fig4TeX is loaded last. For example, in LaTeX, we should}
    \immediate\write16{ say :}
    \immediate\write16{\BS@ usepackage[french]{babel}}
    \immediate\write16{\BS@ usepackage{fig4tex}}
    \immediate\write16{!!!=============================================================!!!}
    \immediate\write16{}
    \fi}}
\ctr@ld@f\def\FigforTeX{F\kern-.05em i\kern-.05em g\kern-.1em\raise-.14em\hbox{4}\kern-.19em\TeX}
\ctr@ld@f\def\W@rnmesoldA#1{\W@rnmesold}
\ctr@ld@f\def\W@rnmesoldAB#1(#2){\W@rnmesold}
\ctr@ld@f\def\W@rnmesold{%
    \immediate\write16{}
    \immediate\write16{!!!=============================================================!!!}
    \immediate\write16{ Fig4TeX WARNING:}
    \immediate\write16{ The file to be compiled is not compatible with the current version}
    \immediate\write16{ of Fig4TeX. To fix that, upgrade the source file (mainly change \BS@ ps*}
    \immediate\write16{ macros by \BS@ fig* macros), or use fig4tex184.tex instead (\BS@ input fig4tex184}
    \immediate\write16{ or \BS@ usepackage{fig4tex184}).}
    \immediate\write16{!!!=============================================================!!!}
    \immediate\write16{}}
\ctr@ln@m\psbeginfig\let\psbeginfig\W@rnmesoldA
\ctr@ln@m\psset\let\psset\W@rnmesoldAB
\ctr@ln@m\pssetdefault\let\pssetdefault\W@rnmesoldAB
\ctr@ln@m\pssetupdate\let\pssetupdate\W@rnmesoldA
\ctr@ln@w{newdimen}\epsil@n\epsil@n=0.00005pt
\ctr@ln@w{newdimen}\Cepsil@n\Cepsil@n=0.005pt
\ctr@ln@w{newdimen}\dcq@\dcq@=254pt
\ctr@ln@w{newdimen}\PI@\PI@=3.141592pt
\ctr@ln@w{newdimen}\DemiPI@deg\DemiPI@deg=90pt
\ctr@ln@w{newdimen}\PI@deg\PI@deg=180pt
\ctr@ln@w{newdimen}\DePI@deg\DePI@deg=360pt
\ctr@ld@f\chardef\t@n=10
\ctr@ld@f\chardef\c@nt=100
\ctr@ld@f\chardef\@lxxiv=74
\ctr@ld@f\chardef\@xci=91
\ctr@ld@f\mathchardef\@nMnCQn=9949
\ctr@ld@f\chardef\@vi=6
\ctr@ld@f\chardef\@xxx=30
\ctr@ld@f\chardef\@lvi=56
\ctr@ld@f\chardef\@@lxxi=71
\ctr@ld@f\chardef\@lxxxv=85
\ctr@ld@f\mathchardef\@@mmmmlxviii=4068
\ctr@ld@f\mathchardef\@ccclx=360
\ctr@ld@f\mathchardef\@dccxx=720
\ctr@ln@w{newcount}\p@rtent \ctr@ln@w{newcount}\f@ctech \ctr@ln@w{newcount}\result@tent
\ctr@ln@w{newdimen}\v@lmin \ctr@ln@w{newdimen}\v@lmax \ctr@ln@w{newdimen}\v@leur
\ctr@ln@w{newdimen}\result@t\ctr@ln@w{newdimen}\result@@t
\ctr@ln@w{newdimen}\mili@u \ctr@ln@w{newdimen}\c@rre \ctr@ln@w{newdimen}\delt@
\ctr@ld@f\def\degT@rd{0.017453 }  
\ctr@ld@f\def\rdT@deg{57.295779 } 
\ctr@ln@m\v@leurseule
{\catcode`p=12 \catcode`t=12 \gdef\v@leurseule#1pt{#1}}
\ctr@ld@f\def\repdecn@mb#1{\expandafter\v@leurseule\the#1\space}
\ctr@ld@f\def\arct@n#1(#2,#3){{\v@lmin=#2\v@lmax=#3%
    \maxim@m{\mili@u}{-\v@lmin}{\v@lmin}\maxim@m{\c@rre}{-\v@lmax}{\v@lmax}%
    \delt@=\mili@u\m@ech\mili@u%
    \ifdim\c@rre>\@nMnCQn\mili@u\divide\v@lmax\tw@\c@lATAN\v@leur(\z@,\v@lmax)
    \else%
    \maxim@m{\mili@u}{-\v@lmin}{\v@lmin}\maxim@m{\c@rre}{-\v@lmax}{\v@lmax}%
    \m@ech\c@rre%
    \ifdim\mili@u>\@nMnCQn\c@rre\divide\v@lmin\tw@
    \maxim@m{\mili@u}{-\v@lmin}{\v@lmin}\c@lATAN\v@leur(\mili@u,\z@)%
    \else\c@lATAN\v@leur(\delt@,\v@lmax)\fi\fi%
    \ifdim\v@lmin<\z@\v@leur=-\v@leur\ifdim\v@lmax<\z@\advance\v@leur-\PI@%
    \else\advance\v@leur\PI@\fi\fi%
    \global\result@t=\v@leur}#1=\result@t}
\ctr@ld@f\def\m@ech#1{\ifdim#1>1.646pt\divide\mili@u\t@n\divide\c@rre\t@n\m@ech#1\fi}
\ctr@ld@f\def\c@lATAN#1(#2,#3){{\v@lmin=#2\v@lmax=#3\v@leur=\z@\delt@=\tw@ pt%
    \un@iter{0.785398}{\v@lmax<}%
    \un@iter{0.463648}{\v@lmax<}%
    \un@iter{0.244979}{\v@lmax<}%
    \un@iter{0.124355}{\v@lmax<}%
    \un@iter{0.062419}{\v@lmax<}%
    \un@iter{0.031240}{\v@lmax<}%
    \un@iter{0.015624}{\v@lmax<}%
    \un@iter{0.007812}{\v@lmax<}%
    \un@iter{0.003906}{\v@lmax<}%
    \un@iter{0.001953}{\v@lmax<}%
    \un@iter{0.000976}{\v@lmax<}%
    \un@iter{0.000488}{\v@lmax<}%
    \un@iter{0.000244}{\v@lmax<}%
    \un@iter{0.000122}{\v@lmax<}%
    \un@iter{0.000061}{\v@lmax<}%
    \un@iter{0.000030}{\v@lmax<}%
    \un@iter{0.000015}{\v@lmax<}%
    \global\result@t=\v@leur}#1=\result@t}
\ctr@ld@f\def\un@iter#1#2{%
    \divide\delt@\tw@\edef\dpmn@{\repdecn@mb{\delt@}}%
    \mili@u=\v@lmin%
    \ifdim#2\z@%
      \advance\v@lmin-\dpmn@\v@lmax\advance\v@lmax\dpmn@\mili@u%
      \advance\v@leur-#1pt%
    \else%
      \advance\v@lmin\dpmn@\v@lmax\advance\v@lmax-\dpmn@\mili@u%
      \advance\v@leur#1pt%
    \fi}
\ctr@ld@f\def\c@ssin#1#2#3{\expandafter\ifx\csname COS@\number#3\endcsname\relax\c@lCS{#3pt}%
    \expandafter\xdef\csname COS@\number#3\endcsname{\repdecn@mb\result@t}%
    \expandafter\xdef\csname SIN@\number#3\endcsname{\repdecn@mb\result@@t}\fi%
    \edef#1{\csname COS@\number#3\endcsname}\edef#2{\csname SIN@\number#3\endcsname}}
\ctr@ld@f\def\c@lCS#1{{\mili@u=#1\p@rtent=\@ne%
    \relax\ifdim\mili@u<\z@\red@ng<-\else\red@ng>+\fi\f@ctech=\p@rtent%
    \relax\ifdim\mili@u<\z@\mili@u=-\mili@u\f@ctech=-\f@ctech\fi\c@@lCS}}
\ctr@ld@f\def\c@@lCS{\v@lmin=\mili@u\c@rre=-\mili@u\advance\c@rre\DemiPI@deg\v@lmax=\c@rre%
    \mili@u\@@lxxi\mili@u\divide\mili@u\@@mmmmlxviii%
    \edef\v@larg{\repdecn@mb{\mili@u}}\mili@u=-\v@larg\mili@u%
    \edef\v@lmxde{\repdecn@mb{\mili@u}}%
    \c@rre\@@lxxi\c@rre\divide\c@rre\@@mmmmlxviii%
    \edef\v@largC{\repdecn@mb{\c@rre}}\c@rre=-\v@largC\c@rre%
    \edef\v@lmxdeC{\repdecn@mb{\c@rre}}%
    \fctc@s\mili@u\v@lmin\global\result@t\p@rtent\v@leur%
    \let\t@mp=\v@larg\let\v@larg=\v@largC\let\v@largC=\t@mp%
    \let\t@mp=\v@lmxde\let\v@lmxde=\v@lmxdeC\let\v@lmxdeC=\t@mp%
    \fctc@s\c@rre\v@lmax\global\result@@t\f@ctech\v@leur}
\ctr@ld@f\def\fctc@s#1#2{\v@leur=#1\relax\ifdim#2<\@lxxxv\p@\cosser@h\else\sinser@t\fi}
\ctr@ld@f\def\cosser@h{\advance\v@leur\@lvi\p@\divide\v@leur\@lvi%
    \v@leur=\v@lmxde\v@leur\advance\v@leur\@xxx\p@%
    \v@leur=\v@lmxde\v@leur\advance\v@leur\@ccclx\p@%
    \v@leur=\v@lmxde\v@leur\advance\v@leur\@dccxx\p@\divide\v@leur\@dccxx}
\ctr@ld@f\def\sinser@t{\v@leur=\v@lmxdeC\p@\advance\v@leur\@vi\p@%
    \v@leur=\v@largC\v@leur\divide\v@leur\@vi}
\ctr@ld@f\def\red@ng#1#2{\relax\ifdim\mili@u#1#2\DemiPI@deg\advance\mili@u#2-\PI@deg%
    \p@rtent=-\p@rtent\red@ng#1#2\fi}
\ctr@ld@f\def\pr@c@lCS#1#2#3{\ctr@lcsn@m{COS@\number#3 }%
    \expandafter\xdef\csname COS@\number#3\endcsname{#1}%
    \expandafter\xdef\csname SIN@\number#3\endcsname{#2}}
\pr@c@lCS{1}{0}{0}
\pr@c@lCS{0.7071}{0.7071}{45}\pr@c@lCS{0.7071}{-0.7071}{-45}
\pr@c@lCS{0}{1}{90}          \pr@c@lCS{0}{-1}{-90}
\pr@c@lCS{-1}{0}{180}        \pr@c@lCS{-1}{0}{-180}
\pr@c@lCS{0}{-1}{270}        \pr@c@lCS{0}{1}{-270}
\ctr@ld@f\def\invers@#1#2{{\v@leur=#2\maxim@m{\v@lmax}{-\v@leur}{\v@leur}%
    \f@ctech=\@ne\m@inv@rs%
    \multiply\v@leur\f@ctech\edef\v@lv@leur{\repdecn@mb{\v@leur}}%
    \p@rtentiere{\p@rtent}{\v@leur}\v@lmin=\p@\divide\v@lmin\p@rtent%
    \inv@rs@\multiply\v@lmax\f@ctech\global\result@t=\v@lmax}#1=\result@t}
\ctr@ld@f\def\m@inv@rs{\ifdim\v@lmax<\p@\multiply\v@lmax\t@n\multiply\f@ctech\t@n\m@inv@rs\fi}
\ctr@ld@f\def\inv@rs@{\v@lmax=-\v@lmin\v@lmax=\v@lv@leur\v@lmax%
    \advance\v@lmax\tw@ pt\v@lmax=\repdecn@mb{\v@lmin}\v@lmax%
    \delt@=\v@lmax\advance\delt@-\v@lmin\ifdim\delt@<\z@\delt@=-\delt@\fi%
    \ifdim\delt@>\epsil@n\v@lmin=\v@lmax\inv@rs@\fi}
\ctr@ld@f\def\minim@m#1#2#3{\relax\ifdim#2<#3#1=#2\else#1=#3\fi}
\ctr@ld@f\def\maxim@m#1#2#3{\relax\ifdim#2>#3#1=#2\else#1=#3\fi}
\ctr@ld@f\def\p@rtentiere#1#2{#1=#2\divide#1by65536 }
\ctr@ld@f\def\r@undint#1#2{{\v@leur=#2\divide\v@leur\t@n\p@rtentiere{\p@rtent}{\v@leur}%
    \v@leur=\p@rtent pt\global\result@t=\t@n\v@leur}#1=\result@t}
\ctr@ld@f\def\sqrt@#1#2{{\v@leur=#2%
    \minim@m{\v@lmin}{\p@}{\v@leur}\maxim@m{\v@lmax}{\p@}{\v@leur}%
    \f@ctech=\@ne\m@sqrt@\sqrt@@%
    \mili@u=\v@lmin\advance\mili@u\v@lmax\divide\mili@u\tw@\multiply\mili@u\f@ctech%
    \global\result@t=\mili@u}#1=\result@t}
\ctr@ld@f\def\m@sqrt@{\ifdim\v@leur>\dcq@\divide\v@leur\c@nt\v@lmax=\v@leur%
    \multiply\f@ctech\t@n\m@sqrt@\fi}
\ctr@ld@f\def\sqrt@@{\mili@u=\v@lmin\advance\mili@u\v@lmax\divide\mili@u\tw@%
    \c@rre=\repdecn@mb{\mili@u}\mili@u%
    \ifdim\c@rre<\v@leur\v@lmin=\mili@u\else\v@lmax=\mili@u\fi%
    \delt@=\v@lmax\advance\delt@-\v@lmin\ifdim\delt@>\epsil@n\sqrt@@\fi}
\ctr@ld@f\def\extrairelepremi@r#1\de#2{\expandafter\lepremi@r#2@#1#2}
\ctr@ld@f\def\lepremi@r#1,#2@#3#4{\def#3{#1}\def#4{#2}\ignorespaces}
\ctr@ld@f\def\@cfor#1:=#2\do#3{%
  \edef\@fortemp{#2}%
  \ifx\@fortemp\empty\else\@cforloop#2,\@nil,\@nil\@@#1{#3}\fi}
\ctr@ln@m\@nextwhile
\ctr@ld@f\def\@cforloop#1,#2\@@#3#4{%
  \def#3{#1}%
  \ifx#3\Fig@nnil\let\@nextwhile=\Fig@fornoop\else#4\relax\let\@nextwhile=\@cforloop\fi%
  \@nextwhile#2\@@#3{#4}}

\ctr@ld@f\def\@ecfor#1:=#2\do#3{%
  \def\@@cfor{\@cfor#1:=}%
  \edef\@@@cfor{#2}%
  \expandafter\@@cfor\@@@cfor\do{#3}}
\ctr@ld@f\def\Fig@nnil{\@nil}
\ctr@ld@f\def\Fig@fornoop#1\@@#2#3{}
\ctr@ln@m\list@@rg
\ctr@ld@f\def\trtlis@rg#1#2{\def\list@@rg{#1}%
    \@ecfor\p@rv@l:=\list@@rg\do{\expandafter#2\p@rv@l|}}
\ctr@ld@f\def\trtlis@rgtok#1{\let@xte={}\let\n@xt\addt@t@xt\addt@t@xt #1}
\ctr@ln@m\M@cro
\ctr@ln@m\n@xt
\ctr@ld@f\def\addt@t@xt#1{\if#1|\let\n@xt\relax\else%
    \if#1,\expandafter\M@cro\the\let@xte|\let@xte={}%
    \else\let@xte=\expandafter{\the\let@xte #1}\fi\fi\n@xt}
\ctr@ln@w{newbox}\b@xvisu
\ctr@ln@w{newtoks}\let@xte
\ctr@ln@w{newif}\ifitis@K
\ctr@ln@w{newcount}\s@mme
\ctr@ln@w{newcount}\l@mbd@un \ctr@ln@w{newcount}\l@mbd@de
\ctr@ln@w{newcount}\superc@ntr@l\superc@ntr@l=\@ne        
\ctr@ln@w{newcount}\typec@ntr@l\typec@ntr@l=\superc@ntr@l 
\ctr@ln@w{newdimen}\v@lX  \ctr@ln@w{newdimen}\v@lY  \ctr@ln@w{newdimen}\v@lZ
\ctr@ln@w{newdimen}\v@lXa \ctr@ln@w{newdimen}\v@lYa \ctr@ln@w{newdimen}\v@lZa
\ctr@ln@w{newdimen}\unit@\unit@=\p@ 
\ctr@ld@f\def\unit@util{pt}
\ctr@ld@f\def\ptT@ptps{0.996264 }
\ctr@ld@f\def\ptpsT@pt{1.00375 }
\ctr@ld@f\def\ptT@unit@{1} 
\ctr@ld@f\def\setunit@#1{\def\unit@util{#1}\setunit@@#1:\invers@{\result@t}{\unit@}%
    \edef\ptT@unit@{\repdecn@mb\result@t}}
\ctr@ld@f\def\setunit@@#1#2:{\ifcat#1a\unit@=\@ne#1#2\else\unit@=#1#2\fi}
\ctr@ld@f\def\d@fm@cdim#1#2{{\v@leur=#2\v@leur=\ptT@unit@\v@leur\xdef#1{\repdecn@mb\v@leur}}}
\ctr@ln@w{newif}\ifBdingB@x\BdingB@xtrue
\ctr@ln@w{newdimen}\c@@rdXmin \ctr@ln@w{newdimen}\c@@rdYmin  
\ctr@ln@w{newdimen}\c@@rdXmax \ctr@ln@w{newdimen}\c@@rdYmax
\ctr@ld@f\def\b@undb@x#1#2{\ifBdingB@x%
    \relax\ifdim#1<\c@@rdXmin\global\c@@rdXmin=#1\fi%
    \relax\ifdim#2<\c@@rdYmin\global\c@@rdYmin=#2\fi%
    \relax\ifdim#1>\c@@rdXmax\global\c@@rdXmax=#1\fi%
    \relax\ifdim#2>\c@@rdYmax\global\c@@rdYmax=#2\fi\fi}
\ctr@ld@f\def\b@undb@xP#1{{\Figg@tXY{#1}\b@undb@x{\v@lX}{\v@lY}}}
\ctr@ld@f\def\ellBB@x#1;#2,#3(#4,#5,#6){{\s@uvc@ntr@l\et@tellBB@x%
    \setc@ntr@l{2}\figptell-2::#1;#2,#3(#4,#6)\b@undb@xP{-2}%
    \figptell-2::#1;#2,#3(#5,#6)\b@undb@xP{-2}%
    \c@ssin{\C@}{\S@}{#6}\v@lmin=\C@ pt\v@lmax=\S@ pt%
    \mili@u=#3\v@lmin\delt@=#2\v@lmax\arct@n\v@leur(\delt@,\mili@u)%
    \mili@u=-#3\v@lmax\delt@=#2\v@lmin\arct@n\c@rre(\delt@,\mili@u)%
    \v@leur=\rdT@deg\v@leur\advance\v@leur-\DePI@deg%
    \c@rre=\rdT@deg\c@rre\advance\c@rre-\DePI@deg%
    \v@lmin=#4pt\v@lmax=#5pt%
    \loop\ifdim\v@leur<\v@lmax\ifdim\v@leur>\v@lmin%
    \edef\@ngle{\repdecn@mb\v@leur}\figptell-2::#1;#2,#3(\@ngle,#6)%
    \b@undb@xP{-2}\fi\advance\v@leur\PI@deg\repeat%
    \loop\ifdim\c@rre<\v@lmax\ifdim\c@rre>\v@lmin%
    \edef\@ngle{\repdecn@mb\c@rre}\figptell-2::#1;#2,#3(\@ngle,#6)%
    \b@undb@xP{-2}\fi\advance\c@rre\PI@deg\repeat%
    \resetc@ntr@l\et@tellBB@x}\ignorespaces}
\ctr@ld@f\def\initb@undb@x{\c@@rdXmin=\maxdimen\c@@rdYmin=\maxdimen%
    \c@@rdXmax=-\maxdimen\c@@rdYmax=-\maxdimen}
\ctr@ld@f\def\c@ntr@lnum#1{%
    \relax\ifnum\typec@ntr@l=\@ne%
    \ifnum#1<\z@%
    \immediate\write16{*** Forbidden point number (#1). Abort.}\end\fi\fi%
    \set@bjc@de{#1}}
\ctr@ln@m\objc@de
\ctr@ld@f\def\set@bjc@de#1{\edef\objc@de{@BJ\ifnum#1<\z@ M\romannumeral-#1\else\romannumeral#1\fi}}
\s@mme=\m@ne\loop\ifnum\s@mme>-19
  \set@bjc@de{\s@mme}\ctr@lcsn@m\objc@de\ctr@lcsn@m{\objc@de T}
\advance\s@mme\m@ne\repeat
\s@mme=\@ne\loop\ifnum\s@mme<6
  \set@bjc@de{\s@mme}\ctr@lcsn@m\objc@de\ctr@lcsn@m{\objc@de T}
\advance\s@mme\@ne\repeat
\ctr@ld@f\def\setc@ntr@l#1{\ifnum\superc@ntr@l>#1\typec@ntr@l=\superc@ntr@l%
    \else\typec@ntr@l=#1\fi}
\ctr@ld@f\def\resetc@ntr@l#1{\global\superc@ntr@l=#1\setc@ntr@l{#1}}
\ctr@ld@f\def\s@uvc@ntr@l#1{\edef#1{\the\superc@ntr@l}}
\ctr@ln@m\c@lproscal
\ctr@ld@f\def\c@lproscalDD#1[#2,#3]{{\Figg@tXY{#2}%
    \edef\Xu@{\repdecn@mb{\v@lX}}\edef\Yu@{\repdecn@mb{\v@lY}}\Figg@tXY{#3}%
    \global\result@t=\Xu@\v@lX\global\advance\result@t\Yu@\v@lY}#1=\result@t}
\ctr@ld@f\def\c@lproscalTD#1[#2,#3]{{\Figg@tXY{#2}\edef\Xu@{\repdecn@mb{\v@lX}}%
    \edef\Yu@{\repdecn@mb{\v@lY}}\edef\Zu@{\repdecn@mb{\v@lZ}}%
    \Figg@tXY{#3}\global\result@t=\Xu@\v@lX\global\advance\result@t\Yu@\v@lY%
    \global\advance\result@t\Zu@\v@lZ}#1=\result@t}
\ctr@ld@f\def\c@lprovec#1{%
    \det@rmC\v@lZa(\v@lX,\v@lY,\v@lmin,\v@lmax)%
    \det@rmC\v@lXa(\v@lY,\v@lZ,\v@lmax,\v@leur)%
    \det@rmC\v@lYa(\v@lZ,\v@lX,\v@leur,\v@lmin)%
    \Figv@ctCreg#1(\v@lXa,\v@lYa,\v@lZa)}
\ctr@ld@f\def\det@rm#1[#2,#3]{{\Figg@tXY{#2}\Figg@tXYa{#3}%
    \delt@=\repdecn@mb{\v@lX}\v@lYa\advance\delt@-\repdecn@mb{\v@lY}\v@lXa%
    \global\result@t=\delt@}#1=\result@t}
\ctr@ld@f\def\det@rmC#1(#2,#3,#4,#5){{\global\result@t=\repdecn@mb{#2}#5%
    \global\advance\result@t-\repdecn@mb{#3}#4}#1=\result@t}
\ctr@ld@f\def\getredf@ctDD#1(#2,#3){{\maxim@m{\v@lXa}{-#2}{#2}\maxim@m{\v@lYa}{-#3}{#3}%
    \maxim@m{\v@lXa}{\v@lXa}{\v@lYa}
    \ifdim\v@lXa>\@xci pt\divide\v@lXa\@xci%
    \p@rtentiere{\p@rtent}{\v@lXa}\advance\p@rtent\@ne\else\p@rtent=\@ne\fi%
    \global\result@tent=\p@rtent}#1=\result@tent\ignorespaces}
\ctr@ld@f\def\getredf@ctTD#1(#2,#3,#4){{\maxim@m{\v@lXa}{-#2}{#2}\maxim@m{\v@lYa}{-#3}{#3}%
    \maxim@m{\v@lZa}{-#4}{#4}\maxim@m{\v@lXa}{\v@lXa}{\v@lYa}%
    \maxim@m{\v@lXa}{\v@lXa}{\v@lZa}
    \ifdim\v@lXa>\@lxxiv pt\divide\v@lXa\@lxxiv%
    \p@rtentiere{\p@rtent}{\v@lXa}\advance\p@rtent\@ne\else\p@rtent=\@ne\fi%
    \global\result@tent=\p@rtent}#1=\result@tent\ignorespaces}
\ctr@ln@m\getredf@ctB
\ctr@ld@f\def\getredf@ctBDD#1{\getredf@ctDD#1(\v@lX,\v@lY)}
\ctr@ld@f\def\getredf@ctBTD#1{\getredf@ctTD#1(\v@lX,\v@lY,\v@lZ)}
\ctr@ld@f\def\FigptintercircB@zDD#1:#2:#3,#4[#5,#6,#7,#8]{{\s@uvc@ntr@l\et@tfigptintercircB@zDD%
    \setc@ntr@l{2}\figvectPDD-1[#5,#8]\Figg@tXY{-1}\getredf@ctDD\f@ctech(\v@lX,\v@lY)%
    \mili@u=#4\unit@\divide\mili@u\f@ctech\c@rre=\repdecn@mb{\mili@u}\mili@u%
    \figptBezierDD-5::#3[#5,#6,#7,#8]%
    \v@lmin=#3\p@\v@lmax=\v@lmin\advance\v@lmax0.1\p@%
    \loop\edef\T@{\repdecn@mb{\v@lmax}}\figptBezierDD-2::\T@[#5,#6,#7,#8]%
    \figvectPDD-1[-5,-2]\n@rmeucCDD{\delt@}{-1}\ifdim\delt@<\c@rre\v@lmin=\v@lmax%
    \advance\v@lmax0.1\p@\repeat%
    \loop\mili@u=\v@lmin\advance\mili@u\v@lmax%
    \divide\mili@u\tw@\edef\T@{\repdecn@mb{\mili@u}}\figptBezierDD-2::\T@[#5,#6,#7,#8]%
    \figvectPDD-1[-5,-2]\n@rmeucCDD{\delt@}{-1}\ifdim\delt@>\c@rre\v@lmax=\mili@u%
    \else\v@lmin=\mili@u\fi\v@leur=\v@lmax\advance\v@leur-\v@lmin%
    \ifdim\v@leur>\epsil@n\repeat\figptcopyDD#1:#2/-2/%
    \resetc@ntr@l\et@tfigptintercircB@zDD}\ignorespaces}
\ctr@ln@m\figptinterlines
\ctr@ld@f\def\inters@cDD#1:#2[#3,#4;#5,#6]{{\s@uvc@ntr@l\et@tinters@cDD%
    \setc@ntr@l{2}\vecunit@{-1}{#4}\vecunit@{-2}{#6}%
    \Figg@tXY{-1}\setc@ntr@l{1}\Figg@tXYa{#3}%
    \edef\A@{\repdecn@mb{\v@lX}}\edef\B@{\repdecn@mb{\v@lY}}%
    \v@lmin=\B@\v@lXa\advance\v@lmin-\A@\v@lYa%
    \Figg@tXYa{#5}\setc@ntr@l{2}\Figg@tXY{-2}%
    \edef\C@{\repdecn@mb{\v@lX}}\edef\D@{\repdecn@mb{\v@lY}}%
    \v@lmax=\D@\v@lXa\advance\v@lmax-\C@\v@lYa%
    \delt@=\A@\v@lY\advance\delt@-\B@\v@lX%
    \invers@{\v@leur}{\delt@}\edef\v@ldelta{\repdecn@mb{\v@leur}}%
    \v@lXa=\A@\v@lmax\advance\v@lXa-\C@\v@lmin%
    \v@lYa=\B@\v@lmax\advance\v@lYa-\D@\v@lmin%
    \v@lXa=\v@ldelta\v@lXa\v@lYa=\v@ldelta\v@lYa%
    \setc@ntr@l{1}\Figp@intregDD#1:{#2}(\v@lXa,\v@lYa)%
    \resetc@ntr@l\et@tinters@cDD}\ignorespaces}
\ctr@ld@f\def\inters@cTD#1:#2[#3,#4;#5,#6]{{\s@uvc@ntr@l\et@tinters@cTD%
    \setc@ntr@l{2}\figvectNVTD-1[#4,#6]\figvectNVTD-2[#6,-1]\figvectPTD-1[#3,#5]%
    \r@pPSTD\v@leur[-2,-1,#4]\edef\v@lcoef{\repdecn@mb{\v@leur}}%
    \figpttraTD#1:{#2}=#3/\v@lcoef,#4/\resetc@ntr@l\et@tinters@cTD}\ignorespaces}
\ctr@ld@f\def\r@pPSTD#1[#2,#3,#4]{{\Figg@tXY{#2}\edef\Xu@{\repdecn@mb{\v@lX}}%
    \edef\Yu@{\repdecn@mb{\v@lY}}\edef\Zu@{\repdecn@mb{\v@lZ}}%
    \Figg@tXY{#3}\v@lmin=\Xu@\v@lX\advance\v@lmin\Yu@\v@lY\advance\v@lmin\Zu@\v@lZ%
    \Figg@tXY{#4}\v@lmax=\Xu@\v@lX\advance\v@lmax\Yu@\v@lY\advance\v@lmax\Zu@\v@lZ%
    \invers@{\v@leur}{\v@lmax}\global\result@t=\repdecn@mb{\v@leur}\v@lmin}%
    #1=\result@t}
\ctr@ln@m\n@rminf
\ctr@ld@f\def\n@rminfDD#1#2{{\Figg@tXY{#2}\maxim@m{\v@lX}{\v@lX}{-\v@lX}%
    \maxim@m{\v@lY}{\v@lY}{-\v@lY}\maxim@m{\global\result@t}{\v@lX}{\v@lY}}%
    #1=\result@t}
\ctr@ld@f\def\n@rminfTD#1#2{{\Figg@tXY{#2}\maxim@m{\v@lX}{\v@lX}{-\v@lX}%
    \maxim@m{\v@lY}{\v@lY}{-\v@lY}\maxim@m{\v@lZ}{\v@lZ}{-\v@lZ}%
    \maxim@m{\v@lX}{\v@lX}{\v@lY}\maxim@m{\global\result@t}{\v@lX}{\v@lZ}}%
    #1=\result@t}
\ctr@ln@m\n@rmeucC
\ctr@ld@f\def\n@rmeucCDD#1#2{\Figg@tXY{#2}\divide\v@lX\f@ctech\divide\v@lY\f@ctech%
    #1=\repdecn@mb{\v@lX}\v@lX\v@lX=\repdecn@mb{\v@lY}\v@lY\advance#1\v@lX}
\ctr@ld@f\def\n@rmeucCTD#1#2{\Figg@tXY{#2}%
    \divide\v@lX\f@ctech\divide\v@lY\f@ctech\divide\v@lZ\f@ctech%
    #1=\repdecn@mb{\v@lX}\v@lX\v@lX=\repdecn@mb{\v@lY}\v@lY\advance#1\v@lX%
    \v@lX=\repdecn@mb{\v@lZ}\v@lZ\advance#1\v@lX}
\ctr@ln@m\n@rmeucSV
\ctr@ld@f\def\n@rmeucSVDD#1#2{{\Figg@tXY{#2}%
    \v@lXa=\repdecn@mb{\v@lX}\v@lX\v@lYa=\repdecn@mb{\v@lY}\v@lY%
    \advance\v@lXa\v@lYa\sqrt@{\global\result@t}{\v@lXa}}#1=\result@t}
\ctr@ld@f\def\n@rmeucSVTD#1#2{{\Figg@tXY{#2}\v@lXa=\repdecn@mb{\v@lX}\v@lX%
    \v@lYa=\repdecn@mb{\v@lY}\v@lY\v@lZa=\repdecn@mb{\v@lZ}\v@lZ%
    \advance\v@lXa\v@lYa\advance\v@lXa\v@lZa\sqrt@{\global\result@t}{\v@lXa}}#1=\result@t}
\ctr@ln@m\n@rmeuc
\ctr@ld@f\def\n@rmeucDD#1#2{{\Figg@tXY{#2}\getredf@ctDD\f@ctech(\v@lX,\v@lY)%
    \divide\v@lX\f@ctech\divide\v@lY\f@ctech%
    \v@lXa=\repdecn@mb{\v@lX}\v@lX\v@lYa=\repdecn@mb{\v@lY}\v@lY%
    \advance\v@lXa\v@lYa\sqrt@{\global\result@t}{\v@lXa}%
    \global\multiply\result@t\f@ctech}#1=\result@t}
\ctr@ld@f\def\n@rmeucTD#1#2{{\Figg@tXY{#2}\getredf@ctTD\f@ctech(\v@lX,\v@lY,\v@lZ)%
    \divide\v@lX\f@ctech\divide\v@lY\f@ctech\divide\v@lZ\f@ctech%
    \v@lXa=\repdecn@mb{\v@lX}\v@lX%
    \v@lYa=\repdecn@mb{\v@lY}\v@lY\v@lZa=\repdecn@mb{\v@lZ}\v@lZ%
    \advance\v@lXa\v@lYa\advance\v@lXa\v@lZa\sqrt@{\global\result@t}{\v@lXa}%
    \global\multiply\result@t\f@ctech}#1=\result@t}
\ctr@ln@m\vecunit@
\ctr@ld@f\def\vecunit@DD#1#2{{\Figg@tXY{#2}\getredf@ctDD\f@ctech(\v@lX,\v@lY)%
    \divide\v@lX\f@ctech\divide\v@lY\f@ctech%
    \Figv@ctCreg#1(\v@lX,\v@lY)\n@rmeucSV{\v@lYa}{#1}%
    \invers@{\v@lXa}{\v@lYa}\edef\v@lv@lXa{\repdecn@mb{\v@lXa}}%
    \v@lX=\v@lv@lXa\v@lX\v@lY=\v@lv@lXa\v@lY%
    \Figv@ctCreg#1(\v@lX,\v@lY)\multiply\v@lYa\f@ctech\global\result@t=\v@lYa}}
\ctr@ld@f\def\vecunit@TD#1#2{{\Figg@tXY{#2}\getredf@ctTD\f@ctech(\v@lX,\v@lY,\v@lZ)%
    \divide\v@lX\f@ctech\divide\v@lY\f@ctech\divide\v@lZ\f@ctech%
    \Figv@ctCreg#1(\v@lX,\v@lY,\v@lZ)\n@rmeucSV{\v@lYa}{#1}%
    \invers@{\v@lXa}{\v@lYa}\edef\v@lv@lXa{\repdecn@mb{\v@lXa}}%
    \v@lX=\v@lv@lXa\v@lX\v@lY=\v@lv@lXa\v@lY\v@lZ=\v@lv@lXa\v@lZ%
    \Figv@ctCreg#1(\v@lX,\v@lY,\v@lZ)\multiply\v@lYa\f@ctech\global\result@t=\v@lYa}}
\ctr@ld@f\def\vecunitC@TD[#1,#2]{\Figg@tXYa{#1}\Figg@tXY{#2}%
    \advance\v@lX-\v@lXa\advance\v@lY-\v@lYa\advance\v@lZ-\v@lZa\c@lvecunitTD}
\ctr@ld@f\def\vecunitCV@TD#1{\Figg@tXY{#1}\c@lvecunitTD}
\ctr@ld@f\def\c@lvecunitTD{\getredf@ctTD\f@ctech(\v@lX,\v@lY,\v@lZ)%
    \divide\v@lX\f@ctech\divide\v@lY\f@ctech\divide\v@lZ\f@ctech%
    \v@lXa=\repdecn@mb{\v@lX}\v@lX%
    \v@lYa=\repdecn@mb{\v@lY}\v@lY\v@lZa=\repdecn@mb{\v@lZ}\v@lZ%
    \advance\v@lXa\v@lYa\advance\v@lXa\v@lZa\sqrt@{\v@lYa}{\v@lXa}%
    \invers@{\v@lXa}{\v@lYa}\edef\v@lv@lXa{\repdecn@mb{\v@lXa}}%
    \v@lX=\v@lv@lXa\v@lX\v@lY=\v@lv@lXa\v@lY\v@lZ=\v@lv@lXa\v@lZ}
\ctr@ln@m\figgetangle
\ctr@ld@f\def\figgetangleDD#1[#2,#3,#4]{\ifGR@cri{\s@uvc@ntr@l\et@tfiggetangleDD\setc@ntr@l{2}%
    \figvectPDD-1[#2,#3]\figvectPDD-2[#2,#4]\vecunit@{-1}{-1}%
    \c@lproscalDD\delt@[-2,-1]\figvectNVDD-1[-1]\c@lproscalDD\v@leur[-2,-1]%
    \arct@n\v@lmax(\delt@,\v@leur)\v@lmax=\rdT@deg\v@lmax%
    \ifdim\v@lmax<\z@\advance\v@lmax\DePI@deg\fi\xdef#1{\repdecn@mb{\v@lmax}}%
    \resetc@ntr@l\et@tfiggetangleDD}\ignorespaces\fi}
\ctr@ld@f\def\figgetangleTD#1[#2,#3,#4,#5]{\ifGR@cri{\s@uvc@ntr@l\et@tfiggetangleTD\setc@ntr@l{2}%
    \figvectPTD-1[#2,#3]\figvectPTD-2[#2,#5]\figvectNVTD-3[-1,-2]%
    \figvectPTD-2[#2,#4]\figvectNVTD-4[-3,-1]%
    \vecunit@{-1}{-1}\c@lproscalTD\delt@[-2,-1]\c@lproscalTD\v@leur[-2,-4]%
    \arct@n\v@lmax(\delt@,\v@leur)\v@lmax=\rdT@deg\v@lmax%
    \ifdim\v@lmax<\z@\advance\v@lmax\DePI@deg\fi\xdef#1{\repdecn@mb{\v@lmax}}%
    \resetc@ntr@l\et@tfiggetangleTD}\ignorespaces\fi}    
\ctr@ld@f\def\figgetdist#1[#2,#3]{\ifGR@cri{\s@uvc@ntr@l\et@tfiggetdist\setc@ntr@l{2}%
    \figvectP-1[#2,#3]\n@rmeuc{\v@lX}{-1}\v@lX=\ptT@unit@\v@lX\xdef#1{\repdecn@mb{\v@lX}}%
    \resetc@ntr@l\et@tfiggetdist}\ignorespaces\fi}
\ctr@ld@f\def\figget#1=#2[#3]{\keln@mun#1|%
    \def\n@mref{a}\ifx\l@debut\n@mref\figgetangle#2[#3]\else
    \def\n@mref{d}\ifx\l@debut\n@mref\figgetdist#2[#3]\else
    \W@rnmeskwd{figget}{#1}\fi\fi\ignorespaces}
\ctr@ld@f\def\Figg@tT#1{\c@ntr@lnum{#1}%
    {\expandafter\expandafter\expandafter\extr@ctT\csname\objc@de\endcsname:%
     \ifnum\B@@ltxt=\z@\ptn@me{#1}\else\csname\objc@de T\endcsname\fi}}
\ctr@ld@f\def\extr@ctT#1,#2,#3/#4:{\def\B@@ltxt{#3}}
\ctr@ld@f\def\Figg@tXY#1{\c@ntr@lnum{#1}%
    \expandafter\expandafter\expandafter\extr@ctC\csname\objc@de\endcsname:}
\ctr@ln@m\extr@ctC
\ctr@ld@f\def\extr@ctCDD#1/#2,#3,#4:{\v@lX=#2\v@lY=#3}
\ctr@ld@f\def\extr@ctCTD#1/#2,#3,#4:{\v@lX=#2\v@lY=#3\v@lZ=#4}
\ctr@ld@f\def\Figg@tXYa#1{\c@ntr@lnum{#1}%
    \expandafter\expandafter\expandafter\extr@ctCa\csname\objc@de\endcsname:}
\ctr@ln@m\extr@ctCa
\ctr@ld@f\def\extr@ctCaDD#1/#2,#3,#4:{\v@lXa=#2\v@lYa=#3}
\ctr@ld@f\def\extr@ctCaTD#1/#2,#3,#4:{\v@lXa=#2\v@lYa=#3\v@lZa=#4}
\ctr@ln@m\t@xt@
\ctr@ld@f\def\figinit#1{\t@stc@tcodech@nge\initpr@lim\Figinit@#1,:\initpss@ttings\ignorespaces}
\ctr@ld@f\def\Figinit@#1,#2:{\setunit@{#1}\def\t@xt@{#2}\ifx\t@xt@\empty\else\Figinit@@#2:\fi}
\ctr@ld@f\def\Figinit@@#1#2:{\if#12 \else\Figs@tproj{#1}\initTD@\fi}
\ctr@ln@w{newif}\ifTr@isDim
\ctr@ld@f\def\UnD@fined{UNDEFINED}
\ctr@ln@m\@utoFN
\ctr@ln@m\@utoFInDone
\ctr@ln@m\disob@unit
\ctr@ld@f\def\initpr@lim{\initb@undb@x\figsetmark{}\figsetptname{$A_{##1}$}\def\Sc@leFact{1}%
    \initDD@\figsetroundcoord{yes}\GR@critrue\expandafter\setupd@te\D@FTupdate:%
    \edef\disob@unit{\UnD@fined}\edef\t@rgetpt{\UnD@fined}\gdef\@utoFInDone{1}\gdef\@utoFN{0}}
\ctr@ld@f\def\initDD@{\Tr@isDimfalse%
    \ifPDFm@ke%
     \let\Ps@rcerc=\Ps@rcercBz%
     \let\Ps@rell=\Ps@rellBz%
    \fi
    \let\c@lDCUn=\c@lDCUnDD%
    \let\c@lDCDeux=\c@lDCDeuxDD%
    \let\c@ldefproj=\relax%
    \let\c@lproscal=\c@lproscalDD%
    \let\c@lprojSP=\relax%
    \let\extr@ctC=\extr@ctCDD%
    \let\extr@ctCa=\extr@ctCaDD%
    \let\extr@ctCF=\extr@ctCFDD%
    \let\Figp@intreg=\Figp@intregDD%
    \let\Figpts@xes=\Figpts@xesDD%
    \let\getredf@ctB=\getredf@ctBDD%
    \let\n@rmeucSV=\n@rmeucSVDD\let\n@rmeuc=\n@rmeucDD\let\n@rmeucC\n@rmeucCDD\let\n@rminf=\n@rminfDD%
    \let\pr@dMatV=\pr@dMatVDD%
    \let\Q@@xes=\Q@@xesDD%
    \let\vecunit@=\vecunit@DD%
    \let\figcoord=\figcoordDD%
    \let\figgetangle=\figgetangleDD%
    \let\figpt=\figptDD%
    \let\figptBezier=\figptBezierDD%
    \let\figptbary=\figptbaryDD%
    \let\figptcirc=\figptcircDD%
    \let\figptcircumcenter=\figptcircumcenterDD%
    \let\figptcopy=\figptcopyDD%
    \let\figptcurvcenter=\figptcurvcenterDD%
    \let\figptell=\figptellDD%
    \let\figptendnormal=\figptendnormalDD%
    \let\figptinterlineplane=\figptinterlineplaneDD%
    \let\figptinterlines=\inters@cDD%
    \let\figptorthocenter=\figptorthocenterDD%
    \let\figptorthoprojline=\figptorthoprojlineDD%
    \let\figptorthoprojplane=\figptorthoprojplaneDD%
    \let\figptrot=\figptrotDD%
    \let\figptscontrol=\figptscontrolDD%
    \let\figptsintercirc=\figptsintercircDD%
    \let\figptsinterlinell=\figptsinterlinellDD%
    \let\figptsorthoprojline=\figptsorthoprojlineDD%
    \let\figptorthoprojplane=\figptorthoprojplaneDD%
    \let\figptsrot=\figptsrotDD%
    \let\figptssym=\figptssymDD%
    \let\figptstra=\figptstraDD%
    \let\figptsym=\figptsymDD%
    \let\figpttraC=\figpttraCDD%
    \let\figpttra=\figpttraDD%
    \let\figptvisilimSL=\figptvisilimSLDD%
    \let\figsetobdist=\figsetobdistDD%
    \let\figsettarget=\figsettargetDD%
    \let\figsetview=\figsetviewDD%
    \let\figvectDBezier=\figvectDBezierDD%
    \let\figvectN=\figvectNDD%
    \let\figvectNV=\figvectNVDD%
    \let\figvectP=\figvectPDD%
    \let\figvectU=\figvectUDD%
    \let\figdrawarccircP=\Q@arccircPDD%
    \let\figdrawarccirc=\Q@arccircDD%
    \let\figdrawarcell=\Q@arcellDD%
    \let\figdrawarcellPA=\Q@arcellPADD%
    \let\figdrawarrowBezier=\Q@arrowBezierDD%
    \let\figdrawarrowcircP=\Q@arrowcircPDD%
    \let\figdrawarrowcirc=\Q@arrowcircDD%
    \let\figdrawarrowhead=\Q@arrowheadDD%
    \let\figdrawarrow=\Q@arrowDD%
    \let\figdrawBezier=\Q@BezierDD%
    \let\figdrawcirc=\Q@circDD%
    \let\figdrawcurve=\Q@curveDD%
    \let\figdrawnormal=\Q@normalDD%
    }
\ctr@ld@f\def\initTD@{\Tr@isDimtrue\initb@undb@xTD\newt@rgetptfalse\newdis@bfalse%
    \let\c@lDCUn=\c@lDCUnTD%
    \let\c@lDCDeux=\c@lDCDeuxTD%
    \let\c@ldefproj=\c@ldefprojTD%
    \let\c@lproscal=\c@lproscalTD%
    \let\extr@ctC=\extr@ctCTD%
    \let\extr@ctCa=\extr@ctCaTD%
    \let\extr@ctCF=\extr@ctCFTD%
    \let\Figp@intreg=\Figp@intregTD%
    \let\Figpts@xes=\Figpts@xesTD%
    \let\getredf@ctB=\getredf@ctBTD%
    \let\n@rmeucSV=\n@rmeucSVTD\let\n@rmeuc=\n@rmeucTD\let\n@rmeucC\n@rmeucCTD\let\n@rminf=\n@rminfTD%
    \let\pr@dMatV=\pr@dMatVTD%
    \let\Q@@xes=\Q@@xesTD%
    \let\vecunit@=\vecunit@TD%
    \let\figcoord=\figcoordTD%
    \let\figgetangle=\figgetangleTD%
    \let\figpt=\figptTD%
    \let\figptBezier=\figptBezierTD%
    \let\figptbary=\figptbaryTD%
    \let\figptcirc=\figptcircTD%
    \let\figptcircumcenter=\figptcircumcenterTD%
    \let\figptcopy=\figptcopyTD%
    \let\figptcurvcenter=\figptcurvcenterTD%
    \let\figptinterlineplane=\figptinterlineplaneTD%
    \let\figptinterlines=\inters@cTD%
    \let\figptorthocenter=\figptorthocenterTD%
    \let\figptorthoprojline=\figptorthoprojlineTD%
    \let\figptorthoprojplane=\figptorthoprojplaneTD%
    \let\figptrot=\figptrotTD%
    \let\figptscontrol=\figptscontrolTD%
    \let\figptsintercirc=\figptsintercircTD%
    \let\figptsorthoprojline=\figptsorthoprojlineTD%
    \let\figptsorthoprojplane=\figptsorthoprojplaneTD%
    \let\figptsrot=\figptsrotTD%
    \let\figptssym=\figptssymTD%
    \let\figptstra=\figptstraTD%
    \let\figptsym=\figptsymTD%
    \let\figpttraC=\figpttraCTD%
    \let\figpttra=\figpttraTD%
    \let\figptvisilimSL=\figptvisilimSLTD%
    \let\figsetobdist=\figsetobdistTD%
    \let\figsettarget=\figsettargetTD%
    \let\figsetview=\figsetviewTD%
    \let\figvectDBezier=\figvectDBezierTD%
    \let\figvectN=\figvectNTD%
    \let\figvectNV=\figvectNVTD%
    \let\figvectP=\figvectPTD%
    \let\figvectU=\figvectUTD%
    \let\figdrawarccircP=\Q@arccircPTD%
    \let\figdrawarccirc=\Q@arccircTD%
    \let\figdrawarcell=\Q@arcellTD%
    \let\figdrawarcellPA=\Q@arcellPATD%
    \let\figdrawarrowBezier=\Q@arrowBezierTD%
    \let\figdrawarrowcircP=\Q@arrowcircPTD%
    \let\figdrawarrowcirc=\Q@arrowcircTD%
    \let\figdrawarrowhead=\Q@arrowheadTD%
    \let\figdrawarrow=\Q@arrowTD%
    \let\figdrawBezier=\Q@BezierTD%
    \let\figdrawcirc=\Q@circTD%
    \let\figdrawcurve=\Q@curveTD%
    }
\ctr@ld@f\def\un@v@ilable#1{\immediate\write16{*** The macro #1 is not available in the current context.}}
\ctr@ld@f\def\figinsert#1{{\def\t@xt@{#1}\relax%
    \ifx\t@xt@\empty\ifnum\@utoFInDone>\z@\Figinsert@\DefGIfilen@me,:\fi%
    \else\expandafter\FiginsertNu@#1 :\fi}\ignorespaces}
\ctr@ld@f\def\FiginsertNu@#1 #2:{\def\t@xt@{#1}\relax\ifx\t@xt@\empty\def\t@xt@{#2}%
    \ifx\t@xt@\empty\ifnum\@utoFInDone>\z@\Figinsert@\DefGIfilen@me,:\fi%
    \else\FiginsertNu@#2:\fi\else\expandafter\FiginsertNd@#1 #2:\fi}
\ctr@ld@f\def\FiginsertNd@#1#2:{\ifcat#1a\Figinsert@#1#2,:\else%
    \ifnum\@utoFInDone>\z@\Figinsert@\DefGIfilen@me,#1#2,:\fi\fi}
\ctr@ln@m\Sc@leFact
\ctr@ld@f\def\Figinsert@#1,#2:{\def\t@xt@{#2}\ifx\t@xt@\empty\xdef\Sc@leFact{1}\else%
    \X@rgdeux@#2\xdef\Sc@leFact{\@rgdeux}\fi%
    \Figdisc@rdLTS{#1}{\t@xt@}\@psfgetbb{\t@xt@}%
    \v@lX=\@psfllx\p@\v@lX=\ptpsT@pt\v@lX\v@lX=\Sc@leFact\v@lX%
    \v@lY=\@psflly\p@\v@lY=\ptpsT@pt\v@lY\v@lY=\Sc@leFact\v@lY%
    \b@undb@x{\v@lX}{\v@lY}%
    \v@lX=\@psfurx\p@\v@lX=\ptpsT@pt\v@lX\v@lX=\Sc@leFact\v@lX%
    \v@lY=\@psfury\p@\v@lY=\ptpsT@pt\v@lY\v@lY=\Sc@leFact\v@lY%
    \b@undb@x{\v@lX}{\v@lY}%
    \ifPDFm@ke\Figinclud@PDF{\t@xt@}{\Sc@leFact}\else%
    \v@lX=\c@nt pt\v@lX=\Sc@leFact\v@lX\edef\F@ct{\repdecn@mb{\v@lX}}%
    \ifx\TeXturesonMacOSltX\special{postscriptfile #1 vscale=\F@ct\space hscale=\F@ct}%
    \else\includegraphics{#1}\fi\fi%
    \message{[\t@xt@]}\ignorespaces}
\ctr@ld@f\def\Figdisc@rdLTS#1#2{\expandafter\Figdisc@rdLTS@#1 :#2}
\ctr@ld@f\def\Figdisc@rdLTS@#1 #2:#3{\def#3{#1}\relax\ifx#3\empty\expandafter\Figdisc@rdLTS@#2:#3\fi}
\ctr@ld@f\def\figinsertE#1{\FiginsertE@#1,:\ignorespaces}
\ctr@ld@f\def\FiginsertE@#1,#2:{{\def\t@xt@{#2}\ifx\t@xt@\empty\xdef\Sc@leFact{1}\else%
    \X@rgdeux@#2\xdef\Sc@leFact{\@rgdeux}\fi%
    \Figdisc@rdLTS{#1}{\t@xt@}\pdfximage{\t@xt@}%
    \setbox\Gb@x=\hbox{\pdfrefximage\pdflastximage}%
    \v@lX=\z@\v@lY=-\Sc@leFact\dp\Gb@x\b@undb@x{\v@lX}{\v@lY}%
    \advance\v@lX\Sc@leFact\wd\Gb@x\advance\v@lY\Sc@leFact\dp\Gb@x%
    \advance\v@lY\Sc@leFact\ht\Gb@x\b@undb@x{\v@lX}{\v@lY}%
    \v@lX=\Sc@leFact\wd\Gb@x\pdfximage width \v@lX {\t@xt@}%
    \rlap{\pdfrefximage\pdflastximage}\message{[\t@xt@]}}\ignorespaces}
\ctr@ld@f\def\X@rgdeux@#1,{\edef\@rgdeux{#1}}
\ctr@ln@m\figpt
\ctr@ld@f\def\figptDD#1:#2(#3,#4){\ifGR@cri\c@ntr@lnum{#1}%
    {\v@lX=#3\unit@\v@lY=#4\unit@\Fig@dmpt{#2}{\z@}}\ignorespaces\fi}
\ctr@ld@f\def\Fig@dmpt#1#2{\def\t@xt@{#1}\ifx\t@xt@\empty\def\B@@ltxt{\z@}%
    \else\expandafter\gdef\csname\objc@de T\endcsname{#1}\def\B@@ltxt{\@ne}\fi%
    \expandafter\xdef\csname\objc@de\endcsname{\ifitis@vect@r\C@dCl@svect%
    \else\C@dCl@spt\fi,\z@,\B@@ltxt/\the\v@lX,\the\v@lY,#2}}
\ctr@ld@f\def\C@dCl@spt{P}
\ctr@ld@f\def\C@dCl@svect{V}
\ctr@ln@m\c@@rdYZ
\ctr@ln@m\c@@rdY
\ctr@ld@f\def\figptTD#1:#2(#3,#4){\ifGR@cri\c@ntr@lnum{#1}%
    \def\c@@rdYZ{#4,0,0}\extrairelepremi@r\c@@rdY\de\c@@rdYZ%
    \extrairelepremi@r\c@@rdZ\de\c@@rdYZ%
    {\v@lX=#3\unit@\v@lY=\c@@rdY\unit@\v@lZ=\c@@rdZ\unit@\Fig@dmpt{#2}{\the\v@lZ}%
    \b@undb@xTD{\v@lX}{\v@lY}{\v@lZ}}\ignorespaces\fi}
\ctr@ln@m\Figp@intreg
\ctr@ld@f\def\Figp@intregDD#1:#2(#3,#4){\c@ntr@lnum{#1}%
    {\result@t=#4\v@lX=#3\v@lY=\result@t\Fig@dmpt{#2}{\z@}}\ignorespaces}
\ctr@ld@f\def\Figp@intregTD#1:#2(#3,#4){\c@ntr@lnum{#1}%
    \def\c@@rdYZ{#4,\z@,\z@}\extrairelepremi@r\c@@rdY\de\c@@rdYZ%
    \extrairelepremi@r\c@@rdZ\de\c@@rdYZ%
    {\v@lX=#3\v@lY=\c@@rdY\v@lZ=\c@@rdZ\Fig@dmpt{#2}{\the\v@lZ}%
    \b@undb@xTD{\v@lX}{\v@lY}{\v@lZ}}\ignorespaces}
\ctr@ln@m\figptBezier
\ctr@ld@f\def\figptBezierDD#1:#2:#3[#4,#5,#6,#7]{\ifGR@cri{\s@uvc@ntr@l\et@tfigptBezierDD%
    \FigptBezier@#3[#4,#5,#6,#7]\Figp@intregDD#1:{#2}(\v@lX,\v@lY)%
    \resetc@ntr@l\et@tfigptBezierDD}\ignorespaces\fi}
\ctr@ld@f\def\figptBezierTD#1:#2:#3[#4,#5,#6,#7]{\ifGR@cri{\s@uvc@ntr@l\et@tfigptBezierTD%
    \FigptBezier@#3[#4,#5,#6,#7]\Figp@intregTD#1:{#2}(\v@lX,\v@lY,\v@lZ)%
    \resetc@ntr@l\et@tfigptBezierTD}\ignorespaces\fi}
\ctr@ld@f\def\FigptBezier@#1[#2,#3,#4,#5]{\setc@ntr@l{2}%
    \edef\T@{#1}\v@leur=\p@\advance\v@leur-#1pt\edef\UNmT@{\repdecn@mb{\v@leur}}%
    \figptcopy-4:/#2/\figptcopy-3:/#3/\figptcopy-2:/#4/\figptcopy-1:/#5/%
    \l@mbd@un=-4 \l@mbd@de=-\thr@@\p@rtent=\m@ne\c@lDecast%
    \l@mbd@un=-4 \l@mbd@de=-\thr@@\p@rtent=-\tw@\c@lDecast%
    \l@mbd@un=-4 \l@mbd@de=-\thr@@\p@rtent=-\thr@@\c@lDecast\Figg@tXY{-4}}
\ctr@ln@m\c@lDCUn
\ctr@ld@f\def\c@lDCUnDD#1#2{\Figg@tXY{#1}\v@lX=\UNmT@\v@lX\v@lY=\UNmT@\v@lY%
    \Figg@tXYa{#2}\advance\v@lX\T@\v@lXa\advance\v@lY\T@\v@lYa%
    \Figp@intregDD#1:(\v@lX,\v@lY)}
\ctr@ld@f\def\c@lDCUnTD#1#2{\Figg@tXY{#1}\v@lX=\UNmT@\v@lX\v@lY=\UNmT@\v@lY\v@lZ=\UNmT@\v@lZ%
    \Figg@tXYa{#2}\advance\v@lX\T@\v@lXa\advance\v@lY\T@\v@lYa\advance\v@lZ\T@\v@lZa%
    \Figp@intregTD#1:(\v@lX,\v@lY,\v@lZ)}
\ctr@ld@f\def\c@lDecast{\relax\ifnum\l@mbd@un<\p@rtent\c@lDCUn{\l@mbd@un}{\l@mbd@de}%
    \advance\l@mbd@un\@ne\advance\l@mbd@de\@ne\c@lDecast\fi}
\ctr@ld@f\def\figptmap#1:#2=#3/#4/#5/{\ifGR@cri{\s@uvc@ntr@l\et@tfigptmap%
    \setc@ntr@l{2}\figvectP-1[#4,#3]\Figg@tXY{-1}%
    \pr@dMatV/#5/\figpttra#1:{#2}=#4/1,-1/%
    \resetc@ntr@l\et@tfigptmap}\ignorespaces\fi}
\ctr@ln@m\pr@dMatV
\ctr@ld@f\def\pr@dMatVDD/#1,#2;#3,#4/{\v@lXa=#1\v@lX\advance\v@lXa#2\v@lY%
    \v@lYa=#3\v@lX\advance\v@lYa#4\v@lY\Figv@ctCreg-1(\v@lXa,\v@lYa)}
\ctr@ld@f\def\pr@dMatVTD/#1,#2,#3;#4,#5,#6;#7,#8,#9/{%
    \v@lXa=#1\v@lX\advance\v@lXa#2\v@lY\advance\v@lXa#3\v@lZ%
    \v@lYa=#4\v@lX\advance\v@lYa#5\v@lY\advance\v@lYa#6\v@lZ%
    \v@lZa=#7\v@lX\advance\v@lZa#8\v@lY\advance\v@lZa#9\v@lZ%
    \Figv@ctCreg-1(\v@lXa,\v@lYa,\v@lZa)}
\ctr@ln@m\figptbary
\ctr@ld@f\def\figptbaryDD#1:#2[#3;#4]{\ifGR@cri{\edef\list@num{#3}\extrairelepremi@r\p@int\de\list@num%
    \s@mme=\z@\@ecfor\c@ef:=#4\do{\advance\s@mme\c@ef}%
    \edef\listec@ef{#4,0}\extrairelepremi@r\c@ef\de\listec@ef%
    \Figg@tXY{\p@int}\divide\v@lX\s@mme\divide\v@lY\s@mme%
    \multiply\v@lX\c@ef\multiply\v@lY\c@ef%
    \@ecfor\p@int:=\list@num\do{\extrairelepremi@r\c@ef\de\listec@ef%
           \Figg@tXYa{\p@int}\divide\v@lXa\s@mme\divide\v@lYa\s@mme%
           \multiply\v@lXa\c@ef\multiply\v@lYa\c@ef%
           \advance\v@lX\v@lXa\advance\v@lY\v@lYa}%
    \Figp@intregDD#1:{#2}(\v@lX,\v@lY)}\ignorespaces\fi}
\ctr@ld@f\def\figptbaryTD#1:#2[#3;#4]{\ifGR@cri{\edef\list@num{#3}\extrairelepremi@r\p@int\de\list@num%
    \s@mme=\z@\@ecfor\c@ef:=#4\do{\advance\s@mme\c@ef}%
    \edef\listec@ef{#4,0}\extrairelepremi@r\c@ef\de\listec@ef%
    \Figg@tXY{\p@int}\divide\v@lX\s@mme\divide\v@lY\s@mme\divide\v@lZ\s@mme%
    \multiply\v@lX\c@ef\multiply\v@lY\c@ef\multiply\v@lZ\c@ef%
    \@ecfor\p@int:=\list@num\do{\extrairelepremi@r\c@ef\de\listec@ef%
           \Figg@tXYa{\p@int}\divide\v@lXa\s@mme\divide\v@lYa\s@mme\divide\v@lZa\s@mme%
           \multiply\v@lXa\c@ef\multiply\v@lYa\c@ef\multiply\v@lZa\c@ef%
           \advance\v@lX\v@lXa\advance\v@lY\v@lYa\advance\v@lZ\v@lZa}%
    \Figp@intregTD#1:{#2}(\v@lX,\v@lY,\v@lZ)}\ignorespaces\fi}
\ctr@ld@f\def\figptbaryR#1:#2[#3;#4]{\ifGR@cri{%
    \v@leur=\z@\@ecfor\c@ef:=#4\do{\maxim@m{\v@lmax}{\c@ef pt}{-\c@ef pt}%
    \ifdim\v@lmax>\v@leur\v@leur=\v@lmax\fi}%
    \ifdim\v@leur<\p@\f@ctech=\@M\else\ifdim\v@leur<\t@n\p@\f@ctech=\@m\else%
    \ifdim\v@leur<\c@nt\p@\f@ctech=\c@nt\else\ifdim\v@leur<\@m\p@\f@ctech=\t@n\else%
    \f@ctech=\@ne\fi\fi\fi\fi%
    \def\listec@ef{0}%
    \@ecfor\c@ef:=#4\do{\sc@lec@nvRI{\c@ef pt}\edef\listec@ef{\listec@ef,\the\s@mme}}%
    \extrairelepremi@r\c@ef\de\listec@ef\figptbary#1:#2[#3;\listec@ef]}\ignorespaces\fi}
\ctr@ld@f\def\sc@lec@nvRI#1{\v@leur=#1\p@rtentiere{\s@mme}{\v@leur}\advance\v@leur-\s@mme\p@%
    \multiply\v@leur\f@ctech\p@rtentiere{\p@rtent}{\v@leur}%
    \multiply\s@mme\f@ctech\advance\s@mme\p@rtent}
\ctr@ln@m\figptcirc
\ctr@ld@f\def\figptcircDD#1:#2:#3;#4(#5){\ifGR@cri{\s@uvc@ntr@l\et@tfigptcircDD%
    \c@lptellDD#1:{#2}:#3;#4,#4(#5)\resetc@ntr@l\et@tfigptcircDD}\ignorespaces\fi}
\ctr@ld@f\def\figptcircTD#1:#2:#3,#4,#5;#6(#7){\ifGR@cri{\s@uvc@ntr@l\et@tfigptcircTD%
    \setc@ntr@l{2}\c@lExtAxes#3,#4,#5(#6)\figptellP#1:{#2}:#3,-4,-5(#7)%
    \resetc@ntr@l\et@tfigptcircTD}\ignorespaces\fi}
\ctr@ln@m\figptcircumcenter
\ctr@ld@f\def\figptcircumcenterDD#1:#2[#3,#4,#5]{\ifGR@cri{\s@uvc@ntr@l\et@tfigptcircumcenterDD%
    \setc@ntr@l{2}\figvectNDD-5[#3,#4]\figptbaryDD-3:[#3,#4;1,1]%
                  \figvectNDD-6[#4,#5]\figptbaryDD-4:[#4,#5;1,1]%
    \resetc@ntr@l{2}\inters@cDD#1:{#2}[-3,-5;-4,-6]%
    \resetc@ntr@l\et@tfigptcircumcenterDD}\ignorespaces\fi}
\ctr@ld@f\def\figptcircumcenterTD#1:#2[#3,#4,#5]{\ifGR@cri{\s@uvc@ntr@l\et@tfigptcircumcenterTD%
    \setc@ntr@l{2}\figvectNTD-1[#3,#4,#5]%
    \figvectPTD-3[#3,#4]\figvectNVTD-5[-1,-3]\figptbaryTD-3:[#3,#4;1,1]%
    \figvectPTD-4[#4,#5]\figvectNVTD-6[-1,-4]\figptbaryTD-4:[#4,#5;1,1]%
    \resetc@ntr@l{2}\inters@cTD#1:{#2}[-3,-5;-4,-6]%
    \resetc@ntr@l\et@tfigptcircumcenterTD}\ignorespaces\fi}
\ctr@ln@m\figptcopy
\ctr@ld@f\def\figptcopyDD#1:#2/#3/{\ifGR@cri{\Figg@tXY{#3}%
    \Figp@intregDD#1:{#2}(\v@lX,\v@lY)}\ignorespaces\fi}
\ctr@ld@f\def\figptcopyTD#1:#2/#3/{\ifGR@cri{\Figg@tXY{#3}%
    \Figp@intregTD#1:{#2}(\v@lX,\v@lY,\v@lZ)}\ignorespaces\fi}
\ctr@ln@m\figptcurvcenter
\ctr@ld@f\def\figptcurvcenterDD#1:#2:#3[#4,#5,#6,#7]{\ifGR@cri{\s@uvc@ntr@l\et@tfigptcurvcenterDD%
    \setc@ntr@l{2}\c@lcurvradDD#3[#4,#5,#6,#7]\edef\Sprim@{\repdecn@mb{\result@t}}%
    \figptBezierDD-1::#3[#4,#5,#6,#7]\figpttraDD#1:{#2}=-1/\Sprim@,-5/%
    \resetc@ntr@l\et@tfigptcurvcenterDD}\ignorespaces\fi}
\ctr@ld@f\def\figptcurvcenterTD#1:#2:#3[#4,#5,#6,#7]{\ifGR@cri{\s@uvc@ntr@l\et@tfigptcurvcenterTD%
    \setc@ntr@l{2}\figvectDBezierTD -5:1,#3[#4,#5,#6,#7]%
    \figvectDBezierTD -6:2,#3[#4,#5,#6,#7]\vecunit@TD{-5}{-5}%
    \edef\Sprim@{\repdecn@mb{\result@t}}\figvectNVTD-1[-6,-5]%
    \figvectNVTD-5[-5,-1]\c@lproscalTD\v@leur[-6,-5]%
    \invers@{\v@leur}{\v@leur}\v@leur=\Sprim@\v@leur\v@leur=\Sprim@\v@leur%
    \figptBezierTD-1::#3[#4,#5,#6,#7]\edef\Sprim@{\repdecn@mb{\v@leur}}%
    \figpttraTD#1:{#2}=-1/\Sprim@,-5/\resetc@ntr@l\et@tfigptcurvcenterTD}\ignorespaces\fi}
\ctr@ld@f\def\c@lcurvradDD#1[#2,#3,#4,#5]{{\figvectDBezierDD -5:1,#1[#2,#3,#4,#5]%
    \figvectDBezierDD -6:2,#1[#2,#3,#4,#5]\vecunit@DD{-5}{-5}%
    \edef\Sprim@{\repdecn@mb{\result@t}}\figvectNVDD-5[-5]\c@lproscalDD\v@leur[-6,-5]%
    \invers@{\v@leur}{\v@leur}\v@leur=\Sprim@\v@leur\v@leur=\Sprim@\v@leur%
    \global\result@t=\v@leur}}
\ctr@ln@m\figptell
\ctr@ld@f\def\figptellDD#1:#2:#3;#4,#5(#6,#7){\ifGR@cri{\s@uvc@ntr@l\et@tfigptell%
    \c@lptellDD#1::#3;#4,#5(#6)\figptrotDD#1:{#2}=#1/#3,#7/%
    \resetc@ntr@l\et@tfigptell}\ignorespaces\fi}
\ctr@ld@f\def\c@lptellDD#1:#2:#3;#4,#5(#6){\c@ssin{\C@}{\S@}{#6}\v@lmin=\C@ pt\v@lmax=\S@ pt%
    \v@lmin=#4\v@lmin\v@lmax=#5\v@lmax%
    \edef\Xc@mp{\repdecn@mb{\v@lmin}}\edef\Yc@mp{\repdecn@mb{\v@lmax}}%
    \setc@ntr@l{2}\figvectC-1(\Xc@mp,\Yc@mp)\figpttraDD#1:{#2}=#3/1,-1/}
\ctr@ld@f\def\figptellP#1:#2:#3,#4,#5(#6){\ifGR@cri{\s@uvc@ntr@l\et@tfigptellP%
    \setc@ntr@l{2}\figvectP-1[#3,#4]\figvectP-2[#3,#5]%
    \v@leur=#6pt\c@lptellP{#3}{-1}{-2}\figptcopy#1:{#2}/-3/%
    \resetc@ntr@l\et@tfigptellP}\ignorespaces\fi}
\ctr@ln@m\@ngle
\ctr@ld@f\def\c@lptellP#1#2#3{\edef\@ngle{\repdecn@mb\v@leur}\c@ssin{\C@}{\S@}{\@ngle}%
    \figpttra-3:=#1/\C@,#2/\figpttra-3:=-3/\S@,#3/}
\ctr@ln@m\figptendnormal
\ctr@ld@f\def\figptendnormalDD#1:#2:#3,#4[#5,#6]{\ifGR@cri{\s@uvc@ntr@l\et@tfigptendnormal%
    \Figg@tXYa{#5}\Figg@tXY{#6}%
    \advance\v@lX-\v@lXa\advance\v@lY-\v@lYa%
    \setc@ntr@l{2}\Figv@ctCreg-1(\v@lX,\v@lY)\vecunit@{-1}{-1}\Figg@tXY{-1}%
    \delt@=#3\unit@\maxim@m{\delt@}{\delt@}{-\delt@}\edef\l@ngueur{\repdecn@mb{\delt@}}%
    \v@lX=\l@ngueur\v@lX\v@lY=\l@ngueur\v@lY%
    \delt@=\p@\advance\delt@-#4pt\edef\l@ngueur{\repdecn@mb{\delt@}}%
    \figptbaryR-1:[#5,#6;#4,\l@ngueur]\Figg@tXYa{-1}%
    \advance\v@lXa\v@lY\advance\v@lYa-\v@lX%
    \setc@ntr@l{1}\Figp@intregDD#1:{#2}(\v@lXa,\v@lYa)\resetc@ntr@l\et@tfigptendnormal}%
    \ignorespaces\fi}
\ctr@ld@f\def\figptexcenter#1:#2[#3,#4,#5]{\ifGR@cri{\let@xte={-}
    \Figptexinsc@nter#1:#2[#3,#4,#5]}\ignorespaces\fi}
\ctr@ld@f\def\figptincenter#1:#2[#3,#4,#5]{\ifGR@cri{\let@xte={}
    \Figptexinsc@nter#1:#2[#3,#4,#5]}\ignorespaces\fi}
\ctr@ld@f\let\figptinscribedcenter=\figptincenter
\ctr@ld@f\def\Figptexinsc@nter#1:#2[#3,#4,#5]{%
    \figgetdist\LA@[#4,#5]\figgetdist\LB@[#3,#5]\figgetdist\LC@[#3,#4]%
    \figptbaryR#1:{#2}[#3,#4,#5;\the\let@xte\LA@,\LB@,\LC@]}
\ctr@ln@m\figptinterlineplane
\ctr@ld@f\def\figptinterlineplaneDD{\un@v@ilable{figptinterlineplane}}
\ctr@ld@f\def\figptinterlineplaneTD#1:#2[#3,#4;#5,#6]{\ifGR@cri{\s@uvc@ntr@l\et@tfigptinterlineplane%
    \setc@ntr@l{2}\figvectPTD-1[#3,#5]\vecunit@TD{-2}{#6}%
    \r@pPSTD\v@leur[-2,-1,#4]\edef\v@lcoef{\repdecn@mb{\v@leur}}%
    \figpttraTD#1:{#2}=#3/\v@lcoef,#4/\resetc@ntr@l\et@tfigptinterlineplane}\ignorespaces\fi}
\ctr@ln@m\figptorthocenter
\ctr@ld@f\def\figptorthocenterDD#1:#2[#3,#4,#5]{\ifGR@cri{\s@uvc@ntr@l\et@tfigptorthocenterDD%
    \setc@ntr@l{2}\figvectNDD-3[#3,#4]\figvectNDD-4[#4,#5]%
    \resetc@ntr@l{2}\inters@cDD#1:{#2}[#5,-3;#3,-4]%
    \resetc@ntr@l\et@tfigptorthocenterDD}\ignorespaces\fi}
\ctr@ld@f\def\figptorthocenterTD#1:#2[#3,#4,#5]{\ifGR@cri{\s@uvc@ntr@l\et@tfigptorthocenterTD%
    \setc@ntr@l{2}\figvectNTD-1[#3,#4,#5]%
    \figvectPTD-2[#3,#4]\figvectNVTD-3[-1,-2]%
    \figvectPTD-2[#4,#5]\figvectNVTD-4[-1,-2]%
    \resetc@ntr@l{2}\inters@cTD#1:{#2}[#5,-3;#3,-4]%
    \resetc@ntr@l\et@tfigptorthocenterTD}\ignorespaces\fi}
\ctr@ln@m\figptorthoprojline
\ctr@ld@f\def\figptorthoprojlineDD#1:#2=#3/#4,#5/{\ifGR@cri{\s@uvc@ntr@l\et@tfigptorthoprojlineDD%
    \setc@ntr@l{2}\figvectPDD-3[#4,#5]\figvectNVDD-4[-3]\resetc@ntr@l{2}%
    \inters@cDD#1:{#2}[#3,-4;#4,-3]\resetc@ntr@l\et@tfigptorthoprojlineDD}\ignorespaces\fi}
\ctr@ld@f\def\figptorthoprojlineTD#1:#2=#3/#4,#5/{\ifGR@cri{\s@uvc@ntr@l\et@tfigptorthoprojlineTD%
    \setc@ntr@l{2}\figvectPTD-1[#4,#3]\figvectPTD-2[#4,#5]\vecunit@TD{-2}{-2}%
    \c@lproscalTD\v@leur[-1,-2]\edef\v@lcoef{\repdecn@mb{\v@leur}}%
    \figpttraTD#1:{#2}=#4/\v@lcoef,-2/\resetc@ntr@l\et@tfigptorthoprojlineTD}\ignorespaces\fi}
\ctr@ln@m\figptorthoprojplane
\ctr@ld@f\def\figptorthoprojplaneDD{\un@v@ilable{figptorthoprojplane}}
\ctr@ld@f\def\figptorthoprojplaneTD#1:#2=#3/#4,#5/{\ifGR@cri{\s@uvc@ntr@l\et@tfigptorthoprojplane%
    \setc@ntr@l{2}\figvectPTD-1[#3,#4]\vecunit@TD{-2}{#5}%
    \c@lproscalTD\v@leur[-1,-2]\edef\v@lcoef{\repdecn@mb{\v@leur}}%
    \figpttraTD#1:{#2}=#3/\v@lcoef,-2/\resetc@ntr@l\et@tfigptorthoprojplane}\ignorespaces\fi}
\ctr@ld@f\def\figpthom#1:#2=#3/#4,#5/{\ifGR@cri{\s@uvc@ntr@l\et@tfigpthom%
    \setc@ntr@l{2}\figvectP-1[#4,#3]\figpttra#1:{#2}=#4/#5,-1/%
    \resetc@ntr@l\et@tfigpthom}\ignorespaces\fi}
\ctr@ld@f\def\figptinv#1:#2=#3/#4,#5/{\ifGR@cri{\s@uvc@ntr@l\et@tfigptinv%
    \setc@ntr@l{2}\figvectP-1[#4,#3]\Figg@tXY{-1}%
    \getredf@ctB\f@ctech\n@rmeucC{\delt@}{-1}%
    \delt@=\ptT@unit@\delt@\delt@=\ptT@unit@\delt@%
    \invers@{\delt@}{\delt@}\multiply\f@ctech\f@ctech\divide\delt@\f@ctech%
    \delt@=#5\delt@\edef\v@lcoef{\repdecn@mb{\delt@}}\figpttra#1:{#2}=#4/\v@lcoef,-1/%
    \resetc@ntr@l\et@tfigptinv}\ignorespaces\fi}
\ctr@ln@m\figptrot
\ctr@ld@f\def\figptrotDD#1:#2=#3/#4,#5/{\ifGR@cri{\s@uvc@ntr@l\et@tfigptrotDD%
    \c@ssin{\C@}{\S@}{#5}\setc@ntr@l{2}\figvectPDD-1[#4,#3]\Figg@tXY{-1}%
    \v@lXa=\C@\v@lX\advance\v@lXa-\S@\v@lY%
    \v@lYa=\S@\v@lX\advance\v@lYa\C@\v@lY%
    \Figv@ctCreg-1(\v@lXa,\v@lYa)\figpttraDD#1:{#2}=#4/1,-1/%
    \resetc@ntr@l\et@tfigptrotDD}\ignorespaces\fi}
\ctr@ld@f\def\figptrotTD#1:#2=#3/#4,#5,#6/{\ifGR@cri{\s@uvc@ntr@l\et@tfigptrotTD%
    \c@ssin{\C@}{\S@}{#5}%
    \setc@ntr@l{2}\figptorthoprojplaneTD-3:=#4/#3,#6/\figvectPTD-2[-3,#3]%
    \n@rmeucTD\v@leur{-2}\ifdim\v@leur<\Cepsil@n\Figg@tXYa{#3}\else%
    \edef\v@lcoef{\repdecn@mb{\v@leur}}\figvectNVTD-1[#6,-2]%
    \Figg@tXYa{-1}\v@lXa=\v@lcoef\v@lXa\v@lYa=\v@lcoef\v@lYa\v@lZa=\v@lcoef\v@lZa%
    \v@lXa=\S@\v@lXa\v@lYa=\S@\v@lYa\v@lZa=\S@\v@lZa\Figg@tXY{-2}%
    \advance\v@lXa\C@\v@lX\advance\v@lYa\C@\v@lY\advance\v@lZa\C@\v@lZ%
    \Figg@tXY{-3}\advance\v@lXa\v@lX\advance\v@lYa\v@lY\advance\v@lZa\v@lZ\fi%
    \Figp@intregTD#1:{#2}(\v@lXa,\v@lYa,\v@lZa)\resetc@ntr@l\et@tfigptrotTD}\ignorespaces\fi}
\ctr@ln@m\figptsym
\ctr@ld@f\def\figptsymDD#1:#2=#3/#4,#5/{\ifGR@cri{\s@uvc@ntr@l\et@tfigptsymDD%
    \resetc@ntr@l{2}\figptorthoprojlineDD-5:=#3/#4,#5/\figvectPDD-2[#3,-5]%
    \figpttraDD#1:{#2}=#3/2,-2/\resetc@ntr@l\et@tfigptsymDD}\ignorespaces\fi}
\ctr@ld@f\def\figptsymTD#1:#2=#3/#4,#5/{\ifGR@cri{\s@uvc@ntr@l\et@tfigptsymTD%
    \resetc@ntr@l{2}\figptorthoprojplaneTD-3:=#3/#4,#5/\figvectPTD-2[#3,-3]%
    \figpttraTD#1:{#2}=#3/2,-2/\resetc@ntr@l\et@tfigptsymTD}\ignorespaces\fi}
\ctr@ln@m\figpttra
\ctr@ld@f\def\figpttraDD#1:#2=#3/#4,#5/{\ifGR@cri{\Figg@tXYa{#5}\v@lXa=#4\v@lXa\v@lYa=#4\v@lYa%
    \Figg@tXY{#3}\advance\v@lX\v@lXa\advance\v@lY\v@lYa%
    \Figp@intregDD#1:{#2}(\v@lX,\v@lY)}\ignorespaces\fi}
\ctr@ld@f\def\figpttraTD#1:#2=#3/#4,#5/{\ifGR@cri{\Figg@tXYa{#5}\v@lXa=#4\v@lXa\v@lYa=#4\v@lYa%
    \v@lZa=#4\v@lZa\Figg@tXY{#3}\advance\v@lX\v@lXa\advance\v@lY\v@lYa%
    \advance\v@lZ\v@lZa\Figp@intregTD#1:{#2}(\v@lX,\v@lY,\v@lZ)}\ignorespaces\fi}
\ctr@ln@m\figpttraC
\ctr@ld@f\def\figpttraCDD#1:#2=#3/#4,#5/{\ifGR@cri{\v@lXa=#4\unit@\v@lYa=#5\unit@%
    \Figg@tXY{#3}\advance\v@lX\v@lXa\advance\v@lY\v@lYa%
    \Figp@intregDD#1:{#2}(\v@lX,\v@lY)}\ignorespaces\fi}
\ctr@ld@f\def\figpttraCTD#1:#2=#3/#4,#5,#6/{\ifGR@cri{\v@lXa=#4\unit@\v@lYa=#5\unit@\v@lZa=#6\unit@%
    \Figg@tXY{#3}\advance\v@lX\v@lXa\advance\v@lY\v@lYa\advance\v@lZ\v@lZa%
    \Figp@intregTD#1:{#2}(\v@lX,\v@lY,\v@lZ)}\ignorespaces\fi}
\ctr@ld@f\def\figptsaxes#1:#2(#3){\ifGR@cri{\an@lys@xes#3,:\ifx\t@xt@\empty%
    \ifTr@isDim\Figpts@xes#1:#2(0,#3,0,#3,0,#3)\else\Figpts@xes#1:#2(0,#3,0,#3)\fi%
    \else\Figpts@xes#1:#2(#3)\fi}\ignorespaces\fi}
\ctr@ln@m\Figpts@xes
\ctr@ld@f\def\Figpts@xesDD#1:#2(#3,#4,#5,#6){%
    \s@mme=#1\figpttraC\the\s@mme:$x$=#2/#4,0/%
    \advance\s@mme\@ne\figpttraC\the\s@mme:$y$=#2/0,#6/}
\ctr@ld@f\def\Figpts@xesTD#1:#2(#3,#4,#5,#6,#7,#8){%
    \s@mme=#1\figpttraC\the\s@mme:$x$=#2/#4,0,0/%
    \advance\s@mme\@ne\figpttraC\the\s@mme:$y$=#2/0,#6,0/%
    \advance\s@mme\@ne\figpttraC\the\s@mme:$z$=#2/0,0,#8/}
\ctr@ld@f\def\figptsmap#1=#2/#3/#4/{\ifGR@cri{\s@uvc@ntr@l\et@tfigptsmap%
    \setc@ntr@l{2}\def\list@num{#2}\s@mme=#1%
    \@ecfor\p@int:=\list@num\do{\figvectP-1[#3,\p@int]\Figg@tXY{-1}%
    \pr@dMatV/#4/\figpttra\the\s@mme:=#3/1,-1/\advance\s@mme\@ne}%
    \resetc@ntr@l\et@tfigptsmap}\ignorespaces\fi}
\ctr@ln@m\figptscontrol
\ctr@ld@f\def\figptscontrolDD#1[#2,#3,#4,#5]{\ifGR@cri{\s@uvc@ntr@l\et@tfigptscontrolDD\setc@ntr@l{2}%
    \v@lX=\z@\v@lY=\z@\Figtr@nptDD{-5}{#2}\Figtr@nptDD{2}{#5}%
    \divide\v@lX\@vi\divide\v@lY\@vi%
    \Figtr@nptDD{3}{#3}\Figtr@nptDD{-1.5}{#4}\Figp@intregDD-1:(\v@lX,\v@lY)%
    \v@lX=\z@\v@lY=\z@\Figtr@nptDD{2}{#2}\Figtr@nptDD{-5}{#5}%
    \divide\v@lX\@vi\divide\v@lY\@vi\Figtr@nptDD{-1.5}{#3}\Figtr@nptDD{3}{#4}%
    \s@mme=#1\advance\s@mme\@ne\Figp@intregDD\the\s@mme:(\v@lX,\v@lY)%
    \figptcopyDD#1:/-1/\resetc@ntr@l\et@tfigptscontrolDD}\ignorespaces\fi}
\ctr@ld@f\def\figptscontrolTD#1[#2,#3,#4,#5]{\ifGR@cri{\s@uvc@ntr@l\et@tfigptscontrolTD\setc@ntr@l{2}%
    \v@lX=\z@\v@lY=\z@\v@lZ=\z@\Figtr@nptTD{-5}{#2}\Figtr@nptTD{2}{#5}%
    \divide\v@lX\@vi\divide\v@lY\@vi\divide\v@lZ\@vi%
    \Figtr@nptTD{3}{#3}\Figtr@nptTD{-1.5}{#4}\Figp@intregTD-1:(\v@lX,\v@lY,\v@lZ)%
    \v@lX=\z@\v@lY=\z@\v@lZ=\z@\Figtr@nptTD{2}{#2}\Figtr@nptTD{-5}{#5}%
    \divide\v@lX\@vi\divide\v@lY\@vi\divide\v@lZ\@vi\Figtr@nptTD{-1.5}{#3}\Figtr@nptTD{3}{#4}%
    \s@mme=#1\advance\s@mme\@ne\Figp@intregTD\the\s@mme:(\v@lX,\v@lY,\v@lZ)%
    \figptcopyTD#1:/-1/\resetc@ntr@l\et@tfigptscontrolTD}\ignorespaces\fi}
\ctr@ld@f\def\Figtr@nptDD#1#2{\Figg@tXYa{#2}\v@lXa=#1\v@lXa\v@lYa=#1\v@lYa%
    \advance\v@lX\v@lXa\advance\v@lY\v@lYa}
\ctr@ld@f\def\Figtr@nptTD#1#2{\Figg@tXYa{#2}\v@lXa=#1\v@lXa\v@lYa=#1\v@lYa\v@lZa=#1\v@lZa%
    \advance\v@lX\v@lXa\advance\v@lY\v@lYa\advance\v@lZ\v@lZa}
\ctr@ld@f\def\figptscontrolcurve#1,#2[#3]{\ifGR@cri{\s@uvc@ntr@l\et@tfigptscontrolcurve%
    \def\list@num{#3}\extrairelepremi@r\Ak@\de\list@num%
    \extrairelepremi@r\Ai@\de\list@num\extrairelepremi@r\Aj@\de\list@num%
    \s@mme=#1\figptcopy\the\s@mme:/\Ai@/%
    \setc@ntr@l{2}\figvectP -1[\Ak@,\Aj@]%
    \@ecfor\Ak@:=\list@num\do{\advance\s@mme\@ne\figpttra\the\s@mme:=\Ai@/\curv@roundness,-1/%
       \figvectP -1[\Ai@,\Ak@]\advance\s@mme\@ne\figpttra\the\s@mme:=\Aj@/-\curv@roundness,-1/%
       \advance\s@mme\@ne\figptcopy\the\s@mme:/\Aj@/%
       \edef\Ai@{\Aj@}\edef\Aj@{\Ak@}}\advance\s@mme-#1\divide\s@mme\thr@@%
       \xdef#2{\the\s@mme}%
    \resetc@ntr@l\et@tfigptscontrolcurve}\ignorespaces\fi}
\ctr@ln@m\figptsintercirc
\ctr@ld@f\def\figptsintercircDD#1[#2,#3;#4,#5]{\ifGR@cri{\s@uvc@ntr@l\et@tfigptsintercircDD%
    \setc@ntr@l{2}\let\c@lNVintc=\c@lNVintcDD\Figptsintercirc@#1[#2,#3;#4,#5]%
    \resetc@ntr@l\et@tfigptsintercircDD}\ignorespaces\fi}
\ctr@ld@f\def\figptsintercircTD#1[#2,#3;#4,#5;#6]{\ifGR@cri{\s@uvc@ntr@l\et@tfigptsintercircTD%
    \setc@ntr@l{2}\let\c@lNVintc=\c@lNVintcTD\vecunitC@TD[#2,#6]%
    \Figv@ctCreg-3(\v@lX,\v@lY,\v@lZ)\Figptsintercirc@#1[#2,#3;#4,#5]%
    \resetc@ntr@l\et@tfigptsintercircTD}\ignorespaces\fi}
\ctr@ld@f\def\Figptsintercirc@#1[#2,#3;#4,#5]{\figvectP-1[#2,#4]%
    \vecunit@{-1}{-1}\delt@=\result@t\f@ctech=\result@tent%
    \s@mme=#1\advance\s@mme\@ne\figptcopy#1:/#2/\figptcopy\the\s@mme:/#4/%
    \ifdim\delt@=\z@\else%
    \v@lmin=#3\unit@\v@lmax=#5\unit@\v@leur=\v@lmin\advance\v@leur\v@lmax%
    \ifdim\v@leur>\delt@%
    \v@leur=\v@lmin\advance\v@leur-\v@lmax\maxim@m{\v@leur}{\v@leur}{-\v@leur}%
    \ifdim\v@leur<\delt@%
    \divide\v@lmin\f@ctech\divide\v@lmax\f@ctech\divide\delt@\f@ctech%
    \v@lmin=\repdecn@mb{\v@lmin}\v@lmin\v@lmax=\repdecn@mb{\v@lmax}\v@lmax%
    \invers@{\v@leur}{\delt@}\advance\v@lmax-\v@lmin%
    \v@lmax=-\repdecn@mb{\v@leur}\v@lmax\advance\delt@\v@lmax\delt@=.5\delt@%
    \v@lmax=\delt@\multiply\v@lmax\f@ctech%
    \edef\t@ille{\repdecn@mb{\v@lmax}}\figpttra-2:=#2/\t@ille,-1/%
    \delt@=\repdecn@mb{\delt@}\delt@\advance\v@lmin-\delt@%
    \sqrt@{\v@leur}{\v@lmin}\multiply\v@leur\f@ctech\edef\t@ille{\repdecn@mb{\v@leur}}%
    \c@lNVintc\figpttra#1:=-2/-\t@ille,-1/\figpttra\the\s@mme:=-2/\t@ille,-1/\fi\fi\fi}
\ctr@ld@f\def\c@lNVintcDD{\Figg@tXY{-1}\Figv@ctCreg-1(-\v@lY,\v@lX)} 
\ctr@ld@f\def\c@lNVintcTD{{\Figg@tXY{-3}\v@lmin=\v@lX\v@lmax=\v@lY\v@leur=\v@lZ%
    \Figg@tXY{-1}\c@lprovec{-3}\vecunit@{-3}{-3}
    \Figg@tXY{-1}\v@lmin=\v@lX\v@lmax=\v@lY%
    \v@leur=\v@lZ\Figg@tXY{-3}\c@lprovec{-1}}} 
\ctr@ln@m\figptsinterlinell
\ctr@ld@f\def\figptsinterlinellDD#1[#2,#3,#4,#5;#6,#7]{\ifGR@cri{\s@uvc@ntr@l\et@tfigptsinterlinellDD%
    \figptcopy#1:/#6/\s@mme=#1\advance\s@mme\@ne\figptcopy\the\s@mme:/#7/%
    \v@lmin=#3\unit@\v@lmax=#4\unit@
    \setc@ntr@l{2}\figptbaryDD-4:[#6,#7;1,1]\figptsrotDD-3=-4,#7/#2,-#5/
    \Figg@tXY{-3}\Figg@tXYa{#2}\advance\v@lX-\v@lXa\advance\v@lY-\v@lYa
    \figvectP-1[-3,-2]\Figg@tXYa{-1}\figvectP-3[-4,#7]\Figptsint@rLE{#1}
    \resetc@ntr@l\et@tfigptsinterlinellDD}\ignorespaces\fi}
\ctr@ld@f\def\figptsinterlinellP#1[#2,#3,#4;#5,#6]{\ifGR@cri{\s@uvc@ntr@l\et@tfigptsinterlinellP%
    \figptcopy#1:/#5/\s@mme=#1\advance\s@mme\@ne\figptcopy\the\s@mme:/#6/\setc@ntr@l{2}%
    \figvectP-1[#2,#3]\vecunit@{-1}{-1}\v@lmin=\result@t
    \figvectP-2[#2,#4]\vecunit@{-2}{-2}\v@lmax=\result@t
    \figptbary-4:[#5,#6;1,1]
    \figvectP-3[#2,-4]\c@lproscal\v@lX[-3,-1]\c@lproscal\v@lY[-3,-2]
    \figvectP-3[-4,#6]\c@lproscal\v@lXa[-3,-1]\c@lproscal\v@lYa[-3,-2]
    \Figptsint@rLE{#1}\resetc@ntr@l\et@tfigptsinterlinellP}\ignorespaces\fi}
\ctr@ld@f\def\Figptsint@rLE#1{%
    \getredf@ctDD\f@ctech(\v@lmin,\v@lmax)%
    \getredf@ctDD\p@rtent(\v@lX,\v@lY)\ifnum\p@rtent>\f@ctech\f@ctech=\p@rtent\fi%
    \getredf@ctDD\p@rtent(\v@lXa,\v@lYa)\ifnum\p@rtent>\f@ctech\f@ctech=\p@rtent\fi%
    \divide\v@lmin\f@ctech\divide\v@lmax\f@ctech\divide\v@lX\f@ctech\divide\v@lY\f@ctech%
    \divide\v@lXa\f@ctech\divide\v@lYa\f@ctech%
    \c@rre=\repdecn@mb\v@lXa\v@lmax\mili@u=\repdecn@mb\v@lYa\v@lmin%
    \getredf@ctDD\f@ctech(\c@rre,\mili@u)%
    \c@rre=\repdecn@mb\v@lX\v@lmax\mili@u=\repdecn@mb\v@lY\v@lmin%
    \getredf@ctDD\p@rtent(\c@rre,\mili@u)\ifnum\p@rtent>\f@ctech\f@ctech=\p@rtent\fi%
    \divide\v@lmin\f@ctech\divide\v@lmax\f@ctech\divide\v@lX\f@ctech\divide\v@lY\f@ctech%
    \divide\v@lXa\f@ctech\divide\v@lYa\f@ctech%
    \v@lmin=\repdecn@mb{\v@lmin}\v@lmin\v@lmax=\repdecn@mb{\v@lmax}\v@lmax%
    \edef\G@xde{\repdecn@mb\v@lmin}\edef\P@xde{\repdecn@mb\v@lmax}%
    \c@rre=-\v@lmax\v@leur=\repdecn@mb\v@lY\v@lY\advance\c@rre\v@leur\c@rre=\G@xde\c@rre%
    \v@leur=\repdecn@mb\v@lX\v@lX\v@leur=\P@xde\v@leur\advance\c@rre\v@leur
    \v@lmin=\repdecn@mb\v@lYa\v@lmin\v@lmax=\repdecn@mb\v@lXa\v@lmax%
    \mili@u=\repdecn@mb\v@lX\v@lmax\advance\mili@u\repdecn@mb\v@lY\v@lmin
    \v@lmax=\repdecn@mb\v@lXa\v@lmax\advance\v@lmax\repdecn@mb\v@lYa\v@lmin
    \ifdim\v@lmax>\epsil@n%
    \maxim@m{\v@leur}{\c@rre}{-\c@rre}\maxim@m{\v@lmin}{\mili@u}{-\mili@u}%
    \maxim@m{\v@leur}{\v@leur}{\v@lmin}\maxim@m{\v@lmin}{\v@lmax}{-\v@lmax}%
    \maxim@m{\v@leur}{\v@leur}{\v@lmin}\p@rtentiere{\p@rtent}{\v@leur}\advance\p@rtent\@ne%
    \divide\c@rre\p@rtent\divide\mili@u\p@rtent\divide\v@lmax\p@rtent%
    \delt@=\repdecn@mb{\mili@u}\mili@u\v@leur=\repdecn@mb{\v@lmax}\c@rre%
    \advance\delt@-\v@leur\ifdim\delt@<\z@\else\sqrt@\delt@\delt@%
    \invers@\v@lmax\v@lmax\edef\Uns@rAp{\repdecn@mb\v@lmax}%
    \v@leur=-\mili@u\advance\v@leur-\delt@\v@leur=\Uns@rAp\v@leur%
    \edef\t@ille{\repdecn@mb\v@leur}\figpttra#1:=-4/\t@ille,-3/\s@mme=#1\advance\s@mme\@ne%
    \v@leur=-\mili@u\advance\v@leur\delt@\v@leur=\Uns@rAp\v@leur%
    \edef\t@ille{\repdecn@mb\v@leur}\figpttra\the\s@mme:=-4/\t@ille,-3/\fi\fi}
\ctr@ln@m\figptsorthoprojline
\ctr@ld@f\def\figptsorthoprojlineDD#1=#2/#3,#4/{\ifGR@cri{\s@uvc@ntr@l\et@tfigptsorthoprojlineDD%
    \setc@ntr@l{2}\figvectPDD-3[#3,#4]\figvectNVDD-4[-3]\resetc@ntr@l{2}%
    \def\list@num{#2}\s@mme=#1\@ecfor\p@int:=\list@num\do{%
    \inters@cDD\the\s@mme:[\p@int,-4;#3,-3]\advance\s@mme\@ne}%
    \resetc@ntr@l\et@tfigptsorthoprojlineDD}\ignorespaces\fi}
\ctr@ld@f\def\figptsorthoprojlineTD#1=#2/#3,#4/{\ifGR@cri{\s@uvc@ntr@l\et@tfigptsorthoprojlineTD%
    \setc@ntr@l{2}\figvectPTD-2[#3,#4]\vecunit@TD{-2}{-2}%
    \def\list@num{#2}\s@mme=#1\@ecfor\p@int:=\list@num\do{%
    \figvectPTD-1[#3,\p@int]\c@lproscalTD\v@leur[-1,-2]%
    \edef\v@lcoef{\repdecn@mb{\v@leur}}\figpttraTD\the\s@mme:=#3/\v@lcoef,-2/%
    \advance\s@mme\@ne}\resetc@ntr@l\et@tfigptsorthoprojlineTD}\ignorespaces\fi}
\ctr@ln@m\figptsorthoprojplane
\ctr@ld@f\def\figptsorthoprojplaneDD{\un@v@ilable{figptsorthoprojplane}}
\ctr@ld@f\def\figptsorthoprojplaneTD#1=#2/#3,#4/{\ifGR@cri{\s@uvc@ntr@l\et@tfigptsorthoprojplane%
    \setc@ntr@l{2}\vecunit@TD{-2}{#4}%
    \def\list@num{#2}\s@mme=#1\@ecfor\p@int:=\list@num\do{\figvectPTD-1[\p@int,#3]%
    \c@lproscalTD\v@leur[-1,-2]\edef\v@lcoef{\repdecn@mb{\v@leur}}%
    \figpttraTD\the\s@mme:=\p@int/\v@lcoef,-2/\advance\s@mme\@ne}%
    \resetc@ntr@l\et@tfigptsorthoprojplane}\ignorespaces\fi}
\ctr@ld@f\def\figptshom#1=#2/#3,#4/{\ifGR@cri{\s@uvc@ntr@l\et@tfigptshom%
    \setc@ntr@l{2}\def\list@num{#2}\s@mme=#1%
    \@ecfor\p@int:=\list@num\do{\figvectP-1[#3,\p@int]%
    \figpttra\the\s@mme:=#3/#4,-1/\advance\s@mme\@ne}%
    \resetc@ntr@l\et@tfigptshom}\ignorespaces\fi}
\ctr@ld@f\def\figptsinv#1=#2/#3,#4/{\ifGR@cri{\s@uvc@ntr@l\et@tfigptsinv%
    \setc@ntr@l{2}\def\list@num{#2}\s@mme=#1%
    \@ecfor\p@int:=\list@num\do{\figvectP-1[#3,\p@int]\Figg@tXY{-1}%
    \getredf@ctB\f@ctech\n@rmeucC{\delt@}{-1}%
    \delt@=\ptT@unit@\delt@\delt@=\ptT@unit@\delt@%
    \invers@{\delt@}{\delt@}\multiply\f@ctech\f@ctech\divide\delt@\f@ctech%
    \delt@=#4\delt@\edef\v@lcoef{\repdecn@mb{\delt@}}\figpttra\the\s@mme:=#3/\v@lcoef,-1/%
    \advance\s@mme\@ne}\resetc@ntr@l\et@tfigptsinv}\ignorespaces\fi}
\ctr@ln@m\figptsrot
\ctr@ld@f\def\figptsrotDD#1=#2/#3,#4/{\ifGR@cri{\s@uvc@ntr@l\et@tfigptsrotDD%
    \c@ssin{\C@}{\S@}{#4}\setc@ntr@l{2}\def\list@num{#2}\s@mme=#1%
    \@ecfor\p@int:=\list@num\do{\figvectPDD-1[#3,\p@int]\Figg@tXY{-1}%
    \v@lXa=\C@\v@lX\advance\v@lXa-\S@\v@lY%
    \v@lYa=\S@\v@lX\advance\v@lYa\C@\v@lY%
    \Figv@ctCreg-1(\v@lXa,\v@lYa)\figpttraDD\the\s@mme:=#3/1,-1/\advance\s@mme\@ne}%
    \resetc@ntr@l\et@tfigptsrotDD}\ignorespaces\fi}
\ctr@ld@f\def\figptsrotTD#1=#2/#3,#4,#5/{\ifGR@cri{\s@uvc@ntr@l\et@tfigptsrotTD%
    \c@ssin{\C@}{\S@}{#4}%
    \setc@ntr@l{2}\def\list@num{#2}\s@mme=#1%
    \@ecfor\p@int:=\list@num\do{\figptorthoprojplaneTD-3:=#3/\p@int,#5/%
    \figvectPTD-2[-3,\p@int]%
    \figvectNVTD-1[#5,-2]\n@rmeucTD\v@leur{-2}\edef\v@lcoef{\repdecn@mb{\v@leur}}%
    \Figg@tXYa{-1}\v@lXa=\v@lcoef\v@lXa\v@lYa=\v@lcoef\v@lYa\v@lZa=\v@lcoef\v@lZa%
    \v@lXa=\S@\v@lXa\v@lYa=\S@\v@lYa\v@lZa=\S@\v@lZa\Figg@tXY{-2}%
    \advance\v@lXa\C@\v@lX\advance\v@lYa\C@\v@lY\advance\v@lZa\C@\v@lZ%
    \Figg@tXY{-3}\advance\v@lXa\v@lX\advance\v@lYa\v@lY\advance\v@lZa\v@lZ%
    \Figp@intregTD\the\s@mme:(\v@lXa,\v@lYa,\v@lZa)\advance\s@mme\@ne}%
    \resetc@ntr@l\et@tfigptsrotTD}\ignorespaces\fi}
\ctr@ln@m\figptssym
\ctr@ld@f\def\figptssymDD#1=#2/#3,#4/{\ifGR@cri{\s@uvc@ntr@l\et@tfigptssymDD%
    \setc@ntr@l{2}\figvectPDD-3[#3,#4]\Figg@tXY{-3}\Figv@ctCreg-4(-\v@lY,\v@lX)%
    \resetc@ntr@l{2}\def\list@num{#2}\s@mme=#1%
    \@ecfor\p@int:=\list@num\do{\inters@cDD-5:[#3,-3;\p@int,-4]\figvectPDD-2[\p@int,-5]%
    \figpttraDD\the\s@mme:=\p@int/2,-2/\advance\s@mme\@ne}%
    \resetc@ntr@l\et@tfigptssymDD}\ignorespaces\fi}
\ctr@ld@f\def\figptssymTD#1=#2/#3,#4/{\ifGR@cri{\s@uvc@ntr@l\et@tfigptssymTD%
    \setc@ntr@l{2}\vecunit@TD{-2}{#4}\def\list@num{#2}\s@mme=#1%
    \@ecfor\p@int:=\list@num\do{\figvectPTD-1[\p@int,#3]%
    \c@lproscalTD\v@leur[-1,-2]\v@leur=2\v@leur\edef\v@lcoef{\repdecn@mb{\v@leur}}%
    \figpttraTD\the\s@mme:=\p@int/\v@lcoef,-2/\advance\s@mme\@ne}%
    \resetc@ntr@l\et@tfigptssymTD}\ignorespaces\fi}
\ctr@ln@m\figptstra
\ctr@ld@f\def\figptstraDD#1=#2/#3,#4/{\ifGR@cri{\Figg@tXYa{#4}\v@lXa=#3\v@lXa\v@lYa=#3\v@lYa%
    \def\list@num{#2}\s@mme=#1\@ecfor\p@int:=\list@num\do{\Figg@tXY{\p@int}%
    \advance\v@lX\v@lXa\advance\v@lY\v@lYa%
    \Figp@intregDD\the\s@mme:(\v@lX,\v@lY)\advance\s@mme\@ne}}\ignorespaces\fi}
\ctr@ld@f\def\figptstraTD#1=#2/#3,#4/{\ifGR@cri{\Figg@tXYa{#4}\v@lXa=#3\v@lXa\v@lYa=#3\v@lYa%
    \v@lZa=#3\v@lZa\def\list@num{#2}\s@mme=#1\@ecfor\p@int:=\list@num\do{\Figg@tXY{\p@int}%
    \advance\v@lX\v@lXa\advance\v@lY\v@lYa\advance\v@lZ\v@lZa%
    \Figp@intregTD\the\s@mme:(\v@lX,\v@lY,\v@lZ)\advance\s@mme\@ne}}\ignorespaces\fi}
\ctr@ln@m\figptvisilimSL
\ctr@ld@f\def\figptvisilimSLDD{\un@v@ilable{figptvisilimSL}}
\ctr@ld@f\def\figptvisilimSLTD#1:#2[#3,#4;#5,#6]{\ifGR@cri{\s@uvc@ntr@l\et@tfigptvisilimSLTD%
    \setc@ntr@l{2}\figvectP-1[#3,#4]\n@rminf{\delt@}{-1}%
    \ifcase\CUR@proj\v@lX=\cxa@\p@\v@lY=-\p@\v@lZ=\cxb@\p@
    \Figv@ctCreg-2(\v@lX,\v@lY,\v@lZ)\figvectP-3[#5,#6]\figvectNV-1[-2,-3]%
    \or\figvectP-1[#5,#6]\vecunitCV@TD{-1}\v@lmin=\v@lX\v@lmax=\v@lY
    \v@leur=\v@lZ\v@lX=\cza@\p@\v@lY=\czb@\p@\v@lZ=\czc@\p@\c@lprovec{-1}%
    \or\c@ley@pt{-2}\figvectN-1[#5,#6,-2]\fi
    \edef\Ai@{#3}\edef\Aj@{#4}\figvectP-2[#5,\Ai@]\c@lproscal\v@leur[-1,-2]%
    \ifdim\v@leur>\z@\p@rtent=\@ne\else\p@rtent=\m@ne\fi%
    \figvectP-2[#5,\Aj@]\c@lproscal\v@leur[-1,-2]%
    \ifdim\p@rtent\v@leur>\z@\figptcopy#1:#2/#3/%
    \message{*** \BS@ figptvisilimSL: points are on the same side.}\else%
    \figptcopy-3:/#3/\figptcopy-4:/#4/%
    \loop\figptbary-5:[-3,-4;1,1]\figvectP-2[#5,-5]\c@lproscal\v@leur[-1,-2]%
    \ifdim\p@rtent\v@leur>\z@\figptcopy-3:/-5/\else\figptcopy-4:/-5/\fi%
    \divide\delt@\tw@\ifdim\delt@>\epsil@n\repeat%
    \figptbary#1:#2[-3,-4;1,1]\fi\resetc@ntr@l\et@tfigptvisilimSLTD}\ignorespaces\fi}
\ctr@ld@f\def\c@ley@pt#1{\t@stp@r\ifitis@K\v@lX=\cza@\p@\v@lY=\czb@\p@\v@lZ=\czc@\p@%
    \Figv@ctCreg-1(\v@lX,\v@lY,\v@lZ)\Figp@intreg-2:(\wd\Bt@rget,\ht\Bt@rget,\dp\Bt@rget)%
    \figpttra#1:=-2/-\disob@intern,-1/\else\end\fi}
\ctr@ld@f\def\t@stp@r{\itis@Ktrue\ifnewt@rgetpt\else\itis@Kfalse%
    \message{*** \BS@ figptvisilimXX: target point undefined.}\fi\ifnewdis@b\else%
    \itis@Kfalse\message{*** \BS@ figptvisilimXX: observation distance undefined.}\fi%
    \ifitis@K\else\message{*** This macro must be called after \BS@ figdrawbegin or after
    having set the missing parameter(s) with \BS@ figset proj()}\fi}
\ctr@ld@f\def\figscan#1(#2,#3){{\s@uvc@ntr@l\et@tfigscan\@psfgetbb{#1}\if@psfbbfound\else%
    \def\@psfllx{0}\def\@psflly{20}\def\@psfurx{540}\def\@psfury{640}\fi\figscan@{#2}{#3}%
    \resetc@ntr@l\et@tfigscan}\ignorespaces}
\ctr@ld@f\def\figscan@#1#2{%
    \unit@=\@ne bp\setc@ntr@l{2}\figsetmark{}%
    \def\minst@p{20pt}%
    \v@lX=\@psfllx\p@\v@lX=\Sc@leFact\v@lX\r@undint\v@lX\v@lX%
    \v@lY=\@psflly\p@\v@lY=\Sc@leFact\v@lY\ifdim\v@lY>\z@\r@undint\v@lY\v@lY\fi%
    \delt@=\@psfury\p@\delt@=\Sc@leFact\delt@%
    \advance\delt@-\v@lY\v@lXa=\@psfurx\p@\v@lXa=\Sc@leFact\v@lXa\v@leur=\minst@p%
    \edef\valv@lY{\repdecn@mb{\v@lY}}\edef\LgTr@it{\the\delt@}%
    \loop\ifdim\v@lX<\v@lXa\edef\valv@lX{\repdecn@mb{\v@lX}}%
    \figptDD -1:(\valv@lX,\valv@lY)\figwriten -1:\hbox{\vrule height\LgTr@it}(0)%
    \ifdim\v@leur<\minst@p\else\figsetmark{\raise-8bp\hbox{$\scriptscriptstyle\triangle$}}%
    \figwrites -1:\@ffichnb{0}{\valv@lX}(6)\v@leur=\z@\figsetmark{}\fi%
    \advance\v@leur#1pt\advance\v@lX#1pt\repeat%
    \def\minst@p{10pt}%
    \v@lX=\@psfllx\p@\v@lX=\Sc@leFact\v@lX\ifdim\v@lX>\z@\r@undint\v@lX\v@lX\fi%
    \v@lY=\@psflly\p@\v@lY=\Sc@leFact\v@lY\r@undint\v@lY\v@lY%
    \delt@=\@psfurx\p@\delt@=\Sc@leFact\delt@%
    \advance\delt@-\v@lX\v@lYa=\@psfury\p@\v@lYa=\Sc@leFact\v@lYa\v@leur=\minst@p%
    \edef\valv@lX{\repdecn@mb{\v@lX}}\edef\LgTr@it{\the\delt@}%
    \loop\ifdim\v@lY<\v@lYa\edef\valv@lY{\repdecn@mb{\v@lY}}%
    \figptDD -1:(\valv@lX,\valv@lY)\figwritee -1:\vbox{\hrule width\LgTr@it}(0)%
    \ifdim\v@leur<\minst@p\else\figsetmark{$\triangleright$\kern4bp}%
    \figwritew -1:\@ffichnb{0}{\valv@lY}(6)\v@leur=\z@\figsetmark{}\fi%
    \advance\v@leur#2pt\advance\v@lY#2pt\repeat}
\ctr@ld@f\let\figscanI=\figscan
\ctr@ld@f\def\figscan@E#1(#2,#3){{\s@uvc@ntr@l\et@tfigscan@E%
    \Figdisc@rdLTS{#1}{\t@xt@}\pdfximage{\t@xt@}%
    \setbox\Gb@x=\hbox{\pdfrefximage\pdflastximage}%
    \edef\@psfllx{0}\v@lY=-\dp\Gb@x\edef\@psflly{\repdecn@mb{\v@lY}}%
    \edef\@psfurx{\repdecn@mb{\wd\Gb@x}}%
    \v@lY=\dp\Gb@x\advance\v@lY\ht\Gb@x\edef\@psfury{\repdecn@mb{\v@lY}}%
    \figscan@{#2}{#3}\resetc@ntr@l\et@tfigscan@E}\ignorespaces}
\ctr@ld@f\def\figshowpts[#1,#2]{{\figsetmark{$\bullet$}\figsetptname{\bf ##1}%
    \p@rtent=#2\relax\ifnum\p@rtent<\z@\p@rtent=\z@\fi%
    \s@mme=#1\relax\ifnum\s@mme<\z@\s@mme=\z@\fi%
    \loop\ifnum\s@mme<\p@rtent\pt@rvect{\s@mme}%
    \ifitis@K\figwriten{\the\s@mme}:(4pt)\fi\advance\s@mme\@ne\repeat%
    \pt@rvect{\s@mme}\ifitis@K\figwriten{\the\s@mme}:(4pt)\fi}\ignorespaces}
\ctr@ld@f\def\pt@rvect#1{\set@bjc@de{#1}%
    \expandafter\expandafter\expandafter\inqpt@rvec\csname\objc@de\endcsname:}
\ctr@ld@f\def\inqpt@rvec#1#2:{\if#1\C@dCl@spt\itis@Ktrue\else\itis@Kfalse\fi}
\ctr@ld@f\def\figshowsettings{{%
    \immediate\write16{====================================================================}%
    \immediate\write16{ Current settings are (DDV means "with dynamic default value"):}%
    \immediate\write16{ --- GENERAL ---}%
    \immediate\write16{Scale factor and Unit = \unit@util\space (\the\unit@)
     \space -> \BS@ figinit{ScaleFactorUnit}}%
    \immediate\write16{Update mode = \ifGRupdatem@de yes\else no\fi
     \space-> \BS@ figset(update=yes/no) or \BS@ figsetdefault(update=yes/no)}%
    \immediate\write16{ --- WRITING ---}%
    \immediate\write16{Implicit point name = \ptn@me{i} \space-> \BS@ figset write(ptname={Name})}%
    \immediate\write16{Point marker = \the\c@nsymb \space -> \BS@ figset write(mark=Mark)}%
    \immediate\write16{Print rounded coordinates = \ifr@undcoord yes\else no\fi
     \space-> \BS@ figset write(roundcoord=yes/no)}%
    \immediate\write16{ --- GRAPHICAL (general) ---}%
    \immediate\write16{Color = \CUR@color \space-> \BS@ figset(color=ColorDefinition)}%
    \immediate\write16{Filling mode = \iffillm@de yes\else no\fi
     \space-> \BS@ figset(fillmode=yes/no)}%
    \immediate\write16{Line join = \CUR@join \space-> \BS@ figset(join=miter/round/bevel)}%
    \immediate\write16{Line style = \CUR@dash \space-> \BS@ figset(dash=Index/Pattern)}%
    \immediate\write16{Line width = \CUR@width
     \space-> \BS@ figset(width=real in PostScript units)}%
    \immediate\write16{ --- GRAPHICAL (specific) ---}%
    \immediate\write16{Altitude (all the following attributes are DDV):}%
    \immediate\write16{ Base line color =
     \ifx\DDV@blcolor\D@FTref general color\else\DDV@blcolor\fi
     \space-> \BS@ figset altitude(blcolor=ColorDefinition)}%
    \immediate\write16{ Base line style =
     \ifx\DDV@bldash\D@FTref general style\else\DDV@bldash\fi
     \space-> \BS@ figset altitude(bldash=Index/Pattern)}%
    \immediate\write16{ Base line width =
     \ifx\DDV@blwidth\D@FTref general width\else\DDV@blwidth\fi
     \space-> \BS@ figset altitude(blwidth=real in PostScript units)}%
    \immediate\write16{ Square line color =
     \ifx\DDV@sqcolor\D@FTref general color\else\DDV@sqcolor\fi
     \space-> \BS@ figset altitude(sqcolor=ColorDefinition)}%
    \immediate\write16{ Square line style =
     \ifx\DDV@sqdash\D@FTref general style\else\DDV@sqdash\fi
     \space-> \BS@ figset altitude(sqdash=Index/Pattern)}%
    \immediate\write16{ Square line width =
     \ifx\DDV@sqwidth\D@FTref general width\else\DDV@sqwidth\fi
     \space-> \BS@ figset altitude(sqwidth=real in PostScript units)}%
    \immediate\write16{Arrowhead:}%
    \immediate\write16{ (half-)Angle = \@rrowheadangle
     \space-> \BS@ figset arrowhead(angle=real in degrees)}%
    \immediate\write16{ Filling mode = \if@rrowhfill yes\else no\fi
     \space-> \BS@ figset arrowhead(fillmode=yes/no)}%
    \immediate\write16{ "Outside" = \if@rrowhout yes\else no\fi
     \space-> \BS@ figset arrowhead(out=yes/no)}%
    \immediate\write16{ Length = \@rrowheadlength
     \if@rrowratio\space(not active)\else\space(active)\fi
     \space-> \BS@ figset arrowhead(length=real in user coord.)}%
    \immediate\write16{ Ratio = \@rrowheadratio
     \if@rrowratio\space(active)\else\space(not active)\fi
     \space-> \BS@ figset arrowhead(ratio=real in [0,1])}%
    \immediate\write16{Curve:}%
    \immediate\write16{ Roundness = \curv@roundness
     \space-> \BS@ figset curve(roundness=real in [0,0.5])}%
    \immediate\write16{Flow chart:}%
    \immediate\write16{ Arrow position = \@rrowp@s
     \space-> \BS@ figset flowchart(arrowposition=real in [0,1])}%
    \immediate\write16{ Arrow reference point = \ifcase\@rrowr@fpt start\else end\fi
     \space-> \BS@ figset flowchart(arrowrefpt = start/end)}%
    \immediate\write16{ Background color = \fcbgc@lor
     \space-> \BS@ figset flowchart(bgcolor=ColorDefinition)}%
    \immediate\write16{ Line type = \ifcase\fclin@typ@ curve\else polygon\fi
     \space-> \BS@ figset flowchart(line=polygon/curve)}%
    \immediate\write16{ Padding = (\Xp@dd, \Yp@dd)
     \space-> \BS@ figset flowchart(padding = real in user coord.)}%
    \immediate\write16{\space\space\space\space(or
     \BS@ figset flowchart(xpadding=real, ypadding=real) )}%
    \immediate\write16{ Radius = \fclin@r@d
     \space-> \BS@ figset flowchart(radius=positive real in user coord.)}%
    \immediate\write16{ Shape = \fcsh@pe
     \space-> \BS@ figset flowchart(shape = rectangle, ellipse or lozenge)}%
    \immediate\write16{ Thickness color (DDV) = 
     \ifx\DDV@thickcolor\D@FTref general color\else\DDV@thickcolor\fi
     \space-> \BS@ figset flowchart(thickcolor=ColorDefinition)}%
    \immediate\write16{ Thickness = \thickn@ss
     \space-> \BS@ figset flowchart(thickness = real in user coord.)}%
    \immediate\write16{Mesh:}%
    \immediate\write16{ Diagonal = \c@ntrolmesh
     \space-> \BS@ figset mesh(diag=integer in {-1,0,1})}%
    \immediate\write16{ Lines color (DDV) =
     \ifx\DDV@meshcolor\D@FTref general color\else\DDV@meshcolor\fi
     \space-> \BS@ figset mesh(color=ColorDefinition)}%
    \immediate\write16{ Lines style (DDV) =
     \ifx\DDV@meshdash\D@FTref general style\else\DDV@meshdash\fi
     \space-> \BS@ figset mesh(dash=Index/Pattern)}%
    \immediate\write16{ Lines width (DDV) =
     \ifx\DDV@meshwidth\D@FTref general width\else\DDV@meshwidth\fi
     \space-> \BS@ figset mesh(width=real in PostScript units)}%
    \immediate\write16{Trimesh:}%
    \immediate\write16{ Lines color (DDV) =
     \ifx\DDV@tmeshcolor\D@FTref general color\else\DDV@tmeshcolor\fi
     \space-> \BS@ figset trimesh(color=ColorDefinition)}%
    \immediate\write16{ Lines style (DDV) =
     \ifx\DDV@tmeshdash\D@FTref general style\else\DDV@tmeshdash\fi
     \space-> \BS@ figset trimesh(dash=Index/Pattern)}%
    \immediate\write16{ Lines width (DDV) =
     \ifx\DDV@tmeshwidth\D@FTref general width\else\DDV@tmeshwidth\fi
     \space-> \BS@ figset trimesh(width=real in PostScript units)}%
    \ifTr@isDim%
    \immediate\write16{ --- 3D to 2D PROJECTION ---}%
    \immediate\write16{Projection : \typ@proj \space-> \BS@ figinit{ScaleFactorUnit, ProjType}}%
    \immediate\write16{Longitude (psi) = \v@lPsi \space-> \BS@ figset proj(psi=real in degrees)}%
    \ifcase\CUR@proj\immediate\write16{Depth coeff. (Lambda)
     \space = \v@lTheta \space-> \BS@ figset proj(lambda=real in [0,1])}%
    \else\immediate\write16{Latitude (theta)
     \space = \v@lTheta \space-> \BS@ figset proj(theta=real in degrees)}%
    \fi%
    \ifnum\CUR@proj=\tw@%
    \immediate\write16{Observation distance = \disob@unit
     \space-> \BS@ figset proj(dist=real in user coord.)}%
    \immediate\write16{Target point = \t@rgetpt \space-> \BS@ figset proj(targetpt=pt number)}%
     \v@lX=\ptT@unit@\wd\Bt@rget\v@lY=\ptT@unit@\ht\Bt@rget\v@lZ=\ptT@unit@\dp\Bt@rget%
    \immediate\write16{ Its coordinates are
     (\repdecn@mb{\v@lX}, \repdecn@mb{\v@lY}, \repdecn@mb{\v@lZ})}%
    \fi%
    \fi%
    \immediate\write16{====================================================================}%
    \ignorespaces}}
\ctr@ln@w{newif}\ifitis@vect@r
\ctr@ld@f\def\figvectC#1(#2,#3){{\itis@vect@rtrue\figpt#1:(#2,#3)}\ignorespaces}
\ctr@ld@f\def\Figv@ctCreg#1(#2,#3){{\itis@vect@rtrue\Figp@intreg#1:(#2,#3)}\ignorespaces}
\ctr@ln@m\figvectDBezier
\ctr@ld@f\def\figvectDBezierDD#1:#2,#3[#4,#5,#6,#7]{\ifGR@cri{\s@uvc@ntr@l\et@tfigvectDBezierDD%
    \FigvectDBezier@#2,#3[#4,#5,#6,#7]\v@lX=\c@ef\v@lX\v@lY=\c@ef\v@lY%
    \Figv@ctCreg#1(\v@lX,\v@lY)\resetc@ntr@l\et@tfigvectDBezierDD}\ignorespaces\fi}
\ctr@ld@f\def\figvectDBezierTD#1:#2,#3[#4,#5,#6,#7]{\ifGR@cri{\s@uvc@ntr@l\et@tfigvectDBezierTD%
    \FigvectDBezier@#2,#3[#4,#5,#6,#7]\v@lX=\c@ef\v@lX\v@lY=\c@ef\v@lY\v@lZ=\c@ef\v@lZ%
    \Figv@ctCreg#1(\v@lX,\v@lY,\v@lZ)\resetc@ntr@l\et@tfigvectDBezierTD}\ignorespaces\fi}
\ctr@ld@f\def\FigvectDBezier@#1,#2[#3,#4,#5,#6]{\setc@ntr@l{2}%
    \edef\T@{#2}\v@leur=\p@\advance\v@leur-#2pt\edef\UNmT@{\repdecn@mb{\v@leur}}%
    \ifnum#1=\tw@\def\c@ef{6}\else\def\c@ef{3}\fi%
    \figptcopy-4:/#3/\figptcopy-3:/#4/\figptcopy-2:/#5/\figptcopy-1:/#6/%
    \l@mbd@un=-4 \l@mbd@de=-\thr@@\p@rtent=\m@ne\c@lDecast%
    \ifnum#1=\tw@\c@lDCDeux{-4}{-3}\c@lDCDeux{-3}{-2}\c@lDCDeux{-4}{-3}\else%
    \l@mbd@un=-4 \l@mbd@de=-\thr@@\p@rtent=-\tw@\c@lDecast%
    \c@lDCDeux{-4}{-3}\fi\Figg@tXY{-4}}
\ctr@ln@m\c@lDCDeux
\ctr@ld@f\def\c@lDCDeuxDD#1#2{\Figg@tXY{#2}\Figg@tXYa{#1}%
    \advance\v@lX-\v@lXa\advance\v@lY-\v@lYa\Figp@intregDD#1:(\v@lX,\v@lY)}
\ctr@ld@f\def\c@lDCDeuxTD#1#2{\Figg@tXY{#2}\Figg@tXYa{#1}\advance\v@lX-\v@lXa%
    \advance\v@lY-\v@lYa\advance\v@lZ-\v@lZa\Figp@intregTD#1:(\v@lX,\v@lY,\v@lZ)}
\ctr@ln@m\figvectN
\ctr@ld@f\def\figvectNDD#1[#2,#3]{\ifGR@cri{\Figg@tXYa{#2}\Figg@tXY{#3}%
    \advance\v@lX-\v@lXa\advance\v@lY-\v@lYa%
    \Figv@ctCreg#1(-\v@lY,\v@lX)}\ignorespaces\fi}
\ctr@ld@f\def\figvectNTD#1[#2,#3,#4]{\ifGR@cri{\vecunitC@TD[#2,#4]\v@lmin=\v@lX\v@lmax=\v@lY%
    \v@leur=\v@lZ\vecunitC@TD[#2,#3]\c@lprovec{#1}}\ignorespaces\fi}
\ctr@ln@m\figvectNV
\ctr@ld@f\def\figvectNVDD#1[#2]{\ifGR@cri{\Figg@tXY{#2}\Figv@ctCreg#1(-\v@lY,\v@lX)}\ignorespaces\fi}
\ctr@ld@f\def\figvectNVTD#1[#2,#3]{\ifGR@cri{\vecunitCV@TD{#3}\v@lmin=\v@lX\v@lmax=\v@lY%
    \v@leur=\v@lZ\vecunitCV@TD{#2}\c@lprovec{#1}}\ignorespaces\fi}
\ctr@ln@m\figvectP
\ctr@ld@f\def\figvectPDD#1[#2,#3]{\ifGR@cri{\Figg@tXYa{#2}\Figg@tXY{#3}%
    \advance\v@lX-\v@lXa\advance\v@lY-\v@lYa%
    \Figv@ctCreg#1(\v@lX,\v@lY)}\ignorespaces\fi}
\ctr@ld@f\def\figvectPTD#1[#2,#3]{\ifGR@cri{\Figg@tXYa{#2}\Figg@tXY{#3}%
    \advance\v@lX-\v@lXa\advance\v@lY-\v@lYa\advance\v@lZ-\v@lZa%
    \Figv@ctCreg#1(\v@lX,\v@lY,\v@lZ)}\ignorespaces\fi}
\ctr@ln@m\figvectU
\ctr@ld@f\def\figvectUDD#1[#2]{\ifGR@cri{\n@rmeuc\v@leur{#2}\invers@\v@leur\v@leur%
    \delt@=\repdecn@mb{\v@leur}\unit@\edef\v@ldelt@{\repdecn@mb{\delt@}}%
    \Figg@tXY{#2}\v@lX=\v@ldelt@\v@lX\v@lY=\v@ldelt@\v@lY%
    \Figv@ctCreg#1(\v@lX,\v@lY)}\ignorespaces\fi}
\ctr@ld@f\def\figvectUTD#1[#2]{\ifGR@cri{\n@rmeuc\v@leur{#2}\invers@\v@leur\v@leur%
    \delt@=\repdecn@mb{\v@leur}\unit@\edef\v@ldelt@{\repdecn@mb{\delt@}}%
    \Figg@tXY{#2}\v@lX=\v@ldelt@\v@lX\v@lY=\v@ldelt@\v@lY\v@lZ=\v@ldelt@\v@lZ%
    \Figv@ctCreg#1(\v@lX,\v@lY,\v@lZ)}\ignorespaces\fi}
\ctr@ld@f\def\figvisu#1#2#3{\c@ldefproj\initb@undb@x\xdef\figforTeXFigno{\figforTeXnextFigno}%
    \s@mme=\figforTeXnextFigno\advance\s@mme\@ne\xdef\figforTeXnextFigno{\number\s@mme}%
    \setbox\b@xvisu=\hbox{\ifnum\@utoFN>\z@\figinsert{}\gdef\@utoFInDone{0}\fi\ignorespaces#3}%
    \gdef\@utoFInDone{1}\gdef\@utoFN{0}%
    \v@lXa=-\c@@rdYmin\v@lYa=\c@@rdYmax\advance\v@lYa-\c@@rdYmin%
    \v@lX=\c@@rdXmax\advance\v@lX-\c@@rdXmin%
    \setbox#1=\hbox{#2}\v@lY=-\v@lX\maxim@m{\v@lX}{\v@lX}{\wd#1}%
    \advance\v@lY\v@lX\divide\v@lY\tw@\advance\v@lY-\c@@rdXmin%
    \setbox#1=\vbox{\parindent\z@\hsize=\v@lX\vskip\v@lYa%
    \rlap{\hskip\v@lY\smash{\raise\v@lXa\box\b@xvisu}}%
    \def\t@xt@{#2}\ifx\t@xt@\empty\else\medskip\centerline{#2}\fi}\wd#1=\v@lX}
\ctr@ld@f\def\figDecrementFigno{{\xdef\figforTeXnextFigno{\figforTeXFigno}%
    \s@mme=\figforTeXFigno\advance\s@mme\m@ne\xdef\figforTeXFigno{\number\s@mme}}}
\ctr@ln@w{newbox}\Bt@rget\setbox\Bt@rget=\null
\ctr@ln@w{newbox}\BminTD@\setbox\BminTD@=\null
\ctr@ln@w{newbox}\BmaxTD@\setbox\BmaxTD@=\null
\ctr@ln@w{newif}\ifnewt@rgetpt\ctr@ln@w{newif}\ifnewdis@b
\ctr@ld@f\def\b@undb@xTD#1#2#3{%
    \relax\ifdim#1<\wd\BminTD@\global\wd\BminTD@=#1\fi%
    \relax\ifdim#2<\ht\BminTD@\global\ht\BminTD@=#2\fi%
    \relax\ifdim#3<\dp\BminTD@\global\dp\BminTD@=#3\fi%
    \relax\ifdim#1>\wd\BmaxTD@\global\wd\BmaxTD@=#1\fi%
    \relax\ifdim#2>\ht\BmaxTD@\global\ht\BmaxTD@=#2\fi%
    \relax\ifdim#3>\dp\BmaxTD@\global\dp\BmaxTD@=#3\fi}
\ctr@ld@f\def\c@ldefdisob{{\ifdim\wd\BminTD@<\maxdimen\v@leur=\wd\BmaxTD@\advance\v@leur-\wd\BminTD@%
    \delt@=\ht\BmaxTD@\advance\delt@-\ht\BminTD@\maxim@m{\v@leur}{\v@leur}{\delt@}%
    \delt@=\dp\BmaxTD@\advance\delt@-\dp\BminTD@\maxim@m{\v@leur}{\v@leur}{\delt@}%
    \v@leur=5\v@leur\else\v@leur=800pt\fi\c@ldefdisob@{\v@leur}}}
\ctr@ln@m\disob@intern
\ctr@ln@m\disob@
\ctr@ln@m\divf@ctproj
\ctr@ld@f\def\c@ldefdisob@#1{{\v@leur=#1\ifdim\v@leur<\p@\v@leur=800pt\fi%
    \xdef\disob@intern{\repdecn@mb{\v@leur}}%
    \delt@=\ptT@unit@\v@leur\xdef\disob@unit{\repdecn@mb{\delt@}}%
    \f@ctech=\@ne\loop\ifdim\v@leur>\t@n pt\divide\v@leur\t@n\multiply\f@ctech\t@n\repeat%
    \xdef\disob@{\repdecn@mb{\v@leur}}\xdef\divf@ctproj{\the\f@ctech}}%
    \global\newdis@btrue}
\ctr@ln@m\t@rgetpt
\ctr@ld@f\def\c@ldeft@rgetpt{\newt@rgetpttrue\def\t@rgetpt{CenterBoundBox}{%
    \delt@=\wd\BmaxTD@\advance\delt@-\wd\BminTD@\divide\delt@\tw@%
    \v@leur=\wd\BminTD@\advance\v@leur\delt@\global\wd\Bt@rget=\v@leur%
    \delt@=\ht\BmaxTD@\advance\delt@-\ht\BminTD@\divide\delt@\tw@%
    \v@leur=\ht\BminTD@\advance\v@leur\delt@\global\ht\Bt@rget=\v@leur%
    \delt@=\dp\BmaxTD@\advance\delt@-\dp\BminTD@\divide\delt@\tw@%
    \v@leur=\dp\BminTD@\advance\v@leur\delt@\global\dp\Bt@rget=\v@leur}}
\ctr@ln@m\c@ldefproj
\ctr@ld@f\def\c@ldefprojTD{\ifnewt@rgetpt\else\c@ldeft@rgetpt\fi\ifnewdis@b\else\c@ldefdisob\fi}
\ctr@ld@f\def\c@lprojcav{
    \v@lZa=\cxa@\v@lY\advance\v@lX\v@lZa%
    \v@lZa=\cxb@\v@lY\v@lY=\v@lZ\advance\v@lY\v@lZa\ignorespaces}
\ctr@ln@m\v@lcoef
\ctr@ld@f\def\c@lprojrea{
    \advance\v@lX-\wd\Bt@rget\advance\v@lY-\ht\Bt@rget\advance\v@lZ-\dp\Bt@rget%
    \v@lZa=\cza@\v@lX\advance\v@lZa\czb@\v@lY\advance\v@lZa\czc@\v@lZ%
    \divide\v@lZa\divf@ctproj\advance\v@lZa\disob@ pt\invers@{\v@lZa}{\v@lZa}%
    \v@lZa=\disob@\v@lZa\edef\v@lcoef{\repdecn@mb{\v@lZa}}%
    \v@lXa=\cxa@\v@lX\advance\v@lXa\cxb@\v@lY\v@lXa=\v@lcoef\v@lXa%
    \v@lY=\cyb@\v@lY\advance\v@lY\cya@\v@lX\advance\v@lY\cyc@\v@lZ%
    \v@lY=\v@lcoef\v@lY\v@lX=\v@lXa\ignorespaces}
\ctr@ld@f\def\c@lprojort{
    \v@lXa=\cxa@\v@lX\advance\v@lXa\cxb@\v@lY%
    \v@lY=\cyb@\v@lY\advance\v@lY\cya@\v@lX\advance\v@lY\cyc@\v@lZ%
    \v@lX=\v@lXa\ignorespaces}
\ctr@ld@f\def\Figptpr@j#1:#2/#3/{{\Figg@tXY{#3}\superc@lprojSP%
    \Figp@intregDD#1:{#2}(\v@lX,\v@lY)}\ignorespaces}
\ctr@ln@m\figsetobdist
\ctr@ld@f\def\figsetobdistDD{\un@v@ilable{figsetobdist}}
\ctr@ld@f\def\figsetobdistTD(#1){{\ifCUR@PS\W@rnmesIgn{figset proj(dist=...)}%
    \else\v@leur=#1\unit@\c@ldefdisob@{\v@leur}\fi}\ignorespaces}
\ctr@ln@m\c@lprojSP
\ctr@ln@m\CUR@proj
\ctr@ln@m\typ@proj
\ctr@ln@m\superc@lprojSP
\ctr@ld@f\def\Figs@tproj#1{%
    \if#13 \def@ultproj\else\if#1c\def@ultproj%
    \else\if#1o\xdef\CUR@proj{1}\xdef\typ@proj{orthogonal}%
         \figsetviewTD(\def@ultpsi,\def@ulttheta)%
         \global\let\c@lprojSP=\c@lprojort\global\let\superc@lprojSP=\c@lprojort%
    \else\if#1r\xdef\CUR@proj{2}\xdef\typ@proj{realistic}%
         \figsetviewTD(\def@ultpsi,\def@ulttheta)%
         \global\let\c@lprojSP=\c@lprojrea\global\let\superc@lprojSP=\c@lprojrea%
    \else\def@ultproj\message{*** Unknown projection. Cavalier projection assumed.}%
    \fi\fi\fi\fi}
\ctr@ld@f\def\def@ultproj{\xdef\CUR@proj{0}\xdef\typ@proj{cavalier}\figsetviewTD(\def@ultpsi,0.5)%
         \global\let\c@lprojSP=\c@lprojcav\global\let\superc@lprojSP=\c@lprojcav}
\ctr@ln@m\figsettarget
\ctr@ld@f\def\figsettargetDD{\un@v@ilable{figsettarget}}
\ctr@ld@f\def\figsettargetTD[#1]{{\ifCUR@PS\W@rnmesIgn{figset proj(targetpt=...)}%
    \else\global\newt@rgetpttrue\xdef\t@rgetpt{#1}\Figg@tXY{#1}\global\wd\Bt@rget=\v@lX%
    \global\ht\Bt@rget=\v@lY\global\dp\Bt@rget=\v@lZ\fi}\ignorespaces}
\ctr@ln@m\figsetview
\ctr@ld@f\def\figsetviewDD{\un@v@ilable{figsetview}}
\ctr@ld@f\def\figsetviewTD(#1){\ifCUR@PS\W@rnmesIgn{figset proj(Psi|Theta|Lambda=...)}%
     \else\Figsetview@#1,:\fi\ignorespaces}
\ctr@ld@f\def\Figsetview@#1,#2:{{\xdef\v@lPsi{#1}\def\t@xt@{#2}%
    \ifx\t@xt@\empty\def\@rgdeux{\v@lTheta}\else\X@rgdeux@#2\fi%
    \c@ssin{\costhet@}{\sinthet@}{#1}\v@lmin=\costhet@ pt\v@lmax=\sinthet@ pt%
    \ifcase\CUR@proj%
    \v@leur=\@rgdeux\v@lmin\xdef\cxa@{\repdecn@mb{\v@leur}}%
    \v@leur=\@rgdeux\v@lmax\xdef\cxb@{\repdecn@mb{\v@leur}}\v@leur=\@rgdeux pt%
    \relax\ifdim\v@leur>\p@\message{*** Lambda too large ! See \BS@ figset proj() !}\fi%
    \else%
    \v@lmax=-\v@lmax\xdef\cxa@{\repdecn@mb{\v@lmax}}\xdef\cxb@{\costhet@}%
    \ifx\t@xt@\empty\edef\@rgdeux{\def@ulttheta}\fi\c@ssin{\C@}{\S@}{\@rgdeux}%
    \v@lmax=-\S@ pt%
    \v@leur=\v@lmax\v@leur=\costhet@\v@leur\xdef\cya@{\repdecn@mb{\v@leur}}%
    \v@leur=\v@lmax\v@leur=\sinthet@\v@leur\xdef\cyb@{\repdecn@mb{\v@leur}}%
    \xdef\cyc@{\C@}\v@lmin=-\C@ pt%
    \v@leur=\v@lmin\v@leur=\costhet@\v@leur\xdef\cza@{\repdecn@mb{\v@leur}}%
    \v@leur=\v@lmin\v@leur=\sinthet@\v@leur\xdef\czb@{\repdecn@mb{\v@leur}}%
    \xdef\czc@{\repdecn@mb{\v@lmax}}\fi%
    \xdef\v@lTheta{\@rgdeux}}}
\ctr@ld@f\def\def@ultpsi{40}
\ctr@ld@f\def\def@ulttheta{25}
\ctr@ln@m\l@debut
\ctr@ln@m\n@mref
\ctr@ld@f\def\Figsetpr@j#1=#2|{\keln@mtr#1|%
    \def\n@mref{dep}\ifx\l@debut\n@mref\Figsetd@p{#2}\else
    \def\n@mref{dis}\ifx\l@debut\n@mref%
     \ifnum\CUR@proj=\tw@\figsetobdist(#2)\else\Figset@rr\fi\else
    \def\n@mref{lam}\ifx\l@debut\n@mref\Figsetd@p{#2}\else
    \def\n@mref{lat}\ifx\l@debut\n@mref\Figsetth@{#2}\else
    \def\n@mref{lon}\ifx\l@debut\n@mref\figsetview(#2)\else
    \def\n@mref{psi}\ifx\l@debut\n@mref\figsetview(#2)\else
    \def\n@mref{tar}\ifx\l@debut\n@mref%
     \ifnum\CUR@proj=\tw@\figsettarget[#2]\else\Figset@rr\fi\else
    \def\n@mref{the}\ifx\l@debut\n@mref\Figsetth@{#2}\else
    \W@rnmesAttr{figset proj}{#1}\fi\fi\fi\fi\fi\fi\fi\fi}
\ctr@ld@f\def\Figsetd@p#1{\ifnum\CUR@proj=\z@\figsetview(\v@lPsi,#1)\else\Figset@rr\fi}
\ctr@ld@f\def\Figsetth@#1{\ifnum\CUR@proj=\z@\Figset@rr\else\figsetview(\v@lPsi,#1)\fi}
\ctr@ld@f\def\Figset@rr{\message{*** \BS@ figset proj(): Attribute "\n@mref" ignored, incompatible
    with current projection}}
\ctr@ld@f\def\initb@undb@xTD{\wd\BminTD@=\maxdimen\ht\BminTD@=\maxdimen\dp\BminTD@=\maxdimen%
    \wd\BmaxTD@=-\maxdimen\ht\BmaxTD@=-\maxdimen\dp\BmaxTD@=-\maxdimen}
\ctr@ln@w{newbox}\Gb@x      
\ctr@ln@w{newbox}\Gb@xSC    
\ctr@ln@w{newtoks}\c@nsymb  
\ctr@ln@w{newif}\ifr@undcoord\ctr@ln@w{newif}\ifunitpr@sent
\ctr@ld@f\def\unssqrttw@{0.707106 }
\ctr@ld@f\def\figAst{\raise-1.15ex\hbox{$\ast$}}
\ctr@ld@f\def\figBullet{\raise-1.15ex\hbox{$\bullet$}}
\ctr@ld@f\def\figCirc{\raise-1.15ex\hbox{$\circ$}}
\ctr@ld@f\def\figDiamond{\raise-1.15ex\hbox{$\diamond$}}%
\ctr@ld@f\def\boxit#1#2{\leavevmode\hbox{\vrule\vbox{\hrule\vglue#1%
    \vtop{\hbox{\kern#1{#2}\kern#1}\vglue#1\hrule}}\vrule}}
\ctr@ld@f\def\centertext#1#2{\vbox{\hsize#1\parindent0cm%
    \leftskip=0pt plus 1fil\rightskip=0pt plus 1fil\parfillskip=0pt{#2}}}
\ctr@ld@f\def\lefttext#1#2{\vbox{\hsize#1\parindent0cm\rightskip=0pt plus 1fil#2}}
\ctr@ld@f\def\c@nterpt{\ignorespaces%
    \kern-.5\wd\Gb@xSC%
    \raise-.5\ht\Gb@xSC\rlap{\hbox{\raise.5\dp\Gb@xSC\hbox{\copy\Gb@xSC}}}%
    \kern .5\wd\Gb@xSC\ignorespaces}
\ctr@ld@f\def\b@undb@xSC#1#2{{\v@lXa=#1\v@lYa=#2%
    \v@leur=\ht\Gb@xSC\advance\v@leur\dp\Gb@xSC%
    \advance\v@lXa-.5\wd\Gb@xSC\advance\v@lYa-.5\v@leur\b@undb@x{\v@lXa}{\v@lYa}%
    \advance\v@lXa\wd\Gb@xSC\advance\v@lYa\v@leur\b@undb@x{\v@lXa}{\v@lYa}}}
\ctr@ln@m\Dist@n
\ctr@ln@m\l@suite
\ctr@ld@f\def\@keldist#1#2{\edef\Dist@n{#2}\y@tiunit{\Dist@n}%
    \ifunitpr@sent#1=\Dist@n\else#1=\Dist@n\unit@\fi}
\ctr@ld@f\def\y@tiunit#1{\unitpr@sentfalse\expandafter\y@tiunit@#1:}
\ctr@ld@f\def\y@tiunit@#1#2:{\ifcat#1a\unitpr@senttrue\else\def\l@suite{#2}%
    \ifx\l@suite\empty\else\y@tiunit@#2:\fi\fi}
\ctr@ln@m\figcoord
\ctr@ld@f\def\figcoordDD#1{{\v@lX=\ptT@unit@\v@lX\v@lY=\ptT@unit@\v@lY%
    \ifr@undcoord\ifcase#1\v@leur=0.5pt\or\v@leur=0.05pt\or\v@leur=0.005pt%
    \or\v@leur=0.0005pt\else\v@leur=\z@\fi%
    \ifdim\v@lX<\z@\advance\v@lX-\v@leur\else\advance\v@lX\v@leur\fi%
    \ifdim\v@lY<\z@\advance\v@lY-\v@leur\else\advance\v@lY\v@leur\fi\fi%
    (\@ffichnb{#1}{\repdecn@mb{\v@lX}},\ifmmode\else\thinspace\fi%
    \@ffichnb{#1}{\repdecn@mb{\v@lY}})}}
\ctr@ld@f\def\@ffichnb#1#2{{\def\@@ffich{\@ffich#1(}\edef\n@mbre{#2}%
    \expandafter\@@ffich\n@mbre)}}
\ctr@ld@f\def\@ffich#1(#2.#3){{#2\ifnum#1>\z@.\fi\def\dig@ts{#3}\s@mme=\z@%
    \loop\ifnum\s@mme<#1\expandafter\@ffichdec\dig@ts:\advance\s@mme\@ne\repeat}}
\ctr@ld@f\def\@ffichdec#1#2:{\relax#1\def\dig@ts{#20}}
\ctr@ld@f\def\figcoordTD#1{{\v@lX=\ptT@unit@\v@lX\v@lY=\ptT@unit@\v@lY\v@lZ=\ptT@unit@\v@lZ%
    \ifr@undcoord\ifcase#1\v@leur=0.5pt\or\v@leur=0.05pt\or\v@leur=0.005pt%
    \or\v@leur=0.0005pt\else\v@leur=\z@\fi%
    \ifdim\v@lX<\z@\advance\v@lX-\v@leur\else\advance\v@lX\v@leur\fi%
    \ifdim\v@lY<\z@\advance\v@lY-\v@leur\else\advance\v@lY\v@leur\fi%
    \ifdim\v@lZ<\z@\advance\v@lZ-\v@leur\else\advance\v@lZ\v@leur\fi\fi%
    (\@ffichnb{#1}{\repdecn@mb{\v@lX}},\ifmmode\else\thinspace\fi%
     \@ffichnb{#1}{\repdecn@mb{\v@lY}},\ifmmode\else\thinspace\fi%
     \@ffichnb{#1}{\repdecn@mb{\v@lZ}})}}
\ctr@ld@f\def\figsetroundcoord#1{\expandafter\Figsetr@undcoord#1:\ignorespaces}
\ctr@ld@f\def\Figsetr@undcoord#1#2:{\if#1n\r@undcoordfalse\else\r@undcoordtrue\fi}
\ctr@ld@f\def\Figsetwr@te#1=#2|{\keln@mun#1|%
    \def\n@mref{m}\ifx\l@debut\n@mref\figsetmark{#2}\else
    \def\n@mref{p}\ifx\l@debut\n@mref\figsetptname{#2}\else
    \def\n@mref{r}\ifx\l@debut\n@mref\figsetroundcoord{#2}\else
    \W@rnmesAttr{figset write}{#1}\fi\fi\fi}
\ctr@ld@f\def\figsetmark#1{\c@nsymb={#1}\setbox\Gb@xSC=\hbox{\the\c@nsymb}\ignorespaces}
\ctr@ln@m\ptn@me
\ctr@ld@f\def\figsetptname#1{\def\ptn@me##1{#1}\ignorespaces}
\ctr@ld@f\def\FigWrit@L#1:#2(#3,#4){\ignorespaces\@keldist\v@leur{#3}\@keldist\delt@{#4}%
    \C@rp@r@m\def\list@num{#1}\@ecfor\p@int:=\list@num\do{\FigWrit@pt{\p@int}{#2}}}
\ctr@ld@f\def\FigWrit@pt#1#2{\FigWp@r@m{#1}{#2}\Vc@rrect\figWp@si%
    \ifdim\wd\Gb@xSC>\z@\b@undb@xSC{\v@lX}{\v@lY}\fi\figWBB@x}
\ctr@ld@f\def\FigWp@r@m#1#2{\Figg@tXY{#1}%
    \setbox\Gb@x=\hbox{\def\t@xt@{#2}\ifx\t@xt@\empty\Figg@tT{#1}\else#2\fi}\c@lprojSP}
\ctr@ld@f\let\Vc@rrect=\relax
\ctr@ld@f\let\C@rp@r@m=\relax
\ctr@ld@f\def\figwrite[#1]#2{{\ignorespaces\def\list@num{#1}\@ecfor\p@int:=\list@num\do{%
    \setbox\Gb@x=\hbox{\def\t@xt@{#2}\ifx\t@xt@\empty\Figg@tT{\p@int}\else#2\fi}%
    \Figwrit@{\p@int}}}\ignorespaces}
\ctr@ld@f\def\Figwrit@#1{\Figg@tXY{#1}\c@lprojSP%
    \rlap{\kern\v@lX\raise\v@lY\hbox{\unhcopy\Gb@x}}\v@leur=\v@lY%
    \advance\v@lY\ht\Gb@x\b@undb@x{\v@lX}{\v@lY}\advance\v@lX\wd\Gb@x%
    \v@lY=\v@leur\advance\v@lY-\dp\Gb@x\b@undb@x{\v@lX}{\v@lY}}
\ctr@ld@f\def\figwritec[#1]#2{{\ignorespaces\def\list@num{#1}%
    \@ecfor\p@int:=\list@num\do{\Figwrit@c{\p@int}{#2}}}\ignorespaces}
\ctr@ld@f\def\Figwrit@c#1#2{\FigWp@r@m{#1}{#2}%
    \rlap{\kern\v@lX\raise\v@lY\hbox{\rlap{\kern-.5\wd\Gb@x%
    \raise-.5\ht\Gb@x\hbox{\raise.5\dp\Gb@x\hbox{\unhcopy\Gb@x}}}}}%
    \v@leur=\ht\Gb@x\advance\v@leur\dp\Gb@x%
    \advance\v@lX-.5\wd\Gb@x\advance\v@lY-.5\v@leur\b@undb@x{\v@lX}{\v@lY}%
    \advance\v@lX\wd\Gb@x\advance\v@lY\v@leur\b@undb@x{\v@lX}{\v@lY}}
\ctr@ld@f\def\figwritep[#1]{{\ignorespaces\def\list@num{#1}\setbox\Gb@x=\hbox{\c@nterpt}%
    \@ecfor\p@int:=\list@num\do{\Figwrit@{\p@int}}}\ignorespaces}
\ctr@ld@f\def\figwritew#1:#2(#3){\figwritegcw#1:{#2}(#3,0pt)}
\ctr@ld@f\def\figwritee#1:#2(#3){\figwritegce#1:{#2}(#3,0pt)}
\ctr@ld@f\def\figwriten#1:#2(#3){{\def\Vc@rrect{\v@lZ=\v@leur\advance\v@lZ\dp\Gb@x}%
    \Figwrit@NS#1:{#2}(#3)}\ignorespaces}
\ctr@ld@f\def\figwrites#1:#2(#3){{\def\Vc@rrect{\v@lZ=-\v@leur\advance\v@lZ-\ht\Gb@x}%
    \Figwrit@NS#1:{#2}(#3)}\ignorespaces}
\ctr@ld@f\def\Figwrit@NS#1:#2(#3){\let\figWp@si=\FigWp@siNS\let\figWBB@x=\FigWBB@xNS%
    \FigWrit@L#1:{#2}(#3,0pt)}
\ctr@ld@f\def\FigWp@siNS{\rlap{\kern\v@lX\raise\v@lY\hbox{\rlap{\kern-.5\wd\Gb@x%
    \raise\v@lZ\hbox{\unhcopy\Gb@x}}\c@nterpt}}}
\ctr@ld@f\def\FigWBB@xNS{\advance\v@lY\v@lZ%
    \advance\v@lY-\dp\Gb@x\advance\v@lX-.5\wd\Gb@x\b@undb@x{\v@lX}{\v@lY}%
    \advance\v@lY\ht\Gb@x\advance\v@lY\dp\Gb@x%
    \advance\v@lX\wd\Gb@x\b@undb@x{\v@lX}{\v@lY}}
\ctr@ld@f\def\figwritenw#1:#2(#3){{\let\figWp@si=\FigWp@sigW\let\figWBB@x=\FigWBB@xgWE%
    \def\C@rp@r@m{\v@leur=\unssqrttw@\v@leur\delt@=\v@leur%
    \ifdim\delt@=\z@\delt@=\epsil@n\fi}\let@xte={-}\FigWrit@L#1:{#2}(#3,0pt)}\ignorespaces}
\ctr@ld@f\def\figwritesw#1:#2(#3){{\let\figWp@si=\FigWp@sigW\let\figWBB@x=\FigWBB@xgWE%
    \def\C@rp@r@m{\v@leur=\unssqrttw@\v@leur\delt@=-\v@leur%
    \ifdim\delt@=\z@\delt@=-\epsil@n\fi}\let@xte={-}\FigWrit@L#1:{#2}(#3,0pt)}\ignorespaces}
\ctr@ld@f\def\figwritene#1:#2(#3){{\let\figWp@si=\FigWp@sigE\let\figWBB@x=\FigWBB@xgWE%
    \def\C@rp@r@m{\v@leur=\unssqrttw@\v@leur\delt@=\v@leur%
    \ifdim\delt@=\z@\delt@=\epsil@n\fi}\let@xte={}\FigWrit@L#1:{#2}(#3,0pt)}\ignorespaces}
\ctr@ld@f\def\figwritese#1:#2(#3){{\let\figWp@si=\FigWp@sigE\let\figWBB@x=\FigWBB@xgWE%
    \def\C@rp@r@m{\v@leur=\unssqrttw@\v@leur\delt@=-\v@leur%
    \ifdim\delt@=\z@\delt@=-\epsil@n\fi}\let@xte={}\FigWrit@L#1:{#2}(#3,0pt)}\ignorespaces}
\ctr@ld@f\def\figwritegw#1:#2(#3,#4){{\let\figWp@si=\FigWp@sigW\let\figWBB@x=\FigWBB@xgWE%
    \let@xte={-}\FigWrit@L#1:{#2}(#3,#4)}\ignorespaces}
\ctr@ld@f\def\figwritege#1:#2(#3,#4){{\let\figWp@si=\FigWp@sigE\let\figWBB@x=\FigWBB@xgWE%
    \let@xte={}\FigWrit@L#1:{#2}(#3,#4)}\ignorespaces}
\ctr@ld@f\def\FigWp@sigW{\v@lXa=\z@\v@lYa=\ht\Gb@x\advance\v@lYa\dp\Gb@x%
    \ifdim\delt@>\z@\relax%
    \rlap{\kern\v@lX\raise\v@lY\hbox{\rlap{\kern-\wd\Gb@x\kern-\v@leur%
          \raise\delt@\hbox{\raise\dp\Gb@x\hbox{\unhcopy\Gb@x}}}\c@nterpt}}%
    \else\ifdim\delt@<\z@\relax\v@lYa=-\v@lYa%
    \rlap{\kern\v@lX\raise\v@lY\hbox{\rlap{\kern-\wd\Gb@x\kern-\v@leur%
          \raise\delt@\hbox{\raise-\ht\Gb@x\hbox{\unhcopy\Gb@x}}}\c@nterpt}}%
    \else\v@lXa=-.5\v@lYa%
    \rlap{\kern\v@lX\raise\v@lY\hbox{\rlap{\kern-\wd\Gb@x\kern-\v@leur%
          \raise-.5\ht\Gb@x\hbox{\raise.5\dp\Gb@x\hbox{\unhcopy\Gb@x}}}\c@nterpt}}%
    \fi\fi}
\ctr@ld@f\def\FigWp@sigE{\v@lXa=\z@\v@lYa=\ht\Gb@x\advance\v@lYa\dp\Gb@x%
    \ifdim\delt@>\z@\relax%
    \rlap{\kern\v@lX\raise\v@lY\hbox{\c@nterpt\kern\v@leur%
          \raise\delt@\hbox{\raise\dp\Gb@x\hbox{\unhcopy\Gb@x}}}}%
    \else\ifdim\delt@<\z@\relax\v@lYa=-\v@lYa%
    \rlap{\kern\v@lX\raise\v@lY\hbox{\c@nterpt\kern\v@leur%
          \raise\delt@\hbox{\raise-\ht\Gb@x\hbox{\unhcopy\Gb@x}}}}%
    \else\v@lXa=-.5\v@lYa%
    \rlap{\kern\v@lX\raise\v@lY\hbox{\c@nterpt\kern\v@leur%
          \raise-.5\ht\Gb@x\hbox{\raise.5\dp\Gb@x\hbox{\unhcopy\Gb@x}}}}%
    \fi\fi}
\ctr@ld@f\def\FigWBB@xgWE{\advance\v@lY\delt@%
    \advance\v@lX\the\let@xte\v@leur\advance\v@lY\v@lXa\b@undb@x{\v@lX}{\v@lY}%
    \advance\v@lX\the\let@xte\wd\Gb@x\advance\v@lY\v@lYa\b@undb@x{\v@lX}{\v@lY}}
\ctr@ld@f\def\figwritegcw#1:#2(#3,#4){{\let\figWp@si=\FigWp@sigcW\let\figWBB@x=\FigWBB@xgcWE%
    \let@xte={-}\FigWrit@L#1:{#2}(#3,#4)}\ignorespaces}
\ctr@ld@f\def\figwritegce#1:#2(#3,#4){{\let\figWp@si=\FigWp@sigcE\let\figWBB@x=\FigWBB@xgcWE%
    \let@xte={}\FigWrit@L#1:{#2}(#3,#4)}\ignorespaces}
\ctr@ld@f\def\FigWp@sigcW{\rlap{\kern\v@lX\raise\v@lY\hbox{\rlap{\kern-\wd\Gb@x\kern-\v@leur%
     \raise-.5\ht\Gb@x\hbox{\raise\delt@\hbox{\raise.5\dp\Gb@x\hbox{\unhcopy\Gb@x}}}}%
     \c@nterpt}}}
\ctr@ld@f\def\FigWp@sigcE{\rlap{\kern\v@lX\raise\v@lY\hbox{\c@nterpt\kern\v@leur%
    \raise-.5\ht\Gb@x\hbox{\raise\delt@\hbox{\raise.5\dp\Gb@x\hbox{\unhcopy\Gb@x}}}}}}
\ctr@ld@f\def\FigWBB@xgcWE{\v@lZ=\ht\Gb@x\advance\v@lZ\dp\Gb@x%
    \advance\v@lX\the\let@xte\v@leur\advance\v@lY\delt@\advance\v@lY.5\v@lZ%
    \b@undb@x{\v@lX}{\v@lY}%
    \advance\v@lX\the\let@xte\wd\Gb@x\advance\v@lY-\v@lZ\b@undb@x{\v@lX}{\v@lY}}
\ctr@ld@f\def\figwritebn#1:#2(#3){{\def\Vc@rrect{\v@lZ=\v@leur}\Figwrit@NS#1:{#2}(#3)}\ignorespaces}
\ctr@ld@f\def\figwritebs#1:#2(#3){{\def\Vc@rrect{\v@lZ=-\v@leur}\Figwrit@NS#1:{#2}(#3)}\ignorespaces}
\ctr@ld@f\def\figwritebw#1:#2(#3){{\let\figWp@si=\FigWp@sibW\let\figWBB@x=\FigWBB@xbWE%
    \let@xte={-}\FigWrit@L#1:{#2}(#3,0pt)}\ignorespaces}
\ctr@ld@f\def\figwritebe#1:#2(#3){{\let\figWp@si=\FigWp@sibE\let\figWBB@x=\FigWBB@xbWE%
    \let@xte={}\FigWrit@L#1:{#2}(#3,0pt)}\ignorespaces}
\ctr@ld@f\def\FigWp@sibW{\rlap{\kern\v@lX\raise\v@lY\hbox{\rlap{\kern-\wd\Gb@x\kern-\v@leur%
          \hbox{\unhcopy\Gb@x}}\c@nterpt}}}
\ctr@ld@f\def\FigWp@sibE{\rlap{\kern\v@lX\raise\v@lY\hbox{\c@nterpt\kern\v@leur%
          \hbox{\unhcopy\Gb@x}}}}
\ctr@ld@f\def\FigWBB@xbWE{\v@lZ=\ht\Gb@x\advance\v@lZ\dp\Gb@x%
    \advance\v@lX\the\let@xte\v@leur\advance\v@lY\ht\Gb@x\b@undb@x{\v@lX}{\v@lY}%
    \advance\v@lX\the\let@xte\wd\Gb@x\advance\v@lY-\v@lZ\b@undb@x{\v@lX}{\v@lY}}
\ctr@ln@w{newread}\frf@g  \ctr@ln@w{newwrite}\fwf@g
\ctr@ln@w{newif}\ifCUR@PS
\ctr@ln@w{newif}\ifGR@cri
\ctr@ln@w{newif}\ifUse@llipse
\ctr@ln@w{newif}\ifGRdebugm@de \GRdebugm@defalse 
\ctr@ln@w{newif}\ifPDFm@ke
\ifx\pdfliteral\undefined\else\ifnum\pdfoutput>\z@\PDFm@ketrue\fi\fi
\ctr@ld@f\def\initPDF@rDVI{%
\ifPDFm@ke
 \let\figscan=\figscan@E
 \let\newGr@FN=\newGr@FNPDF
 \ctr@ld@f\def\c@mcurveto{c}
 \ctr@ld@f\def\c@mfill{f}
 \ctr@ld@f\def\c@mgsave{q}
 \ctr@ld@f\def\c@mgrestore{Q}
 \ctr@ld@f\def\c@mlineto{l}
 \ctr@ld@f\def\c@mmoveto{m}
 \ctr@ld@f\def\c@msetgray{g}     \ctr@ld@f\def\c@msetgrayStroke{G}
 \ctr@ld@f\def\c@msetcmykcolor{k}\ctr@ld@f\def\c@msetcmykcolorStroke{K}
 \ctr@ld@f\def\c@msetrgbcolor{rg}\ctr@ld@f\def\c@msetrgbcolorStroke{RG}
 \ctr@ld@f\def\d@fprimarC@lor{\CUR@color\space\CUR@colorc@md%
               \space\CUR@color\space\CUR@colorc@mdStroke}
 \ctr@ld@f\def\c@msetdash{d}
 \ctr@ld@f\def\c@msetlinejoin{j}
 \ctr@ld@f\def\c@msetlinewidth{w}
 \ctr@ld@f\def\f@gclosestroke{\immediate\write\fwf@g{s}}
 \ctr@ld@f\def\f@gfill{\immediate\write\fwf@g{\fillc@md}}
 \ctr@ld@f\def\f@gnewpath{}
 \ctr@ld@f\def\f@gstroke{\immediate\write\fwf@g{S}}
\else
 \let\figinsertE=\figinsert
 \let\newGr@FN=\newGr@FNDVI
 \ctr@ld@f\def\c@mcurveto{curveto}
 \ctr@ld@f\def\c@mfill{fill}
 \ctr@ld@f\def\c@mgsave{gsave}
 \ctr@ld@f\def\c@mgrestore{grestore}
 \ctr@ld@f\def\c@mlineto{lineto}
 \ctr@ld@f\def\c@mmoveto{moveto}
 \ctr@ld@f\def\c@msetgray{setgray}          \ctr@ld@f\def\c@msetgrayStroke{}
 \ctr@ld@f\def\c@msetcmykcolor{setcmykcolor}\ctr@ld@f\def\c@msetcmykcolorStroke{}
 \ctr@ld@f\def\c@msetrgbcolor{setrgbcolor}  \ctr@ld@f\def\c@msetrgbcolorStroke{}
 \ctr@ld@f\def\d@fprimarC@lor{\CUR@color\space\CUR@colorc@md}
 \ctr@ld@f\def\c@msetdash{setdash}
 \ctr@ld@f\def\c@msetlinejoin{setlinejoin}
 \ctr@ld@f\def\c@msetlinewidth{setlinewidth}
 \ctr@ld@f\def\f@gclosestroke{\immediate\write\fwf@g{closepath\space stroke}}
 \ctr@ld@f\def\f@gfill{\immediate\write\fwf@g{\fillc@md}}
 \ctr@ld@f\def\f@gnewpath{\immediate\write\fwf@g{newpath}}
 \ctr@ld@f\def\f@gstroke{\immediate\write\fwf@g{stroke}}
\fi}
\ctr@ld@f\def\c@pypsfile#1#2{\c@pyfil@{\immediate\write#1}{#2}}
\ctr@ld@f\def\Figinclud@PDF#1#2{\openin\frf@g=#1\pdfliteral{q #2 0 0 #2 0 0 cm}%
    \c@pyfil@{\pdfliteral}{\frf@g}\pdfliteral{Q}\closein\frf@g}
\ctr@ln@w{newif}\ifmored@ta
\ctr@ln@m\bl@nkline
\ctr@ld@f\def\c@pyfil@#1#2{\def\bl@nkline{\par}{\catcode`\%=12
    \loop\ifeof#2\mored@tafalse\else\mored@tatrue\immediate\read#2 to\tr@c
    \ifx\tr@c\bl@nkline\else#1{\tr@c}\fi\fi\ifmored@ta\repeat}}
\ctr@ld@f\def\keln@mun#1#2|{\def\l@debut{#1}\def\l@suite{#2}}
\ctr@ld@f\def\keln@mde#1#2#3|{\def\l@debut{#1#2}\def\l@suite{#3}}
\ctr@ld@f\def\keln@mtr#1#2#3#4|{\def\l@debut{#1#2#3}\def\l@suite{#4}}
\ctr@ld@f\def\keln@mqu#1#2#3#4#5|{\def\l@debut{#1#2#3#4}\def\l@suite{#5}}
\ctr@ld@f\let\@psffilein=\frf@g 
\ctr@ln@w{newif}\if@psffileok    
\ctr@ln@w{newif}\if@psfbbfound   
\ctr@ln@w{newif}\if@psfverbose   
\@psfverbosetrue
\ctr@ln@m\@psfllx \ctr@ln@m\@psflly
\ctr@ln@m\@psfurx \ctr@ln@m\@psfury
\ctr@ln@m\resetcolonc@tcode
\ctr@ld@f\def\@psfgetbb#1{\global\@psfbbfoundfalse%
\global\def\@psfllx{0}\global\def\@psflly{0}%
\global\def\@psfurx{30}\global\def\@psfury{30}%
\openin\@psffilein=#1\relax
\ifeof\@psffilein\errmessage{I couldn't open #1, will ignore it}\else
   \edef\resetcolonc@tcode{\catcode`\noexpand\:\the\catcode`\:\relax}%
   {\@psffileoktrue \chardef\other=12
    \def\do##1{\catcode`##1=\other}\dospecials \catcode`\ =10 \resetcolonc@tcode
    \loop
       \read\@psffilein to \@psffileline
       \ifeof\@psffilein\@psffileokfalse\else
          \expandafter\@psfaux\@psffileline:. \\%
       \fi
   \if@psffileok\repeat
   \if@psfbbfound\else
    \if@psfverbose\message{No bounding box comment in #1; using defaults}\fi\fi
   }\closein\@psffilein\fi}%
\ctr@ln@m\@psfbblit
\ctr@ln@m\@psfpercent
{\catcode`\%=12 \global\let\@psfpercent=
\ctr@ln@m\@psfaux
\long\def\@psfaux#1#2:#3\\{\ifx#1\@psfpercent
   \def\testit{#2}\ifx\testit\@psfbblit
      \@psfgrab #3 . . . \\%
      \@psffileokfalse
      \global\@psfbbfoundtrue
   \fi\else\ifx#1\par\else\@psffileokfalse\fi\fi}%
\ctr@ld@f\def\@psfempty{}%
\ctr@ld@f\def\@psfgrab #1 #2 #3 #4 #5\\{%
\global\def\@psfllx{#1}\ifx\@psfllx\@psfempty
      \@psfgrab #2 #3 #4 #5 .\\\else
   \global\def\@psflly{#2}%
   \global\def\@psfurx{#3}\global\def\@psfury{#4}\fi}%
\ctr@ld@f\def\PSwrit@cmd#1#2#3{{\Figg@tXY{#1}\c@lprojSP\b@undb@x{\v@lX}{\v@lY}%
    \v@lX=\ptT@ptps\v@lX\v@lY=\ptT@ptps\v@lY%
    \immediate\write#3{\repdecn@mb{\v@lX}\space\repdecn@mb{\v@lY}\space#2}}}
\ctr@ld@f\def\PSwrit@cmdS#1#2#3#4#5{{\Figg@tXY{#1}\c@lprojSP\b@undb@x{\v@lX}{\v@lY}%
    \global\result@t=\v@lX\global\result@@t=\v@lY%
    \v@lX=\ptT@ptps\v@lX\v@lY=\ptT@ptps\v@lY%
    \immediate\write#3{\repdecn@mb{\v@lX}\space\repdecn@mb{\v@lY}\space#2}}%
    \edef#4{\the\result@t}\edef#5{\the\result@@t}}
\ctr@ld@f\def\update@ttr#1#2#3{\Figdisc@rdLTS{#3}{\n@mref}%
    \ifx\n@mref\D@FTref#2{#1}\else#2{#3}\fi}
\ctr@ld@f\def\D@FTref{default}
\ctr@ld@f\def\W@rnmesAttr#1#2{%
    \immediate\write16{*** Unknown attribute: \BS@ #1(..., #2=...)}}
\ctr@ld@f\def\W@rnmeskwd#1#2{%
    \immediate\write16{*** Unknown keyword #2 in \BS@ #1}}
\ctr@ld@f\def\W@rnmesIgn#1{\immediate\write16{*** \BS@ #1 is ignored inside a
     \BS@ figdrawbegin-\BS@ figdrawend block.}}
\ctr@ld@f\def\Psset@lti#1=#2|{\keln@mtr#1|%
    \def\n@mref{blc}\ifx\l@debut\n@mref\update@ttr\D@FTref\P@setblcolor{#2}\else
    \def\n@mref{bld}\ifx\l@debut\n@mref\update@ttr\D@FTref\P@setbldash{#2}\else
    \def\n@mref{blw}\ifx\l@debut\n@mref\update@ttr\D@FTref\P@setblwidth{#2}\else
    \def\n@mref{sqc}\ifx\l@debut\n@mref\update@ttr\D@FTref\P@setsqcolor{#2}\else
    \def\n@mref{sqd}\ifx\l@debut\n@mref\update@ttr\D@FTref\P@setsqdash{#2}\else
    \def\n@mref{sqw}\ifx\l@debut\n@mref\update@ttr\D@FTref\P@setsqwidth{#2}\else
    \W@rnmesAttr{figset altitude}{#1}\fi\fi\fi\fi\fi\fi}
\ctr@ln@m\DDV@blcolor
\ctr@ld@f\def\P@setblcolor#1{\edef\DDV@blcolor{#1}}
\ctr@ln@m\DDV@bldash
\ctr@ld@f\def\P@setbldash#1{\edef\DDV@bldash{#1}}
\ctr@ln@m\DDV@blwidth
\ctr@ld@f\def\P@setblwidth#1{\edef\DDV@blwidth{#1}}
\ctr@ln@m\DDV@sqcolor
\ctr@ld@f\def\P@setsqcolor#1{\edef\DDV@sqcolor{#1}}
\ctr@ln@m\DDV@sqdash
\ctr@ld@f\def\P@setsqdash#1{\edef\DDV@sqdash{#1}}
\ctr@ln@m\DDV@sqwidth
\ctr@ld@f\def\P@setsqwidth#1{\edef\DDV@sqwidth{#1}}
\ctr@ld@f\def\figdrawaltitude#1[#2,#3,#4]{{\ifCUR@PS\ifGR@cri%
    \PSc@mment{altitude Square Dim=#1, Triangle=[#2 / #3,#4]}%
    \s@uvc@ntr@l\et@tpsaltitude\resetc@ntr@l{2}\figptorthoprojline-5:=#2/#3,#4/%
    \figvectP -1[#3,#4]\n@rminf{\v@leur}{-1}\vecunit@{-3}{-1}%
    \figvectP -1[-5,#3]\n@rminf{\v@lmin}{-1}\figvectP -2[-5,#4]\n@rminf{\v@lmax}{-2}%
    \ifdim\v@lmin<\v@lmax\s@mme=#3\else\v@lmax=\v@lmin\s@mme=#4\fi%
    \figvectP -4[-5,#2]\vecunit@{-4}{-4}\delt@=#1\unit@%
    \edef\t@ille{\repdecn@mb{\delt@}}\figpttra-1:=-5/\t@ille,-3/%
    \figptstra-3=-5,-1/\t@ille,-4/\figdrawline[#2,-5]%
    \Pss@tspecifSt{color=\DDV@sqcolor,dash=\DDV@sqdash,width=\DDV@sqwidth}%
    \figdrawline[-1,-2,-3]%
    \Psrest@reSt{color=\DDV@sqcolor,dash=\DDV@sqdash,width=\DDV@sqwidth}%
    \ifdim\v@leur<\v@lmax%
    \Pss@tspecifSt{color=\DDV@blcolor,dash=\DDV@bldash,width=\DDV@blwidth}%
    \figdrawline[-5,\the\s@mme]%
    \Psrest@reSt{color=\DDV@blcolor,dash=\DDV@bldash,width=\DDV@blwidth}%
    \fi\PSc@mment{End altitude}\resetc@ntr@l\et@tpsaltitude\fi\fi}}
\ctr@ld@f\def\Ps@rcerc#1;#2(#3,#4){\ellBB@x#1;#2,#2(#3,#4,0)%
    \f@gnewpath{\delt@=#2\unit@\delt@=\ptT@ptps\delt@%
    \BdingB@xfalse%
    \PSwrit@cmd{#1}{\repdecn@mb{\delt@}\space #3\space #4\space arc}{\fwf@g}}}
\ctr@ln@m\figdrawarccirc
\ctr@ld@f\def\Q@arccircDD#1;#2(#3,#4){\ifCUR@PS\ifGR@cri%
    \PSc@mment{arccircDD Center=#1 ; Radius=#2 (Ang1=#3, Ang2=#4)}%
    \iffillm@de\Ps@rcerc#1;#2(#3,#4)%
    \f@gfill%
    \else\Ps@rcerc#1;#2(#3,#4)\f@gstroke\fi%
    \PSc@mment{End arccircDD}\fi\fi}
\ctr@ld@f\def\Q@arccircTD#1,#2,#3;#4(#5,#6){{\ifCUR@PS\ifGR@cri\s@uvc@ntr@l\et@tpsarccircTD%
    \PSc@mment{arccircTD Center=#1,P1=#2,P2=#3 ; Radius=#4 (Ang1=#5, Ang2=#6)}%
    \setc@ntr@l{2}\c@lExtAxes#1,#2,#3(#4)\Q@arcellPATD#1,-4,-5(#5,#6)%
    \PSc@mment{End arccircTD}\resetc@ntr@l\et@tpsarccircTD\fi\fi}}
\ctr@ld@f\def\c@lExtAxes#1,#2,#3(#4){%
    \figvectPTD-5[#1,#2]\vecunit@{-5}{-5}\figvectNTD-4[#1,#2,#3]\vecunit@{-4}{-4}%
    \figvectNVTD-3[-4,-5]\delt@=#4\unit@\edef\r@yon{\repdecn@mb{\delt@}}%
    \figpttra-4:=#1/\r@yon,-5/\figpttra-5:=#1/\r@yon,-3/}
\ctr@ln@m\figdrawarccircP
\ctr@ld@f\def\Q@arccircPDD#1;#2[#3,#4]{{\ifCUR@PS\ifGR@cri\s@uvc@ntr@l\et@tpsarccircPDD%
    \PSc@mment{arccircPDD Center=#1; Radius=#2, [P1=#3, P2=#4]}%
    \Ps@ngleparam#1;#2[#3,#4]\ifdim\v@lmin>\v@lmax\advance\v@lmax\DePI@deg\fi%
    \edef\@ngdeb{\repdecn@mb{\v@lmin}}\edef\@ngfin{\repdecn@mb{\v@lmax}}%
    \figdrawarccirc#1;\r@dius(\@ngdeb,\@ngfin)%
    \PSc@mment{End arccircPDD}\resetc@ntr@l\et@tpsarccircPDD\fi\fi}}
\ctr@ld@f\def\Q@arccircPTD#1;#2[#3,#4,#5]{{\ifCUR@PS\ifGR@cri\s@uvc@ntr@l\et@tpsarccircPTD%
    \PSc@mment{arccircPTD Center=#1; Radius=#2, [P1=#3, P2=#4, P3=#5]}%
    \setc@ntr@l{2}\c@lExtAxes#1,#3,#5(#2)\figdrawarcellPP#1,-4,-5[#3,#4]%
    \PSc@mment{End arccircPTD}\resetc@ntr@l\et@tpsarccircPTD\fi\fi}}
\ctr@ld@f\def\Ps@ngleparam#1;#2[#3,#4]{\setc@ntr@l{2}%
    \figvectPDD-1[#1,#3]\vecunit@{-1}{-1}\Figg@tXY{-1}\arct@n\v@lmin(\v@lX,\v@lY)%
    \figvectPDD-2[#1,#4]\vecunit@{-2}{-2}\Figg@tXY{-2}\arct@n\v@lmax(\v@lX,\v@lY)%
    \v@lmin=\rdT@deg\v@lmin\v@lmax=\rdT@deg\v@lmax%
    \v@leur=#2pt\maxim@m{\mili@u}{-\v@leur}{\v@leur}%
    \edef\r@dius{\repdecn@mb{\mili@u}}}
\ctr@ld@f\def\Ps@rcercBz#1;#2(#3,#4){\Ps@rellBz#1;#2,#2(#3,#4,0)}
\ctr@ld@f\def\Ps@rellBz#1;#2,#3(#4,#5,#6){%
    \ellBB@x#1;#2,#3(#4,#5,#6)\BdingB@xfalse%
    \c@lNbarcs{#4}{#5}\v@leur=#4pt\setc@ntr@l{2}\figptell-13::#1;#2,#3(#4,#6)%
    \f@gnewpath\PSwrit@cmd{-13}{\c@mmoveto}{\fwf@g}%
    \s@mme=\z@\bcl@rellBz#1;#2,#3(#6)\BdingB@xtrue}
\ctr@ld@f\def\bcl@rellBz#1;#2,#3(#4){\relax%
    \ifnum\s@mme<\p@rtent\advance\s@mme\@ne%
    \advance\v@leur\delt@\edef\@ngle{\repdecn@mb\v@leur}\figptell-14::#1;#2,#3(\@ngle,#4)%
    \advance\v@leur\delt@\edef\@ngle{\repdecn@mb\v@leur}\figptell-15::#1;#2,#3(\@ngle,#4)%
    \advance\v@leur\delt@\edef\@ngle{\repdecn@mb\v@leur}\figptell-16::#1;#2,#3(\@ngle,#4)%
    \figptscontrolDD-18[-13,-14,-15,-16]%
    \PSwrit@cmd{-18}{}{\fwf@g}\PSwrit@cmd{-17}{}{\fwf@g}%
    \PSwrit@cmd{-16}{\c@mcurveto}{\fwf@g}%
    \figptcopyDD-13:/-16/\bcl@rellBz#1;#2,#3(#4)\fi}
\ctr@ld@f\def\Ps@rell#1;#2,#3(#4,#5,#6){\ellBB@x#1;#2,#3(#4,#5,#6)%
    \f@gnewpath{\v@lmin=#2\unit@\v@lmin=\ptT@ptps\v@lmin%
    \v@lmax=#3\unit@\v@lmax=\ptT@ptps\v@lmax\BdingB@xfalse%
    \PSwrit@cmd{#1}%
    {#6\space\repdecn@mb{\v@lmin}\space\repdecn@mb{\v@lmax}\space #4\space #5\space ellipse}{\fwf@g}}%
    \global\Use@llipsetrue}
\ctr@ln@m\figdrawarcell
\ctr@ld@f\def\Q@arcellDD#1;#2,#3(#4,#5,#6){{\ifCUR@PS\ifGR@cri%
    \PSc@mment{arcellDD Center=#1 ; XRad=#2, YRad=#3 (Ang1=#4, Ang2=#5, Inclination=#6)}%
    \iffillm@de\Ps@rell#1;#2,#3(#4,#5,#6)%
    \f@gfill%
    \else\Ps@rell#1;#2,#3(#4,#5,#6)\f@gstroke\fi%
    \PSc@mment{End arcellDD}\fi\fi}}
\ctr@ld@f\def\Q@arcellTD#1;#2,#3(#4,#5,#6){{\ifCUR@PS\ifGR@cri\s@uvc@ntr@l\et@tpsarcellTD%
    \PSc@mment{arcellTD Center=#1 ; XRad=#2, YRad=#3 (Ang1=#4, Ang2=#5, Inclination=#6)}%
    \setc@ntr@l{2}\figpttraC -8:=#1/#2,0,0/\figpttraC -7:=#1/0,#3,0/%
    \figvectC -4(0,0,1)\figptsrot -8=-8,-7/#1,#6,-4/\Q@arcellPATD#1,-8,-7(#4,#5)%
    \PSc@mment{End arcellTD}\resetc@ntr@l\et@tpsarcellTD\fi\fi}}
\ctr@ln@m\figdrawarcellPA
\ctr@ld@f\def\Q@arcellPADD#1,#2,#3(#4,#5){{\ifCUR@PS\ifGR@cri\s@uvc@ntr@l\et@tpsarcellPADD%
    \PSc@mment{arcellPADD Center=#1,PtAxis1=#2,PtAxis2=#3 (Ang1=#4, Ang2=#5)}%
    \setc@ntr@l{2}\figvectPDD-1[#1,#2]\vecunit@DD{-1}{-1}\v@lX=\ptT@unit@\result@t%
    \edef\XR@d{\repdecn@mb{\v@lX}}\Figg@tXY{-1}\arct@n\v@lmin(\v@lX,\v@lY)%
    \v@lmin=\rdT@deg\v@lmin\edef\Inclin@{\repdecn@mb{\v@lmin}}%
    \figgetdist\YR@d[#1,#3]\Q@arcellDD#1;\XR@d,\YR@d(#4,#5,\Inclin@)%
    \PSc@mment{End arcellPADD}\resetc@ntr@l\et@tpsarcellPADD\fi\fi}}
\ctr@ld@f\def\Q@arcellPATD#1,#2,#3(#4,#5){{\ifCUR@PS\ifGR@cri\s@uvc@ntr@l\et@tpsarcellPATD%
    \PSc@mment{arcellPATD Center=#1,PtAxis1=#2,PtAxis2=#3 (Ang1=#4, Ang2=#5)}%
    \iffillm@de\Ps@rellPATD#1,#2,#3(#4,#5)%
    \f@gfill%
    \else\Ps@rellPATD#1,#2,#3(#4,#5)\f@gstroke\fi%
    \PSc@mment{End arcellPATD}\resetc@ntr@l\et@tpsarcellPATD\fi\fi}}
\ctr@ld@f\def\Ps@rellPATD#1,#2,#3(#4,#5){\let\c@lprojSP=\relax%
    \setc@ntr@l{2}\figvectPTD-1[#1,#2]\figvectPTD-2[#1,#3]\c@lNbarcs{#4}{#5}%
    \v@leur=#4pt\c@lptellP{#1}{-1}{-2}\Figptpr@j-5:/-3/%
    \f@gnewpath\PSwrit@cmdS{-5}{\c@mmoveto}{\fwf@g}{\X@un}{\Y@un}%
    \edef\C@nt@r{#1}\s@mme=\z@\bcl@rellPATD}
\ctr@ld@f\def\bcl@rellPATD{\relax%
    \ifnum\s@mme<\p@rtent\advance\s@mme\@ne%
    \advance\v@leur\delt@\c@lptellP{\C@nt@r}{-1}{-2}\Figptpr@j-4:/-3/%
    \advance\v@leur\delt@\c@lptellP{\C@nt@r}{-1}{-2}\Figptpr@j-6:/-3/%
    \advance\v@leur\delt@\c@lptellP{\C@nt@r}{-1}{-2}\Figptpr@j-3:/-3/%
    \v@lX=\z@\v@lY=\z@\Figtr@nptDD{-5}{-5}\Figtr@nptDD{2}{-3}%
    \divide\v@lX\@vi\divide\v@lY\@vi%
    \Figtr@nptDD{3}{-4}\Figtr@nptDD{-1.5}{-6}\v@lmin=\v@lX\v@lmax=\v@lY%
    \v@lX=\z@\v@lY=\z@\Figtr@nptDD{2}{-5}\Figtr@nptDD{-5}{-3}%
    \divide\v@lX\@vi\divide\v@lY\@vi\Figtr@nptDD{-1.5}{-4}\Figtr@nptDD{3}{-6}%
    \BdingB@xfalse%
    \Figp@intregDD-4:(\v@lmin,\v@lmax)\PSwrit@cmdS{-4}{}{\fwf@g}{\X@de}{\Y@de}%
    \Figp@intregDD-4:(\v@lX,\v@lY)\PSwrit@cmdS{-4}{}{\fwf@g}{\X@tr}{\Y@tr}%
    \BdingB@xtrue\PSwrit@cmdS{-3}{\c@mcurveto}{\fwf@g}{\X@qu}{\Y@qu}%
    \B@zierBB@x{1}{\Y@un}(\X@un,\X@de,\X@tr,\X@qu)%
    \B@zierBB@x{2}{\X@un}(\Y@un,\Y@de,\Y@tr,\Y@qu)%
    \edef\X@un{\X@qu}\edef\Y@un{\Y@qu}\figptcopyDD-5:/-3/\bcl@rellPATD\fi}
\ctr@ld@f\def\c@lNbarcs#1#2{%
    \delt@=#2pt\advance\delt@-#1pt\maxim@m{\v@lmax}{\delt@}{-\delt@}%
    \v@leur=\v@lmax\divide\v@leur45 \p@rtentiere{\p@rtent}{\v@leur}\advance\p@rtent\@ne%
    \s@mme=\p@rtent\multiply\s@mme\thr@@\divide\delt@\s@mme}
\ctr@ld@f\def\figdrawarcellPP#1,#2,#3[#4,#5]{{\ifCUR@PS\ifGR@cri\s@uvc@ntr@l\et@tpsarcellPP%
    \PSc@mment{arcellPP Center=#1,PtAxis1=#2,PtAxis2=#3 [Point1=#4, Point2=#5]}%
    \setc@ntr@l{2}\figvectP-2[#1,#3]\vecunit@{-2}{-2}\v@lmin=\result@t%
    \invers@{\v@lmax}{\v@lmin}%
    \figvectP-1[#1,#2]\vecunit@{-1}{-1}\v@leur=\result@t%
    \v@leur=\repdecn@mb{\v@lmax}\v@leur\edef\AsB@{\repdecn@mb{\v@leur}}
    \c@lAngle{#1}{#4}{\v@lmin}\edef\@ngdeb{\repdecn@mb{\v@lmin}}%
    \c@lAngle{#1}{#5}{\v@lmax}\ifdim\v@lmin>\v@lmax\advance\v@lmax\DePI@deg\fi%
    \edef\@ngfin{\repdecn@mb{\v@lmax}}\figdrawarcellPA#1,#2,#3(\@ngdeb,\@ngfin)%
    \PSc@mment{End arcellPP}\resetc@ntr@l\et@tpsarcellPP\fi\fi}}
\ctr@ld@f\def\c@lAngle#1#2#3{\figvectP-3[#1,#2]%
    \c@lproscal\delt@[-3,-1]\c@lproscal\v@leur[-3,-2]%
    \v@leur=\AsB@\v@leur\arct@n#3(\delt@,\v@leur)#3=\rdT@deg#3}
\ctr@ln@w{newif}\if@rrowratio\@rrowratiotrue
\ctr@ln@w{newif}\if@rrowhfill
\ctr@ln@w{newif}\if@rrowhout
\ctr@ld@f\def\Psset@rrowhe@d#1=#2|{\keln@mun#1|%
    \def\n@mref{a}\ifx\l@debut\n@mref\update@ttr\D@FTarrowheadangle\Q@s@tarrowheadangle{#2}\else
    \def\n@mref{f}\ifx\l@debut\n@mref\update@ttr\D@FTarrowheadfill\Q@s@tarrowheadfill{#2}\else
    \def\n@mref{l}\ifx\l@debut\n@mref\update@ttr\D@FTarrowheadlength\Q@s@tarrowheadlength{#2}\else
    \def\n@mref{o}\ifx\l@debut\n@mref\update@ttr\D@FTarrowheadout\Q@s@tarrowheadout{#2}\else
    \def\n@mref{r}\ifx\l@debut\n@mref\update@ttr\D@FTarrowheadratio\Q@s@tarrowheadratio{#2}\else
    \W@rnmesAttr{figset arrowhead}{#1}\fi\fi\fi\fi\fi}
\ctr@ln@m\@rrowheadangle
\ctr@ln@m\C@AHANG \ctr@ln@m\S@AHANG \ctr@ln@m\UNSS@N
\ctr@ld@f\def\Q@s@tarrowheadangle#1{\edef\@rrowheadangle{#1}{\c@ssin{\C@}{\S@}{#1}%
    \xdef\C@AHANG{\C@}\xdef\S@AHANG{\S@}\v@lmax=\S@ pt%
    \invers@{\v@leur}{\v@lmax}\maxim@m{\v@leur}{\v@leur}{-\v@leur}%
    \xdef\UNSS@N{\the\v@leur}}}
\ctr@ld@f\def\Q@s@tarrowheadfill#1{\expandafter\set@rrowhfill#1:}
\ctr@ld@f\def\set@rrowhfill#1#2:{\if#1n\@rrowhfillfalse\else\@rrowhfilltrue\fi}
\ctr@ld@f\def\Q@s@tarrowheadout#1{\expandafter\set@rrowhout#1:}
\ctr@ld@f\def\set@rrowhout#1#2:{\if#1n\@rrowhoutfalse\else\@rrowhouttrue\fi}
\ctr@ln@m\@rrowheadlength
\ctr@ld@f\def\Q@s@tarrowheadlength#1{\edef\@rrowheadlength{#1}\@rrowratiofalse}
\ctr@ln@m\@rrowheadratio
\ctr@ld@f\def\Q@s@tarrowheadratio#1{\edef\@rrowheadratio{#1}\@rrowratiotrue}
\ctr@ln@m\D@FTarrowheadlength
\ctr@ld@f\def\figresetarrowhead{%
    \Q@s@tarrowheadangle{\D@FTarrowheadangle}%
    \Q@s@tarrowheadfill{\D@FTarrowheadfill}%
    \Q@s@tarrowheadout{\D@FTarrowheadout}%
    \Q@s@tarrowheadratio{\D@FTarrowheadratio}%
    \d@fm@cdim\D@FTarrowheadlength{\D@FTh@rdahlength}
    \Q@s@tarrowheadlength{\D@FTarrowheadlength}}
\ctr@ld@f\def\D@FTarrowheadratio{0.1}
\ctr@ld@f\def\D@FTarrowheadangle{20}
\ctr@ld@f\def\D@FTarrowheadfill{no}
\ctr@ld@f\def\D@FTarrowheadout{no}
\ctr@ld@f\def\D@FTh@rdahlength{8pt}
\ctr@ln@m\figdrawarrow
\ctr@ld@f\def\Q@arrowDD[#1,#2]{{\ifCUR@PS\ifGR@cri\s@uvc@ntr@l\et@tpsarrow%
    \PSc@mment{arrowDD [Pt1,Pt2]=[#1,#2]}\Q@s@tfillmode{no}%
    \Q@arrowheadDD[#1,#2]\setc@ntr@l{2}\figdrawline[#1,-3]%
    \PSc@mment{End arrowDD}\resetc@ntr@l\et@tpsarrow\fi\fi}}
\ctr@ld@f\def\Q@arrowTD[#1,#2]{{\ifCUR@PS\ifGR@cri\s@uvc@ntr@l\et@tpsarrowTD%
    \PSc@mment{arrowTD [Pt1,Pt2]=[#1,#2]}\resetc@ntr@l{2}%
    \Figptpr@j-5:/#1/\Figptpr@j-6:/#2/\let\c@lprojSP=\relax\Q@arrowDD[-5,-6]%
    \PSc@mment{End arrowTD}\resetc@ntr@l\et@tpsarrowTD\fi\fi}}
\ctr@ln@m\figdrawarrowhead
\ctr@ld@f\def\Q@arrowheadDD[#1,#2]{{\ifCUR@PS\ifGR@cri\s@uvc@ntr@l\et@tpsarrowheadDD%
    \if@rrowhfill\def\@hangle{-\@rrowheadangle}\else\def\@hangle{\@rrowheadangle}\fi%
    \if@rrowratio%
    \if@rrowhout\def\@hratio{-\@rrowheadratio}\else\def\@hratio{\@rrowheadratio}\fi%
    \PSc@mment{arrowheadDD Ratio=\@hratio, Angle=\@hangle, [Pt1,Pt2]=[#1,#2]}%
    \Ps@rrowhead\@hratio,\@hangle[#1,#2]%
    \else%
    \if@rrowhout\def\@hlength{-\@rrowheadlength}\else\def\@hlength{\@rrowheadlength}\fi%
    \PSc@mment{arrowheadDD Length=\@hlength, Angle=\@hangle, [Pt1,Pt2]=[#1,#2]}%
    \Ps@rrowheadfd\@hlength,\@hangle[#1,#2]%
    \fi%
    \PSc@mment{End arrowheadDD}\resetc@ntr@l\et@tpsarrowheadDD\fi\fi}}
\ctr@ld@f\def\Q@arrowheadTD[#1,#2]{{\ifCUR@PS\ifGR@cri\s@uvc@ntr@l\et@tpsarrowheadTD%
    \PSc@mment{arrowheadTD [Pt1,Pt2]=[#1,#2]}\resetc@ntr@l{2}%
    \Figptpr@j-5:/#1/\Figptpr@j-6:/#2/\let\c@lprojSP=\relax\Q@arrowheadDD[-5,-6]%
    \PSc@mment{End arrowheadTD}\resetc@ntr@l\et@tpsarrowheadTD\fi\fi}}
\ctr@ld@f\def\Ps@rrowhead#1,#2[#3,#4]{\v@leur=#1\p@\maxim@m{\v@leur}{\v@leur}{-\v@leur}%
    \ifdim\v@leur>\Cepsil@n{
    \PSc@mment{@rrowhead Ratio=#1, Angle=#2, [Pt1,Pt2]=[#3,#4]}\v@leur=\UNSS@N%
    \v@leur=\CUR@width\v@leur\v@leur=\ptpsT@pt\v@leur\delt@=.5\v@leur
    \setc@ntr@l{2}\figvectPDD-3[#4,#3]%
    \Figg@tXY{-3}\v@lX=#1\v@lX\v@lY=#1\v@lY\Figv@ctCreg-3(\v@lX,\v@lY)%
    \vecunit@{-4}{-3}\mili@u=\result@t%
    \ifdim#2pt>\z@\v@lXa=-\C@AHANG\delt@%
     \edef\c@ef{\repdecn@mb{\v@lXa}}\figpttraDD-3:=-3/\c@ef,-4/\fi%
    \edef\c@ef{\repdecn@mb{\delt@}}%
    \v@lXa=\mili@u\v@lXa=\C@AHANG\v@lXa%
    \v@lYa=\ptpsT@pt\p@\v@lYa=\CUR@width\v@lYa\v@lYa=\sDcc@ngle\v@lYa%
    \advance\v@lXa-\v@lYa\gdef\sDcc@ngle{0}%
    \ifdim\v@lXa>\v@leur\edef\c@efendpt{\repdecn@mb{\v@leur}}%
    \else\edef\c@efendpt{\repdecn@mb{\v@lXa}}\fi%
    \Figg@tXY{-3}\v@lmin=\v@lX\v@lmax=\v@lY%
    \v@lXa=\C@AHANG\v@lmin\v@lYa=\S@AHANG\v@lmax\advance\v@lXa\v@lYa%
    \v@lYa=-\S@AHANG\v@lmin\v@lX=\C@AHANG\v@lmax\advance\v@lYa\v@lX%
    \setc@ntr@l{1}\Figg@tXY{#4}\advance\v@lX\v@lXa\advance\v@lY\v@lYa%
    \setc@ntr@l{2}\Figp@intregDD-2:(\v@lX,\v@lY)%
    \v@lXa=\C@AHANG\v@lmin\v@lYa=-\S@AHANG\v@lmax\advance\v@lXa\v@lYa%
    \v@lYa=\S@AHANG\v@lmin\v@lX=\C@AHANG\v@lmax\advance\v@lYa\v@lX%
    \setc@ntr@l{1}\Figg@tXY{#4}\advance\v@lX\v@lXa\advance\v@lY\v@lYa%
    \setc@ntr@l{2}\Figp@intregDD-1:(\v@lX,\v@lY)%
    \ifdim#2pt<\z@\fillm@detrue\figdrawline[-2,#4,-1]
    \else\figptstraDD-3=#4,-2,-1/\c@ef,-4/\s@uvdash{\typ@dash}\Q@s@tdash{\D@FTdash}%
    \figdrawline[-2,-3,-1]\Q@s@tdash{\typ@dash}\fi
    \ifdim#1pt>\z@\figpttraDD-3:=#4/\c@efendpt,-4/\else\figptcopyDD-3:/#4/\fi%
    \PSc@mment{End @rrowhead}}\fi}
\ctr@ld@f\def\sDcc@ngle{0}
\ctr@ld@f\def\Ps@rrowheadfd#1,#2[#3,#4]{{%
    \PSc@mment{@rrowheadfd Length=#1, Angle=#2, [Pt1,Pt2]=[#3,#4]}%
    \setc@ntr@l{2}\figvectPDD-1[#3,#4]\n@rmeucDD{\v@leur}{-1}\v@leur=\ptT@unit@\v@leur%
    \invers@{\v@leur}{\v@leur}\v@leur=#1\v@leur\edef\R@tio{\repdecn@mb{\v@leur}}%
    \Ps@rrowhead\R@tio,#2[#3,#4]\PSc@mment{End @rrowheadfd}}}
\ctr@ln@m\figdrawarrowBezier
\ctr@ld@f\def\Q@arrowBezierDD[#1,#2,#3,#4]{{\ifCUR@PS\ifGR@cri\s@uvc@ntr@l\et@tpsarrowBezierDD%
    \PSc@mment{arrowBezierDD Control points=#1,#2,#3,#4}\setc@ntr@l{2}%
    \if@rrowratio\c@larclengthDD\v@leur,10[#1,#2,#3,#4]\else\v@leur=\z@\fi%
    \Ps@rrowB@zDD\v@leur[#1,#2,#3,#4]%
    \PSc@mment{End arrowBezierDD}\resetc@ntr@l\et@tpsarrowBezierDD\fi\fi}}
\ctr@ld@f\def\Q@arrowBezierTD[#1,#2,#3,#4]{{\ifCUR@PS\ifGR@cri\s@uvc@ntr@l\et@tpsarrowBezierTD%
    \PSc@mment{arrowBezierTD Control points=#1,#2,#3,#4}\resetc@ntr@l{2}%
    \Figptpr@j-7:/#1/\Figptpr@j-8:/#2/\Figptpr@j-9:/#3/\Figptpr@j-10:/#4/%
    \let\c@lprojSP=\relax\ifnum\CUR@proj<\tw@\Q@arrowBezierDD[-7,-8,-9,-10]%
    \else\f@gnewpath\PSwrit@cmd{-7}{\c@mmoveto}{\fwf@g}%
    \if@rrowratio\c@larclengthDD\mili@u,10[-7,-8,-9,-10]\else\mili@u=\z@\fi%
    \p@rtent=\NBz@rcs\advance\p@rtent\m@ne\subB@zierTD\p@rtent[#1,#2,#3,#4]%
    \f@gstroke%
    \advance\v@lmin\p@rtent\delt@
    \v@leur=\v@lmin\advance\v@leur0.33333 \delt@\edef\unti@rs{\repdecn@mb{\v@leur}}%
    \v@leur=\v@lmin\advance\v@leur0.66666 \delt@\edef\deti@rs{\repdecn@mb{\v@leur}}%
    \figptcopyDD-8:/-10/\c@lsubBzarc\unti@rs,\deti@rs[#1,#2,#3,#4]%
    \figptcopyDD-8:/-4/\figptcopyDD-9:/-3/\Ps@rrowB@zDD\mili@u[-7,-8,-9,-10]\fi%
    \PSc@mment{End arrowBezierTD}\resetc@ntr@l\et@tpsarrowBezierTD\fi\fi}}
\ctr@ld@f\def\c@larclengthDD#1,#2[#3,#4,#5,#6]{{\p@rtent=#2\figptcopyDD-5:/#3/%
    \delt@=\p@\divide\delt@\p@rtent\c@rre=\z@\v@leur=\z@\s@mme=\z@%
    \loop\ifnum\s@mme<\p@rtent\advance\s@mme\@ne\advance\v@leur\delt@%
    \edef\T@{\repdecn@mb{\v@leur}}\figptBezierDD-6::\T@[#3,#4,#5,#6]%
    \figvectPDD-1[-5,-6]\n@rmeucDD{\mili@u}{-1}\advance\c@rre\mili@u%
    \figptcopyDD-5:/-6/\repeat\global\result@t=\ptT@unit@\c@rre}#1=\result@t}
\ctr@ld@f\def\Ps@rrowB@zDD#1[#2,#3,#4,#5]{{\Q@s@tfillmode{no}%
    \if@rrowratio\delt@=\@rrowheadratio#1\else\delt@=\@rrowheadlength pt\fi%
    \v@leur=\C@AHANG\delt@\edef\R@dius{\repdecn@mb{\v@leur}}%
    \FigptintercircB@zDD-5::0,\R@dius[#5,#4,#3,#2]%
    \Q@s@tarrowheadlength{\repdecn@mb{\delt@}}\Q@arrowheadDD[-5,#5]%
    \let\n@rmeuc=\n@rmeucDD\figgetdist\R@dius[#5,-3]%
    \FigptintercircB@zDD-6::0,\R@dius[#5,#4,#3,#2]%
    \figptBezierDD-5::0.33333[#5,#4,#3,#2]\figptBezierDD-3::0.66666[#5,#4,#3,#2]%
    \figptscontrolDD-5[-6,-5,-3,#2]\Q@BezierDD1[-6,-5,-4,#2]}}
\ctr@ln@m\figdrawarrowcirc
\ctr@ld@f\def\Q@arrowcircDD#1;#2(#3,#4){{\ifCUR@PS\ifGR@cri\s@uvc@ntr@l\et@tpsarrowcircDD%
    \PSc@mment{arrowcircDD Center=#1 ; Radius=#2 (Ang1=#3,Ang2=#4)}%
    \Q@s@tfillmode{no}\Pscirc@rrowhead#1;#2(#3,#4)%
    \setc@ntr@l{2}\figvectPDD -4[#1,-3]\vecunit@{-4}{-4}%
    \Figg@tXY{-4}\arct@n\v@lmin(\v@lX,\v@lY)%
    \v@lmin=\rdT@deg\v@lmin\v@leur=#4pt\advance\v@leur-\v@lmin%
    \maxim@m{\v@leur}{\v@leur}{-\v@leur}%
    \ifdim\v@leur>\DemiPI@deg\relax\ifdim\v@lmin<#4pt\advance\v@lmin\DePI@deg%
    \else\advance\v@lmin-\DePI@deg\fi\fi\edef\ar@ngle{\repdecn@mb{\v@lmin}}%
    \ifdim#3pt<#4pt\figdrawarccirc#1;#2(#3,\ar@ngle)\else\figdrawarccirc#1;#2(\ar@ngle,#3)\fi%
    \PSc@mment{End arrowcircDD}\resetc@ntr@l\et@tpsarrowcircDD\fi\fi}}
\ctr@ld@f\def\Q@arrowcircTD#1,#2,#3;#4(#5,#6){{\ifCUR@PS\ifGR@cri\s@uvc@ntr@l\et@tpsarrowcircTD%
    \PSc@mment{arrowcircTD Center=#1,P1=#2,P2=#3 ; Radius=#4 (Ang1=#5, Ang2=#6)}%
    \resetc@ntr@l{2}\c@lExtAxes#1,#2,#3(#4)\let\c@lprojSP=\relax%
    \figvectPTD-11[#1,-4]\figvectPTD-12[#1,-5]\c@lNbarcs{#5}{#6}%
    \if@rrowratio\v@lmax=\degT@rd\v@lmax\edef\D@lpha{\repdecn@mb{\v@lmax}}\fi%
    \advance\p@rtent\m@ne\mili@u=\z@%
    \v@leur=#5pt\c@lptellP{#1}{-11}{-12}\Figptpr@j-9:/-3/%
    \f@gnewpath\PSwrit@cmdS{-9}{\c@mmoveto}{\fwf@g}{\X@un}{\Y@un}%
    \edef\C@nt@r{#1}\s@mme=\z@\bcl@rcircTD\f@gstroke%
    \advance\v@leur\delt@\c@lptellP{#1}{-11}{-12}\Figptpr@j-5:/-3/%
    \advance\v@leur\delt@\c@lptellP{#1}{-11}{-12}\Figptpr@j-6:/-3/%
    \advance\v@leur\delt@\c@lptellP{#1}{-11}{-12}\Figptpr@j-10:/-3/%
    \figptscontrolDD-8[-9,-5,-6,-10]%
    \if@rrowratio\c@lcurvradDD0.5[-9,-8,-7,-10]\advance\mili@u\result@t%
    \maxim@m{\mili@u}{\mili@u}{-\mili@u}\mili@u=\ptT@unit@\mili@u%
    \mili@u=\D@lpha\mili@u\advance\p@rtent\@ne\divide\mili@u\p@rtent\fi%
    \Ps@rrowB@zDD\mili@u[-9,-8,-7,-10]%
    \PSc@mment{End arrowcircTD}\resetc@ntr@l\et@tpsarrowcircTD\fi\fi}}
\ctr@ld@f\def\bcl@rcircTD{\relax%
    \ifnum\s@mme<\p@rtent\advance\s@mme\@ne%
    \advance\v@leur\delt@\c@lptellP{\C@nt@r}{-11}{-12}\Figptpr@j-5:/-3/%
    \advance\v@leur\delt@\c@lptellP{\C@nt@r}{-11}{-12}\Figptpr@j-6:/-3/%
    \advance\v@leur\delt@\c@lptellP{\C@nt@r}{-11}{-12}\Figptpr@j-10:/-3/%
    \figptscontrolDD-8[-9,-5,-6,-10]\BdingB@xfalse%
    \PSwrit@cmdS{-8}{}{\fwf@g}{\X@de}{\Y@de}\PSwrit@cmdS{-7}{}{\fwf@g}{\X@tr}{\Y@tr}%
    \BdingB@xtrue\PSwrit@cmdS{-10}{\c@mcurveto}{\fwf@g}{\X@qu}{\Y@qu}%
    \if@rrowratio\c@lcurvradDD0.5[-9,-8,-7,-10]\advance\mili@u\result@t\fi%
    \B@zierBB@x{1}{\Y@un}(\X@un,\X@de,\X@tr,\X@qu)%
    \B@zierBB@x{2}{\X@un}(\Y@un,\Y@de,\Y@tr,\Y@qu)%
    \edef\X@un{\X@qu}\edef\Y@un{\Y@qu}\figptcopyDD-9:/-10/\bcl@rcircTD\fi}
\ctr@ld@f\def\Pscirc@rrowhead#1;#2(#3,#4){{%
    \PSc@mment{circ@rrowhead Center=#1 ; Radius=#2 (Ang1=#3,Ang2=#4)}%
    \v@leur=#2\unit@\edef\s@glen{\repdecn@mb{\v@leur}}\v@lY=\z@\v@lX=\v@leur%
    \resetc@ntr@l{2}\Figv@ctCreg-3(\v@lX,\v@lY)\figpttraDD-5:=#1/1,-3/%
    \figptrotDD-5:=-5/#1,#4/%
    \figvectPDD-3[#1,-5]\Figg@tXY{-3}\v@leur=\v@lX%
    \ifdim#3pt<#4pt\v@lX=\v@lY\v@lY=-\v@leur\else\v@lX=-\v@lY\v@lY=\v@leur\fi%
    \Figv@ctCreg-3(\v@lX,\v@lY)\vecunit@{-3}{-3}%
    \if@rrowratio\v@leur=#4pt\advance\v@leur-#3pt\maxim@m{\mili@u}{-\v@leur}{\v@leur}%
    \mili@u=\degT@rd\mili@u\v@leur=\s@glen\mili@u\edef\s@glen{\repdecn@mb{\v@leur}}%
    \mili@u=#2\mili@u\mili@u=\@rrowheadratio\mili@u\else\mili@u=\@rrowheadlength pt\fi%
    \figpttraDD-6:=-5/\s@glen,-3/\v@leur=#2pt\v@leur=2\v@leur%
    \invers@{\v@leur}{\v@leur}\c@rre=\repdecn@mb{\v@leur}\mili@u
    \mili@u=\c@rre\mili@u=\repdecn@mb{\c@rre}\mili@u%
    \v@leur=\p@\advance\v@leur-\mili@u
    \invers@{\mili@u}{2\v@leur}\delt@=\c@rre\delt@=\repdecn@mb{\mili@u}\delt@%
    \xdef\sDcc@ngle{\repdecn@mb{\delt@}}
    \sqrt@{\mili@u}{\v@leur}\arct@n\v@leur(\mili@u,\c@rre)%
    \v@leur=\rdT@deg\v@leur
    \ifdim#3pt<#4pt\v@leur=-\v@leur\fi%
    \if@rrowhout\v@leur=-\v@leur\fi\edef\cor@ngle{\repdecn@mb{\v@leur}}%
    \figptrotDD-6:=-6/-5,\cor@ngle/\Q@arrowheadDD[-6,-5]%
    \PSc@mment{End circ@rrowhead}}}
\ctr@ln@m\figdrawarrowcircP
\ctr@ld@f\def\Q@arrowcircPDD#1;#2[#3,#4]{{\ifCUR@PS\ifGR@cri%
    \PSc@mment{arrowcircPDD Center=#1; Radius=#2, [P1=#3,P2=#4]}%
    \s@uvc@ntr@l\et@tpsarrowcircPDD\Ps@ngleparam#1;#2[#3,#4]%
    \ifdim\v@leur>\z@\ifdim\v@lmin>\v@lmax\advance\v@lmax\DePI@deg\fi%
    \else\ifdim\v@lmin<\v@lmax\advance\v@lmin\DePI@deg\fi\fi%
    \edef\@ngdeb{\repdecn@mb{\v@lmin}}\edef\@ngfin{\repdecn@mb{\v@lmax}}%
    \figdrawarrowcirc#1;\r@dius(\@ngdeb,\@ngfin)%
    \PSc@mment{End arrowcircPDD}\resetc@ntr@l\et@tpsarrowcircPDD\fi\fi}}
\ctr@ld@f\def\Q@arrowcircPTD#1;#2[#3,#4,#5]{{\ifCUR@PS\ifGR@cri\s@uvc@ntr@l\et@tpsarrowcircPTD%
    \PSc@mment{arrowcircPTD Center=#1; Radius=#2, [P1=#3,P2=#4,P3=#5]}%
    \figgetangleTD\@ngfin[#1,#3,#4,#5]\v@leur=#2pt%
    \maxim@m{\mili@u}{-\v@leur}{\v@leur}\edef\r@dius{\repdecn@mb{\mili@u}}%
    \ifdim\v@leur<\z@\v@lmax=\@ngfin pt\advance\v@lmax-\DePI@deg%
    \edef\@ngfin{\repdecn@mb{\v@lmax}}\fi\Q@arrowcircTD#1,#3,#5;\r@dius(0,\@ngfin)%
    \PSc@mment{End arrowcircPTD}\resetc@ntr@l\et@tpsarrowcircPTD\fi\fi}}
\ctr@ld@f\def\figdrawaxes#1(#2){{\ifCUR@PS\ifGR@cri\s@uvc@ntr@l\et@tpsaxes%
    \PSc@mment{axes Origin=#1 Range=(#2)}\an@lys@xes#2,:\resetc@ntr@l{2}%
    \ifx\t@xt@\empty\ifTr@isDim\Q@@xes#1(0,#2,0,#2,0,#2)\else\Q@@xes#1(0,#2,0,#2)\fi%
    \else\Q@@xes#1(#2)\fi\PSc@mment{End axes}\resetc@ntr@l\et@tpsaxes\fi\fi}}
\ctr@ld@f\def\an@lys@xes#1,#2:{\def\t@xt@{#2}}
\ctr@ln@m\Q@@xes
\ctr@ld@f\def\Q@@xesDD#1(#2,#3,#4,#5){%
    \figpttraC-5:=#1/#2,0/\figpttraC-6:=#1/#3,0/\Q@arrowDD[-5,-6]%
    \figpttraC-5:=#1/0,#4/\figpttraC-6:=#1/0,#5/\Q@arrowDD[-5,-6]}
\ctr@ld@f\def\Q@@xesTD#1(#2,#3,#4,#5,#6,#7){%
    \figpttraC-7:=#1/#2,0,0/\figpttraC-8:=#1/#3,0,0/\Q@arrowTD[-7,-8]%
    \figpttraC-7:=#1/0,#4,0/\figpttraC-8:=#1/0,#5,0/\Q@arrowTD[-7,-8]%
    \figpttraC-7:=#1/0,0,#6/\figpttraC-8:=#1/0,0,#7/\Q@arrowTD[-7,-8]}
\ctr@ln@m\newGr@FN
\ctr@ld@f\def\newGr@FNPDF#1{\s@mme=\Gr@FNb\advance\s@mme\@ne\xdef\Gr@FNb{\number\s@mme}}
\ctr@ld@f\def\newGr@FNDVI#1{\newGr@FNPDF{}\xdef#1{\jobname GI\Gr@FNb.anx}}
\ctr@ld@f\def\figdrawbegin#1{\newGr@FN\DefGIfilen@me\gdef\@utoFN{0}%
    \def\t@xt@{#1}\relax\ifx\t@xt@\empty\GRupdatem@detrue%
    \gdef\@utoFN{1}\Psb@ginfig\DefGIfilen@me\else\expandafter\Psb@ginfigNu@#1 :\fi}
\ctr@ld@f\def\Psb@ginfigNu@#1 #2:{\def\t@xt@{#1}\relax\ifx\t@xt@\empty\def\t@xt@{#2}%
    \ifx\t@xt@\empty\GRupdatem@detrue\gdef\@utoFN{1}\Psb@ginfig\DefGIfilen@me%
    \else\Psb@ginfigNu@#2:\fi\else\Psb@ginfig{#1}\fi}
\ctr@ln@m\PSfilen@me \ctr@ln@m\auxfilen@me
\ctr@ld@f\def\Psb@ginfig#1{\ifCUR@PS\else%
    \edef\PSfilen@me{#1}\edef\auxfilen@me{\jobname.anx}%
    \ifGRupdatem@de\GR@critrue\else\openin\frf@g=\PSfilen@me\relax%
    \ifeof\frf@g\GR@critrue\else\GR@crifalse\fi\closein\frf@g\fi%
    \CUR@PStrue\c@ldefproj\expandafter\setupd@te\D@FTupdate:%
    \ifGR@cri\initb@undb@x%
    \immediate\openout\fwf@g=\auxfilen@me\initpss@ttings\fi%
    \fi}
\ctr@ld@f\def\Gr@FNb{0}
\ctr@ld@f\def\figforTeXFileno{\Gr@FNb}
\ctr@ld@f\def\figforTeXFigno{0 }
\ctr@ld@f\def\figforTeXnextFigno{1 }
\ctr@ld@f\edef\DefGIfilen@me{\jobname GI.anx}
\ctr@ld@f\def\initpss@ttings{\figreset{altitude,arrowhead,curve,general,flowchart,mesh,trimesh}%
    \Use@llipsefalse}
\ctr@ld@f\def\B@zierBB@x#1#2(#3,#4,#5,#6){{\c@rre=\t@n\epsil@n
    \v@lmax=#4\advance\v@lmax-#5\v@lmax=\thr@@\v@lmax\advance\v@lmax#6\advance\v@lmax-#3%
    \mili@u=#4\mili@u=-\tw@\mili@u\advance\mili@u#3\advance\mili@u#5%
    \v@lmin=#4\advance\v@lmin-#3\maxim@m{\v@leur}{-\v@lmax}{\v@lmax}%
    \maxim@m{\delt@}{-\mili@u}{\mili@u}\maxim@m{\v@leur}{\v@leur}{\delt@}%
    \maxim@m{\delt@}{-\v@lmin}{\v@lmin}\maxim@m{\v@leur}{\v@leur}{\delt@}%
    \ifdim\v@leur>\c@rre\invers@{\v@leur}{\v@leur}\edef\Uns@rM@x{\repdecn@mb{\v@leur}}%
    \v@lmax=\Uns@rM@x\v@lmax\mili@u=\Uns@rM@x\mili@u\v@lmin=\Uns@rM@x\v@lmin%
    \maxim@m{\v@leur}{-\v@lmax}{\v@lmax}\ifdim\v@leur<\c@rre%
    \maxim@m{\v@leur}{-\mili@u}{\mili@u}\ifdim\v@leur<\c@rre\else%
    \invers@{\mili@u}{\mili@u}\v@leur=-0.5\v@lmin%
    \v@leur=\repdecn@mb{\mili@u}\v@leur\m@jBBB@x{\v@leur}{#1}{#2}(#3,#4,#5,#6)\fi%
    \else\delt@=\repdecn@mb{\mili@u}\mili@u\v@leur=\repdecn@mb{\v@lmax}\v@lmin%
    \advance\delt@-\v@leur\ifdim\delt@<\z@\else\invers@{\v@lmax}{\v@lmax}%
    \edef\Uns@rAp{\repdecn@mb{\v@lmax}}\sqrt@{\delt@}{\delt@}%
    \v@leur=-\mili@u\advance\v@leur\delt@\v@leur=\Uns@rAp\v@leur%
    \m@jBBB@x{\v@leur}{#1}{#2}(#3,#4,#5,#6)%
    \v@leur=-\mili@u\advance\v@leur-\delt@\v@leur=\Uns@rAp\v@leur%
    \m@jBBB@x{\v@leur}{#1}{#2}(#3,#4,#5,#6)\fi\fi\fi}}
\ctr@ld@f\def\m@jBBB@x#1#2#3(#4,#5,#6,#7){{\relax\ifdim#1>\z@\ifdim#1<\p@%
    \edef\T@{\repdecn@mb{#1}}\v@lX=\p@\advance\v@lX-#1\edef\UNmT@{\repdecn@mb{\v@lX}}%
    \v@lX=#4\v@lY=#5\v@lZ=#6\v@lXa=#7\v@lX=\UNmT@\v@lX\advance\v@lX\T@\v@lY%
    \v@lY=\UNmT@\v@lY\advance\v@lY\T@\v@lZ\v@lZ=\UNmT@\v@lZ\advance\v@lZ\T@\v@lXa%
    \v@lX=\UNmT@\v@lX\advance\v@lX\T@\v@lY\v@lY=\UNmT@\v@lY\advance\v@lY\T@\v@lZ%
    \v@lX=\UNmT@\v@lX\advance\v@lX\T@\v@lY%
    \ifcase#2\or\v@lY=#3\or\v@lY=\v@lX\v@lX=#3\fi\b@undb@x{\v@lX}{\v@lY}\fi\fi}}
\ctr@ld@f\def\PsB@zier#1[#2]{{\f@gnewpath%
    \s@mme=\z@\def\list@num{#2,0}\extrairelepremi@r\p@int\de\list@num%
    \PSwrit@cmdS{\p@int}{\c@mmoveto}{\fwf@g}{\X@un}{\Y@un}\p@rtent=#1\bclB@zier}}
\ctr@ld@f\def\bclB@zier{\relax%
    \ifnum\s@mme<\p@rtent\advance\s@mme\@ne\BdingB@xfalse%
    \extrairelepremi@r\p@int\de\list@num\PSwrit@cmdS{\p@int}{}{\fwf@g}{\X@de}{\Y@de}%
    \extrairelepremi@r\p@int\de\list@num\PSwrit@cmdS{\p@int}{}{\fwf@g}{\X@tr}{\Y@tr}%
    \BdingB@xtrue%
    \extrairelepremi@r\p@int\de\list@num\PSwrit@cmdS{\p@int}{\c@mcurveto}{\fwf@g}{\X@qu}{\Y@qu}%
    \B@zierBB@x{1}{\Y@un}(\X@un,\X@de,\X@tr,\X@qu)%
    \B@zierBB@x{2}{\X@un}(\Y@un,\Y@de,\Y@tr,\Y@qu)%
    \edef\X@un{\X@qu}\edef\Y@un{\Y@qu}\bclB@zier\fi}
\ctr@ln@m\figdrawBezier
\ctr@ld@f\def\Q@BezierDD#1[#2]{\ifCUR@PS\ifGR@cri%
    \PSc@mment{BezierDD N arcs=#1, Control points=#2}%
    \iffillm@de\PsB@zier#1[#2]%
    \f@gfill%
    \else\PsB@zier#1[#2]\f@gstroke\fi%
    \PSc@mment{End BezierDD}\fi\fi}
\ctr@ln@m\et@tpsBezierTD
\ctr@ld@f\def\Q@BezierTD#1[#2]{\ifCUR@PS\ifGR@cri\s@uvc@ntr@l\et@tpsBezierTD%
    \PSc@mment{BezierTD N arcs=#1, Control points=#2}%
    \iffillm@de\PsB@zierTD#1[#2]%
    \f@gfill%
    \else\PsB@zierTD#1[#2]\f@gstroke\fi%
    \PSc@mment{End BezierTD}\resetc@ntr@l\et@tpsBezierTD\fi\fi}
\ctr@ld@f\def\PsB@zierTD#1[#2]{\ifnum\CUR@proj<\tw@\PsB@zier#1[#2]\else\PsB@zier@TD#1[#2]\fi}
\ctr@ld@f\def\PsB@zier@TD#1[#2]{{\f@gnewpath%
    \s@mme=\z@\def\list@num{#2,0}\extrairelepremi@r\p@int\de\list@num%
    \let\c@lprojSP=\relax\setc@ntr@l{2}\Figptpr@j-7:/\p@int/%
    \PSwrit@cmd{-7}{\c@mmoveto}{\fwf@g}%
    \loop\ifnum\s@mme<#1\advance\s@mme\@ne\extrairelepremi@r\p@intun\de\list@num%
    \extrairelepremi@r\p@intde\de\list@num\extrairelepremi@r\p@inttr\de\list@num%
    \subB@zierTD\NBz@rcs[\p@int,\p@intun,\p@intde,\p@inttr]\edef\p@int{\p@inttr}\repeat}}
\ctr@ld@f\def\subB@zierTD#1[#2,#3,#4,#5]{\delt@=\p@\divide\delt@\NBz@rcs\v@lmin=\z@%
    {\Figg@tXY{-7}\edef\X@un{\the\v@lX}\edef\Y@un{\the\v@lY}%
    \s@mme=\z@\loop\ifnum\s@mme<#1\advance\s@mme\@ne%
    \v@leur=\v@lmin\advance\v@leur0.33333 \delt@\edef\unti@rs{\repdecn@mb{\v@leur}}%
    \v@leur=\v@lmin\advance\v@leur0.66666 \delt@\edef\deti@rs{\repdecn@mb{\v@leur}}%
    \advance\v@lmin\delt@\edef\trti@rs{\repdecn@mb{\v@lmin}}%
    \figptBezierTD-8::\trti@rs[#2,#3,#4,#5]\Figptpr@j-8:/-8/%
    \c@lsubBzarc\unti@rs,\deti@rs[#2,#3,#4,#5]\BdingB@xfalse%
    \PSwrit@cmdS{-4}{}{\fwf@g}{\X@de}{\Y@de}\PSwrit@cmdS{-3}{}{\fwf@g}{\X@tr}{\Y@tr}%
    \BdingB@xtrue\PSwrit@cmdS{-8}{\c@mcurveto}{\fwf@g}{\X@qu}{\Y@qu}%
    \B@zierBB@x{1}{\Y@un}(\X@un,\X@de,\X@tr,\X@qu)%
    \B@zierBB@x{2}{\X@un}(\Y@un,\Y@de,\Y@tr,\Y@qu)%
    \edef\X@un{\X@qu}\edef\Y@un{\Y@qu}\figptcopyDD-7:/-8/\repeat}}
\ctr@ld@f\def\NBz@rcs{2}
\ctr@ld@f\def\c@lsubBzarc#1,#2[#3,#4,#5,#6]{\figptBezierTD-5::#1[#3,#4,#5,#6]%
    \figptBezierTD-6::#2[#3,#4,#5,#6]\Figptpr@j-4:/-5/\Figptpr@j-5:/-6/%
    \figptscontrolDD-4[-7,-4,-5,-8]}
\ctr@ln@m\figdrawcirc
\ctr@ld@f\def\Q@circDD#1(#2){\ifCUR@PS\ifGR@cri\PSc@mment{circDD Center=#1 (Radius=#2)}%
    \Q@arccircDD#1;#2(0,360)\PSc@mment{End circDD}\fi\fi}
\ctr@ld@f\def\Q@circTD#1,#2,#3(#4){\ifCUR@PS\ifGR@cri%
    \PSc@mment{circTD Center=#1,P1=#2,P2=#3 (Radius=#4)}%
    \Q@arccircTD#1,#2,#3;#4(0,360)\PSc@mment{End circTD}\fi\fi}
\ctr@ln@m\p@urcent
{\catcode`\%=12\gdef\p@urcent{
\ctr@ld@f\def\PSc@mment#1{\ifGRdebugm@de\immediate\write\fwf@g{\p@urcent\space#1}\fi}
\ctr@ln@m\acc@louv \ctr@ln@m\acc@lfer
{\catcode`\[=1\catcode`\{=12\gdef\acc@louv[{}}
{\catcode`\]=2\catcode`\}=12\gdef\acc@lfer{}]]
\ctr@ld@f\def\PSdict@{\ifUse@llipse%
    \immediate\write\fwf@g{/ellipsedict 9 dict def ellipsedict /mtrx matrix put}%
    \immediate\write\fwf@g{/ellipse \acc@louv ellipsedict begin}%
    \immediate\write\fwf@g{ /endangle exch def /startangle exch def}%
    \immediate\write\fwf@g{ /yrad exch def /xrad exch def}%
    \immediate\write\fwf@g{ /rotangle exch def /y exch def /x exch def}%
    \immediate\write\fwf@g{ /savematrix mtrx currentmatrix def}%
    \immediate\write\fwf@g{ x y translate rotangle rotate xrad yrad scale}%
    \immediate\write\fwf@g{ 0 0 1 startangle endangle arc}%
    \immediate\write\fwf@g{ savematrix setmatrix end\acc@lfer def}%
    \fi\PShe@der{EndProlog}}
\ctr@ld@f\def\Pssetc@rve#1=#2|{\keln@mun#1|%
    \def\n@mref{r}\ifx\l@debut\n@mref\update@ttr\D@FTroundness\Q@s@troundness{#2}\else
    \W@rnmesAttr{figset curve}{#1}\fi}
\ctr@ln@m\curv@roundness
\ctr@ld@f\def\Q@s@troundness#1{\edef\curv@roundness{#1}}
\ctr@ld@f\def\D@FTroundness{0.2} 
\ctr@ln@m\figdrawcurve
\ctr@ld@f\def\Q@curveDD[#1]{{\ifCUR@PS\ifGR@cri\PSc@mment{curveDD Points=#1}%
    \s@uvc@ntr@l\et@tpscurveDD%
    \iffillm@de\Psc@rveDD\curv@roundness[#1]%
    \f@gfill%
    \else\Psc@rveDD\curv@roundness[#1]\f@gstroke\fi%
    \PSc@mment{End curveDD}\resetc@ntr@l\et@tpscurveDD\fi\fi}}
\ctr@ld@f\def\Q@curveTD[#1]{{\ifCUR@PS\ifGR@cri%
    \PSc@mment{curveTD Points=#1}\s@uvc@ntr@l\et@tpscurveTD\let\c@lprojSP=\relax%
    \iffillm@de\Psc@rveTD\curv@roundness[#1]%
    \f@gfill%
    \else\Psc@rveTD\curv@roundness[#1]\f@gstroke\fi%
    \PSc@mment{End curveTD}\resetc@ntr@l\et@tpscurveTD\fi\fi}}
\ctr@ld@f\def\Psc@rveDD#1[#2]{%
    \def\list@num{#2}\extrairelepremi@r\Ak@\de\list@num%
    \extrairelepremi@r\Ai@\de\list@num\extrairelepremi@r\Aj@\de\list@num%
    \f@gnewpath\PSwrit@cmdS{\Ai@}{\c@mmoveto}{\fwf@g}{\X@un}{\Y@un}%
    \setc@ntr@l{2}\figvectPDD -1[\Ak@,\Aj@]%
    \@ecfor\Ak@:=\list@num\do{\figpttraDD-2:=\Ai@/#1,-1/\BdingB@xfalse%
       \PSwrit@cmdS{-2}{}{\fwf@g}{\X@de}{\Y@de}%
       \figvectPDD -1[\Ai@,\Ak@]\figpttraDD-2:=\Aj@/-#1,-1/%
       \PSwrit@cmdS{-2}{}{\fwf@g}{\X@tr}{\Y@tr}\BdingB@xtrue%
       \PSwrit@cmdS{\Aj@}{\c@mcurveto}{\fwf@g}{\X@qu}{\Y@qu}%
       \B@zierBB@x{1}{\Y@un}(\X@un,\X@de,\X@tr,\X@qu)%
       \B@zierBB@x{2}{\X@un}(\Y@un,\Y@de,\Y@tr,\Y@qu)%
       \edef\X@un{\X@qu}\edef\Y@un{\Y@qu}\edef\Ai@{\Aj@}\edef\Aj@{\Ak@}}}
\ctr@ld@f\def\Psc@rveTD#1[#2]{\ifnum\CUR@proj<\tw@\Psc@rvePPTD#1[#2]\else\Psc@rveCPTD#1[#2]\fi}
\ctr@ld@f\def\Psc@rvePPTD#1[#2]{\setc@ntr@l{2}%
    \def\list@num{#2}\extrairelepremi@r\Ak@\de\list@num\Figptpr@j-5:/\Ak@/%
    \extrairelepremi@r\Ai@\de\list@num\Figptpr@j-3:/\Ai@/%
    \extrairelepremi@r\Aj@\de\list@num\Figptpr@j-4:/\Aj@/%
    \f@gnewpath\PSwrit@cmdS{-3}{\c@mmoveto}{\fwf@g}{\X@un}{\Y@un}%
    \figvectPDD -1[-5,-4]%
    \@ecfor\Ak@:=\list@num\do{\Figptpr@j-5:/\Ak@/\figpttraDD-2:=-3/#1,-1/%
       \BdingB@xfalse\PSwrit@cmdS{-2}{}{\fwf@g}{\X@de}{\Y@de}%
       \figvectPDD -1[-3,-5]\figpttraDD-2:=-4/-#1,-1/%
       \PSwrit@cmdS{-2}{}{\fwf@g}{\X@tr}{\Y@tr}\BdingB@xtrue%
       \PSwrit@cmdS{-4}{\c@mcurveto}{\fwf@g}{\X@qu}{\Y@qu}%
       \B@zierBB@x{1}{\Y@un}(\X@un,\X@de,\X@tr,\X@qu)%
       \B@zierBB@x{2}{\X@un}(\Y@un,\Y@de,\Y@tr,\Y@qu)%
       \edef\X@un{\X@qu}\edef\Y@un{\Y@qu}\figptcopyDD-3:/-4/\figptcopyDD-4:/-5/}}
\ctr@ld@f\def\Psc@rveCPTD#1[#2]{\setc@ntr@l{2}%
    \def\list@num{#2}\extrairelepremi@r\Ak@\de\list@num%
    \extrairelepremi@r\Ai@\de\list@num\extrairelepremi@r\Aj@\de\list@num%
    \Figptpr@j-7:/\Ai@/%
    \f@gnewpath\PSwrit@cmd{-7}{\c@mmoveto}{\fwf@g}%
    \figvectPTD -9[\Ak@,\Aj@]%
    \@ecfor\Ak@:=\list@num\do{\figpttraTD-10:=\Ai@/#1,-9/%
       \figvectPTD -9[\Ai@,\Ak@]\figpttraTD-11:=\Aj@/-#1,-9/%
       \subB@zierTD\NBz@rcs[\Ai@,-10,-11,\Aj@]\edef\Ai@{\Aj@}\edef\Aj@{\Ak@}}}
\ctr@ld@f\def\figdrawend{\ifCUR@PS\ifGR@cri\immediate\closeout\fwf@g%
    \immediate\openout\fwf@g=\PSfilen@me\relax%
    \ifPDFm@ke\PSBdingB@x\else%
    \immediate\write\fwf@g{\p@urcent\string!PS-Adobe-2.0 EPSF-2.0}%
    \PShe@der{Creator\string: TeX (fig4tex.tex)}%
    \PShe@der{Title\string: \PSfilen@me}%
    \PShe@der{CreationDate\string: \the\day/\the\month/\the\year}%
    \PSBdingB@x%
    \PShe@der{EndComments}\PSdict@\fi%
    \immediate\write\fwf@g{\c@mgsave}%
    \openin\frf@g=\auxfilen@me\c@pypsfile\fwf@g\frf@g\closein\frf@g%
    \immediate\write\fwf@g{\c@mgrestore}%
    \PSc@mment{End of file.}\immediate\closeout\fwf@g%
    \immediate\openout\fwf@g=\auxfilen@me\immediate\closeout\fwf@g%
    \immediate\write16{File \PSfilen@me\space created.}\fi\fi\CUR@PSfalse\GR@critrue}
\ctr@ld@f\def\PShe@der#1{\immediate\write\fwf@g{\p@urcent\p@urcent#1}}
\ctr@ld@f\def\PSBdingB@x{{\v@lX=\ptT@ptps\c@@rdXmin\v@lY=\ptT@ptps\c@@rdYmin%
     \v@lXa=\ptT@ptps\c@@rdXmax\v@lYa=\ptT@ptps\c@@rdYmax%
     \PShe@der{BoundingBox\string: \repdecn@mb{\v@lX}\space\repdecn@mb{\v@lY}%
     \space\repdecn@mb{\v@lXa}\space\repdecn@mb{\v@lYa}}}}
\ctr@ld@f\def\figdrawfcconnect[#1]{{\ifCUR@PS\ifGR@cri\PSc@mment{fcconnect Points=#1}%
    \Q@s@tfillmode{no}\s@uvc@ntr@l\et@tpsfcconnect\resetc@ntr@l{2}%
    \fcc@nnect@[#1]\resetc@ntr@l\et@tpsfcconnect\PSc@mment{End fcconnect}\fi\fi}}
\ctr@ld@f\def\fcc@nnect@[#1]{\let\N@rm=\n@rmeucDD\def\list@num{#1}%
    \extrairelepremi@r\Ai@\de\list@num\edef\pr@m{\Ai@}\v@leur=\z@\p@rtent=\@ne\c@llgtot%
    \ifcase\fclin@typ@\edef\list@num{[\pr@m,#1,\Ai@}\expandafter\figdrawcurve\list@num]%
    \else\ifdim\fclin@r@d\p@>\z@\Pslin@conge[#1]\else\figdrawline[#1]\fi\fi%
    \v@leur=\@rrowp@s\v@leur\edef\list@num{#1,\Ai@,0}%
    \extrairelepremi@r\Ai@\de\list@num\mili@u=\epsil@n\c@llgpart%
    \advance\mili@u-\epsil@n\advance\mili@u-\delt@\advance\v@leur-\mili@u%
    \ifcase\fclin@typ@\invers@\mili@u\delt@%
    \ifnum\@rrowr@fpt>\z@\advance\delt@-\v@leur\v@leur=\delt@\fi%
    \v@leur=\repdecn@mb\v@leur\mili@u\edef\v@lt{\repdecn@mb\v@leur}%
    \extrairelepremi@r\Ak@\de\list@num%
    \figvectPDD-1[\pr@m,\Aj@]\figpttraDD-6:=\Ai@/\curv@roundness,-1/%
    \figvectPDD-1[\Ak@,\Ai@]\figpttraDD-7:=\Aj@/\curv@roundness,-1/%
    \delt@=\@rrowheadlength\p@\delt@=\C@AHANG\delt@\edef\R@dius{\repdecn@mb{\delt@}}%
    \ifcase\@rrowr@fpt%
    \FigptintercircB@zDD-8::\v@lt,\R@dius[\Ai@,-6,-7,\Aj@]\Q@arrowheadDD[-5,-8]\else%
    \FigptintercircB@zDD-8::\v@lt,\R@dius[\Aj@,-7,-6,\Ai@]\Q@arrowheadDD[-8,-5]\fi%
    \else\advance\delt@-\v@leur%
    \p@rtentiere{\p@rtent}{\delt@}\edef\C@efun{\the\p@rtent}%
    \p@rtentiere{\p@rtent}{\v@leur}\edef\C@efde{\the\p@rtent}%
    \figptbaryDD-5:[\Ai@,\Aj@;\C@efun,\C@efde]\ifcase\@rrowr@fpt%
    \delt@=\@rrowheadlength\unit@\delt@=\C@AHANG\delt@\edef\t@ille{\repdecn@mb{\delt@}}%
    \figvectPDD-2[\Ai@,\Aj@]\vecunit@{-2}{-2}\figpttraDD-5:=-5/\t@ille,-2/\fi%
    \Q@arrowheadDD[\Ai@,-5]\fi}
\ctr@ld@f\def\c@llgtot{\@ecfor\Aj@:=\list@num\do{\figvectP-1[\Ai@,\Aj@]\N@rm\delt@{-1}%
    \advance\v@leur\delt@\advance\p@rtent\@ne\edef\Ai@{\Aj@}}}
\ctr@ld@f\def\c@llgpart{\extrairelepremi@r\Aj@\de\list@num\figvectP-1[\Ai@,\Aj@]\N@rm\delt@{-1}%
    \advance\mili@u\delt@\ifdim\mili@u<\v@leur\edef\pr@m{\Ai@}\edef\Ai@{\Aj@}\c@llgpart\fi}
\ctr@ld@f\def\Pslin@conge[#1]{\ifnum\p@rtent>\tw@{\def\list@num{#1}%
    \extrairelepremi@r\Ai@\de\list@num\extrairelepremi@r\Aj@\de\list@num%
    \figptcopy-6:/\Ai@/\figvectP-3[\Ai@,\Aj@]\vecunit@{-3}{-3}\v@lmax=\result@t%
    \@ecfor\Ak@:=\list@num\do{\figvectP-4[\Aj@,\Ak@]\vecunit@{-4}{-4}%
    \minim@m\v@lmin\v@lmax\result@t\v@lmax=\result@t%
    \det@rm\delt@[-3,-4]\maxim@m\mili@u{\delt@}{-\delt@}\ifdim\mili@u>\Cepsil@n%
    \ifdim\delt@>\z@\figgetangleDD\Angl@[\Aj@,\Ak@,\Ai@]\else%
    \figgetangleDD\Angl@[\Aj@,\Ai@,\Ak@]\fi%
    \v@leur=\PI@deg\advance\v@leur-\Angl@\p@\divide\v@leur\tw@%
    \edef\Angl@{\repdecn@mb\v@leur}\c@ssin{\C@}{\S@}{\Angl@}\v@leur=\fclin@r@d\unit@%
    \v@leur=\S@\v@leur\mili@u=\C@\p@\invers@\mili@u\mili@u%
    \v@leur=\repdecn@mb{\mili@u}\v@leur%
    \minim@m\v@leur\v@leur\v@lmin\edef\t@ille{\repdecn@mb{\v@leur}}%
    \figpttra-5:=\Aj@/-\t@ille,-3/\figdrawline[-6,-5]\figpttra-6:=\Aj@/\t@ille,-4/%
    \figvectNVDD-3[-3]\figvectNVDD-8[-4]\inters@cDD-7:[-5,-3;-6,-8]%
    \ifdim\delt@>\z@\figdrawarccircP-7;\fclin@r@d[-5,-6]\else\figdrawarccircP-7;\fclin@r@d[-6,-5]\fi%
    \else\figdrawline[-6,\Aj@]\figptcopy-6:/\Aj@/\fi
    \edef\Ai@{\Aj@}\edef\Aj@{\Ak@}\figptcopy-3:/-4/}\figdrawline[-6,\Aj@]}\else\figdrawline[#1]\fi}
\ctr@ld@f\def\figdrawfcnode[#1]#2{{\ifCUR@PS\ifGR@cri\PSc@mment{fcnode Points=#1}%
    \s@uvc@ntr@l\et@tpsfcnode\resetc@ntr@l{2}%
    \def\t@xt@{#2}\ifx\t@xt@\empty\def\g@tt@xt{\setbox\Gb@x=\hbox{\Figg@tT{\p@int}}}%
    \else\def\g@tt@xt{\setbox\Gb@x=\hbox{#2}}\fi%
    \v@lmin=\h@rdfcXp@dd\advance\v@lmin\Xp@dd\unit@\multiply\v@lmin\tw@%
    \v@lmax=\h@rdfcYp@dd\advance\v@lmax\Yp@dd\unit@\multiply\v@lmax\tw@%
    \Figv@ctCreg-8(\unit@,-\unit@)\def\list@num{#1}%
    \delt@=\CUR@width bp\divide\delt@\tw@%
    \fcn@de\PSc@mment{End fcnode}\resetc@ntr@l\et@tpsfcnode\fi\fi}}
\ctr@ld@f\def\d@butn@de{\g@tt@xt\v@lX=\wd\Gb@x%
    \v@lY=\ht\Gb@x\advance\v@lY\dp\Gb@x\advance\v@lX\v@lmin\advance\v@lY\v@lmax}
\ctr@ld@f\def\fcn@deE{%
    \@ecfor\p@int:=\list@num\do{\d@butn@de\v@lX=\unssqrttw@\v@lX\v@lY=\unssqrttw@\v@lY%
    \ifdim\thickn@ss\p@>\z@
    \v@lXa=\v@lX\advance\v@lXa\delt@\v@lXa=\ptT@unit@\v@lXa\edef\XR@d{\repdecn@mb\v@lXa}%
    \v@lYa=\v@lY\advance\v@lYa\delt@\v@lYa=\ptT@unit@\v@lYa\edef\YR@d{\repdecn@mb\v@lYa}%
    \arct@n\v@leur(\v@lXa,\v@lYa)\v@leur=\rdT@deg\v@leur\edef\@nglde{\repdecn@mb\v@leur}%
    {\c@lptellDD-2::\p@int;\XR@d,\YR@d(\@nglde)}
    \advance\v@leur-\PI@deg\edef\@nglun{\repdecn@mb\v@leur}%
    {\c@lptellDD-3::\p@int;\XR@d,\YR@d(\@nglun)}%
    \figptstra-6=-3,-2,\p@int/\thickn@ss,-8/\Q@s@tfillmode{yes}%
    \Pss@tspecifSt{color=\DDV@thickcolor}%
    \figdrawline[-2,-3,-6,-5]\figdrawarcell-4;\XR@d,\YR@d(\@nglun,\@nglde,0)%
    \Psrest@reSt{color=\DDV@thickcolor}\fi
    \v@lX=\ptT@unit@\v@lX\v@lY=\ptT@unit@\v@lY%
    \edef\XR@d{\repdecn@mb\v@lX}\edef\YR@d{\repdecn@mb\v@lY}%
    \Q@s@tfillmode{yes}\Pss@tspecifSt{color=\fcbgc@lor}%
    \figdrawarcell\p@int;\XR@d,\YR@d(0,360,0)%
    \Q@s@tfillmode{no}\Psrest@reSt{color=\fcbgc@lor}\figdrawarcell\p@int;\XR@d,\YR@d(0,360,0)}}
\ctr@ld@f\def\fcn@deL{\delt@=\ptT@unit@\delt@\edef\t@ille{\repdecn@mb\delt@}%
    \@ecfor\p@int:=\list@num\do{\Figg@tXYa{\p@int}\d@butn@de%
    \ifdim\v@lX>\v@lY\itis@Ktrue\else\itis@Kfalse\fi%
    \advance\v@lXa-\v@lX\Figp@intreg-1:(\v@lXa,\v@lYa)%
    \advance\v@lXa\v@lX\advance\v@lYa-\v@lY\Figp@intreg-2:(\v@lXa,\v@lYa)%
    \advance\v@lXa\v@lX\advance\v@lYa\v@lY\Figp@intreg-3:(\v@lXa,\v@lYa)%
    \advance\v@lXa-\v@lX\advance\v@lYa\v@lY\Figp@intreg-4:(\v@lXa,\v@lYa)%
    \ifdim\thickn@ss\p@>\z@
    \Figg@tXYa{\p@int}\Q@s@tfillmode{yes}\Pss@tspecifSt{color=\DDV@thickcolor}%
    \c@lpt@xt{-1}{-4}\c@lpt@xt@\v@lXa\v@lYa\v@lX\v@lY\c@rre\delt@%
    \Figp@intregDD-9:(\v@lZ,\v@lYa)\Figp@intregDD-11:(\v@lZa,\v@lYa)%
    \c@lpt@xt{-4}{-3}\c@lpt@xt@\v@lYa\v@lXa\v@lY\v@lX\delt@\c@rre%
    \Figp@intregDD-12:(\v@lXa,\v@lZ)\Figp@intregDD-10:(\v@lXa,\v@lZa)%
    \ifitis@K\figptstra-7=-9,-10,-11/\thickn@ss,-8/\figdrawline[-9,-11,-5,-6,-7]\else%
    \figptstra-7=-10,-11,-12/\thickn@ss,-8/\figdrawline[-10,-12,-5,-6,-7]\fi%
    \Psrest@reSt{color=\DDV@thickcolor}\fi
    \Q@s@tfillmode{yes}\Pss@tspecifSt{color=\fcbgc@lor}\figdrawline[-1,-2,-3,-4]%
    \Q@s@tfillmode{no}\Psrest@reSt{color=\fcbgc@lor}\figdrawline[-1,-2,-3,-4,-1]}}
\ctr@ld@f\def\c@lpt@xt#1#2{\figvectN-7[#1,#2]\vecunit@{-7}{-7}\figpttra-5:=#1/\t@ille,-7/%
    \figvectP-7[#1,#2]\Figg@tXY{-7}\c@rre=\v@lX\delt@=\v@lY\Figg@tXY{-5}}
\ctr@ld@f\def\c@lpt@xt@#1#2#3#4#5#6{\v@lZ=#6\invers@{\v@lZ}{\v@lZ}\v@leur=\repdecn@mb{#5}\v@lZ%
    \v@lZ=#2\advance\v@lZ-#4\mili@u=\repdecn@mb{\v@leur}\v@lZ%
    \v@lZ=#3\advance\v@lZ\mili@u\v@lZa=-\v@lZ\advance\v@lZa\tw@#1}
\ctr@ld@f\def\fcn@deR{\@ecfor\p@int:=\list@num\do{\Figg@tXYa{\p@int}\d@butn@de%
    \advance\v@lXa-0.5\v@lX\advance\v@lYa-0.5\v@lY\Figp@intreg-1:(\v@lXa,\v@lYa)%
    \advance\v@lXa\v@lX\Figp@intreg-2:(\v@lXa,\v@lYa)%
    \advance\v@lYa\v@lY\Figp@intreg-3:(\v@lXa,\v@lYa)%
    \advance\v@lXa-\v@lX\Figp@intreg-4:(\v@lXa,\v@lYa)%
    \ifdim\thickn@ss\p@>\z@
    \Q@s@tfillmode{yes}\Pss@tspecifSt{color=\DDV@thickcolor}%
    \Figv@ctCreg-5(-\delt@,-\delt@)\figpttra-9:=-1/1,-5/%
    \Figv@ctCreg-5(\delt@,-\delt@)\figpttra-10:=-2/1,-5/%
    \Figv@ctCreg-5(\delt@,\delt@)\figpttra-11:=-3/1,-5/%
    \figptstra-7=-9,-10,-11/\thickn@ss,-8/\figdrawline[-9,-11,-5,-6,-7]%
    \Psrest@reSt{color=\DDV@thickcolor}\fi
    \Q@s@tfillmode{yes}\Pss@tspecifSt{color=\fcbgc@lor}\figdrawline[-1,-2,-3,-4]%
    \Q@s@tfillmode{no}\Psrest@reSt{color=\fcbgc@lor}\figdrawline[-1,-2,-3,-4,-1]}}
\ctr@ld@f\def\Pssetfl@wchart#1=#2|{\keln@mtr#1|%
    \def\n@mref{arr}\ifx\l@debut\n@mref\expandafter\keln@mtr\l@suite|%
     \def\n@mref{owp}\ifx\l@debut\n@mref\update@ttr\D@FTfcarrowposition\P@setfcarrowposition{#2}\else
     \def\n@mref{owr}\ifx\l@debut\n@mref\update@ttr\D@FTfcarrowrefpt\P@setfcarrowrefpt{#2}\else
     \W@rnmesAttr{figset flowchart}{#1}\fi\fi\else%
    \def\n@mref{bgc}\ifx\l@debut\n@mref\update@ttr\D@FTfcbgcolor\P@setfcbgcolor{#2}\else
    \def\n@mref{lin}\ifx\l@debut\n@mref\update@ttr\D@FTfcline\P@setfcline{#2}\else
    \def\n@mref{pad}\ifx\l@debut\n@mref\update@ttr\D@FTfcxpadding\P@setfcxpadding{#2}%
                                       \update@ttr\D@FTfcypadding\P@setfcypadding{#2}\else
    \def\n@mref{rad}\ifx\l@debut\n@mref\update@ttr\D@FTfcradius\P@setfcradius{#2}\else
    \def\n@mref{sha}\ifx\l@debut\n@mref\update@ttr\D@FTfcshape\P@setfcshape{#2}\else
    \def\n@mref{thi}\ifx\l@debut\n@mref\expandafter\keln@mtr\l@suite|%
     \def\n@mref{ckc}\ifx\l@debut\n@mref\update@ttr\D@FTref\P@setfcthickcolor{#2}\else
     \def\n@mref{ckn}\ifx\l@debut\n@mref\update@ttr\D@FTfcthickness\P@setfcthickness{#2}\else
     \W@rnmesAttr{figset flowchart}{#1}\fi\fi\else%
    \def\n@mref{xpa}\ifx\l@debut\n@mref\update@ttr\D@FTfcxpadding\P@setfcxpadding{#2}\else
    \def\n@mref{ypa}\ifx\l@debut\n@mref\update@ttr\D@FTfcypadding\P@setfcypadding{#2}\else
    \W@rnmesAttr{figset flowchart}{#1}\fi\fi\fi\fi\fi\fi\fi\fi\fi}
\ctr@ln@m\@rrowp@s
\ctr@ld@f\def\P@setfcarrowposition#1{\edef\@rrowp@s{#1}}
\ctr@ln@m\@rrowr@fpt
\ctr@ld@f\def\P@setfcarrowrefpt#1{\setfcr@fpt#1|}
\ctr@ld@f\def\setfcr@fpt#1#2|{\if#1e\def\@rrowr@fpt{1}\else\def\@rrowr@fpt{0}\fi}
\ctr@ln@m\fcbgc@lor
\ctr@ld@f\def\P@setfcbgcolor#1{\edef\fcbgc@lor{#1}}
\ctr@ln@m\fclin@typ@
\ctr@ld@f\def\P@setfcline#1{\setfccurv@#1|}
\ctr@ld@f\def\setfccurv@#1#2|{\if#1c\def\fclin@typ@{0}\else\def\fclin@typ@{1}\fi}
\ctr@ln@m\fclin@r@d
\ctr@ld@f\def\P@setfcradius#1{\edef\fclin@r@d{#1}}
\ctr@ln@m\fcn@de \ctr@ln@m\fcsh@pe
\ctr@ln@m\h@rdfcXp@dd \ctr@ln@m\h@rdfcYp@dd
\ctr@ld@f\def\P@setfcshape#1{\setfcshap@#1|}
\ctr@ld@f\def\setfcshap@#1#2|{%
    \if#1e\let\fcn@de=\fcn@deE\def\h@rdfcXp@dd{4pt}\def\h@rdfcYp@dd{4pt}%
     \edef\fcsh@pe{ellipse}\else%
    \if#1l\let\fcn@de=\fcn@deL\def\h@rdfcXp@dd{4pt}\def\h@rdfcYp@dd{4pt}%
     \edef\fcsh@pe{lozenge}\else%
          \let\fcn@de=\fcn@deR\def\h@rdfcXp@dd{6pt}\def\h@rdfcYp@dd{6pt}%
     \edef\fcsh@pe{rectangle}\fi\fi}
\ctr@ln@m\DDV@thickcolor
\ctr@ld@f\def\P@setfcthickcolor#1{\edef\DDV@thickcolor{#1}}
\ctr@ln@m\thickn@ss
\ctr@ld@f\def\P@setfcthickness#1{\edef\thickn@ss{#1}}
\ctr@ln@m\Xp@dd
\ctr@ld@f\def\P@setfcxpadding#1{\edef\Xp@dd{#1}}
\ctr@ln@m\Yp@dd
\ctr@ld@f\def\P@setfcypadding#1{\edef\Yp@dd{#1}}
\ctr@ld@f\def\figdrawline[#1]{{\ifCUR@PS\ifGR@cri\PSc@mment{line Points=#1}%
    \let\figdrawlign@=\Pslign@P\Pslin@{#1}\PSc@mment{End line}\fi\fi}}
\ctr@ld@f\def\figdrawlineF#1{{\ifCUR@PS\ifGR@cri\PSc@mment{lineF Filename=#1}%
    \let\figdrawlign@=\Pslign@F\Pslin@{#1}\PSc@mment{End lineF}\fi\fi}}
\ctr@ld@f\def\figdrawlineC(#1){{\ifCUR@PS\ifGR@cri\PSc@mment{lineC}%
    \let\figdrawlign@=\Pslign@C\Pslin@{#1}\PSc@mment{End lineC}\fi\fi}}
\ctr@ld@f\def\Pslin@#1{\iffillm@de\figdrawlign@{#1}%
    \f@gfill%
    \else\figdrawlign@{#1}\ifx\derp@int\premp@int%
    \f@gclosestroke%
    \else\f@gstroke\fi\fi}
\ctr@ld@f\def\Pslign@P#1{\def\list@num{#1}\extrairelepremi@r\p@int\de\list@num%
    \edef\premp@int{\p@int}\f@gnewpath%
    \PSwrit@cmd{\p@int}{\c@mmoveto}{\fwf@g}%
    \@ecfor\p@int:=\list@num\do{\PSwrit@cmd{\p@int}{\c@mlineto}{\fwf@g}%
    \edef\derp@int{\p@int}}}
\ctr@ld@f\def\Pslign@F#1{\s@uvc@ntr@l\et@tPslign@F\setc@ntr@l{2}\openin\frf@g=#1\relax%
    \ifeof\frf@g\message{*** File #1 not found !}\end\else%
    \read\frf@g to\tr@c\edef\premp@int{\tr@c}\expandafter\extr@ctCF\tr@c:%
    \f@gnewpath\PSwrit@cmd{-1}{\c@mmoveto}{\fwf@g}%
    \loop\read\frf@g to\tr@c\ifeof\frf@g\mored@tafalse\else\mored@tatrue\fi%
    \ifmored@ta\expandafter\extr@ctCF\tr@c:\PSwrit@cmd{-1}{\c@mlineto}{\fwf@g}%
    \edef\derp@int{\tr@c}\repeat\fi\closein\frf@g\resetc@ntr@l\et@tPslign@F}
\ctr@ln@m\extr@ctCF
\ctr@ld@f\def\extr@ctCFDD#1 #2:{\v@lX=#1\unit@\v@lY=#2\unit@\Figp@intregDD-1:(\v@lX,\v@lY)}
\ctr@ld@f\def\extr@ctCFTD#1 #2 #3:{\v@lX=#1\unit@\v@lY=#2\unit@\v@lZ=#3\unit@%
    \Figp@intregTD-1:(\v@lX,\v@lY,\v@lZ)}
\ctr@ld@f\def\Pslign@C#1{\s@uvc@ntr@l\et@tPslign@C\setc@ntr@l{2}%
    \def\list@num{#1}\extrairelepremi@r\p@int\de\list@num%
    \edef\premp@int{\p@int}\f@gnewpath%
    \expandafter\Pslign@C@\p@int:\PSwrit@cmd{-1}{\c@mmoveto}{\fwf@g}%
    \@ecfor\p@int:=\list@num\do{\expandafter\Pslign@C@\p@int:%
    \PSwrit@cmd{-1}{\c@mlineto}{\fwf@g}\edef\derp@int{\p@int}}%
    \resetc@ntr@l\et@tPslign@C}
\ctr@ld@f\def\Pslign@C@#1 #2:{{\def\t@xt@{#1}\ifx\t@xt@\empty\Pslign@C@#2:
    \else\extr@ctCF#1 #2:\fi}}
\ctr@ld@f\def\Pssetm@sh#1=#2|{\keln@mde#1|%
    \def\n@mref{co}\ifx\l@debut\n@mref\update@ttr\D@FTref\P@setmeshcolor{#2}\else
    \def\n@mref{da}\ifx\l@debut\n@mref\update@ttr\D@FTref\P@setmeshdash{#2}\else
    \def\n@mref{di}\ifx\l@debut\n@mref\update@ttr\D@FTmeshdiag\Q@s@tmeshdiag{#2}\else
    \def\n@mref{wi}\ifx\l@debut\n@mref\update@ttr\D@FTref\P@setmeshwidth{#2}\else
    \W@rnmesAttr{figset mesh}{#1}\fi\fi\fi\fi}
\ctr@ln@m\c@ntrolmesh
\ctr@ld@f\def\Q@s@tmeshdiag#1{\edef\c@ntrolmesh{#1}}
\ctr@ld@f\def\D@FTmeshdiag{0}    
\ctr@ln@m\DDV@meshcolor
\ctr@ld@f\def\P@setmeshcolor#1{\edef\DDV@meshcolor{#1}}
\ctr@ln@m\DDV@meshdash
\ctr@ld@f\def\P@setmeshdash#1{\edef\DDV@meshdash{#1}}
\ctr@ln@m\DDV@meshwidth
\ctr@ld@f\def\P@setmeshwidth#1{\edef\DDV@meshwidth{#1}}
\ctr@ld@f\def\figdrawmesh#1,#2[#3,#4,#5,#6]{{\ifCUR@PS\ifGR@cri%
    \PSc@mment{mesh N1=#1, N2=#2, Quadrangle=[#3,#4,#5,#6]}\s@uvc@ntr@l\et@tpsmesh%
    \Pss@tspecifSt{color=\DDV@meshcolor,dash=\DDV@meshdash,width=\DDV@meshwidth}%
    \setc@ntr@l{2}%
    \ifnum#1>\@ne\Psmeshp@rt#1[#3,#4,#5,#6]\fi%
    \ifnum#2>\@ne\Psmeshp@rt#2[#4,#5,#6,#3]\fi%
    \ifnum\c@ntrolmesh>\z@\Psmeshdi@g#1,#2[#3,#4,#5,#6]\fi%
    \ifnum\c@ntrolmesh<\z@\Psmeshdi@g#2,#1[#4,#5,#6,#3]\fi%
    \Psrest@reSt{color=\DDV@meshcolor,dash=\DDV@meshdash,width=\DDV@meshwidth}%
    \figdrawline[#3,#4,#5,#6,#3]\PSc@mment{End mesh}\resetc@ntr@l\et@tpsmesh\fi\fi}}
\ctr@ld@f\def\Psmeshp@rt#1[#2,#3,#4,#5]{{\l@mbd@un=\@ne\l@mbd@de=#1\loop%
    \ifnum\l@mbd@un<#1\advance\l@mbd@de\m@ne\figptbary-1:[#2,#3;\l@mbd@de,\l@mbd@un]%
    \figptbary-2:[#5,#4;\l@mbd@de,\l@mbd@un]\figdrawline[-1,-2]\advance\l@mbd@un\@ne\repeat}}
\ctr@ld@f\def\Psmeshdi@g#1,#2[#3,#4,#5,#6]{\figptcopy-2:/#3/\figptcopy-3:/#6/%
    \l@mbd@un=\z@\l@mbd@de=#1\loop\ifnum\l@mbd@un<#1%
    \advance\l@mbd@un\@ne\advance\l@mbd@de\m@ne\figptcopy-1:/-2/\figptcopy-4:/-3/%
    \figptbary-2:[#3,#4;\l@mbd@de,\l@mbd@un]%
    \figptbary-3:[#6,#5;\l@mbd@de,\l@mbd@un]\Psmeshdi@gp@rt#2[-1,-2,-3,-4]\repeat}
\ctr@ld@f\def\Psmeshdi@gp@rt#1[#2,#3,#4,#5]{{\l@mbd@un=\z@\l@mbd@de=#1\loop%
    \ifnum\l@mbd@un<#1\figptbary-5:[#2,#5;\l@mbd@de,\l@mbd@un]%
    \advance\l@mbd@de\m@ne\advance\l@mbd@un\@ne%
    \figptbary-6:[#3,#4;\l@mbd@de,\l@mbd@un]\figdrawline[-5,-6]\repeat}}
\ctr@ln@m\figdrawnormal
\ctr@ld@f\def\Q@normalDD#1,#2[#3,#4]{{\ifCUR@PS\ifGR@cri%
    \PSc@mment{normal Length=#1, Lambda=#2 [Pt1,Pt2]=[#3,#4]}%
    \s@uvc@ntr@l\et@tpsnormal\resetc@ntr@l{2}\figptendnormal-6::#1,#2[#3,#4]%
    \figptcopyDD-5:/-1/\figdrawarrow[-5,-6]%
    \PSc@mment{End normal}\resetc@ntr@l\et@tpsnormal\fi\fi}}
\ctr@ld@f\def\figreset#1{\trtlis@rg{#1}{\Psreset@}}
\ctr@ld@f\def\Psreset@#1|{\def\t@xt@{#1}\ifx\t@xt@\empty\P@resetg@n
    \else\keln@mde#1|%
    \def\n@mref{al}\ifx\l@debut\n@mref%
        \figset altitude(blcolor=default,bldash=default,blwidth=default,%
        sqcolor=default,sqdash=default,sqwidth=default)\else
    \def\n@mref{ar}\ifx\l@debut\n@mref\figresetarrowhead\else
    \def\n@mref{cu}\ifx\l@debut\n@mref\figset curve(roundness=\D@FTroundness)\else
    \def\n@mref{ge}\ifx\l@debut\n@mref\P@resetg@n\else
    \def\n@mref{fl}\ifx\l@debut\n@mref%
        \figset flowchart(arrowp=\D@FTfcarrowposition,arrowr=\D@FTfcarrowrefpt,%
	bgcolor=\D@FTfcbgcolor,line=\D@FTfcline,radius=\D@FTfcradius,%
	shape=\D@FTfcshape,thickcolor=default,thickness=\D@FTfcthickness,%
	xpadd=\D@FTfcxpadding,ypadd=\D@FTfcypadding)\else
    \def\n@mref{me}\ifx\l@debut\n@mref\figset mesh(diag=\D@FTmeshdiag,%
        color=default,dash=default,width=default)\else
    \def\n@mref{tr}\ifx\l@debut\n@mref%
        \figset trimesh(color=default,dash=default,width=default)\else
    \W@rnmeskwd{figreset}{#1}\fi\fi\fi\fi\fi\fi\fi\fi}
\ctr@ld@f\def\P@resetg@n{\figset (color=\D@FTcolor,dash=\D@FTdash,fill=\D@FTfill,%
    join=\D@FTjoin,width=\D@FTwidth)}
\ctr@ld@f\def\figset#1(#2){\def\t@xt@{#1}\ifx\t@xt@\empty\trtlis@rg{#2}{\Pssetg@n}
    \else\keln@mde#1|%
    \def\n@mref{al}\ifx\l@debut\n@mref\trtlis@rg{#2}{\Psset@lti}\else
    \def\n@mref{ar}\ifx\l@debut\n@mref\trtlis@rg{#2}{\Psset@rrowhe@d}\else
    \def\n@mref{cu}\ifx\l@debut\n@mref\trtlis@rg{#2}{\Pssetc@rve}\else
    \def\n@mref{fl}\ifx\l@debut\n@mref\trtlis@rg{#2}{\Pssetfl@wchart}\else
    \def\n@mref{ge}\ifx\l@debut\n@mref\trtlis@rg{#2}{\Pssetg@n}\else
    \def\n@mref{me}\ifx\l@debut\n@mref\trtlis@rg{#2}{\Pssetm@sh}\else
    \def\n@mref{pr}\ifx\l@debut\n@mref\ifCUR@PS\W@rnmesIgn{figset proj(...)}%
     \else\trtlis@rg{#2}{\Figsetpr@j}\fi\else
    \def\n@mref{tr}\ifx\l@debut\n@mref\trtlis@rg{#2}{\Pssettrim@sh}\else
    \def\n@mref{wr}\ifx\l@debut\n@mref\let\M@cro=\Figsetwr@te\trtlis@rgtok{#2,|}\else
    \W@rnmeskwd{figset}{#1}\fi\fi\fi\fi\fi\fi\fi\fi\fi\fi\ignorespaces}
\ctr@ld@f\def\figsetdefault#1(#2){\ifCUR@PS\W@rnmesIgn{figsetdefault}\else%
    \def\t@xt@{#1}\ifx\t@xt@\empty\trtlis@rg{#2}{\Pssd@g@n}\else\keln@mun#1|
    \def\n@mref{a}\ifx\l@debut\n@mref\trtlis@rg{#2}{\Pssd@@rrowhe@d}\else
    \def\n@mref{c}\ifx\l@debut\n@mref\trtlis@rg{#2}{\Pssd@c@rve}\else
    \def\n@mref{g}\ifx\l@debut\n@mref\trtlis@rg{#2}{\Pssd@g@n}\else
    \def\n@mref{f}\ifx\l@debut\n@mref\trtlis@rg{#2}{\Pssd@fl@wchart}\else
    \def\n@mref{m}\ifx\l@debut\n@mref\trtlis@rg{#2}{\Pssd@m@sh}\else
    \W@rnmeskwd{figsetdefault}{#1}\fi\fi\fi\fi\fi\fi\initpss@ttings\fi}
\ctr@ld@f\def\Pssd@g@n#1=#2|{\keln@mun#1|%
    \def\n@mref{c}\ifx\l@debut\n@mref\edef\D@FTcolor{#2}\else
    \def\n@mref{d}\ifx\l@debut\n@mref\edef\D@FTdash{#2}\else
    \def\n@mref{f}\ifx\l@debut\n@mref\edef\D@FTfill{#2}\else
    \def\n@mref{j}\ifx\l@debut\n@mref\edef\D@FTjoin{#2}\else
    \def\n@mref{u}\ifx\l@debut\n@mref\edef\D@FTupdate{#2}\Q@s@tupdate{#2}\else
    \def\n@mref{w}\ifx\l@debut\n@mref\edef\D@FTwidth{#2}\else
    \W@rnmesAttr{figsetdefault}{#1}\fi\fi\fi\fi\fi\fi}
\ctr@ld@f\def\Pssd@@rrowhe@d#1=#2|{\keln@mun#1|%
    \def\n@mref{a}\ifx\l@debut\n@mref\edef\D@FTarrowheadangle{#2}\else
    \def\n@mref{f}\ifx\l@debut\n@mref\edef\D@FTarrowheadfill{#2}\else
    \def\n@mref{l}\ifx\l@debut\n@mref\y@tiunit{#2}\ifunitpr@sent%
     \edef\D@FTh@rdahlength{#2}\else\edef\D@FTh@rdahlength{#2pt}%
     \message{*** \BS@ figsetdefault (..., #1=#2, ...) : unit is missing, pt is assumed.}%
     \fi\else
    \def\n@mref{o}\ifx\l@debut\n@mref\edef\D@FTarrowheadout{#2}\else
    \def\n@mref{r}\ifx\l@debut\n@mref\edef\D@FTarrowheadratio{#2}\else
    \W@rnmesAttr{figsetdefault arrowhead}{#1}\fi\fi\fi\fi\fi}
\ctr@ld@f\def\Pssd@c@rve#1=#2|{\keln@mun#1|%
    \def\n@mref{r}\ifx\l@debut\n@mref\edef\D@FTroundness{#2}\else%
    \W@rnmesAttr{figsetdefault curve}{#1}\fi}
\ctr@ld@f\def\Pssd@fl@wchart#1=#2|{\keln@mtr#1|%
    \def\n@mref{arr}\ifx\l@debut\n@mref\expandafter\keln@mtr\l@suite|%
     \def\n@mref{owp}\ifx\l@debut\n@mref\edef\D@FTfcarrowposition{#2}\else
     \def\n@mref{owr}\ifx\l@debut\n@mref\edef\D@FTfcarrowrefpt{#2}\else
                     \W@rnmesAttr{figsetdefault flowchart}{#1}\fi\fi\else%
    \def\n@mref{bgc}\ifx\l@debut\n@mref\edef\D@FTfcbgcolor{#2}\else
    \def\n@mref{lin}\ifx\l@debut\n@mref\edef\D@FTfcline{#2}\else
    \def\n@mref{pad}\ifx\l@debut\n@mref\edef\D@FTfcxpadding{#2}%
                    \edef\D@FTfcypadding{#2}\else
    \def\n@mref{rad}\ifx\l@debut\n@mref\edef\D@FTfcradius{#2}\else
    \def\n@mref{sha}\ifx\l@debut\n@mref\edef\D@FTfcshape{#2}\else
    \def\n@mref{thi}\ifx\l@debut\n@mref\expandafter\keln@mtr\l@suite|%
     \def\n@mref{ckn}\ifx\l@debut\n@mref\edef\D@FTfcthickness{#2}\else
                     \W@rnmesAttr{figsetdefault flowchart}{#1}\fi\else%
    \def\n@mref{xpa}\ifx\l@debut\n@mref\edef\D@FTfcxpadding{#2}\else
    \def\n@mref{ypa}\ifx\l@debut\n@mref\edef\D@FTfcypadding{#2}\else
    \W@rnmesAttr{figsetdefault flowchart}{#1}\fi\fi\fi\fi\fi\fi\fi\fi\fi}
\ctr@ld@f\def\D@FTfcarrowposition{0.5}
\ctr@ld@f\def\D@FTfcarrowrefpt{start}
\ctr@ld@f\def\D@FTfcbgcolor{1}
\ctr@ld@f\def\D@FTfcline{polygon}
\ctr@ld@f\def\D@FTfcradius{0}
\ctr@ld@f\def\D@FTfcshape{rectangle}
\ctr@ld@f\def\D@FTfcthickness{0}
\ctr@ld@f\def\D@FTfcxpadding{0}
\ctr@ld@f\def\D@FTfcypadding{0}
\ctr@ld@f\def\Pssd@m@sh#1=#2|{\keln@mun#1|%
    \def\n@mref{d}\ifx\l@debut\n@mref\edef\D@FTmeshdiag{#2}\else%
    \W@rnmesAttr{figsetdefault mesh}{#1}\fi}
\ctr@ln@w{newif}\iffillm@de
\ctr@ld@f\def\Q@s@tfillmode#1{\expandafter\setfillm@de#1:}
\ctr@ld@f\def\setfillm@de#1#2:{\if#1n\fillm@defalse\else\fillm@detrue\fi}
\ctr@ld@f\def\D@FTfill{no}     
\ctr@ln@w{newif}\ifGRupdatem@de
\ctr@ld@f\def\Q@s@tupdate#1{\ifCUR@PS\W@rnmesIgn{figset (update=...)}%
    \else\expandafter\setupd@te#1:\fi}
\ctr@ld@f\def\setupd@te#1#2:{\if#1n\GRupdatem@defalse\else\GRupdatem@detrue\fi}
\ctr@ld@f\def\D@FTupdate{no}     
\ctr@ln@m\CUR@color \ctr@ln@m\CUR@colorc@md
\ctr@ld@f\def\s@uvcolor#1{\edef#1{\CUR@color}}
\ctr@ld@f\def\D@FTcolor{0}       
\ctr@ld@f\def\Pssetc@lor#1{\ifGR@cri\result@tent=\@ne\expandafter\c@lnbV@l#1 :%
    \def\CUR@color{}\def\CUR@colorc@md{}%
    \ifcase\result@tent\or\Q@s@tgray{#1}\or\or\Q@s@trgb{#1}\or\Q@s@tcmyk{#1}\fi\fi}
\ctr@ln@m\CUR@colorc@mdStroke
\ctr@ld@f\def\Q@s@tcmyk#1{\ifGR@cri\def\CUR@color{#1}\def\CUR@colorc@md{\c@msetcmykcolor}%
    \def\CUR@colorc@mdStroke{\c@msetcmykcolorStroke}%
    \ifCUR@PS\PSc@mment{setcmyk Color=#1}\us@primarC@lor\fi\fi}
\ctr@ld@f\def\Q@s@trgb#1{\ifGR@cri\def\CUR@color{#1}\def\CUR@colorc@md{\c@msetrgbcolor}%
    \def\CUR@colorc@mdStroke{\c@msetrgbcolorStroke}%
    \ifCUR@PS\PSc@mment{setrgb Color=#1}\us@primarC@lor\fi\fi}
\ctr@ld@f\def\Q@s@tgray#1{\ifGR@cri\def\CUR@color{#1}\def\CUR@colorc@md{\c@msetgray}%
    \def\CUR@colorc@mdStroke{\c@msetgrayStroke}%
    \ifCUR@PS\PSc@mment{setgray Gray level=#1}\us@primarC@lor\fi\fi}
\ctr@ln@m\fillc@md
\ctr@ld@f\def\us@primarC@lor{\immediate\write\fwf@g{\d@fprimarC@lor}%
    \let\fillc@md=\prfillc@md}
\ctr@ld@f\def\prfillc@md{\d@fprimarC@lor\space\c@mfill}
\ctr@ld@f\def\c@lnbV@l#1 #2:{\def\t@xt@{#1}\relax\ifx\t@xt@\empty\c@lnbV@l#2:
    \else\c@lnbV@l@#1 #2:\fi}
\ctr@ld@f\def\c@lnbV@l@#1 #2:{\def\t@xt@{#2}\ifx\t@xt@\empty%
    \def\t@xt@{#1}\ifx\t@xt@\empty\advance\result@tent\m@ne\fi
    \else\advance\result@tent\@ne\c@lnbV@l@#2:\fi}
\ctr@ld@f\def\Blackcmyk{0 0 0 1}
\ctr@ld@f\def\Whitecmyk{0 0 0 0}
\ctr@ld@f\def\Cyancmyk{1 0 0 0}
\ctr@ld@f\def\Magentacmyk{0 1 0 0}
\ctr@ld@f\def\Yellowcmyk{0 0 1 0}
\ctr@ld@f\def\Redcmyk{0 1 1 0}
\ctr@ld@f\def\Greencmyk{1 0 1 0}
\ctr@ld@f\def\Bluecmyk{1 1 0 0}
\ctr@ld@f\def\Graycmyk{0 0 0 0.50}
\ctr@ld@f\def\BrickRedcmyk{0 0.89 0.94 0.28} 
\ctr@ld@f\def\Browncmyk{0 0.81 1 0.60} 
\ctr@ld@f\def\ForestGreencmyk{0.91 0 0.88 0.12} 
\ctr@ld@f\def\Goldenrodcmyk{ 0 0.10 0.84 0} 
\ctr@ld@f\def\Marooncmyk{0 0.87 0.68 0.32} 
\ctr@ld@f\def\Orangecmyk{0 0.61 0.87 0} 
\ctr@ld@f\def\Purplecmyk{0.45 0.86 0 0} 
\ctr@ld@f\def\RoyalBluecmyk{1. 0.50 0 0} 
\ctr@ld@f\def\Violetcmyk{0.79 0.88 0 0} 
\ctr@ld@f\def\Blackrgb{0 0 0}
\ctr@ld@f\def\Whitergb{1 1 1}
\ctr@ld@f\def\Redrgb{1 0 0}
\ctr@ld@f\def\Greenrgb{0 1 0}
\ctr@ld@f\def\Bluergb{0 0 1}
\ctr@ld@f\def\Cyanrgb{0 1 1}
\ctr@ld@f\def\Magentargb{1 0 1}
\ctr@ld@f\def\Yellowrgb{1 1 0}
\ctr@ld@f\def\Grayrgb{0.5 0.5 0.5}
\ctr@ld@f\def\Chocolatergb{0.824 0.412 0.118}
\ctr@ld@f\def\DarkGoldenrodrgb{0.722 0.525 0.043}
\ctr@ld@f\def\DarkOrangergb{1 0.549 0}
\ctr@ld@f\def\Firebrickrgb{0.698 0.133 0.133}
\ctr@ld@f\def\ForestGreenrgb{0.133 0.545 0.133}
\ctr@ld@f\def\Goldrgb{1 0.843 0}
\ctr@ld@f\def\HotPinkrgb{1 0.412 0.706}
\ctr@ld@f\def\Maroonrgb{0.690 0.188 0.376}
\ctr@ld@f\def\Pinkrgb{1 0.753 0.796}
\ctr@ld@f\def\RoyalBluergb{0.255 0.412 0.882}
\ctr@ld@f\def\Pssetg@n#1=#2|{\keln@mun#1|%
    \def\n@mref{c}\ifx\l@debut\n@mref\update@ttr\D@FTcolor\Pssetc@lor{#2}\else
    \def\n@mref{d}\ifx\l@debut\n@mref\update@ttr\D@FTdash\Q@s@tdash{#2}\else
    \def\n@mref{f}\ifx\l@debut\n@mref\update@ttr\D@FTfill\Q@s@tfillmode{#2}\else
    \def\n@mref{j}\ifx\l@debut\n@mref\update@ttr\D@FTjoin\Q@s@tjoin{#2}\else
    \def\n@mref{u}\ifx\l@debut\n@mref\update@ttr\D@FTupdate\Q@s@tupdate{#2}\else
    \def\n@mref{w}\ifx\l@debut\n@mref\update@ttr\D@FTwidth\Q@s@twidth{#2}\else
    \W@rnmesAttr{figset}{#1}\fi\fi\fi\fi\fi\fi}
\ctr@ln@m\CUR@dash
\ctr@ld@f\def\s@uvdash#1{\edef#1{\CUR@dash}}
\ctr@ld@f\def\D@FTdash{1}        
\ctr@ld@f\def\Q@s@tdash#1{\ifGR@cri\edef\CUR@dash{#1}\ifCUR@PS\expandafter\Pssetd@sh#1 :\fi\fi}
\ctr@ld@f\def\Pssetd@shI#1{\PSc@mment{setdash Index=#1}\ifcase#1%
    \or\immediate\write\fwf@g{[] 0 \c@msetdash}
    \or\immediate\write\fwf@g{[6 2] 0 \c@msetdash}
    \or\immediate\write\fwf@g{[4 2] 0 \c@msetdash}
    \or\immediate\write\fwf@g{[2 2] 0 \c@msetdash}
    \or\immediate\write\fwf@g{[1 2] 0 \c@msetdash}
    \or\immediate\write\fwf@g{[2 4] 0 \c@msetdash}
    \or\immediate\write\fwf@g{[3 5] 0 \c@msetdash}
    \or\immediate\write\fwf@g{[3 3] 0 \c@msetdash}
    \or\immediate\write\fwf@g{[3 5 1 5] 0 \c@msetdash}
    \or\immediate\write\fwf@g{[6 4 2 4] 0 \c@msetdash}
    \fi}
\ctr@ld@f\def\Pssetd@sh#1 #2:{{\def\t@xt@{#1}\ifx\t@xt@\empty\Pssetd@sh#2:
    \else\def\t@xt@{#2}\ifx\t@xt@\empty\Pssetd@shI{#1}\else\s@mme=\@ne\def\debutp@t{#1}%
    \an@lysd@sh#2:\ifodd\s@mme\edef\debutp@t{\debutp@t\space\finp@t}\def\finp@t{0}\fi%
    \PSc@mment{setdash Pattern=#1 #2}%
    \immediate\write\fwf@g{[\debutp@t] \finp@t\space\c@msetdash}\fi\fi}}
\ctr@ld@f\def\an@lysd@sh#1 #2:{\def\t@xt@{#2}\ifx\t@xt@\empty\def\finp@t{#1}\else%
    \edef\debutp@t{\debutp@t\space#1}\advance\s@mme\@ne\an@lysd@sh#2:\fi}
\ctr@ln@m\CUR@width
\ctr@ld@f\def\s@uvwidth#1{\edef#1{\CUR@width}}
\ctr@ld@f\def\D@FTwidth{0.4}     
\ctr@ld@f\def\Q@s@twidth#1{\ifGR@cri\edef\CUR@width{#1}\ifCUR@PS%
    \PSc@mment{setwidth Width=#1}\immediate\write\fwf@g{#1 \c@msetlinewidth}\fi\fi}
\ctr@ln@m\CUR@join
\ctr@ld@f\def\s@uvjoin#1{\edef#1{\CUR@join}}
\ctr@ld@f\def\D@FTjoin{miter}   
\ctr@ld@f\def\Q@s@tjoin#1{\ifGR@cri\edef\CUR@join{#1}\ifCUR@PS\expandafter\Pssetj@in#1:\fi\fi}
\ctr@ld@f\def\Pssetj@in#1#2:{\PSc@mment{setjoin join=#1}%
    \if#1r\def\t@xt@{1}\else\if#1b\def\t@xt@{2}\else\def\t@xt@{0}\fi\fi%
    \immediate\write\fwf@g{\t@xt@\space\c@msetlinejoin}}
\ctr@ld@f\def\Pss@tspecifSt#1{\trtlis@rg{#1}{\Pss@tspecifSt@}}
\ctr@ld@f\def\Pss@tspecifSt@#1=#2|{\keln@mun#1|%
    \def\n@mref{c}\ifx\l@debut\n@mref\def\n@mref{#2}\ifx\n@mref\D@FTref\else%
     \s@uvcolor{\typ@color}\Pssetc@lor{#2}\fi\else
    \def\n@mref{d}\ifx\l@debut\n@mref\def\n@mref{#2}\ifx\n@mref\D@FTref\else%
     \s@uvdash{\typ@dash}\Q@s@tdash{#2}\fi\else
    \def\n@mref{j}\ifx\l@debut\n@mref\def\n@mref{#2}\ifx\n@mref\D@FTref\else%
     \s@uvjoin{\typ@join}\Q@s@tjoin{#2}\fi\else
    \def\n@mref{w}\ifx\l@debut\n@mref\def\n@mref{#2}\ifx\n@mref\D@FTref\else%
     \s@uvwidth{\typ@width}\Q@s@twidth{#2}\fi\else
    \W@rnmeskwd{Pss@tspecifSt}{#1}\fi\fi\fi\fi}
\ctr@ld@f\def\Psrest@reSt#1{\trtlis@rg{#1}{\Psrest@reSt@}}
\ctr@ld@f\def\Psrest@reSt@#1=#2|{\keln@mun#1|%
    \def\n@mref{c}\ifx\l@debut\n@mref\def\n@mref{#2}\ifx\n@mref\D@FTref\else%
     \Pssetc@lor{\typ@color}\fi\else
    \def\n@mref{d}\ifx\l@debut\n@mref\def\n@mref{#2}\ifx\n@mref\D@FTref\else%
     \Q@s@tdash{\typ@dash}\fi\else
    \def\n@mref{j}\ifx\l@debut\n@mref\def\n@mref{#2}\ifx\n@mref\D@FTref\else%
     \Q@s@tjoin{\typ@join}\fi\else
    \def\n@mref{w}\ifx\l@debut\n@mref\def\n@mref{#2}\ifx\n@mref\D@FTref\else%
     \Q@s@twidth{\typ@width}\fi\else
    \W@rnmeskwd{Psrest@reSt}{#1}\fi\fi\fi\fi}
\ctr@ld@f\def\Pssettrim@sh#1=#2|{\keln@mde#1|%
    \def\n@mref{co}\ifx\l@debut\n@mref\update@ttr\D@FTref\P@settmeshcolor{#2}\else
    \def\n@mref{da}\ifx\l@debut\n@mref\update@ttr\D@FTref\P@settmeshdash{#2}\else
    \def\n@mref{wi}\ifx\l@debut\n@mref\update@ttr\D@FTref\P@settmeshwidth{#2}\else
    \W@rnmesAttr{figset trimesh}{#1}\fi\fi\fi}
\ctr@ln@m\DDV@tmeshcolor
\ctr@ld@f\def\P@settmeshcolor#1{\edef\DDV@tmeshcolor{#1}}
\ctr@ln@m\DDV@tmeshdash
\ctr@ld@f\def\P@settmeshdash#1{\edef\DDV@tmeshdash{#1}}
\ctr@ln@m\DDV@tmeshwidth
\ctr@ld@f\def\P@settmeshwidth#1{\edef\DDV@tmeshwidth{#1}}
\ctr@ld@f\def\figdrawtrimesh#1[#2,#3,#4]{{\ifCUR@PS\ifGR@cri%
    \PSc@mment{trimesh Type=#1, Triangle=[#2,#3,#4]}%
    \s@uvc@ntr@l\et@tpstrimesh\ifnum#1>\@ne%
    \Pss@tspecifSt{color=\DDV@tmeshcolor,dash=\DDV@tmeshdash,width=\DDV@tmeshwidth}%
    \setc@ntr@l{2}%
    \Pstrimeshp@rt#1[#2,#3,#4]\Pstrimeshp@rt#1[#3,#4,#2]\Pstrimeshp@rt#1[#4,#2,#3]%
    \Psrest@reSt{color=\DDV@tmeshcolor,dash=\DDV@tmeshdash,width=\DDV@tmeshwidth}%
    \fi\figdrawline[#2,#3,#4,#2]%
    \PSc@mment{End trimesh}\resetc@ntr@l\et@tpstrimesh\fi\fi}}
\ctr@ld@f\def\Pstrimeshp@rt#1[#2,#3,#4]{{\l@mbd@un=\@ne\l@mbd@de=#1\loop\ifnum\l@mbd@de>\@ne%
    \advance\l@mbd@de\m@ne\figptbary-1:[#2,#3;\l@mbd@de,\l@mbd@un]%
    \figptbary-2:[#2,#4;\l@mbd@de,\l@mbd@un]\figdrawline[-1,-2]%
    \advance\l@mbd@un\@ne\repeat}}
\initpr@lim\initpss@ttings\initPDF@rDVI
\ctr@ln@w{newbox}\figBoxA
\ctr@ln@w{newbox}\figBoxB
\ctr@ln@w{newbox}\figBoxC
\catcode`\@=12